\documentclass[12pt]{amsart}
\pdfoutput=1

\usepackage[lmargin=1in,rmargin=1in,tmargin=1in,bmargin=1in]{geometry}
\usepackage{amsmath}
\usepackage{amsfonts}
\usepackage{amsthm}
\usepackage{amssymb}
\usepackage{graphicx}
\usepackage{tabularx}
\usepackage{overpic}
\usepackage{tikz}
\usetikzlibrary{arrows}
\usepackage{rotating}
\usepackage{hyperref}
\usepackage{framed}
\usepackage{xcolor}
\usepackage{placeins}
\usepackage{booktabs}
\usepackage{pinlabel}
 \usepackage{picins}
\usepackage{color}
\usepackage{multirow}

\def\co{\colon\thinspace}
\newcommand{\bsmat}{\left(\begin{smallmatrix}}
\newcommand{\esmat}{\end{smallmatrix}\right)}

\newcommand{\into}{\hookrightarrow}

\newcommand{\Ts}{\widehat{T}_{\operatorname{st}}}
\newcommand{\Td}{\widehat{T}_{\operatorname{du}}}
\newcommand{\stand}{{\operatorname{st}}}
\newcommand{\dual}{{\operatorname{du}}}

\newcommand{\subL}{l}

\newcommand{\w}{ {\bf w} }
\newcommand{\s}{ {\bf s} }

\newcommand{\R}{ \mathbb{R} }
\newcommand{\Q}{ \mathbb{Q} }

\newcommand{\F}{ \mathbb{F} }
\newcommand{\Z}{ \mathbb{Z} }
\newcommand{\Ztwo}{ \mathbb{Z}/2\mathbb{Z} }

\newcommand{\Alg}{ \mathcal{A} }

\newcommand{\spinc}{ \mathfrak{s} }
\newcommand{\Ainfty}{ \mathcal{A}_\infty }

\newcommand{\CFDA}{ \widehat{\mathit{CFDA}} }

\newcommand{\CFAA}{ \widehat{\mathit{CFAA}} }
\newcommand{\CFDD}{ \widehat{\mathit{CFDD}} }
\newcommand{\CFAAid}{ \CFAA(\mathbb{I}) }
\newcommand{\CFDDid}{ \CFDD(\mathbb{I}) }

\newcommand{\HFhat}{ \widehat{\mathit{HF}} }
\newcommand{\CFhat}{ \widehat{\mathit{CF}} }

\newcommand{\pants}{ \mathcal{P} }

\DeclareMathOperator{\id}{id}

\DeclareMathOperator{\gr}{gr}

\newcommand{\CFD}{ \widehat{\mathit{CFD}} }
\newcommand{\CFA}{ \widehat{\mathit{CFA}} }
\newcommand{\CFDAA}{\widehat{\mathit{CFDAA}}}

\newcommand{\tw}{\textsc{t}}  
\newcommand{\du}{\textsc{h}} 
\newcommand{\twi}{\textsc{t}^{-1}}  
\newcommand{\dui}{\textsc{h}^{-1}} 
\newcommand{\ex}{\textsc{e}}  %
\newcommand{\me}{\textsc{m}}  %

\newcommand{\lp}{\boldsymbol\ell}

\newcommand{\twist}{\mathcal{T}}
\newcommand{\extend}{\mathcal{E}}
\newcommand{\merge}{\mathcal{M}}

\newcommand{\proplambda}{{\bf property} $\boldsymbol{\lambda}$}

\newcommand{\proplambdaDUAL}{{\bf property} $\boldsymbol{\lambda}^*$}

\newcommand{\frkA}{\mathfrak{A}}

\newtheorem{thm}{Theorem}[section]
\newtheorem{prop}[thm]{Proposition}
\newtheorem{lem}[thm]{Lemma}
\newtheorem{cor}[thm]{Corollary}

\newtheorem{conj}[thm]{Conjecture}

\newtheorem{obs}[thm]{Observation}

\theoremstyle{definition}
\newtheorem{definition}[thm]{Definition}
\newtheorem{example}[thm]{Example}

\theoremstyle{remark}
\newtheorem{remark}[thm]{Remark}

\title{A calculus for bordered Floer homology}

\author[Jonathan Hanselman]{Jonathan Hanselman}
\thanks{The first author was partially supported by NSF RTG grant DMS-1148490}
\address {Department of Mathematics, University of Texas at Austin.\newline {\it E-mail address:} {\tt hanselman@math.utexas.edu}}

\author[Liam Watson]{Liam Watson}
\thanks{The second author was partially supported by a Marie Curie Career Integration Grant (HFFUNDGRP)}
\address {School of Mathematics and Statistics, University of Glasgow.\newline {\it E-mail address:} {\tt liam.watson@glasgow.ac.uk}}

\begin{document}

\begin{abstract}
We consider a class of manifolds with torus boundary admitting bordered Heegaard Floer homology of a particularly simple form, namely, the type D structure may be described graphically by a disjoint union of loops. We develop a calculus for studying bordered invariants of this form and, in particular, provide a complete description of slopes giving rise to L-space Dehn fillings as well as necessary and sufficient conditions for L-spaces resulting from identifying two such manifolds along their boundaries. As an application, we show that Seifert fibered spaces with torus boundary fall into this class, leading to a proof that, among graph manifolds containing a single JSJ torus, the property of being an L-space is equivalent to non-left-orderability of the fundamental group and to the non-existence of a coorientable taut foliation. 
\end{abstract}

\maketitle

\section{Introduction}\label{sec:intro}

This paper is concerned with developing Heegaard Floer theory with a view to better understanding the relationship between coorientable taut foliations, left-orderable fundamental groups, and manifolds that do not have {\em simple} Heegaard Floer homology. Recall that manifolds with simplest possible Heegaard Floer homology, called L-spaces, are rational homology spheres $Y$ for which $\dim\HFhat(Y)=|H_1(Y;\Z)|$ (all Heegaard Floer-theoretic objects in this work will take coefficients in $\Z/2\Z$). On the other hand, a group $G$ is left-orderable if there exists a non-empty set $P\subset G$, called a positive cone, that is a closed sub-semigroup of $G$ and gives a partition of the group in the sense that $G=P\amalg\{1\}\amalg P^{-1}$. This auxiliary structure is equivalent to $G$ admitting an effective action on $\R$ by order-preserving homeomorphisms. The work of Boyer, Rolfsen and Wiest is a good introduction to left-orderable groups in the context of three-manifold topology, including the interaction with taut foliations \cite{BRW2005}; for closed, orientable, irreducible, three-manifolds it is conjectured that being an L-space is equivalent to having a non-left-orderable fundamental group \cite{BGW2013}. This conjecture holds for Seifert fibered spaces and, as a natural extension of this case, graph manifolds are a key family of interest; see \cite{BB,BC,CLW2013,Hanselman2013,Hanselman2014,Mauricio}. Towards establishing the conjecture for graph manifolds, we prove:

\begin{thm}\label{thm:main}Suppose that $Y$ is a graph manifold with a single JSJ torus, that is, $Y$ is constructed by identifying two Seifert fibered manifolds with torus boundary along their boundaries. Then the following are equivalent:
\begin{enumerate}
\item[(i)] $Y$ is an L-space;
\item[(ii)] $\pi_1(Y)$ is not left-orderable;
\item[(iii)] $Y$ does not admit a coorientable taut foliation. 
\end{enumerate}
\end{thm}
The equivalence between (ii) and (iii) is due to Boyer and Clay \cite{BC}; the focus of this paper is understanding the behaviour of Heegaard Floer homology in this setting. To do this, we make use of bordered Heegaard Floer homology, a variant of Heegaard Floer homology adapted to cut-and-paste arguments. Briefly, this theory assigns a differential graded module over a particular algebra to each manifold with torus boundary. A chain complex for the Heegaard Floer homology of the associated closed manifold is obtained from a pairing theorem due to Lipshitz, Ozsv\'ath and Thurston \cite{LOT-bordered}.
  
Our approach to this problem is to work in a more general setting. We consider a particular class of differential graded modules which we call loop-type (Definition \ref{def:loop}), and introduce a calculus for studying loops; the bulk of this paper is devoted to developing this calculus in detail. Given a three-manifold $M$ with torus boundary, $M$ is called loop-type if its associated bordered invariants are loop-type up to homotopy (Definition \ref{def:loop-type}). Recall that a slope in $\partial M$ is the isotopy class of an essential simple closed curve in $\partial M$, and denote by $M(\gamma)$ the closed three manifold resulting from Dehn filling along a slope $\gamma$. The set of slopes may be (non-canonically) identified with the extended rationals $\Q\cup\{\frac{1}{0}\}$.  Of central interest  is the subset $\mathcal{L}_M$ consisting of those slopes giving rise to L-spaces after Dehn filling. We prove:

\begin{thm}[Detection]\label{thm:detection}
Suppose that $M$ is a loop-type rational homology torus. Then there is a complete, combinatorial description of the set $\mathcal{L}_M$ in terms of loop calculus. In particular, $\mathcal{L}_M$ may be identified with the restriction to $\Q\cup\{\frac{1}{0}\}$ of a connected interval in $\R\cup\{\infty\}$. Moreover, if $M$ is simple loop-type, this interval has rational endpoints.  
\end{thm}

We remark that this is the expected behaviour from the foliations/orderability vantage point, at least for graph manifolds. It is interesting that the analogous behaviour on the Heegaard Floer side appears to be intrinsic to the algebraic structures that arise. Namely, this is a statement that makes sense for loop-type bordered invariants, without reference to any three-manifold. The introduction of the technical notion of a {\em simple loop} (Definition \ref{def:simple}) also allows us to state and prove a gluing theorem. Let $\mathcal{L}_M^\circ$ denote the interior of $\mathcal{L}_M$. 

\begin{thm}[Gluing]\label{thm:gluing} Suppose that $M_1$ and $M_2$ are simple loop-type rational homology tori and neither is solid torus-like. Then given a homeomorphism $h\co \partial M_1\to\partial M_2$, the closed manifold $M_1\cup_h M_2$ is an L-space if and only if, for every slope $\gamma$ in $\partial M_1$, either $\gamma\in\mathcal{L}_{M_1}^\circ$ or $h(\gamma)\in\mathcal{L}_{M_2}^\circ$.
\end{thm}

Note that Theorem \ref{thm:gluing} does not hold if either $M_1$ or $M_2$ is a solid torus. If $M_1$ is a solid torus with meridian $m=\partial D^2\times\{\text{pt}\}$ then, according to the definition of an L-space slope, $M_1 \cup_h M_2$ is an L-space if and only if $h(m)\in\mathcal{L}_{M_2}$; equivalently, the statement of Theorem \ref{thm:gluing} holds if we use $\mathcal{L}_{M_i}$ in place of $\mathcal{L}_{M_i}^\circ$. When both $M_1$ and $M_2$ are solid tori this simply amounts to the construction of lens spaces interpreted in our notation. More generally, however, there are bordered invariants that arise in the loop setting that behave just like solid tori with respect to gluing. We will need to deal with these explicitly; this amounts to defining a class of manifolds and loops which are referred to as solid torus-like (Definition \ref{def:solid torus-like}).

Both Theorem \ref{thm:detection} and Theorem \ref{thm:gluing} follow from working with loops in the abstract. Towards the proof of Theorem \ref{thm:main}, and in the interest of establishing an existence result, a key class of loop-type manifolds is provided by Seifert fibered spaces. 

\begin{thm}\label{thm:Seifert-simple}Suppose $M$ is a rational homology solid torus admitting a Seifert fibered structure. Then $M$ has simple loop-type bordered Heegaard Floer homology. \end{thm}

Let $N$ denote the twisted $I$-bundle over the Klein bottle, and recall that the rational longitude $\lambda$ for this manifold with torus boundary may be identified with a fiber in a Seifert structure over the M\"obius band. As a Seifert fibered rational homology solid torus, $N$ is a simple loop-type manifold; compare \cite{BGW2013}. The twisted $I$-bundle over the Klein bottle allows for an alternative detection statement for L-space slopes.  

\begin{thm}[Detection via the twisted $I$-bundle over the Klein bottle]\label{thm:detection-via-TIBOKB}
Suppose that $M$ is a loop-type rational homology torus. Then $\gamma\in\mathcal{L}_M^\circ$ if and only if $N\cup_h M$ is an L-space where $h(\lambda)=\gamma$.  \end{thm}

This answers a question of Boyer and Clay in the case of connected boundary; compare \cite[Question 1.8]{BC}. In particular, our notion of detection aligns precisely with the characterization given by Boyer and Clay \cite{BC}. Indeed, we prove that, more generally, the twisted $I$-bundle over the Klein bottle used in Theorem \ref{thm:detection-via-TIBOKB} may be replaced with any simple loop-type manifold for which every non-longitudinal filling is an L-space; see Theorem \ref{thm:characterize}. There are many examples of these provided by {\em Heegaard Floer homology solid tori}; for more on this class of manifolds see \cite{Watson}. 

Note that the interior of $\mathcal{L}_M$, denoted $\mathcal{L}_M^\circ$, is the set of strict L-space slopes.  The complement of $\mathcal{L}_M^\circ$, according to Theorem \ref{thm:detection-via-TIBOKB}, corresponds to the set of {\it non-L-space (NLS) detected slopes} in the the sense of Boyer and Clay \cite[Definition 7.16]{BC}. 

According to Theorem \ref{thm:Seifert-simple}, the exterior of any torus knot in the three-sphere gives an example of a simple loop-type manifold (indeed, this follows from work of Lipshitz, Ozsv\'ath and Thurston \cite{LOT-bordered}). More generally, if $K$ is a knot in the three-sphere admitting an L-space surgery (an L-space knot), then  $S^3\smallsetminus \nu(K)$ is a simple loop-type manifold. In this setting it is well known that $\mathcal{L}_M$ is (the restriction to the rationals of) $[2g-1,\infty]$ where $g$ is the Seifert genus of $K$ whenever $K$ admits a positive L-space surgery. Specializing Theorem \ref{thm:gluing} to this setting, we have:

\begin{thm}\label{thm:gen-splice}
Let $K_i$ be an L-space knot and write $M_i=S^3\smallsetminus\nu(K)$ for $i=1,2$. Given a homeomorphism $h\co \partial M_1\to \partial M_2$ the closed manifold $M_1\cup_h M_2$ is a L-space if and only if $h(\mathcal{L}_{M_1}^\circ)\cup \mathcal{L}_{M_2}^\circ \cong\Q\cup\{\frac{1}{0}\}$. 
\end{thm} 

Note that this solves \cite[Problem 1.11]{BC} by resolving (much more generally and in the affirmative) one direction of \cite[Conjecture 1]{Hanselman2014} (a special case of \cite[Conjecture 1.10]{BC} or \cite[Conjecture 4.3]{CLW2013}).  It should also be noted that, owing to the existence of hyperbolic L-space knots, Theorem \ref{thm:gen-splice} provides additional food for thought regarding  \cite[Question 1.12]{BC}. Namely, one would like to know if the order detected/foliation detected slopes in the boundary of the exterior of an L-space knot $K$ coincide with the complement of $\mathcal{L}_M^\circ$, where $M=S^3\smallsetminus\nu(K)$.  Note that for integer homology three-spheres resulting from surgery on a knot in $S^3$, there has been considerable progress in this vein; see Li and Roberts for foliations \cite{LR2014}, Boileau and Boyer for left-orders \cite{BB}, and also Hedden and Levine for splicing results \cite{HL}. 

\subsection*{Related work}
Non-simplicity is somewhat sporadic, however it can be shown explicitly that non-simple graph manifolds exist. On the other hand, it is clear from examples that  certain graph manifolds with torus boundary do not give rise to loop-type bordered invariants, at least not in any obvious way. These subtleties seem particularly interesting when contrasted with the behaviour of foliations and orders: For more complicated graph manifolds, more subtle notions of foliations and orders (relative to the boundary tori) need to be considered \cite{BC}. On the other hand, it is now clear from work of Rasmussen and Rasmussen \cite{RR} that simplicity (for loops) is more than a convenience: The class of simple loop-type manifolds is equivalent to the class of (exteriors of) Floer simple knots; see \cite[Proposition 6]{HRRW} in particular. As a result, Theorem \ref{thm:gluing} may be recast in terms of Floer simple manifolds \cite[Theorem 7]{HRRW}. This observation gives rise to an extension of Theorem \ref{thm:main} to the case of general graph manifolds; this is the main result of a joint paper of the authors with Rasmussen and Rasmussen \cite{HRRW}, illustrating a fruitful overlap between these two independent projects. We note that the methods used by Rasmussen and Rasmussen are different from those used in our work. They appeal to knot Floer homology instead of bordered Heegaard Floer homology. This leads to somewhat divergent results and emphasis: While Rasmussen and Rasmussen give a clear picture of the interval of L-space slopes in terms of classical invariants, our machinery is better suited to gluing results.  

\subsection*{Structure of the paper}
Section \ref{sec:background} collects the essentials of bordered Heegaard Floer homology and puts in place our conventions. In particular, the definitions of L-space and strict L-space slope are found here.

Section \ref{sec:loop} is devoted to defining loops and loop calculus. This represents the key tool; loop calculus provides a combinatorial framework for studying bordered Heegaard Floer homology. Note that we define and work in an {\em a priori} broader setting of abstract loops. It seems likely that many of the loops considered do not  represent the type D structure of any bordered three-manifold.  This  calculus in applied towards two distinct ends: Detection and gluing. 

Section \ref{sec:char} gives characterizations of L-space slopes and strict L-space slopes in the loop setting. In particular, we prove Theorem \ref{thm:detection}. The proof of the second part of this result -- identifying the set of L-space slopes as the restriction of an interval with rational endpoints -- requires a detailed study of non-L-space slopes. While interesting in its own right, this fact is essential to the gluing results that follow.

Section \ref{sec:glue} proves Theorem \ref{thm:gluing} by first establishing the appropriate gluing statements for abstract loops. The section concludes with a complete characterization of L-spaces resulting from generalized splices of L-space knots in the three-sphere, proving Theorem \ref{thm:gen-splice}. 

Section \ref{sec:manifolds} turns to the study loop-type manifolds. We describe an algorithm for constructing rational homology sphere graph manifolds by way of three moves, and determine the effect of these moves on bordered invariants. Using this, we establish a class of graph manifold rational homology tori, subsuming Seifert fibered rational homology tori, for which we now have a complete understanding of gluing in Heegaard Floer homology according to the material in the preceding two sections. 

Section \ref{sec:proofs} collects all of the forgoing material in order to prove our main results: This section includes the proofs of Theorem \ref{thm:Seifert-simple} and Theorem \ref{thm:detection-via-TIBOKB}, and from these results we prove Theorem \ref{thm:main}.  

\subsection*{Acknowledgements} We are very grateful to Jake and Sarah Rasmussen for their interest in this work and for their willingness to share their own \cite{RR}. This strengthened the present work considerably and facilitated the joint paper at the intersection of the two projects \cite{HRRW}. 

\section{Background and conventions}\label{sec:background}

We begin by briefly recalling the essentials of bordered Heegaard Floer homology; for details see \cite{LOT-bordered}. We restrict attention to compact, orientable three-manifolds $M$ with torus boundary. Let $\F$ denote the two-element field. 

\subsection{Bordered structures} A bordered three-manifold, in this setting, is a pair $(M,\Phi)$ where $\Phi\co S^1\times S^1 \to \partial M$ is a fixed homeomorphism satisfying $\Phi(S^1\times\{\text{pt}\})=\alpha$ and $\Phi(\{\text{pt}\}\times S^1)=\beta$ for slopes $\alpha$ and $\beta$ in $\partial M$ satisfying $\Delta(\alpha,\beta)=1$. Recall that a slope in $\partial M$ is the isotopy class of an essential simple closed curve in $\partial M$ or, equivalently, a primitive class in $H_1(\partial M;\Z)/\{\pm 1\}$. The distance $\Delta(\cdot,\cdot)$ is measured by considering the minimal geometric intersection between slopes, thus the requirement that  $\Delta(\alpha,\beta)=1$ ensures that the pair $\{\alpha,\beta\}$, having chosen  orientations, forms a basis for $H_1(\partial M;\Z)$. 

As a result, any bordered manifold $(M,\Phi)$ may be represented by the ordered triple $(M,\alpha,\beta)$, with the understanding that $(M,\alpha,\beta)$ and $(M,\beta,\alpha)$ differ as bordered manifolds (that is, these represent different bordered structures on the same underlying manifold $M$). We will adhere to this convention for describing bordered manifolds as it makes clear that bordered manifolds come with a pair of preferred slopes. 

\begin{definition}For a given bordered manifold $(M,\alpha,\beta)$ the slope $\alpha$ is referred to as the standard slope and the slope $\beta$ is referred to as the dual slope. \end{definition}

\subsection{The torus algebra} The torus algebra $\Alg$ is generated (as a vector space over $\F$) by elements 
\[\iota_0, \iota_1, \rho_1,\rho_2,\rho_3, \rho_{12},\rho_{23},\rho_{123}\]
with multiplication defined by 
\[\rho_1\rho_2=\rho_{12},\quad\rho_2\rho_3=\rho_{23}, \quad \rho_1\rho_{23}=\rho_{123}= \rho_{12}\rho_{3}\]
(all other products $\rho_I\rho_J$ vanish) and
\[\iota_0\rho_1=\rho_1=\rho_1\iota_1, \quad \iota_1\rho_2=\rho_2=\rho_2\iota_0, \quad \iota_0\rho_3=\rho_3=\rho_3\iota_1\]
so that $\iota_0+\iota_1$ is a unit.  Denote by $\mathcal{I}$ the subring of idempotents in $\Alg$ generated by $\iota_0$ and $\iota_1$. This algebra has various geometric interpretations; see \cite{LOT-bordered}. The bordered Heegaard Floer invariants of $(M,\alpha,\beta)$ are modules of various types over $\Alg$, as we now describe. 

\subsection{Type D structures}\label{sub:typeD} A type D structure over $\Alg$ is a $\F$-vector space $N$ equipped with a left action of the idempotent subring $\mathcal{I}$ so that $N=\iota_0 N\oplus \iota_1 N$, together with a map \[\delta^1\co N \to \Alg\otimes_\mathcal{I} N\] 
satisfying a compatibility condition with the multiplication on $\Alg$ \cite{LOT-bordered}. The compatibility ensures that the map \begin{align*}\partial\co &\Alg\otimes_\mathcal{I}N\to \Alg\otimes_\mathcal{I}N\\
& a\otimes x \mapsto a\cdot\delta^1(x) 
\end{align*} promotes $\Alg\otimes N$ to a left differential module over $\Alg$ (in particular, $\partial^2=0$), where $a\cdot(b\otimes y)=ab\otimes y$. While we will generally confuse type D structures and their associated differential modules, the advantage of the type D structure is in an iterative definition \[\delta^k\co N \to \Alg^{\otimes k}\otimes_\mathcal{I} N\]
where $\delta^{k+1}=(\id_{\Alg^{\otimes {k}}}\otimes\delta^1)\circ\delta^{k}$ for $k>1$. The type D structure is bounded if all $\delta^k$ are identically zero for sufficiently large $k$.

\parpic[r]{\begin{tikzpicture}[>=latex] 
\node at (1,1) {$\bullet$}; 
\node at (0.25,0) {$\circ$};
\node at (1.75,0) {$\circ$};
\draw[thick, ->,shorten >=0.1cm, shorten <=0.1cm] (1,1) -- (0.25,0);\node at (0.4,0.65) {$\scriptsize\rho_{3}$};
\draw[thick, ->,shorten >=0.1cm, shorten <=0.1cm] (1,1) -- (1.75,0);\node at (1.6,0.65) {$\scriptsize\rho_{1}$};
\draw[thick, ->,shorten >=0.1cm, shorten <=0.1cm] (0.25,0) -- (1.75,0);\node at (1,-0.25) {$\scriptsize\rho_{23}$};
\end{tikzpicture}}
This structure may be described graphically. An $\Alg$-decorated graph is a directed graph with vertex set labeled by $\{\bullet,\circ\}$ and edge set labelled by elements of $\Alg$ consistent with the edge orientations. The vertex labelings specifies the splitting of the generating set according to the idempotents, while the edge set encodes the differential. For example, the $\Alg$-decorated graph on the right encodes the fact that there is a single generator $x$ in the $\iota_0$-idempotent with $\delta^1(x)=\rho_1\otimes u + \rho_3\otimes v$ (or, $\partial(x)=\rho_1u + \rho_3v$). The higher maps in the type D structure can be extracted from following directed paths in the graph, for example, on the right we have $\delta^2(x)= \rho_3\otimes \rho_{23}\otimes u$.

\begin{definition}
An $\Alg$-decorated graph is reduced if no edge is labeled by the identity element of $\Alg$. 
\end{definition}

We remark that, since we will only be interested in the homotopy class of differential modules represented by an $\Alg$-decorated graph, it is always possible to restrict attention to reduced $\Alg$-decorated graphs. This is due to the edge reduction algorithm, as described for example by Levine \cite[Section 2.6]{Levine2012}. Briefly, any segment of a graph of the form
\[
\begin{tikzpicture}[>=latex] 
		\node at (0,0) {$\bullet$}; 
		\node at (1,0) {$\circ$}; 
		\node at (2,0) {$\circ$};
		\node at (3,0) {$\bullet$}; 
		\draw[thick, ->,shorten >=0.1cm, shorten <=0.1cm] (0,0) -- (1,0); \node at (0.5,0.25) {$\scriptsize\rho_{I}$};
		\draw[thick, <-,shorten >=0.1cm, shorten <=0.1cm] (1,0) -- (2,0); 
		\draw[thick, ->,shorten >=0.1cm, shorten <=0.1cm] (2,0) -- (3,0); \node at (2.5,0.25) {$\scriptsize\rho_{J}$};
	\end{tikzpicture}
\]
may be replaced by a single edge 
\[
\begin{tikzpicture}[>=latex] 
		\node at (0,0) {$\bullet$}; 
		\node at (3,0) {$\bullet$}; 
		\draw[thick, ->,shorten >=0.1cm, shorten <=0.1cm] (0,0) -- (3,0); \node at (1.5,0.25) {$\scriptsize\rho_{I}\rho_{J}$};
	\end{tikzpicture}
\]
where the edge is simply deleted if the product $\rho_{I}\rho_{J}$ vanishes (note that we have chosen specific vertex labelling for illustration only). The result is two different representatives of the same homotopy class of differential modules over $\Alg$.

Given a bordered three-manifold $(M,\alpha,\beta)$, Lipshitz, Ozsv\'ath and Thurston define a type D structure $\CFD(M,\alpha,\beta)$ over $\Alg$ that is an invariant of the bordered manifold up to quasi-isomorphism \cite{LOT-bordered}. As explained above, we will sometimes regard this object as a differential module over $\Alg$. 

\subsection{Type A structures}\label{sub:typeA} A type A structure over $\Alg$ is a $\F$-vector space $M$ equipped with a right action of the idempotent subring $\mathcal{I}$ so that $M=M\iota_0\oplus M\iota_1$, together with maps
\[m_{k-1}\co M\otimes_\mathcal{I}\Alg^{\otimes k} \to M\] satisfying the $\Alg_\infty$ relations; see \cite{LOT-bordered}. That is, a type A structure is a right $\Alg_\infty$-algebra over $\Alg$. A type A structure is bounded if the $m_k$ vanish for all sufficuently large $k$. Given a bordered three-manifold $(M,\alpha,\beta)$, Lipshitz, Ozsv\'ath and Thurston define a type A structure $\CFA(M,\alpha,\beta)$ over $\Alg$ that is an invariant of the bordered manifold up to quasi-isomorphism. 

\parpic[r]{\begin{tikzpicture}[>=latex] 
\node at (1,1) {$\bullet$}; 
\node at (0.25,0) {$\circ$};
\node at (1.75,0) {$\circ$};
\draw[thick, ->,shorten >=0.1cm, shorten <=0.1cm] (1,1) -- (0.25,0);\node at (0.4,0.65) {$\scriptsize \rho_{1}$};
\draw[thick, ->,shorten >=0.1cm, shorten <=0.1cm] (1,1) -- (1.75,0);\node at (1.6,0.65) {$\scriptsize \rho_{3}$};
\draw[thick, ->,shorten >=0.1cm, shorten <=0.1cm] (0.25,0) -- (1.75,0);\node at (1,-0.25) {$\scriptsize \rho_{2},\rho_1$};
\end{tikzpicture}}There is a similar graphical representation for type A structures. Indeed, owing to a duality between type D and type A structures for three-manifolds, $\CFA(M,\alpha,\beta)$ may be deduced from the graph describing $\CFD(M,\alpha,\beta)$ by appealing to an algorithm described by Hedden and Levine \cite{HL}. This algorithm takes subscripts $1\mapsto 3$ and $3\mapsto 1$ while fixing $2$, with the convention that a conversion of the form $23\mapsto21$ must be parsed as $2,1$ (this example is shown on the right, as it occurs in the conversion of the sample graph shown previously). In this type A context, sequences of directed edges must be concatenated in order to obtain all of the multiplication maps. For example, labeling the generators as before, the graph on the right ecodes opperations $m_2(x,\rho_3) = v$, $m_2(x,\rho_1) = u$, and $m_3(u,\rho_2,\rho_1)=v$, as well as   $m_3(x,\rho_{12},\rho_2)=v$.

\subsection{Pairing}\label{sub:pairing} Consider a closed, orientable three-manifold $Y$ decomposed along a (possibly essential) torus so that $Y=M_1\cup_h M_2$ for some homeomorphism $h\co\partial M_1\to\partial M_2$. If $h$ has the property that $h(\alpha_1)=\beta_2$ and $h(\beta_1)=\alpha_2$ then we will write this decomposition as $Y=(M_1,\alpha_1,\beta_1)\cup (M_2,\alpha_2,\beta_2)$. The reason for this convention is to ensure compatibility with the paring theorem established by Lipshitz, Ozsv\'ath and Thurston \cite{LOT-bordered}. In particular, they prove that
\[\CFhat(Y) \cong \CFA(M_1,\alpha_1,\beta_1)\boxtimes \CFD(M_2,\alpha_2,\beta_2) \]
where $\CFhat(Y)$ is a chain complex with homology $\HFhat(Y)$. As a vector space over $\F$, this chain complex is generated by tensors (over $\mathcal{I}\in\Alg$) of the form $x\otimes_\mathcal{I}y$ where $x\in\CFA(M_1,\alpha_1,\beta_1)$ and $y\in \CFD(M_2,\alpha_2,\beta_2)$ the differential
\[\partial^\boxtimes(x\otimes_\mathcal{I}y)=\sum_{k=0}^\infty\big(m_{k+1}\otimes\id\big)\big(x\otimes_\mathcal{I}\delta^k(y)\big)\]
which is a finite sum provided at least one of the modules in the pairing is bounded. 

As a particular special case, consider the pairing theorem in the context of Dehn filling. Given a three-manifold $M$ with torus boundary, write $M(\alpha)$ to denote the result of Dehn filling $M$ along the slope $\alpha$, that is, $M(\alpha)=(D^2\times S^1)\cup_h M$ where the homeomorphism $h$ is determined by $h(\partial D^2\times\{\text{pt}\})=\alpha$. In particular,  \[\CFhat(M(\alpha))\cong \CFA(D^2\times S^1,l,m)\boxtimes\CFD(M,\alpha,\beta)\]
where $m=\partial D^2\times\{\text{pt}\}$. Note that, in this context, any choice of slopes $l$ dual to $m$ and $\beta$ dual to $\alpha$ will do, since the family $(D^2\times S^1, l+nm,m)$ are all homeomorphic as bordered manifolds. This is due to the Alexander trick; the Dehn twist along $m$ in $\partial(D^2\times S^1)$ extends to a homeomorphism of $D^2\times S^1$.

Since every bordered manifold is equipped with a preferred choice of slopes, it will be important to distinguish between the two Dehn fillings along these slopes. 

\begin{definition}\label{def:two-types-of-filling}
Let $(M,\alpha,\beta)$ be a bordered three-manifold. The standard filling is the Dehn filling \[M(\alpha)=(D^2\times S^1,l,m)\cup (M,\alpha,\beta)\] (that is, $m\mapsto \alpha$) and the dual filling is the Dehn filling 
\[M(\beta)=(D^2\times S^1,m,l)\cup (M,\alpha,\beta)\] (that is, $m\mapsto \beta$). \end{definition}

More generally, given $(M,\alpha,\beta)$ we would like to compute $\HFhat(M(\gamma))$ for any slope $\gamma$ expressed in terms of $\alpha$ and $\beta$. In particular, we will always make a choice of orientations so that $\alpha\cdot\beta=+1$ resulting in slopes of the form $\gamma=\pm(p\alpha+q\beta)\in H_1(\partial M;\Z)/\{\pm 1\}$.  As is familiar, the fixed choice $\{\alpha,\beta\}$ gives rise to an identification of the set of slopes and the extended rational numbers $\hat\Q := \Q\cup\{\frac{1}{0}\}$. Our convention is that the slope $p\alpha+q\beta$ is identified with  $\frac{p}{q}\in\hat\Q$.  We will return to a detailed description of the pairing theorem for an arbitrary Dehn filling in the next section, since the following definition will be of  central importance.

\begin{definition}\label{def:L-space}An L-space is a rational homology sphere $Y$ for which $\dim\HFhat(Y)=|H_1(Y;\Z)|$. An L-space slope is a slope $\gamma$ in $\partial M$ for which the result of Dehn filling $M(\gamma)$ is an L-space. For any $M$ with torus boundary, let $\mathcal{L}_M$ denote the set of L-space slopes in $\partial M$.   \end{definition}

We will need to distinguish certain L-space slopes. To do this, consider the natural inclusion \[\hat\Q\into \hat\R = \R\cup\{\textstyle\frac{1}{0}\}\] arising from orienting the basis slopes so that $\alpha\cdot\beta=+1$, and endow $\hat\Q$ with the subspace topology. With this identification, $\mathcal{L}_M\subset\hat\Q$.

\begin{definition} The set of strict L-space slopes, denoted $\mathcal{L}_M^\circ$, is the interior of the subset $\mathcal{L}_M$.\end{definition} 

Recall that if $a,b\in\hat\Q$, then the subsets $(a,b)\cap\hat\Q$ and $[a,b]\cap\hat\Q$ are open and closed, respectively, in $\hat\Q$. By abuse, we will write simply $(a,b)$ and $[a,b]$ with the understanding that these describe subsets of $\hat\Q$. 

A key example to consider is that of the exterior of a non-trivial knot $K$ in $S^3$, with $M=S^3\smallsetminus\nu K$. In this case it is well known that if $M$ admits a non-trivial L-space filling, then $\mathcal{L}_M$ is either $[2g-1,\infty]$ or $[\infty, 1-2g]$ relative to the preferred basis consisting of the knot meridian $\mu$ (corresponding to $\frac{1}{0}$) and the Seifert longitude $\lambda$ (corresponding to $0$), where $g$ denotes the Seifert genus of $K$. Notice that $\mu$ and $(2g-1)\mu+\lambda$ are non-strict L-space slopes by definition. On the other hand, if $K$ is the trivial knot then $M\cong D^2\times S^1$ and $\mathcal{L}_M=\mathcal{L}_M^\circ=\hat\Q\smallsetminus\{0\}$ since these are precisely the fillings that give lens spaces. 

Every bordered manifold comes with a preferred identification of the set of slopes with $\hat\Q$, in particular, the notation $\frac{p}{q}\in\mathcal{L}(M,\alpha,\beta)$ should be understood to mean the slope $\pm(p\alpha+q\beta)\in\mathcal{L}_M$. We will adhere to this convention, and use the two interchangeably where there is no potential for confusion. 

\subsection{Gluing via change of framing}\label{sec:change_of_framing} In the interest of determining the set $\mathcal{L}_M$ we will need a means of describing any slope $\gamma$ in $\partial M$ in terms of a fixed basis of slopes $\{\alpha,\beta\}$. Suppose $(M,\alpha,\beta)$ is given, and we would like to calculate $\HFhat(M(p\alpha+q\beta))$. Then, according to the pairing theorem 
\[\CFhat(M(p\alpha+q\beta)) \cong \CFA(D^2\times S^1,m,l)\boxtimes \CFD(M,r\alpha+s\beta,p\alpha+q\beta) \]
where $\bsmat q&p\\s&r\esmat\in \mathit{SL}_2(\Z)$. Notice that  $p=0$ recovers a chain complex for the dual filling. 

Fixing a basis so that $\bsmat 1\\0\esmat$ represents the standard slope and $\bsmat 0\\1\esmat$ represents the dual slope, notice that the Dehn twist along $\alpha$ carrying $\beta\mapsto\alpha+\beta$ is encoded by the matrix $\bsmat 1&1\\0&1\esmat$. Call this the standard Dehn twist and denote it by $T_\stand$. Similarly, the Dehn twist along $\beta$ carrying $\alpha\mapsto \alpha+\beta$ is encoded by the matrix $\bsmat 1&0\\1&1\esmat$. Call this the (negative) dual Dehn twist and denote it by $T_\dual^{-1}$. Associated with each Dehn twist is a mapping cylinder and to this Lipshitz, Ozsv\'ath and Thurston assign a type DA bimodule. For the Dehn twists $T_\stand^{\pm 1}$ and $T_\dual^{\pm 1}$, we use  $\Ts^{\pm 1}$ and $\Td^{\pm 1}$ to denote the corresponding bimodules\footnote{In the notation of \cite[Section 10]{LOT-bimodules}, we have $\Ts = \CFDA(\tau_m)$ and $\Td = \CFDA(\tau_\ell^{-1})$.}.

Given an odd-length continued fraction expansion $\frac{p}{q}=[a_1,a_2,\ldots,a_n]$ we obtain a decomposition according to Dehn twists:
\[\bsmat q&p\\s&r\esmat = \bsmat 1&1\\0&1\esmat^{a_n}\cdots \bsmat 1&0\\1&1\esmat^{a_2}\bsmat 1&1\\0&1\esmat^{a_1}\]
The bordered manifold $(M,r\alpha+s\beta,p\alpha+q\beta)$ is obtained from $(M,\alpha,\beta)$ by attaching the mapping cylinder of a homeomorphism $h$ with representative $h_*=\bsmat q&p\\s&r\esmat$, and so
\[\CFD(M,r\alpha+s\beta,p\alpha+q\beta) \cong \Ts^{a_n}\cdots\boxtimes\Td^{-a_2} \boxtimes\Ts^{a_1} \boxtimes \CFD(M,\alpha,\beta)\]
where $\Ts^n=\underbrace{\Ts\boxtimes\cdots\boxtimes\Ts}_n$.

More generally, given bordered manifolds $(M_1,\alpha_1,\beta_1)$ and $(M_2,\alpha_2,\beta_2)$, we can calculate\[\CFA(M,\alpha_1,\beta_1) \boxtimes \CFD(M,r\alpha_2+s\beta_2,p\alpha_2+q\beta_2)\] by considering a homeomorphism $h$ as above. This gives a complex computing $\HFhat(Y)$ where \[Y\cong M_1\cup_h M_2 \cong (M_1,\alpha_1,\beta_1) \cup (M_2,r\alpha_2+s\beta_2,p\alpha_2+q\beta_2)\] and
$h\co \partial M_1 \to \partial M_2$ is specified by 
\[\alpha_1 \mapsto r\alpha_2+s\beta_2, \quad \beta_1\mapsto p\alpha_2+q\beta_2.\]
With these gluing conventions in place, we make an observation that will be of use in the sequel:

\begin{prop}\label{prop:slope-trick}
Given type D modules $N_1$ and $N_2$, let $N'_1 = \Ts^n \boxtimes N_1$ and $N'_2 = \Td^{-n}\boxtimes N_2$ for some $n\in\Z$. There is a homotopy equivalence 
\[N_1 \boxtimes \CFAAid \boxtimes N_2 \cong N'_1 \boxtimes \CFAAid \boxtimes N'_2\]
where $\CFAAid$ is the type AA identity bimodule. 
\end{prop}
\begin{remark}The type AA identity bimodule is defined in \cite{LOT-bimodules} and, in particular, gives rise to $\CFA(M,\alpha,\beta)\cong \CFAAid \boxtimes\CFD(M,\alpha,\beta)$. This observation leads to the algorithm described in Section \ref{sub:typeA}. \end{remark}
\begin{proof}[Proof of Proposition \ref{prop:slope-trick}]The right hand side of the claimed equivalence can be written as
\[ \left( \Ts^n \boxtimes N_1\right) \boxtimes \CFAAid \boxtimes \Td^{-n} \boxtimes N_2.\] 
The $DA$-bimodule $\Ts$ is homotopy equivalent to the $AD$-bimodule $\CFAAid \boxtimes \Td \boxtimes \CFDDid$; to see this, note that the Heegaard diagram for $T_{st} = \tau_m$ in \cite[Figure 25 ]{LOT-bimodules} can be obtained from the Heegaard diagram for $T_{\dual} = \tau_l$ by rotating 180 degrees. Thus the right side above simplifies to
\begin{align*}
&N_1\boxtimes \left( \CFAAid \boxtimes \Td \boxtimes \CFDDid \right)^n \boxtimes \CFAAid \boxtimes \Td^{-n} \boxtimes N_2\\
&\cong N_1\boxtimes \CFAAid \boxtimes \Td^n  \boxtimes \Td^{-n} \boxtimes N_2\\
&\cong N_1\boxtimes \CFAAid \boxtimes N_2\qedhere\\
\end{align*}
\end{proof}

\begin{remark}
If $N_1 = \CFD(M_1,\alpha_1,\beta_1)$ and $N_2 = \CFD(M_2,\alpha_2,\beta_2)$, the homotopy equivalence above corresponds to the fact that, with the gluing conventions described above,
\[(M_1,\alpha_1,\beta_1)\cup(M_2,\alpha_2,\beta_2)\cong(M_1,\alpha_1,\beta_1+n\alpha_1)\cup(M_2,\alpha_2+n\beta_2,\beta_2)\] for any $n\in\Z$. 
\end{remark}

\subsection{Grading} \label{sub:grading}
We conclude this discussion with a description of the relative $(\Ztwo)$-grading on the bordered invariants, summarizing the discussion in \cite[Section 2.2]{Hanselman2014}. For more details and developments, see \cite{LHW,Petkova-decategorification}

The torus algebra $\Alg$ may be promoted to a graded algebra by defining $\gr(\rho_1)=\gr(\rho_3)=0$ and $\gr(\rho_2)=1$, and extending according to $\gr(\rho_I\rho_J)\equiv \gr(\rho_I)+\gr(\rho_J)$ reduced modulo 2 (we will always drop this reduction from the notation). In particular, the grading is 1 on all the remaining non-zero products $\rho_{12},\rho_{23},\rho_{123}$.  

The relative $(\Ztwo)$-grading on elements $x$ of a  type D structure is determined by \[\gr(\rho_I\cdot x) \equiv \gr (\rho_I) + \gr(x) \ \text{and} \  \gr(\partial x) \equiv \gr(x) + 1\] In particular, given a reduced $\Alg$-decorated graph this relative grading is determined by choosing the grading on a given vertex, and then noting that only $\rho_1$- and $\rho_3$-labeled edges alter the grading.  

The relative $(\Ztwo)$-grading on elements $x$ of a  type A structure is determined by \[\gr(m_{k+1}(x,\rho_{I_1},\ldots,\rho_{I_k}))-k-1\equiv \gr(x)+ \sum_{j=1}^k \gr(\rho_{I_j})\] Notice that, given a type $D$ structure with a choice of relative grading, a relative grading on the associated type A structure is obtained by switching the grading of each generator with idempotent $\iota_0$. 

If $Y\cong (M_1,\alpha_1,\beta_2) \cup (M_2,\alpha_2,\beta_2)$, choices of relative $(\Ztwo)$-gradings on each of  the objects $\CFA(M_1,\alpha_1,\beta_2)$ and $\CFD(M_2,\alpha_2,\beta_2)$ give rise to a relative grading on $\CFhat(Y)$ via the pairing theorem and the rule $\gr(x\otimes y) = \gr(x)+ \gr(y)$. This agrees with the usual  relative $(\Ztwo)$-grading on $\CFhat(Y)$ so that, in particular,  $|\chi\CFhat(Y)|\ge|H_1(Y;\Z)|$ for any rational homology sphere $Y$. At the level of homology, it follows immediately that $Y$ is an L-space if and only if every generator of $\HFhat(Y)$ has the same grading. 

Where required, we will make use of an additional marking on the vertices of an $\Alg$-decorated graph by $\{\bullet^+,\bullet^-,\circ^+,\circ^-\}$ to indicate a choice of relative $(\Ztwo)$-grading on the underlying differential module. In addition, given a type D structure $N$ with a choice of relative $(\Ztwo)$-grading we define the idempotent Euler characteristics \[\chi_\bullet(N) = \chi(\iota_0 N) \ \text{and} \  \chi_\circ(N)= \chi(\iota_1 N)\]

The following Lemma records how $\chi_\bullet$ and $\chi_\circ$ change under reparametrization of the boundary.
\begin{lem}\label{lem:Euler_char_twists}
Let $\bsmat q&p\\s&r\esmat \in \mathit{SL}_2(\Z)$. Then
\begin{align*}
\pm \chi_\bullet \CFD(M, r\alpha + s\beta, p\alpha + q\beta) &= r \chi_\bullet \CFD(M, \alpha, \beta) + s \chi_\circ \CFD(M, \alpha, \beta) \\
\pm \chi_\circ \CFD(M, r\alpha + s\beta, p\alpha + q\beta) &= p \chi_\bullet \CFD(M, \alpha, \beta) + q \chi_\circ \CFD(M, \alpha, \beta)
\end{align*}
Here the $\pm$ choice depends on the choice of absolute $(\Ztwo)$-grading on $\CFD(M, \alpha, \beta)$ and $\CFD(M, r\alpha + s\beta, p\alpha + q\beta)$.
\end{lem}
\begin{proof}
This is true when $r = q = 1$ and $s = p = 0$. Assuming the claim is true for given $r, s, p$ and $q$, we check that it is true when $p$ and $q$ are replaced by $p+r$ and $q+s$ by examining the effect of the standard Dehn twist. The relevant bimodule $\Ts$ is pictured in Figure \ref{fig:Dehn-twist-bimodules}. It has three generators: one pairs with $\iota_0$ generators to produce $\iota_0$ generators, one pairs with $\iota_1$ generators to produce $\iota_1$ generators, and one pairs with $\iota_0$ generators to produce $\iota_1$ generators. All three generators have the same relative $(\Ztwo)$-grading; we will assume they have grading 0. It follows that 
\begin{align*}
 &\chi_\bullet \CFD(M, r\alpha + s\beta, (p+r)\alpha + (q+s)\beta) \\
 &= \chi_\bullet\left( \Ts \boxtimes \CFD(M, r\alpha + s\beta, p\alpha + q\beta) \right) \\
& = \chi_\bullet \CFD(M, r\alpha + s\beta, p\alpha + q\beta) \\
& = r \chi_\bullet \CFD(M, \alpha, \beta) + s \chi_\circ \CFD(M, \alpha, \beta) \\[4pt]
& \chi_\circ \CFD(M, r\alpha + s\beta, (p+r)\alpha + (q+s)\beta) \\
 &= \chi_\circ\left( \Ts \boxtimes \CFD(M, r\alpha + s\beta, p\alpha + q\beta) \right) \\
&= \chi_\circ \CFD(M, r\alpha + s\beta, p\alpha + q\beta) + \chi_\bullet \CFD(M, r\alpha + s\beta, p\alpha + q\beta) \\
& = (p+r) \chi_\bullet \CFD(M, \alpha, \beta) + (q+s) \chi_\circ \CFD(M, \alpha, \beta) \end{align*}
Similarly, we check the claim for when $r$ and $s$ are replaced with $r-p$ and $s-q$ by considering the dual Dehn twist bimodule $\Td$ (pictured in Figure \ref{fig:Dehn-twist-bimodules-DUAL}). This bimodule has three generators with relative $(\Ztwo)$-grading 0, 0, and 1: the first pairs with $\iota_0$ generators to produce $\iota_0$ generators, the second pairs with $\iota_1$ generators to produce $\iota_1$ generators, and the third pairs with $\iota_1$ generators to produce $\iota_0$ generators. It follows that
\begin{align*} & \chi_\bullet \CFD(M, (r-p)\alpha + (s-q)\beta, p\alpha + q\beta) \\
&= \chi_\bullet\left( \Td \boxtimes \CFD(M, r\alpha + s\beta, p\alpha + q\beta) \right) \\
&= \chi_\bullet \CFD(M, r\alpha + s\beta, p\alpha + q\beta) - \chi_\circ \CFD(M, r\alpha + s\beta, p\alpha + q\beta) \\
& = (r-p) \chi_\bullet \CFD(M, \alpha, \beta) + (s-q) \chi_\circ \CFD(M, \alpha, \beta) \end{align*}
\begin{align*}
\chi_\circ \CFD(M, (r-p)\alpha + (s-q)\beta, p\alpha + q\beta) 
&= \chi_\circ\left( \Td \boxtimes \CFD(M, r\alpha + s\beta, p\alpha + q\beta) \right) \\
&= \chi_\circ \CFD(M, r\alpha + s\beta, p\alpha + q\beta) \\
&= p \chi_\circ \CFD(M, \alpha, \beta) + q \chi_\circ \CFD(M, \alpha, \beta) \end{align*}
The argument extends easily to multiples and inverses of Dehn twists. The claim follows by decomposing an arbitrary element of $\mathit{SL}_2(\Z)$ into Dehn twists.
\end{proof}

Given a rational homology solid torus $M$ with bordered structure $(\alpha, \beta)$, there is a distinguished slope $\frac{p}{q}$ on the boundary for which some multiple of the curve $p \alpha + q \beta$ in $\partial M$ is nullhomologous in $M$. This slope is called the \emph{rational longitude}. The $(\Ztwo)$-grading allows us to recover this data from $\CFD(M, \alpha, \beta)$ as follows:

\begin{prop}\label{prop:rational_longitude}
Let $(M, \alpha, \beta)$ be a bordered manifold, let $\spinc$ be a spin$^c$-structure on $M$, and let $\chi_\bullet$ and $\chi_\circ$ denote the Euler characteristics of $\CFD(M, \alpha, \beta; \spinc)$ in the appropriate idempotent. If $\chi_\bullet = \chi_\circ = 0$, then $M$ is not a rational homology solid torus. Otherwise, $M$ is a rational homology solid torus, and the nullhomologous curves in $\partial M$ are the multiples of $\chi_\circ \alpha - \chi_\bullet \beta$ (in particular, the rational longitude is $-\frac{\chi_\circ}{\chi_\bullet}$).
\end{prop}
\begin{proof}
Given $\frac{p}{q}$ in $\hat\Q$, consider the Euler characteristic of $\HFhat(M(p\alpha + q\beta); \spinc')$, where $\spinc'$ is the spin$^c$-structure that restricts to $\spinc$ on $M$. This is nonzero if and only if $M(p\alpha + q\beta)$ is a rational homology sphere. $M(p\alpha + q\beta)$ is the dual filling of $(M,r\alpha+s\beta,p\alpha+q\beta)$ for any $r, s$ with $rq-ps = 1$. Thus, using Lemma \ref{lem:Euler_char_twists}, $\chi \HFhat(M(p\alpha + q\beta); \spinc')$ is given (up to sign) by
\[\chi_\circ \CFD(M, r\alpha + s\beta, p\alpha + q\beta; \spinc') = p \chi_\bullet + q \chi_\circ\]
Note that if $\chi_\bullet$ and $\chi_\circ$ are both zero, then $M(p\alpha + q\beta)$ is not a rational homology sphere for any $\frac{p}{q}$, and so $M$ is not a rational homology solid torus. Otherwise, $M(p\alpha + q\beta)$ is a rational homology sphere unless $\frac{p}{q} = -\frac{\chi_\circ}{\chi_\bullet} \in \hat\Q$. It follows that $M$ is a rational homology solid torus, and the slope of the rational longitude is $-\frac{\chi_\circ}{\chi_\bullet}$. 

From now on, let $\frac{p}{q} = -\frac{\chi_\circ}{\chi_\bullet} $ with $p,q$ relatively prime. The curve $np\alpha + nq\beta$ is nullhomologous in $M$ for some $n>0$. To determine the smallest such $n$, we consider the standard filling of $(M,r\alpha+s\beta,p\alpha+q\beta)$, for some $r,s$ with $rq-ps = 1$. This manifold is obtained from $M$ by adding a solid torus, identifying the longitude $l$ with $p\alpha+q\beta$. $H_1$ of this manifold is a direct sum of $H_1(M, \partial M)$ and $H_1(D^2 \times S^1)$ (which is generated by the longitude $l$) with the relation $nl = np\alpha + nq\beta = 0$. Recall that $H_1(M, \partial M)$ indexes the spin$^c$-structures on $M$. Thus for each spin$^c$ structure $\spinc$ on $M$, there are $n$ elements of $H_1(M(r\alpha+s\beta))$. It follows that $\chi \HFhat( M(r\alpha + s\beta; \spinc') = n$ for each $\spinc'$. By Lemma \ref{lem:Euler_char_twists}, we have
\[ n = \chi_\bullet \CFD(M, r\alpha + s\beta, p\alpha + q\beta; \spinc') = r \chi_\bullet + s \chi_\circ \]
We also have that $n = nqr - nps$, so up to sign $(\chi_\bullet, \chi_\circ) = (nq, -np)$. The smallest nullhomologous curve in $\partial M$ is $\chi_\circ \alpha - \chi_\bullet \beta$.
\end{proof}

\section{Loop calculus}\label{sec:loop}

A focus of this work is the development of a calculus for studying bordered invariants. This will be achieved by restricting to a class of manifolds whose bordered invariants can be represented by certain valence two $\Alg$-decorated graphs, which we will call loops.

\subsection{Loops and loop-type manifolds}\label{sub:loop-type}
Towards defining loops, consider the following arrows which may appear in an $\Alg$-decorated graph.
\[
\begin{tikzpicture}[>=latex] 
		\node at (0,0) {$\bullet$}; 
		\node at (0,1) {$\bullet$}; 
		\node at (0,2) {$\bullet$};
\draw[thick, ->,shorten <=0.1cm] (0,0) -- (1,0); \node at (0.5,0.25) {$\scriptsize\rho_{123}$};
\draw[thick, ->, shorten <=0.1cm] (0,1) -- (1,1); \node at (0.5,1.25) {$\scriptsize\rho_{12}$};
\draw[thick, ->, shorten <=0.1cm] (0,2) -- (1,2); \node at (0.5,2.25) {$\scriptsize\rho_{1}$}; \node at (0.5,-0.75){\bf{I}$_\bullet$};

		\node at (2,0) {$\bullet$}; 
		\node at (2,1) {$\bullet$}; 
		\node at (2,2) {$\bullet$};
\draw[thick, <-,shorten <=0.1cm] (2,0) -- (3,0); \node at (2.5,0.25) {$\scriptsize\rho_{12}$};
\draw[thick, <-, shorten <=0.1cm] (2,1) -- (3,1); \node at (2.5,1.25) {$\scriptsize\rho_{2}$};
\draw[thick, ->, shorten <=0.1cm] (2,2) -- (3,2); \node at (2.5,2.25) {$\scriptsize\rho_{3}$}; \node at (2.5,-0.75){\bf{II}$_\bullet$};
		
		\node at (4,0) {$\circ$}; 
		\node at (4,1) {$\circ$}; 
		\node at (4,2) {$\circ$};
\draw[thick, <-,shorten <=0.1cm] (4,0) -- (5,0); \node at (4.5,0.25) {$\scriptsize\rho_{1}$};
\draw[thick, ->, shorten <=0.1cm] (4,1) -- (5,1); \node at (4.5,1.25) {$\scriptsize\rho_{23}$};
\draw[thick, ->, shorten <=0.1cm] (4,2) -- (5,2); \node at (4.5,2.25) {$\scriptsize\rho_{2}$}; \node at (4.5,-0.75){\bf{I}$_\circ$};

		\node at (6,0) {$\circ$}; 
		\node at (6,1) {$\circ$}; 
		\node at (6,2) {$\circ$};
\draw[thick, <-,shorten <=0.1cm] (6,0) -- (7,0); \node at (6.5,0.25) {$\scriptsize\rho_{123}$};
\draw[thick, <-, shorten <=0.1cm] (6,1) -- (7,1); \node at (6.5,1.25) {$\scriptsize\rho_{23}$};
\draw[thick, <-, shorten <=0.1cm] (6,2) -- (7,2); \node at (6.5,2.25) {$\scriptsize\rho_{3}$}; \node at (6.5,-0.75){\bf{II}$_\circ$};
	\end{tikzpicture}	\]
We will be interested in valence two $\Alg$-decorated graphs subject to the following restriction:
\begin{itemize}
\item[($\star$)] Each $\iota_0$-vertex is adjacent to an edge of type {\bf I}$_\bullet$ and an edge of type {\bf II}$_\bullet$, and each $\iota_1$-vertex is adjacent to an edge of type {\bf I}$_\circ$ and an edge of type {\bf II}$_\circ$
\end{itemize}
\begin{definition}\label{def:loop}
A loop is a connected valence two $\Alg$-decorated graph satisfying $(\star)$.
\end{definition}

Notice that since any loop describes a differential module over $\Alg$, a loop may be promoted to a graded loop via the $(\Ztwo)$-grading described in Section \ref{sub:grading}. In particular, where needed, the vertex set will be extended to $\{\bullet^+,\bullet^-,\circ^+,\circ^-\}$.

As abstract combinatorial objects loops provide a tractable structure to work with; this section develops a calculus for doing so. The results derived from this calculus apply to the following class of bordered invariants:

\begin{definition}\label{def:loop-type}The bordered invariant $\CFD(M,\alpha,\beta)$ is said to be of loop-type if, up to homotopy, it may be represented by a collection of loops, that is, by a (possibly disconnected) $\Alg$-decorated valence two graph satisfying $(\star)$. For simplicity, in this paper we will make the additional assumption that the number of connected components of the valence two graph describing $\CFD(M,\alpha,\beta)$ coincides with $|\operatorname{Spin}^c(M)|$.  \end{definition}

We will refer to the bordered manifold $(M,\alpha,\beta)$ being of loop-type when the associated bordered invariant has this property. Some motivation for this definition is provided by the following:

\begin{prop}
Let $\mathcal{H}$ be a bordered Heegaard diagram describing $(M,\alpha,\beta)$ and  suppose  $\CFD(M,\alpha,\beta)=\CFD(\mathcal{H})$ is reduced and represented by a valence two $\Alg$-decorated graph having a single connected component per $\operatorname{spin}^c$-structure. Then $\CFD(M,\alpha,\beta)$ is loop-type. 
\end{prop}
\begin{proof}
The hypothesis $\CFD(M,\alpha,\beta)=\CFD(\mathcal{H})$ allows us to use the notion of generalized coefficient maps developed in \cite[Section 11.6]{LOT-bordered}, which force restrictions on the type D modules that can occur as invariants of manifolds with torus boundary. Briefly, generalized coefficient maps are extra differentials obtained by counting holomorphic curves that run over the basepoint. The torus algebra is extended with additional Reeb chords: $\rho_0$, $\rho_{01}$, $\rho_{30}$, $\rho_{012}$, $\rho_{301}$, $\rho_{230}$, $\rho_{0123}$, $\rho_{3012}$, $\rho_{2301}$, and $\rho_{1230}$. With these additional differentials, $\partial^2 = 0$ is no longer satisfied. However, we have instead that
$$\partial^2(x) = \rho_{1230}x + \rho_{3012}x  \: \text{  for any $x$ in}\: \iota_0 \CFD(M,\alpha,\beta)$$
$$\partial^2(x) = \rho_{2301}x + \rho_{1230}x  \: \text{  for any $x$ in}\: \iota_1 \CFD(M,\alpha,\beta)$$

Let $x$ be a generator with idempotent $\iota_0$. Since $\partial^2(x)$ contains the term $\rho_{1230}x$, $\partial(x)$ contains the term $\rho_{I_1} y$ and $\partial(y)$ contains the term $\rho_{I_2} x$ for some $y \in \CFD(M,\alpha,\beta)$ and some (generalized) Reeb chords $\rho_{I_1}$ and $\rho_{I_2}$ such that $\rho_{I_1}\rho_{I_2} = \rho_{1230}$. Since we assumed that $\CFD(M,\alpha,\beta)$ is reduced, $I_1$ and $I_2$ are not $\emptyset$. It follows that $I_1 \in \{1, 12, 123\}$ and, in the graphical representation, the vertex corresponding to to $x$ has an incident edge of type {\bf I}$_\bullet$.

Since $\partial^2(x)$ contains the term $\rho_{3012}x$, $\partial(x)$ contains the term $\rho_{I_1} y$ and $\partial(y)$ contains the term $\rho_{I_2} x$ for some $y \in \CFD(M,\alpha,\beta)$ and some (generalized) Reeb chords $\rho_{I_1}$ and $\rho_{I_2}$ such that $\rho_{I_1}\rho_{I_2} = \rho_{3012}$. It follows that either $I_1 = 3$ or $I_2 \in \{2, 12\}$ and, in the graphical representation, the vertex corresponding to to $x$ has an incident edge of type {\bf II}$_\bullet$. Since any vertex has valence two by assumption, the vertex corresponding to to $x$ must have exactly one edge from each of {\bf I}$_\bullet$ and {\bf II}$_\bullet$.

The argument when $x$ has idempotent $\iota_1$ is similar. Since $\partial^2(x)$ contains the term $\rho_{2301}x$, it follows that the vertex corresponding to $x$ is adjacent an edge of type {\bf I}$_\circ$; and since $\partial^2(x)$ contains the term $\rho_{0123}x$ it follows that the corresponding vertex is adjacent to an edge of type {\bf II}$_\circ$.\end{proof}

\begin{remark}
While this argument required a particular choice of Heegaard diagram, it seems likely that this hypothesis is not a necessary. However, for the purposes of this work the more general statement will not be required -- we simply restrict to the graphs satisfying $(\star)$ by definition -- and leave developing the necessary algebra to future work. 
\end{remark}

It is important to note that loops in the abstract need not be related to three-manifolds: It is not necessarily true that every loop (or disjoint union of loops) arises as the bordered invariant of some three-manifold with torus boundary. This section concludes with an explicit example for illustration.

\subsection{Standard and dual notation}\label{sec:notations_for_loops} It is natural to decompose a loop  into pieces by breaking along vertices corresponding to one of the two idempotents; the constraint $(\star)$ suggests that the pieces resulting from such a decomposition are quite limited. Indeed, breaking  along generators with idempotent $\iota_0$, five essentially different chains are possible and these possibilities are listed in Figure \ref{fig:puzzle_pieces} (on the left hand side). Since $(\star)$ also puts restrictions on how these segments can be concatenated, each piece is depicted with puzzle-piece ends. For instance, a type $a$ piece can be followed by a type $b$ piece, but not by a type $c$ piece. Any given piece may also appear backwards; we denote this with a {\it bar}.

\begin{figure}
\begin{tabular}{lr}
\labellist
\pinlabel {\begin{tikzpicture}[>= latex]
\node at (0,0) {$\bullet$}; \node at (1.5,0) {$\circ$};  \node at (3,0) {$\circ$};  \node at (4.5,0) {$\bullet$};
\draw[->, thick, shorten >=0.1cm,shorten <=0.1cm]  (0,0) to (1.5,0);
\draw[->, thick, shorten >=0.1cm,shorten <=0.1cm, dashed]  (1.5,0) to (3,0);
\draw[->, thick, shorten >=0.1cm,shorten <=0.1cm] (3,0) to (4.5,0);
\node at (0.75,0.2) {$\rho_3$};\node at (2.25,0.2) {$\rho_{23}$};\node at (3.75,0.2) {$\rho_2$};
\end{tikzpicture}} at 76.5 32 \pinlabel {$a_k$} at -30 32 
\pinlabel {$\underbrace{\phantom{aaaaaaa}}_k$} at 76.5 20
\endlabellist
\includegraphics[scale=1]{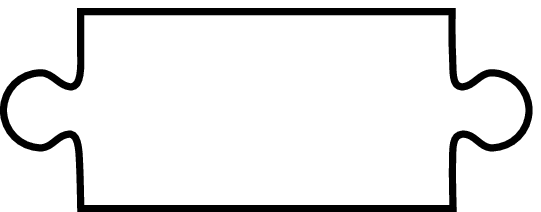}\qquad\qquad\qquad\qquad
&
\labellist
\pinlabel {\begin{tikzpicture}[>= latex]
 \node at (1.5,0) {$\bullet$};  \node at (3,0) {$\bullet$};  
\draw[<-, thick, shorten >=0.1cm,shorten <=0.4cm]  (0,0) to (1.5,0);
\draw[->, thick, shorten >=0.1cm,shorten <=0.1cm, dashed]  (1.5,0) to (3,0);
\draw[->, thick, shorten >=0.4cm,shorten <=0.1cm] (3,0) to (4.5,0);
\node at (0.75,0.2) {$\rho_3$};\node at (2.25,0.2) {$\rho_{12}$};\node at (3.75,0.2) {$\rho_{123}$};
\end{tikzpicture}} at 76.5 32 \pinlabel {$a_k^*$} at -30 32 
\pinlabel {$\underbrace{\phantom{aaaaaaa}}_k$} at 76.5 20
\endlabellist
\includegraphics[scale=1]{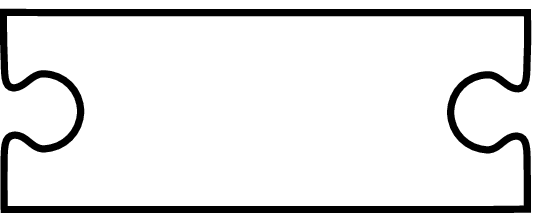}\\[10pt]
\labellist
\pinlabel {\begin{tikzpicture}[>= latex]
 \node at (1.5,0) {$\circ$};  \node at (3,0) {$\circ$}; 
 \draw[->, thick, shorten >=0.1cm,shorten <=0.4cm]  (0,0) to (1.5,0);
\draw[->, thick, shorten >=0.1cm,shorten <=0.1cm, dashed]  (1.5,0) to (3,0);
\draw[<-, thick, shorten >=0.4cm,shorten <=0.1cm] (3,0) to (4.5,0);
\node at (0.75,0.2) {$\rho_{123}$};\node at (2.25,0.2) {$\rho_{23}$};\node at (3.75,0.2) {$\rho_1$};
\end{tikzpicture}} at 76.5 32  \pinlabel {$b_k$} at -30 32
\pinlabel {$\underbrace{\phantom{aaaaaaa}}_k$} at 76.5 18
\endlabellist
\includegraphics[scale=1]{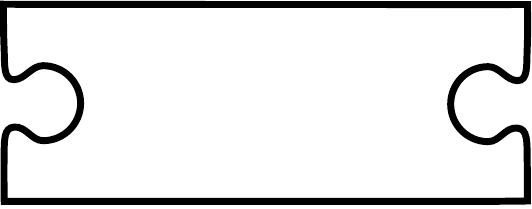}\qquad\qquad\qquad\qquad
&
\labellist
\pinlabel {\begin{tikzpicture}[>= latex]
\node at (0,0) {$\circ$}; \node at (1.5,0) {$\bullet$};  \node at (3,0) {$\bullet$}; \node at (4.5,0) {$\circ$};
 \draw[->, thick, shorten >=0.1cm,shorten <=0.1cm]  (0,0) to (1.5,0);
\draw[->, thick, shorten >=0.1cm,shorten <=0.1cm, dashed]  (1.5,0) to (3,0);
\draw[->, thick, shorten >=0.1cm,shorten <=0.1cm] (3,0) to (4.5,0);
\node at (0.75,0.2) {$\rho_{2}$};\node at (2.25,0.2) {$\rho_{12}$};\node at (3.75,0.2) {$\rho_1$};
\end{tikzpicture}} at 76.5 32  \pinlabel {$b_k^*$} at -30 32
\pinlabel {$\underbrace{\phantom{aaaaaaa}}_k$} at 76.5 18
\endlabellist
\includegraphics[scale=1]{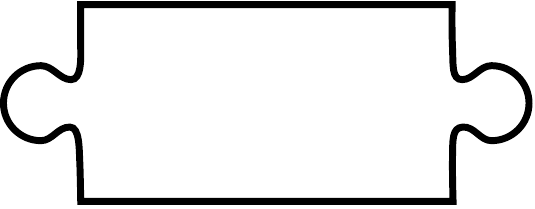}\\[10pt]
\labellist
\pinlabel {\begin{tikzpicture}[>= latex]
\node at (0,0) {$\bullet$}; \node at (1.5,0) {$\circ$};  \node at (3,0) {$\circ$}; 
\draw[->, thick, shorten >=0.1cm,shorten <=0.1cm]  (0,0) to (1.5,0);
\draw[->, thick, shorten >=0.1cm,shorten <=0.1cm, dashed]  (1.5,0) to (3,0);
\draw[<-, thick, shorten >=0.4cm,shorten <=0.1cm] (3,0) to (4.5,0);
\node at (0.75,0.2) {$\rho_3$};\node at (2.25,0.2) {$\rho_{23}$};\node at (3.75,0.2) {$\rho_1$};
\end{tikzpicture}} at 73 32  \pinlabel {$c_k$} at -30 32
\pinlabel {$\underbrace{\phantom{aaaaaaa}}_k$} at 76 20
\endlabellist
\includegraphics[scale=1]{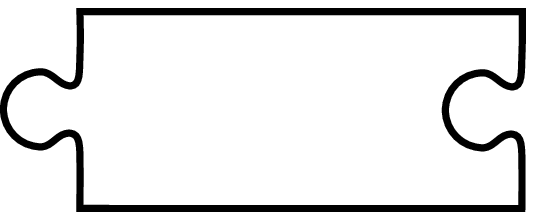}\qquad\qquad\qquad\qquad 
&
\labellist
\pinlabel {\begin{tikzpicture}[>= latex]
 \node at (1.5,0) {$\bullet$};  \node at (3,0) {$\bullet$}; \node at (4.5,0) {$\circ$};
\draw[<-, thick, shorten >=0.1cm,shorten <=0.4cm]  (0,0) to (1.5,0);
\draw[->, thick, shorten >=0.1cm,shorten <=0.1cm, dashed]  (1.5,0) to (3,0);
\draw[->, thick, shorten >=0.1cm,shorten <=0.1cm] (3,0) to (4.5,0);
\node at (0.75,0.2) {$\rho_3$};\node at (2.25,0.2) {$\rho_{12}$};\node at (3.75,0.2) {$\rho_1$};
\end{tikzpicture}} at 79.5 32  \pinlabel {$c_k^*$} at -30 32
\pinlabel {$\underbrace{\phantom{aaaaaaa}}_k$} at 76 20
\endlabellist
\includegraphics[scale=1]{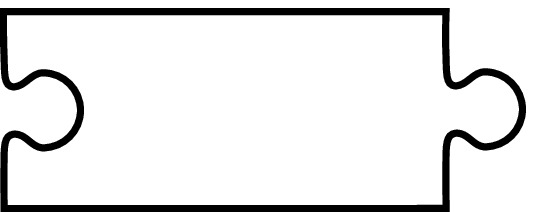}\\[10pt]
\labellist
\pinlabel {\begin{tikzpicture}[>= latex]
\node at (1.5,0) {$\circ$};  \node at (3,0) {$\circ$};  \node at (4.5,0) {$\bullet$};
\draw[->, thick, shorten >=0.1cm,shorten <=0.4cm]  (0,0) to (1.5,0);
\draw[->, thick, shorten >=0.1cm,shorten <=0.1cm, dashed]  (1.5,0) to (3,0);
\draw[->, thick, shorten >=0.1cm,shorten <=0.1cm] (3,0) to (4.5,0);
\node at (0.75,0.2) {$\rho_{123}$};\node at (2.25,0.2) {$\rho_{23}$};\node at (3.75,0.2) {$\rho_2$};
\end{tikzpicture}} at 79.5 32  \pinlabel {$d_k$} at -30 32
\pinlabel {$\underbrace{\phantom{aaaaaaa}}_k$} at 76.5 20
\endlabellist
\includegraphics[scale=1]{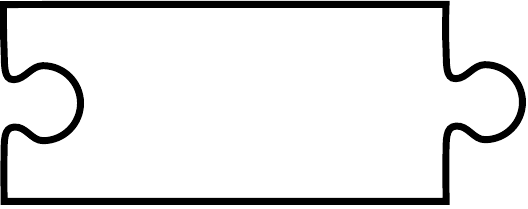}\qquad\qquad\qquad\qquad &
\labellist
\pinlabel {\begin{tikzpicture}[>= latex]
\node at (0,0) {$\circ$};\node at (1.5,0) {$\bullet$};  \node at (3,0) {$\bullet$}; 
\draw[->, thick, shorten >=0.1cm,shorten <=0.1cm]  (0,0) to (1.5,0);
\draw[->, thick, shorten >=0.1cm,shorten <=0.1cm, dashed]  (1.5,0) to (3,0);
\draw[->, thick, shorten >=0.4cm,shorten <=0.1cm] (3,0) to (4.5,0);
\node at (0.75,0.2) {$\rho_{2}$};\node at (2.25,0.2) {$\rho_{12}$};\node at (3.5,0.2) {$\rho_{123}$};
\end{tikzpicture}} at 73 32  \pinlabel {$d_k^*$} at -30 32
\pinlabel {$\underbrace{\phantom{aaaaaaa}}_k$} at 76.5 20
\endlabellist
\includegraphics[scale=1]{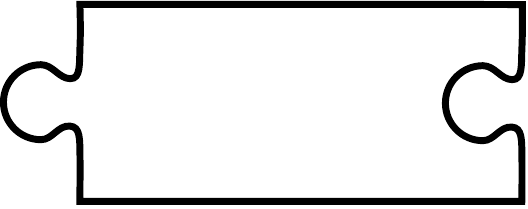}\\[10pt]
\labellist
\pinlabel {\begin{tikzpicture}[>= latex]
 \node at (4.5,0) {$\bullet$};
\draw[->, thick, shorten >=0.1cm,shorten <=0.4cm]  (0,0) to (4.5,0);
\node at (2.25,0.2) {$\rho_{12}$};
\end{tikzpicture}} at 79.5 32  \pinlabel {$e$} at -30 32
\endlabellist
\includegraphics[scale=1]{figures/in-out}\qquad\qquad\qquad\qquad &
\labellist
\pinlabel {\begin{tikzpicture}[>= latex]
\node at (0,0) {$\circ$}; 
\draw[->, thick, shorten >=0.4cm,shorten <=0.1cm]  (0,0) to (4.5,0);
\node at (2.25,0.2) {$\rho_{23}$};
\end{tikzpicture}} at 73 32  \pinlabel {$e^*$} at -30 32
\endlabellist
\includegraphics[scale=1]{figures/out-in}\\[10pt]
\end{tabular}
\caption{Possible segments: Standard notation (left) is obtained by breaking a loop along $\iota_0$ idempotents ($\bullet$-vertices) and dual notation (right) is obtained by breaking a loop along $\iota_1$ idempotents ($\circ$-vertices).}
\end{figure}\label{fig:puzzle_pieces}

Denote the standard alphabet by \[\mathfrak{A}=\{a_i,b_i,c_j,d_j\}\] where $i\in\Z\smallsetminus\{0\}$ and $j\in\Z$. The letters in the standard alphabet correspond directly to the segments depicted in Figure \ref{fig:puzzle_pieces}, with the relationships 
\[a_{-i}=\bar a_i \ \text{and}\  b_{-i}=\bar b_i\] for $i>0$ as well as 
\[c_{-j}=\bar d_j \ \text{and}\ d_{-j}=\bar c_j\]  for $j\ge0$ with  $d_0=e=\bar c_0$ and $\bar d_0=\bar e=c_0$. Throughout this paper we will always assume these shorthand relationships and, by abuse, we will often not distinguish between segments and letters. For instance, the symbols $d_0$ and $e$ will be used interchangeably and may refer either to a letter in the standard alphabet or the corresponding $\Alg$-decorated graph segment. Note that as a result of this equivalence, the notation $\bar s$ makes sense for any standard letter $s$, with the understanding that $\bar{ \bar s} = s$.

We will be interested in the set of cyclic words $W_{\frkA}$ on the alphabet $\frkA$ that are consistent with the puzzle-piece notation of of Figure {\ref{fig:puzzle_pieces}, which encodes the restriction $(\star)$. Thus, for example, the cyclic word $c_1$ is an element of $W_{\frkA}$ while the cyclic word $a_1$ is not. Another immediate restriction is that  an element of $W_{\frkA}$ must contain  an equal number of $a$-type letters and $b$-type letters.

\begin{prop}\label{prp:lp-std}
A loop $\lp$ with at least one $\iota_0$ generator may be represented as a cyclic word $\w_{\lp}$ in $W_{\frkA}$. Moreover, this representation is unique up to overall reversal of the word $\w_{\lp}$; that is, writing $\w_{\lp}$ with the opposite cyclic ordering and replacing each letter $s$ with $\bar s$.

\end{prop}
\begin{proof}Immediate from the definitions.\end{proof}

As a result, we will be interested in the equivalence class which identifies $\w_{\lp}$  and $\overline{\w}_{\lp}$, where $\overline{\w}_{\lp}$ denotes the reversal of $\w_{\lp}$. Denote this equivalence class of cyclic words by $(\w_{\lp})$; by abuse we will continue referring to $(\w_{\lp})$ as a cyclic word. Note that by Proposition \ref{prp:lp-std} this sets up a one-to-one correspondence $\lp \leftrightarrow (\w_{\lp})$ between loops with at least one $\iota_0$-decorated vertex and (equivalence classes of) cyclic words in $W_{\frkA}$.

\begin{definition}\label{def:lp-std} 
When $\lp$ is represented by a cyclic word $(\w_{\lp})$ using the standard alphabet, we say $\lp$ is written in standard notation.
\end{definition}

A loop cannot be written in standard notation if it does not contain an $\iota_0$-decorated vertex. This suggests it will sometimes be useful break a loop along $\iota_1$-decorated vertices. There are again five types of chains possible as listed in Figure \ref{fig:puzzle_pieces} (on the right hand side). Denote the dual alphabet by \[\mathfrak{A}^*=\{a_i^*,b_i^*,c_j^*,d_j^*\} \] ($i\in\Z\smallsetminus\{0\}$ and $j\in\Z$) where, as before, $a_{-i}^*=\bar a_i^*$, $b_{-i}^*=\bar b_i^*$ (for $i>0$), $c_{-j}^*=\bar d_j^*$, $d_{-j}^*=\bar c_j^*$ (for $j\ge0$), and $d_0^*=e^*=\bar c_0^*$. Proceeding as before, let $W_{\frkA^*}$ denote the set of cyclic words in the alphabet $\frkA^*$ which are consistent with the puzzle-piece notation. 

\begin{prop}\label{prp:lp-std}
A loop $\lp$ with at least one $\iota_1$ generator may be represented as a cyclic word $\w_{\lp}^*$ in $W_{\frkA^*}$. Moreover, this representation is unique up to overall reversal of the word $\w_{\lp}^*$. 
\end{prop}
\begin{proof}Immediate from the definitions.\end{proof}

Again, we denote by $(\w_{\lp}^*)$ the equivalence class identifying $\w_{\lp}^*$  and $\overline{\w}^*_{\lp}$, and remark that this gives rise to a second one-to-one correspondence $\lp \leftrightarrow (\w_{\lp}^*)$ for loops $\lp$ with at least one $\iota_1$-decorated vertex.

\begin{definition}\label{def:lp-dul}
When $\lp$ is represented by a cyclic word $(\w_{\lp}^*)$ using the dual alphabet, we say $\lp$ is written in dual notation.
\end{definition}

Collecting these observations, notice that whenever $\lp$ contains instances of both idempotents the pair of correspondences $(\w_{\lp}^*)\leftrightarrow\lp \leftrightarrow (\w_{\lp})$ sets up a natural map between the standard representation and the dual representation. In particular, this allows us to set $(\w_{\lp})^*=(\w_{\lp}^*)$ in a well defined way; we refer to $(\w_{\lp})^*$ as the dual of $(\w_{\lp})$ and remark that $(\w_{\lp})^{**}=(\w_{\lp})$.

This description of loops as cyclic words is the essential starting point for loop calculus. A typical loop will have vertices in both idempotents, it is thus expressible in both standard and dual notation; switching between the two notations will be a key part of the loop calculus. We now make this  process explicit.

First notice that, given $(\w_{\lp})$ representing a loop $\lp$ in standard notation, there is a natural normal form $(\w_{\lp})=( {\bf u}_1 {\bf v}_1{\bf u}_2 {\bf v}_2 \ldots {\bf u}_{n} {\bf v}_n )$ where
\begin{itemize}
\item[(N1)] the subword ${\bf u}_i$ is a standard letter with subscript $k_i \neq 0$; and 
\item[(N2)] the subword ${\bf v}_i$ is $n_i\ge0$ consecutive copies of either $d_0$ or $c_0$  (here $n_i$ may be zero).
\end{itemize}
This normal form makes sense for any $\w_{\lp}$ that does not consist only of $d_0=e$ or $c_0=\bar e$ letters; we may safely ignore these sporadic examples since, in these cases, $\lp$ has no $\iota_1$-decorated vertices and cannot be expressed in dual notation.

Now the dual word $\w^*_{\lp}$ is obtained as follows.

\begin{itemize}
\item[(D1)] Replace each ${\bf v}_i$ by a dual letter with subscript $n_i + 1$ and type\footnote{Recall that a letter of type $a$, $b$, $c$, or $d$ with negative subscript can also be written as a letter of type $\bar a$, $\bar b$, $\bar d$, or $\bar c$ with positive subscript. Here the \emph{type} of a standard letter (other than $d_0$ and $c_0$) refers to the element of $\{a, \bar a, b, \bar b, c, \bar c, d, \bar d\}$ corresponding to the representation with positive subscript. Similarly, a dual letter other than $d^*_0$ and $c^*_0$ has a well-defined type in $\{a^*, \bar a^*, b^*, \bar b^*, c^*, \bar c^*, d^*, \bar d^*\}$.}
 determined by (the types of) the ordered pair  $\{{\bf u}_{i}, {\bf u}_{i+1}\}$ (written ${\bf u}_{i}{\bf u}_{i+1}$ for brevity) from the subword ${\bf u}_{i}{\bf v}_i{\bf u}_{i+1}$ according to the following table:
\[
\begin{array}{c @{\hskip 2 cm} c}
\bar{a}b, \bar{a}d, \bar{c}b, \bar{c}d \to a^* &  \bar{b}a, \bar{b}c, \bar{d}a, \bar{d}c \to \bar{a}^* \\
a\bar{b}, a\bar{c}, d\bar{b}, d\bar{c} \to b^* & b\bar{a}, b\bar{d}, c\bar{a}, c\bar{d} \to \bar{b}^* \\
\bar{a}\bar{b}, \bar{a}\bar{c}, \bar{c}\bar{b}, \bar{c}\bar{c} \to c^* & ba, bc, ca, cc \to \bar{c}^* \\
ab, ad, db, dd \to d^* & \bar{b}\bar{a}, \bar{b}\bar{d}, \bar{d}\bar{a}, \bar{d}\bar{d} \to \bar{d}^*
\end{array}
\] Note that indices are taken mod $n$, so ${\bf u}_{n+1}$ is identified with ${\bf u}_1$.
\item[(D2)]  Replace each ${\bf u}_i$ with the subword consisting of $k_i - 1$ consecutive $d^*_0$ letters if $k_i > 0$ or $1 - k_i$ consecutive $c^*_0$ letters if $k_i < 0$.
\end{itemize}

To convert from dual notation to standard notation, we use exactly the same procedure (interchanging the words {\em standard} and {\em dual}  and adding/removing stars where appropriate in the discussion above). Note that letters in dual (resp. standard) notation correspond to consecutive pairs of letters in standard (resp. dual) notation after we ignore letters with subscript 0. We make the following observation about this correspondence:

\begin{obs}\label{obs:dual_stable_chains}
Stable chains in dual (resp. standard) notation correspond precisely to consecutive pairs of letters in standard (resp. dual) notation, ignoring letters with subscript 0, whose subscripts have opposite signs.
\end{obs}

With the forgoing in place we will not distinguish between loops in standard or in dual notation; for a given loop $\lp$ containing (vertices decorated by) both idempotents, it is always possible to choose a representative for $\lp$ in either of the alphabets $\mathfrak{A}$ or $\mathfrak{A}^*$. In summary: We will regard a loop as both a graph-theoretic object and as an equivalence class of words $(\w_{\lp})$ and $(\w_{\lp}^*)$ modulo dualizing. In particular, we will adopt the abuse of notation that $\lp$ is such an equivalence class of words, where convenient, and think  of loops as graph-theoretic objects and word-theoretic objects interchangeably. In particular, we will let the notation $\lp$ stand in for a choice of cyclic word representative (in either alphabet). 

We will often refer to a sub-words  $\w$ of a given loop $\lp$, so the notation $(\w)$ indicating the cyclic closure of a word (when it exists) will be used to distinguish sub-words from (equivalence classes of) cyclic words. The length of a sub-word is the number of pieces (or, letters) in the word. 

\parpic[r]{\begin{tikzpicture}[>=latex, thick] 
\node (a) at (2,0) {$\bullet$};
\node (b) at (1,0) {$\circ$};
\node (c) at (0,0) {$\bullet$};
\node (d) at (0,1) {$\circ$};
\node (e) at (0,2) {$\bullet$};
\node (f) at (1,2) {$\circ$};
\node (g) at (2,1) {$\circ$};
\draw[->] (a) to node[below]{$\rho_3$} (b);
\draw[->] (b) to node[below]{$\rho_2$} (c);
\draw[->] (c) to node[left]{$\rho_{123}$} (d);
\draw[->] (e) to node[left]{$\rho_1$} (d);
\draw[->] (f) to node[above]{$\rho_2$} (e);
\draw[->] (g) to node[above right]{$\rho_{23}$} (f);
\draw[->] (a) to node[right]{$\rho_{123}$} (g);
\end{tikzpicture}} 
To conclude with a particular example, let $M$ denote the complement of the left-hand trefoil and consider the bordered manifold $(M,\mu,\lambda)$, where $\mu$ is the knot meridian and $\lambda$ is the Seifert longitude. Following \cite[Chapter 11]{LOT-bordered}, $\CFD(M,\mu,\lambda)$ is described by a loop, as shown on the right.  This may be expressed in standard notation as $(a_1b_1\bar{d}_2)$ or in dual notation as $(a^*_1e^*b^*_1\bar{d}^*_1)$.  In particular, according to the discussion above, this pair of words is a representative of the same equivalence  specifying the  loop $\lp$ shown. That is, we regard $\lp\sim(a_1b_1\bar{d}_2)\sim(a^*_1e^*b^*_1\bar{d}^*_1)$. More generally, the homotopy type of the invariant $\CFD(M,\mu,\lambda-(2+n)\mu)$  is represented by the loop $(a_1b_1c_n)$ for $n\in\Z$, following the convention that $c_0=\bar e$ and $c_n=\bar d_{-n}$ when $n<0$.

\subsection{Operations on loops}\label{subsec:operations} We now define two abstract operations on loops: $\tw$ and $\du$. These operations are easy to describe for loops in terms of standard and dual notation, respectively, and we will see in the next subsection that they correspond to important bordered Floer operations.

If a loop $\lp$ cannot be written in standard notation (that is, it is a collection of only $e^*$ segments), then $\tw(\lp) = \lp$. Otherwise, express $\lp$ in standard notation and consider the operation $\tw$ determined on individual letters via
\[\tw(a_i) = a_i, \tw(b_i) = b_i, \tw(c_j)=c_{j-1}, \tw(d_j)=d_{j+1}\]
for any $i \in \Z \smallsetminus \{0\}$ and $j\in\Z$. For collections of loops, we also define $\tw(\{\lp_i\}_{i=1}^n)=\{\tw(\lp_i)\}_{i=1}^n$. The operation $\tw$ is invertible; denote the inverse $\twi$.

Note that the letters $a_i$ and $b_i$ are fixed by $\tw$. As a result, these are sometimes refered to as stable chains (or, standard stable chains). The (standard) unstable chains are the letters $c_j$ and $d_j$. This is consistent with the notion of stable and unstable chains from \cite{LOT-bordered}.

The operator $\du$ is defined similarly, but with respect to dual notation. If a loop $\lp$ cannot be written in dual notation, then $\du(\lp) = \lp$. Otherwise, express $\lp$ in dual notation and consider the operation $\du$ on $\lp$ defined on individual dual letters via
\[\du(a_i^*) = a_i^*, \du(b_i^*) = b_i^*, \du(c_j^*)=c_{j-1}^*, \du(d_j^*)=d_{j+1}^*\]
for any $i \in \Z \smallsetminus \{0\}$ and $j\in\Z$. For collections of loops, we also define $\du(\{\lp_i\}_{i=1}^n)=\{\du(\lp_i)\}_{i=1}^n$. As with $\tw$, $\du$ is an invertible operation with inverse $\du^{-1}$. 

Note that $a^*_i$ and $b^*_i$ are fixed by $\du$. When we have need for it, $a^*$- and $b^*$-type letters will be referred to as  dual stable chains, while $c^*$- and $d^*$-type letters will be referred to as dual unstable chains.

The operation $\tw$ is easy to define for loops in standard notation, but it would be difficult to describe purely in terms of dual letters. Similarly, the operation $\du$ is simple to define in dual notation, but would be complicated in terms of standard letters. This suggests the need to comfortably switch between the two notations. In particular, given a loop $\lp$ expressed in standard notation, finding $\du(\lp)$ in standard notation is a three step process: dualize, apply $\du$, and dualize again.

We will often need to apply combinations of the operations $\tw$ and $\du$ to a loop. The composition $\tw \circ \dui \circ \tw$, in particular, appears often; it will be convenient to regard this composition as another loop operation, which we call $\ex$. The following Lemma describes the effect of $\ex$ on a cyclic word in standard notation; this provides a convenient shortcut compared with computing the operations $\tw$, $\dui$, and $\tw$ individually. We state the Lemma for general loops, but we will only prove it in a special case.

\begin{lem}\label{lem:ex_operation}
If $\lp$ is written in standard notation, then $\ex(\lp)$ is determined by the following action of $\ex$ on standard letters:
\[\ex(a_k) = a^*_{-k}, \ex(b_k) = b^*_{-k}, \ex(c_k)=c^*_{-k}, \ex(d_k)=d^*_{-k}\]
\end{lem}
\begin{proof}
We give the proof in the special case that $\lp$ consists only of $d_k$ segments with $k\ge 0$. The general proof is left to the reader.

Let $\lp = (d_{k_1} d_{k_2} \ldots d_{k_n})$, with $k_i \ge 0$. Then we have $\tw(\lp) = (d_{k_1+1} \ldots d_{k_n+1})$. Writing this in dual notation and applying $\dui$, we have
\[\tw(\lp) = (d^*_1 \underbrace{d^*_0 \ldots d^*_0}_{k_1} d^*_1 \underbrace{d^*_0 \ldots d^*_0}_{k_2} \ldots d^*_1 \underbrace{d^*_0 \ldots d^*_0}_{k_n}) \]
\[ \dui \circ \tw(\lp) = (d^*_0 \underbrace{d^*_{-1} \ldots d^*_{-1}}_{k_1} d^*_0 \underbrace{d^*_{-1} \ldots d^*_{-1}}_{k_2} \ldots d^*_0 \underbrace{d^*_{-1} \ldots d^*_{-1}}_{k_n}) \]
Dualizing and twisting again gives
\[ \dui \circ \tw(\lp) = (c_2 \underbrace{c_1 \ldots c_1}_{k_1-1} c_2 \underbrace{c_1 \ldots c_1}_{k_2-1} \ldots c_2 \underbrace{c_1 \ldots c_1}_{k_n-1}) \]
\[ \tw \circ \dui \circ \tw(\lp) = (c_1 \underbrace{c_0 \ldots c_0}_{k_1-1} c_1 \underbrace{c_0 \ldots c_0}_{k_2-1} \ldots c_1 \underbrace{c_0 \ldots c_0}_{k_n-1}) \]
\[ \tw \circ \dui \circ \tw(\lp) = (d^*_{-k_1} d^*_{-k_2} \ldots d^*_{-k_n}) \qedhere\]
\end{proof}

As suggested by the example of the left-hand trefoil exterior, expressing loops as cyclic words gives rise to clean way of describing a type D structure and its reparametrization. Recall that when $M$ is the exterior of the left-hand trefoil, $\CFD(M,\mu,\lambda-(2+n)\mu)$ is represented by the loop $(a_1 b_1 c_n)$ for $n\in\Z$. This loop is simply $\tw^{-n}(\lp)$, where $\lp=(a_1 b_1 c_0)$.

\subsection{Dehn twists} The operators $\tw$ and $\du$ naturally encode the effect of a Dehn twist on a loop representing the type D structure of a bordered manifold. Recall that \[\Ts^{\pm1}\boxtimes\CFD(M,\alpha,\beta)\cong\CFD(M,\alpha,\beta\pm\alpha)\] and, more generally, for any type D structure $N$  over $\Alg$ the pair of type D structures $\Ts^{\pm1}\boxtimes N$ is well defined.

\begin{figure}[ht!]
\begin{tikzpicture}[>=latex] 
\node at (0,0) {$\ast$}; 
\node at (-2,3) {$\bullet$}; 	
\node at (2,3) {$\circ$};
\draw[thick, shorten <=0.1cm] (2,3) arc (180:360:0.5);\draw[thick, ->] (3,3) arc (0:173:0.5); 
\draw[thick, shorten <=3pt, shorten >= 3pt, ->] (-2,3) to[bend left, looseness=0.75] (2,3); 
\draw[thick, shorten <=3pt, shorten >= 3pt, ->] (0,0) to[bend right, looseness=1] (-2,3);
\draw[thick, shorten <=3pt, shorten >= 3pt, <-] (0,0) to[bend left, looseness=1] (-2,3); 
\draw[thick, shorten <=3pt, shorten >= 3pt, <-] (0,0) to[bend right, looseness=1] (2,3);
\draw[thick, shorten <=3pt, shorten >= 3pt, ->] (0,0) to[bend left, looseness=1] (2,3); 
\node at (2.5,3.7) {$\scriptstyle \rho_{23}\otimes\rho_{23}$};
\node at (-0.3,4.25) {$\scriptstyle \rho_1\otimes \rho_1$};
	\node at (0,3.95) {$\scriptstyle +\rho_{123}\otimes \rho_{123}$};
	\node at (0.12,3.65) {$\scriptstyle +\rho_3\otimes (\rho_3,\rho_{23})$};	
\node at (-1.1,1.9) {$\scriptstyle \rho_{2}\otimes1$};
\node at (-2.1,0.8)  {$\scriptstyle \rho_{123}\otimes\rho_{12}$};\node at (-1.9,0.5)  {$\scriptstyle +\rho_{3}\otimes(\rho_{3},\rho_{2})$};
\node at (1.1,1.9) {$\scriptstyle 1\otimes\rho_{3}$};
\node at (1.9,0.8) {$\scriptstyle \rho_{23}\otimes\rho_{2}$};
\node at (7,0) {$\ast$}; 
\node at (5,3) {$\bullet$}; 	
\node at (9,3) {$\circ$};
\draw[thick, ->, shorten <=0.1cm] (5,3) arc (0:352:0.5);
\draw[thick, shorten <=0.1cm] (9,3) arc (180:360:0.5);\draw[thick, ->] (10,3) arc (0:173:0.5); 
\draw[thick, shorten <=3pt, shorten >= 3pt, ->] (5,3) to[bend left, looseness=0.75] (9,3); 
\draw[thick, shorten <=3pt, shorten >= 3pt, <-] (5,3) to (9,3);
\draw[thick, shorten <=3pt, shorten >= 3pt, ->] (7,0) to[bend right, looseness=1] (5,3);
\draw[thick, shorten <=3pt, shorten >= 3pt, <-] (7,0) to[bend left, looseness=1] (5,3); 
\draw[thick, shorten <=3pt, shorten >= 3pt, <-] (7,0) to[bend right, looseness=1] (9,3);
\draw[thick, shorten <=3pt, shorten >= 3pt, ->] (7,0) to[bend left, looseness=1] (9,3); 
\node at (4.5,3.7) {$\scriptstyle \rho_{12}\otimes(\rho_{123}\otimes\rho_2)$};
\node at (9.5,3.7) {$\scriptstyle \rho_{23}\otimes\rho_{23}$};
\node at (6.7,3.95) {$\scriptstyle \rho_1\otimes \rho_1$};
	\node at (7,3.65) {$\scriptstyle +\rho_{123}\otimes \rho_{123}$};
	\node at (7,2.75) {$\scriptstyle \rho_2\otimes (\rho_{23},\rho_{2})$};	
\filldraw [fill=white,white] (6,1.7) rectangle (7,2);\node at (6.2,1.9) {$\scriptstyle \rho_{2}\otimes(\rho_3,\rho_2)$};
\node at (4.8,0.8)  {$\scriptstyle \rho_{3}\otimes1+\rho_1\otimes\rho_{12}$};
\node at (8.15,1.9) {$\scriptstyle \rho_{23}\otimes\rho_{3}$};
\node at (8.8,0.8) {$\scriptstyle 1\otimes\rho_{2}$};
\end{tikzpicture}		
\caption{Graphical representations of the Dehn twist bimodules $\Ts$ (left) and $\Ts^{-1}$ (right), following \cite[Section 10]{LOT-bimodules}.} 	\label{fig:Dehn-twist-bimodules}
\end{figure}

\begin{prop}
\label{prop:effect_of_Tstd}
If $\lp$ is a loop with corresponding type D structure $N_{\lp}$ then $N_{\lp}^{\pm} = \Ts^{\pm1} \boxtimes N_{\lp}$ is a loop-type module represented by the loop $\tw^\pm(\lp)$.
\end{prop}

\begin{figure}
------------------------------------------------------------------------------------------------------
\begin{tikzpicture}[>=latex]

\node (a) at (0,2) {$\bullet$};
\node (b) at (2,2) {$\circ$};
\node (c) at (4,2) {$\circ$};
\node (d) at (6,2) {$\circ$};
\node (e) at (8,2) {$\bullet$};
\draw[->] (a) to node[above]{$\rho_3$} (b);
\draw[->] (b) to node[above]{$\rho_{23}$} (c);
\draw[->, dashed] (c) to node[above]{$\rho_{23}$} (d);
\draw[->] (d) to node[above]{$\rho_2$} (e);

\node at (4,1.3) {$\Downarrow$};

\node (ap) at (-.2,.5) {$\bullet$};
\node (ar) at (.2,-.5) {$\circ$};
\node (bq) at (2,0) {$\circ$};
\node (cq) at (4,0) {$\circ$};
\node (dq) at (6,0) {$\circ$};
\node (ep) at (8.2,.5) {$\bullet$};
\node (er) at (7.8,-.5) {$\circ$};

\draw[->] (ar) to node[left]{$\rho_2$} (ap);
\draw[->] (er) to node[right]{$\rho_2$} (ep);

\draw[->] (bq) to node[below]{$\rho_{23}$} (cq);
\draw[->, dashed] (cq) to node[below]{$\rho_{23}$} (dq);
\draw[->] (dq) to node[below]{$\rho_{23}$} (er);

\draw[->, bend left = 10] (ap) to node[above]{$\rho_{3}$} (cq);
\draw[->, ultra thick] (ar) to (bq);

\end{tikzpicture}

------------------------------------------------------------------------------------------------------
\begin{tikzpicture}[>=latex]
\node (a) at (0,2) {$\bullet$};
\node (b) at (2,2) {$\circ$};
\node (c) at (4,2) {$\bullet$};
\draw[->] (a) to node[above]{$\rho_3$} (b);
\draw[->] (b) to node[above]{$\rho_2$} (c);
\node at (2,1.3) {$\Downarrow$};

\node (ap) at (-.2,.5) {$\bullet$};
\node (ar) at (.2,-.5) {$\circ$};
\node (bq) at (2,0) {$\circ$};
\node (cp) at (4.2,.5) {$\bullet$};
\node (cr) at (3.8,-.5) {$\circ$};
\draw[->] (ar) to node[left]{$\rho_2$} (ap);
\draw[->] (cr) to node[right]{$\rho_2$} (cp);
\draw[->] (bq) to node[below]{$\rho_{23}$} (cr);
\draw[->, bend left = 20] (ap) to node[above]{$\rho_{3}$} (cr);
\draw[->, ultra thick] (ar) to (bq);

\draw (6,-.5) -- (6,2.5);
\end{tikzpicture} \qquad
\begin{tikzpicture}[>=latex]

\node (a) at (0,2) {$\bullet$};
\node (b) at (2,2) {$\circ$};
\node (c) at (4,2) {$\circ$};
\node (d) at (6,2) {$\bullet$};
\draw[->] (a) to node[above]{$\rho_{123}$} (b);
\draw[->, dashed] (b) to node[above]{$\rho_{23}$} (c);
\draw[->] (d) to node[above]{$\rho_1$} (c);

\node at (3,1.3) {$\Downarrow$};

\node (ap) at (.2,.5) {$\bullet$};
\node (ar) at (-.2,-.5) {$\circ$};
\node (bq) at (2,0) {$\circ$};
\node (cq) at (4,0) {$\circ$};
\node (dp) at (5.8,.5) {$\bullet$};
\node (dr) at (6.2,-.5) {$\circ$};

\draw[->] (ar) to node[left]{$\rho_2$} (ap);
\draw[->] (dr) to node[right]{$\rho_2$} (dp);

\draw[->] (ap) to node[above]{$\rho_{123}$} (bq);
\draw[->, dashed] (bq) to node[below]{$\rho_{23}$} (cq);
\draw[->] (dp) to node[above]{$\rho_1$} (cq);

\end{tikzpicture}

------------------------------------------------------------------------------------------------------
\begin{tikzpicture}[>=latex]

\node (a) at (0,2) {$\bullet$};
\node (b) at (2,2) {$\circ$};
\node (c) at (4,2) {$\circ$};
\node (d) at (6,2) {$\circ$};
\node (e) at (8,2) {$\bullet$};
\draw[->] (a) to node[above]{$\rho_3$} (b);
\draw[->] (b) to node[above]{$\rho_{23}$} (c);
\draw[->, dashed] (c) to node[above]{$\rho_{23}$} (d);
\draw[->] (e) to node[above]{$\rho_1$} (d);

\node at (4,1.3) {$\Downarrow$};

\node (ap) at (-.2,.5) {$\bullet$};
\node (ar) at (.2,-.5) {$\circ$};
\node (bq) at (2,0) {$\circ$};
\node (cq) at (4,0) {$\circ$};
\node (dq) at (6,0) {$\circ$};
\node (ep) at (7.8,.5) {$\bullet$};
\node (er) at (8.2,-.5) {$\circ$};

\draw[->] (ar) to node[left]{$\rho_2$} (ap);
\draw[->] (er) to node[right]{$\rho_2$} (ep);

\draw[->] (bq) to node[below]{$\rho_{23}$} (cq);
\draw[->, dashed] (cq) to node[below]{$\rho_{23}$} (dq);
\draw[->] (ep) to node[below]{$\rho_1$} (dq);

\draw[->, bend left = 10] (ap) to node[above]{$\rho_{3}$} (cq);
\draw[->, ultra thick] (ar) to (bq);

\end{tikzpicture}

------------------------------------------------------------------------------------------------------

\begin{tikzpicture}[>=latex]
\node (a) at (0,2) {$\bullet$};
\node (b) at (2,2) {$\circ$};
\node (c) at (4,2) {$\bullet$};
\draw[->] (a) to node[above]{$\rho_3$} (b);
\draw[->] (c) to node[above]{$\rho_1$} (b);
\node at (2,1.3) {$\Downarrow$};

\node (ap) at (-.2,.5) {$\bullet$};
\node (ar) at (.2,-.5) {$\circ$};
\node (bq) at (2,0) {$\circ$};
\node (cp) at (3.8,.5) {$\bullet$};
\node (cr) at (4.2,-.5) {$\circ$};
\draw[->] (ar) to node[left]{$\rho_2$} (ap);
\draw[->] (cr) to node[right]{$\rho_2$} (cp);
\draw[->] (cp) to node[below]{$\rho_1$} (bq);
\draw[->, ultra thick] (ar) to (bq);

\draw (6,-.5) -- (6,2.5);

\end{tikzpicture} \qquad
\begin{tikzpicture}[>=latex]

\node (a) at (0,2) {$\bullet$};
\node (b) at (2,2) {$\circ$};
\node (c) at (4,2) {$\circ$};
\node (d) at (6,2) {$\bullet$};
\draw[->] (a) to node[above]{$\rho_{123}$} (b);
\draw[->, dashed] (b) to node[above]{$\rho_{23}$} (c);
\draw[->] (c) to node[above]{$\rho_2$} (d);

\node at (3,1.3) {$\Downarrow$};

\node (ap) at (.2,.5) {$\bullet$};
\node (ar) at (-.2,-.5) {$\circ$};
\node (bq) at (2,0) {$\circ$};
\node (cq) at (4,0) {$\circ$};
\node (dp) at (6.2,.5) {$\bullet$};
\node (dr) at (5.8,-.5) {$\circ$};

\draw[->] (ar) to node[left]{$\rho_2$} (ap);
\draw[->] (dr) to node[right]{$\rho_2$} (dp);

\draw[->] (ap) to node[above]{$\rho_{123}$} (bq);
\draw[->, dashed] (bq) to node[below]{$\rho_{23}$} (cq);
\draw[->] (cq) to node[above]{$\rho_{23}$} (dr);

\end{tikzpicture}

------------------------------------------------------------------------------------------------------
\begin{tikzpicture}[>=latex]
\node (a) at (0,2) {$\bullet$};
\node (b) at (2,2) {$\bullet$};
\draw[->] (a) to node[above]{$\rho_{12}$} (b);
\node at (1,1.3) {$\Downarrow$};

\node (ap) at (.2,.5) {$\bullet$};
\node (ar) at (-.2,-.5) {$\circ$};
\node (bp) at (2.2,.5) {$\bullet$};
\node (br) at (1.8,-.5) {$\circ$};
\draw[->] (ar) to node[left]{$\rho_2$} (ap);
\draw[->] (br) to node[right]{$\rho_2$} (bp);
\draw[->] (ap) to node[below]{$\rho_{123}$} (br);

\end{tikzpicture}
------------------------------------------------------------------------------------------------------

\caption{Illustrating the proof of Proposition \ref{prop:effect_of_Tstd}: The effect of box tensoring with $\Ts$ on each of the possible segments occurring in a loop expressed in standard notation. Unmarked edges, which are eliminated using edge reduction described in Section \ref{sub:typeD}, are highlighted.}
\label{fig:effect_of_tau_m}

\end{figure}

\begin{figure}
------------------------------------------------------------------------------------------------------

\begin{tikzpicture}[>=latex]

\node (a) at (0,2) {$\circ$};
\node (b) at (2,2) {$\bullet$};
\node (c) at (4,2) {$\bullet$};
\node (d) at (6,2) {$\bullet$};
\node (e) at (8,2) {$\circ$};
\draw[->] (a) to node[above]{$\rho_2$} (b);
\draw[->] (b) to node[above]{$\rho_{12}$} (c);
\draw[->, dashed] (c) to node[above]{$\rho_{12}$} (d);
\draw[->] (d) to node[above]{$\rho_1$} (e);

\node at (4,1.3) {$\Downarrow$};

\node (ap) at (-.2,.5) {$\circ$};
\node (ar) at (.2,-.5) {$\bullet$};
\node (bq) at (2,0) {$\bullet$};
\node (cq) at (4,0) {$\bullet$};
\node (dq) at (6,0) {$\bullet$};
\node (ep) at (8.2,.5) {$\circ$};
\node (er) at (7.8,-.5) {$\bullet$};

\draw[->] (ar) to node[left]{$\rho_1$} (ap);
\draw[->] (er) to node[right]{$\rho_1$} (ep);

\draw[->] (bq) to node[below]{$\rho_{12}$} (cq);
\draw[->, dashed] (cq) to node[below]{$\rho_{12}$} (dq);
\draw[->] (dq) to node[below]{$\rho_{12}$} (er);

\draw[->, bend left = 10] (ap) to node[above]{$\rho_{2}$} (cq);
\draw[->, ultra thick] (ar) to (bq);

\end{tikzpicture}

------------------------------------------------------------------------------------------------------

\begin{tikzpicture}[>=latex]
\node (a) at (0,2) {$\circ$};
\node (b) at (2,2) {$\bullet$};
\node (c) at (4,2) {$\circ$};
\draw[->] (a) to node[above]{$\rho_2$} (b);
\draw[->] (b) to node[above]{$\rho_1$} (c);
\node at (2,1.3) {$\Downarrow$};

\node (ap) at (-.2,.5) {$\circ$};
\node (ar) at (.2,-.5) {$\bullet$};
\node (bq) at (2,0) {$\bullet$};
\node (cp) at (4.2,.5) {$\circ$};
\node (cr) at (3.8,-.5) {$\bullet$};
\draw[->] (ar) to node[left]{$\rho_1$} (ap);
\draw[->] (cr) to node[right]{$\rho_1$} (cp);
\draw[->] (bq) to node[below]{$\rho_{12}$} (cr);
\draw[->, bend left = 20] (ap) to node[above]{$\rho_{2}$} (cr);
\draw[->, ultra thick] (ar) to (bq);

\draw (6,-.5) -- (6,2.5);
\end{tikzpicture} \qquad
\begin{tikzpicture}[>=latex]

\node (a) at (0,2) {$\circ$};
\node (b) at (2,2) {$\bullet$};
\node (c) at (4,2) {$\bullet$};
\node (d) at (6,2) {$\circ$};
\draw[->] (b) to node[above]{$\rho_{3}$} (a);
\draw[->, dashed] (b) to node[above]{$\rho_{12}$} (c);
\draw[->] (d) to node[above]{$\rho_{123}$} (c);

\node at (3,1.3) {$\Downarrow$};

\node (ap) at (.2,.5) {$\circ$};
\node (ar) at (-.2,-.5) {$\bullet$};
\node (bq) at (2,0) {$\bullet$};
\node (cq) at (4,0) {$\bullet$};
\node (dp) at (5.8,.5) {$\circ$};
\node (dr) at (6.2,-.5) {$\bullet$};

\draw[->] (ar) to node[left]{$\rho_1$} (ap);
\draw[->] (dr) to node[right]{$\rho_1$} (dp);

\draw[->] (bq) to node[above]{$\rho_{3}$} (ap);
\draw[->, dashed] (bq) to node[below]{$\rho_{12}$} (cq);
\draw[->] (cq) to node[above]{$\rho_{123}$} (dp);

\end{tikzpicture}

------------------------------------------------------------------------------------------------------

\begin{tikzpicture}[>=latex]

\node (a) at (0,2) {$\circ$};
\node (b) at (2,2) {$\bullet$};
\node (c) at (4,2) {$\bullet$};
\node (d) at (6,2) {$\circ$};
\draw[->] (b) to node[above]{$\rho_{3}$} (a);
\draw[->, dashed] (b) to node[above]{$\rho_{12}$} (c);
\draw[->] (c) to node[above]{$\rho_1$} (d);

\node at (3,1.3) {$\Downarrow$};

\node (ap) at (.2,.5) {$\circ$};
\node (ar) at (-.2,-.5) {$\bullet$};
\node (bq) at (2,0) {$\bullet$};
\node (cq) at (4,0) {$\bullet$};
\node (dp) at (6.2,.5) {$\circ$};
\node (dr) at (5.8,-.5) {$\bullet$};

\draw[->] (ar) to node[left]{$\rho_1$} (ap);
\draw[->] (dr) to node[right]{$\rho_1$} (dp);

\draw[->] (bq) to node[above]{$\rho_{3}$} (ap);
\draw[->, dashed] (bq) to node[below]{$\rho_{12}$} (cq);
\draw[->] (cq) to node[above]{$\rho_{12}$} (dr);

\draw (8,-.5) -- (8,2.5);

\end{tikzpicture}\qquad
\begin{tikzpicture}[>=latex]
\node (a) at (0,2) {$\circ$};
\node (b) at (2,2) {$\bullet$};
\node (c) at (4,2) {$\circ$};
\draw[->] (a) to node[above]{$\rho_2$} (b);
\draw[->] (b) to node[above]{$\rho_{123}$} (c);
\node at (2,1.3) {$\Downarrow$};

\node (ap) at (-.2,.5) {$\circ$};
\node (ar) at (.2,-.5) {$\bullet$};
\node (bq) at (2,0) {$\bullet$};
\node (cp) at (3.8,.5) {$\circ$};
\node (cr) at (4.2,-.5) {$\bullet$};
\draw[->] (ar) to node[left]{$\rho_1$} (ap);
\draw[->] (cr) to node[right]{$\rho_1$} (cp);
\draw[->] (bq) to node[below]{$\rho_{123}$} (cp);
\draw[->, ultra thick] (ar) to (bq);
\draw[->, bend left = 10, pos = .3] (ap) to node[above]{$\rho_{23}$} (cp);

\end{tikzpicture}

------------------------------------------------------------------------------------------------------

\begin{tikzpicture}[>=latex]

\node (a) at (0,2) {$\circ$};
\node (b) at (2,2) {$\bullet$};
\node (c) at (4,2) {$\bullet$};
\node (d) at (6,2) {$\bullet$};
\node (e) at (8,2) {$\circ$};
\draw[->] (a) to node[above]{$\rho_2$} (b);
\draw[->] (b) to node[above]{$\rho_{12}$} (c);
\draw[->, dashed] (c) to node[above]{$\rho_{12}$} (d);
\draw[->] (d) to node[above]{$\rho_{123}$} (e);

\node at (4,1.3) {$\Downarrow$};

\node (ap) at (-.2,.5) {$\circ$};
\node (ar) at (.2,-.5) {$\bullet$};
\node (bq) at (2,0) {$\bullet$};
\node (cq) at (4,0) {$\bullet$};
\node (dq) at (6,0) {$\bullet$};
\node (ep) at (7.8,.5) {$\circ$};
\node (er) at (8.2,-.5) {$\bullet$};

\draw[->] (ar) to node[left]{$\rho_1$} (ap);
\draw[->] (er) to node[right]{$\rho_1$} (ep);

\draw[->] (bq) to node[below]{$\rho_{12}$} (cq);
\draw[->, dashed] (cq) to node[below]{$\rho_{12}$} (dq);
\draw[->] (dq) to node[below]{$\rho_{123}$} (ep);

\draw[->, bend left = 10] (ap) to node[above]{$\rho_{2}$} (cq);
\draw[->, ultra thick] (ar) to (bq);

\draw (10,-.5) -- (10,2.5);
\end{tikzpicture}\qquad
\begin{tikzpicture}[>=latex]
\node (a) at (0,2) {$\circ$};
\node (b) at (2,2) {$\circ$};
\draw[->] (a) to node[above]{$\rho_{23}$} (b);
\node at (1,1.3) {$\Downarrow$};

\node (ap) at (-.2,.5) {$\circ$};
\node (ar) at (.2,-.5) {$\bullet$};
\node (bp) at (1.8,.5) {$\circ$};
\node (br) at (2.2,-.5) {$\bullet$};
\draw[->] (ar) to node[left]{$\rho_1$} (ap);
\draw[->] (br) to node[right]{$\rho_1$} (bp);
\draw[->] (ar) to node[below]{$\rho_{3}$} (bp);

\end{tikzpicture}
------------------------------------------------------------------------------------------------------

\caption{Illustrating the proof of Proposition \ref{prop:effect_of_Tdul}: The effect of the Dehn twist $\Td$ on each of the possible segments occurring in a loop expressed in dual notation. Unmarked edges, which are eliminated using edge reduction described in Section \ref{sub:typeD}, are highlighted.}
\label{fig:effect_of_Tdual}

\end{figure}

\begin{proof}
We compute the box tensor product $\Ts\boxtimes N_{\lp}$ by considering the effect on one segment at a time in a standard representative for $\lp$.  

Recall from \cite[Section 10]{LOT-bimodules} (see Figure \ref{fig:Dehn-twist-bimodules}) that the type DA bimodule $\Ts$ has three generators $p,q$, and $r$ (denoted by $\bullet$, $\circ$ and $\ast$, respectively, in Figure \ref{fig:Dehn-twist-bimodules}) with idempotents  determined by
$p = \iota_0p\iota_0, q = \iota_1q\iota_1$, and $r = \iota_1 r \iota_0.$
Thus each generator $x\in\iota_1 N_{\lp}$ gives rise to a single generator $q\otimes x\in\Ts\boxtimes N_{\lp}$, and each generator $x\in\iota_0 N_{\lp}$ gives rise to two generators $p\otimes x$ and $r\otimes x$ in the box tensor product. There is always a $\rho_2$ arrow from $r\otimes x$ to $p\otimes x$, that is, $\partial(r\otimes x)$ has a summand $\rho_2\cdot(p\otimes x)$. 

If generators $x, y \in \iota_0 N_{\lp}$ are connected by a single segment $s$ (in the loop $\lp$), we will consider the portion of a loop representing $N^+_{\lp}$ between $p\otimes x$ and $p\otimes y$ and show that (up to homotopy equivalence) it is the segment $\tw(s)$. To talk about the portion \emph{between}, we need a (cyclic) ordering on the elements of $N^+_{\lp}$. This is inherited from a choice of cyclic ordering on the elements of $\lp$, together with a specified order of $p\otimes x$ and $r\otimes x$ for each $x\in\iota_0 N_{\lp}$. If there is a segment $s$ from $x$ to $y$, we say that $r\otimes x$ is between $p\otimes x$ and $p\otimes y$ if the puzzle piece shape of $s$ on the $x$ end agrees with the shape of $a_k$. That is, if $s$ is $a_k, \bar{a}_k, c_k, \bar{d}_k$, or $\bar{e}$. 

Consider first a segment $a_k$ from $x$ to $y$, where $x,y$ are generators of $\iota_0 N_{\lp}$. The cases of $k = 1$ and $k\ge2$ are slightly different; both are pictured in Figure \ref{fig:effect_of_tau_m}. In either case, the effect of tensoring the segment with $\Ts$ is pictured. There is a differential starting at $r\otimes x$; after cancelling this differential, the result is a segment of type  $a_k = \tw(a_k)$ from $p\otimes x$ to $p\otimes y$.

Consider next a segment $b_k$ from $x$ to $y$. Tensoring the segment with $\Ts$, we see that the portion between $p\otimes x$ and $p\otimes y$ is simply a segment of type $b_k= \tw(b_k)$. Note that the generators $r \otimes x$ and $r\otimes y$ are not included in this new segment; they must be included in the segments on either side.

If $x$ and $y$ are connected by a segment $c_k$, then there is a differential starting at $r\otimes x$. If $k>1$, then cancelling this differential leaves a segment from $p\otimes x$ to $p\otimes y$ of type $c_{k-1} = \tw(c_k)$. If $k =1$, then cancelling the differential produces a new $\rho_{12}$ arrow, and thus there is a segment $\bar{e} = \tw(c_1)$ from $p \otimes x$ to $p \otimes y$. If $x$ and $y$ are connected by segments $d_k$ or $e$, we see in Figure \ref{fig:effect_of_tau_m} that the portion of $N^+_{\lp}$ from $x\otimes p$ to $y\otimes p$ is a segment $d_{k+1} = \tw(d_k)$ or $d_1 = \tw(e)$, respectively.

Segments with the opposite orientation behave the same way. A segment of $\bar{s}$ from $x$ to $y$ is the same as a segment $s$ from $y$ to $x$. In the tensor product, this produces a segment $\tw(s)$ from $p\otimes y$ to $p \otimes x$, or a segment $\overline{\tw(s)} = \tw(\bar{s})$ from $p \otimes x$ to $p \otimes y$.

It remains to check that a loop consisting only of $e^*$ segments represents a type D structure $N$ that is fixed by the action of the standard Dehn twist. This is easy to see, since in this case the only relevant operation in $\Ts$ is
\[ m_2(q, \rho_{23}) = \rho_{23}\cdot q\]
Each generator of $\iota_1 N = N$ gives rise to one generator of $N^+$ and each $\rho_{23}$ arrow in $N$ gives a $\rho_{23}$ arrow in $N^+$.

The case of $\Ts^{-1}$ can be deduced from the case of $\Ts$. Let $N'$ be a type D module represented by the loop $\twi(\lp)$. We have just shown 
$$ \Ts\boxtimes N'\cong N $$
and it follows that
$$ \Ts^{-1} \boxtimes N \cong \Ts^{-1} \boxtimes \Ts \boxtimes N'\cong N' $$
since $\Ts^{-1} \boxtimes \Ts$ is homotopy equivalent to the identity bimodule.
\end{proof}

\begin{prop}\label{prop:effect_of_Tdul}
If $\lp$ is a loop with corresponding type D structure $N_{\lp}$ then $N^{\mp}_{\lp}= \Td^{\mp1} \boxtimes N_{\lp}$ is a loop-type module represented by the loop $\du^\pm(\lp)$.
\end{prop}
\begin{proof}
The proof is similar, with the relevant bimodules $\Td$ and $\Td^{-1}$ shown in Figure \ref{fig:Dehn-twist-bimodules-DUAL}. The result of tensoring $\Td$ with each type of dual segment is shown in Figure \ref{fig:effect_of_Tdual}. We see that $a^*$ and $b^*$ segments are fixed, $c^*_k$ segments become $c^*_{k+1}$ segments, and $d^*_k$ segments become $d^*_{k-1}$ segments. In other words, tensoring $\Td$ with a dual segment ${\bf s}^*$ gives $\dui({\bf s}^*)$. Thus for a loop $\lp$, $\Td\boxtimes \lp$ is the loop $\dui(\lp)$. Since $\Td^{-1}$ is the inverse of $\Td$, we conclude that $\Td^{-1}\boxtimes \lp$ is the loop $\du(\lp)$. \end{proof}

\begin{figure}[ht!]
\begin{tikzpicture}[>=latex] 
\node at (0,0) {$\ast$}; 
\node at (-2,3) {$\circ$}; 	
\node at (2,3) {$\bullet$};
\draw[thick, shorten <=0.1cm] (2,3) arc (180:360:0.5);\draw[thick, ->] (3,3) arc (0:173:0.5); 
\draw[thick, shorten <=3pt, shorten >= 3pt, <-] (-2,3) to[bend left, looseness=0.75] (2,3); 
\draw[thick, shorten <=3pt, shorten >= 3pt, <-] (0,0) to[bend right, looseness=1] (-2,3);
\draw[thick, shorten <=3pt, shorten >= 3pt, ->] (0,0) to[bend left, looseness=1] (-2,3); 
\draw[thick, shorten <=3pt, shorten >= 3pt, ->] (0,0) to[bend right, looseness=1] (2,3);
\draw[thick, shorten <=3pt, shorten >= 3pt, <-] (0,0) to[bend left, looseness=1] (2,3); 
\node at (2.5,3.7) {$\scriptstyle \rho_{12}\otimes\rho_{12}$};
\node at (-0.3,4.25) {$\scriptstyle \rho_3\otimes \rho_3$};
	\node at (0,3.95) {$\scriptstyle +\rho_{123}\otimes \rho_{123}$};
	\node at (0.12,3.65) {$\scriptstyle +\rho_1\otimes (\rho_{12},\rho_{1})$};	
\node at (-1.1,1.9) {$\scriptstyle \rho_{2}\otimes1$};
\node at (-2.1,0.8)  {$\scriptstyle \rho_{123}\otimes\rho_{23}$};\node at (-1.9,0.5)  {$\scriptstyle +\rho_{1}\otimes(\rho_{2},\rho_{1})$};
\node at (1.1,1.9) {$\scriptstyle 1\otimes\rho_{1}$};
\node at (1.9,0.8) {$\scriptstyle \rho_{12}\otimes\rho_{2}$};
\node at (7,0) {$\ast$}; 
\node at (5,3) {$\circ$}; 	
\node at (9,3) {$\bullet$};
\draw[thick, ->, shorten <=0.1cm] (5,3) arc (0:352:0.5);
\draw[thick, shorten <=0.1cm] (9,3) arc (180:360:0.5);\draw[thick, ->] (10,3) arc (0:173:0.5); 
\draw[thick, shorten <=3pt, shorten >= 3pt, <-] (5,3) to[bend left, looseness=0.75] (9,3); 
\draw[thick, shorten <=3pt, shorten >= 3pt, ->] (5,3) to (9,3);
\draw[thick, shorten <=3pt, shorten >= 3pt, <-] (7,0) to[bend right, looseness=1] (5,3);
\draw[thick, shorten <=3pt, shorten >= 3pt, ->] (7,0) to[bend left, looseness=1] (5,3); 
\draw[thick, shorten <=3pt, shorten >= 3pt, ->] (7,0) to[bend right, looseness=1] (9,3);
\draw[thick, shorten <=3pt, shorten >= 3pt, <-] (7,0) to[bend left, looseness=1] (9,3); 
\node at (4.5,3.7) {$\scriptstyle \rho_{23}\otimes(\rho_{2}\otimes\rho_{123})$};
\node at (9.5,3.7) {$\scriptstyle \rho_{12}\otimes\rho_{12}$};
\node at (6.7,3.95) {$\scriptstyle \rho_3\otimes \rho_3$};
	\node at (7,3.65) {$\scriptstyle +\rho_{123}\otimes \rho_{123}$};
	\node at (7,2.75) {$\scriptstyle \rho_2\otimes (\rho_{2},\rho_{12})$};	
\filldraw [fill=white,white] (6,1.7) rectangle (7,2);\node at (6.2,1.9) {$\scriptstyle \rho_{2}\otimes(\rho_2,\rho_1)$};
\node at (4.8,0.8)  {$\scriptstyle \rho_{1}\otimes1+\rho_3\otimes\rho_{23}$};
\node at (8.15,1.9) {$\scriptstyle \rho_{12}\otimes\rho_{1}$};
\node at (8.8,0.8) {$\scriptstyle 1\otimes\rho_{2}$};
\end{tikzpicture}		
\caption{Graphical representations of the Dehn twist bimodules $\Td^{-1}$ (left) and $\Td$ (right), following \cite[Section 10]{LOT-bimodules}.} 	\label{fig:Dehn-twist-bimodules-DUAL}
\end{figure}

We conclude this discussion by observing that the notion of a manifold $M$ (with torus boundary) being of loop-type is now a well defined. In particular, since any reparametrization of a loop gives rise to a loop it follows that the property of loop-type (or, having loop-type bordered invariants) is independent of the bordered structure and hence a property of the underlying (unbordered) manifold. Indeed:
\begin{definition}A compact, orientable, connected three-manifold  $M$ with torus boundary is loop-type if $\CFD(M,\alpha,\beta)$ is of loop-type, for any choice of basis slopes $\alpha$ and $\beta$. \end{definition}

\subsection{Solid tori} As a simple example of the loop operations described above, we now describe the computation of $\CFD$ of a solid torus with arbitrary framing. By a $\frac{p}{q}$-framed solid torus, we will mean a bordered solid torus $(D^2\times S^1, \alpha, \beta)$ such that the meridian is $p\alpha + q\beta$. Recall that for the standard (0-framed) and dual ($\infty$-framed) solid tori we have
\[ \CFD(D^2 \times S^1, l, m) = (e), \hspace {1cm} \CFD(D^2 \times S^1, m, l) = (e^*). \]
The bordered invariants for solid tori with other framings can be computed by applying Dehn twists as described in Section \ref{sec:change_of_framing}.

\begin{example}
We compute $\CFD$ of  the $\frac{7}{2}$- framed solid torus  using the continued fraction expansion $-\frac{2}{7} = [-1,2,-2,3]$:
$$\begin{array}{rccc}
\CFD(D^2\times S^1, l, m) :& (e)&& \\
&\qquad \Downarrow \tw^{3} && \\
\CFD(D^2\times S^1, l, m + 3l) :& (d_3) &\sim& (d^*_1 d^*_0 d^*_0)\\
& && \qquad \Downarrow \du^{-2} \\
\CFD(D^2\times S^1, -2m -5l, m + 3l) :& (c_1 c_1 c_0 c_1 c_0) &\sim& (d^*_{-1} d^*_{-2} d^*_{-2})\\
&\qquad \Downarrow \tw^{2} && \\
\CFD(D^2\times S^1, -2m - 5l, -3m - 7l) :& (c_{-1} c_{-1} c_{-2} c_{-1} c_{-2}) &\sim& (c^*_{-1} c^*_{-1} c^*_{-1} c_0^* c^*_{-1} c^*_{-1} c_0^*)\\
& &&\qquad \Downarrow \du^{-1} \\
\CFD(D^2\times S^1, m+2l, -3m -7l) :& (d_{-4} d_{-3}) &\sim& (c_0^* c_0^* c_0^* c_1^* c_0^* c_0^* c^*_1)
\end{array}$$
\end{example}
It is is easy to check: $m = 7(m+2l) + 2(-3m -7l)$, and so $(D^2\times S^1, m+2l, -3m -7l)$ is a $\frac{7}{2}$-framed solid torus. Note that given $\CFD$ of a solid torus we can also check the framing by using Lemma \ref{prop:rational_longitude}. If we choose the $(\Ztwo)$-grading so that endponts of $d_k$ chains in standard notation have grading $0$ and endpoints of $c_k$ chains have grading $1$, it is not difficult to see that the Euler characteristic $\chi_\bullet$ of a loop in standard notation is the number of $d_k$ segments minus the number of $c_k$ segments and $\chi_\circ$ is given by the sum of the subscripts. Thus $(D^2\times S^1, m+2l, -3m -7l)$ has rational longitude $-\frac{ \chi_\circ}{ \chi_\bullet } = - \frac{-7}{2}$.

$\CFD$ for arbirtrarily framed solid tori can be computed by a similar procedure. The result is always a loop of a particularly simple form.

\begin{lem}\label{lem:solid-torus-restrictions}
If $q \neq 0$, then $\CFD(D^2\times S^1, pm+ql, rm+sl)$ can be represented by a single loop $\lp = (d_{k_1} d_{k_2} \ldots d_{k_m})$. Moreover, the difference between $max\{k_i\}$ and $min\{k_i\}$ is at most 1.
\end{lem}
\begin{proof}
If $p= 0$ (this implies that $q = r = 1$) then
\[ \CFD(D^2\times S^1, l, m + sl) = \Ts^s \boxtimes \CFD(D^2\times S^1, l, m) \sim \tw^s( (e) ) = (d_s) .\]
Otherwise, let $[ a_1, \ldots, a_{2n} ]$ be an even length continued fraction for $\frac{q}{p}$ and choose $a_{2n+1}$ so that $[ a_1, \ldots, a_{2n+1} ]$ is a continued fraction for $\frac{s}{r}$. Then 
\begin{align*}
\CFD(D^2\times S^1,pm+ql,rm+sl) &\cong \Ts^{a_{2n+1}}\boxtimes\cdots\boxtimes\Td^{-a_2} \boxtimes\Ts^{a_1} \boxtimes \CFD(M,l,m) \\
&\sim \tw^{a_{2n+1}}\circ\du^{a_{2n}}\circ\cdots\circ\du^{a_2}\circ\tw^{a_1} ( (d_0) ).
\end{align*}
If $\frac{p}{q}$ is positive we may assume that $a_i>0$ for all $1 \le i \le 2n$. We will show by induction on $n$ that $\du^{a_{2n}}\circ\cdots\circ\du^{a_2}\circ\tw^{a_1} ( (d_0) )$ is a loop in standard notation consisting only of $d_0$ and $d_1$ chains. The base case of $n = 0$ is immediate. Assuming that $\du^{a_{2n-2}}\circ\cdots\circ\du^{a_2}\circ\tw^{a_1} ( (d_0) )$ is a loop in standard notation consisting only of $d_0$ and $d_1$ chains, applying the twist $\tw^{a_{2n-1}}$ produces a loop consisting of $d_k$ chains with only positive subscripts. In dual notation such a loop involves only $d^*_0$ and $d^*_1$ segments. Applying the twist $\du^{a_{2n}}$ produces a loop consisting of $d^*_k$ segments with positive subscripts. Switching to standard notation, this loop contains only $d_0$ and $d_1$ segments. Similarly, if $\frac{p}{q}$ is negative we assume that $a_i<0$ for $1\le i \le n$ and observe by induction that $\du^{a_{2n}}\circ\cdots\circ\du^{a_2}\circ\tw^{a_1} ( (d_0) )$ is a loop in standard notation consisting only of $d_0$ and $d_{-1}$ chains.

Finally, applying the twist $\tw^{a_{2n+1}}$ preserves the fact that the loop consists of type $d_k$ unstable chains. It also preserves the relative differences of the subscripts, so the difference between the maximum and minimum subscript remains at most one.
\end{proof}

\subsection{Abstract fillings and abstract slopes} Recall that a loop $\lp$ represents a type D structure, which by slight abuse of notation we denote by $\lp$. Since $\lp$ is reduced, there is an associated type A structure as described in Section \ref{sub:typeA}, which we denote by $\lp^A$. As a result, given loops $\lp_1$ and $\lp_2$ we can represent the chain complex produced by the box tensor product of the associated modules by $\lp_1^A\boxtimes\lp_2$. Note that  since loops are connected $\Alg$-decorated graphs, the type D and type A structures associated to a loop have a well defined relative $(\Ztwo)$-grading, as described in Section \ref{sub:grading}. The gradings on $\lp_1^A$ and $\lp_2$ induce a relative $(\Ztwo)$-grading on $\lp_1^A\boxtimes\lp_2$.

Consider the loops $\lp_\bullet=(e)$ and $\lp_\circ=(e^*)$. Given any loop $\lp$ we have a pair of chain complexes \[C_{\lp}(\textstyle\frac{1}{0}) = \lp_\bullet^A\boxtimes\lp\ \text{and}\ C_{\lp}(0) = \lp_\circ^A\boxtimes\lp \] which, noting that $\lp_\bullet$  and $\lp_\circ$ represent type $D$ structures of the standard and dual (bordered) solid torus, respectively, might be regarded as abstract standard and dual Dehn filings of $\lp$, respectively.  \begin{remark}Note that we do not need to assume $\lp$ is bounded, since if it is not it is homotopy equivalent to a modified loop which is bounded and has the same box tensor product with $\lp^A_\bullet$. Such a modified loop can be obtained by replacing either a $\rho_{12}$ arrow with the homotopy equivalent (but not reduced) sequence
\[
\begin{tikzpicture}[>=latex] 
		\node at (0,0) {$\bullet$}; 
		\node at (1,0) {$\circ$}; 
		\node at (2,0) {$\circ$};
		\node at (3,0) {$\bullet$}; 
		\draw[thick, ->,shorten >=0.1cm, shorten <=0.1cm] (0,0) -- (1,0); \node at (0.5,0.25) {$\scriptsize\rho_{1}$};
		\draw[thick, <-,shorten >=0.1cm, shorten <=0.1cm] (1,0) -- (2,0); 
		\draw[thick, ->,shorten >=0.1cm, shorten <=0.1cm] (2,0) -- (3,0); \node at (2.5,0.25) {$\scriptsize\rho_{2}$};
	\end{tikzpicture}
\]
or a $\rho_{123}$ arrow with
\[
\begin{tikzpicture}[>=latex] 
		\node at (0,0) {$\bullet$}; 
		\node at (1,0) {$\circ$}; 
		\node at (2,0) {$\circ$};
		\node at (3,0) {$\bullet$}; 
		\draw[thick, ->,shorten >=0.1cm, shorten <=0.1cm] (0,0) -- (1,0); \node at (0.5,0.25) {$\scriptsize\rho_{1}$};
		\draw[thick, <-,shorten >=0.1cm, shorten <=0.1cm] (1,0) -- (2,0); 
		\draw[thick, ->,shorten >=0.1cm, shorten <=0.1cm] (2,0) -- (3,0); \node at (2.5,0.25) {$\scriptsize\rho_{23}$};
	\end{tikzpicture}
\]
If $\lp$ can not be written in standard notation -- that is, it is a collection of $e^*$ segments -- then $H_*(\lp^A_\bullet\boxtimes\lp)$ must have two generators of opposite grading. To see this, replace one $e^*$ segment with
\[
\begin{tikzpicture}[>=latex] 
		\node at (0,0) {$\circ$}; 
		\node at (1,0) {$\bullet$}; 
		\node at (2,0) {$\bullet$};
		\node at (3,0) {$\circ$}; 
		\draw[thick, ->,shorten >=0.1cm, shorten <=0.1cm] (0,0) -- (1,0); \node at (0.5,0.25) {$\scriptsize\rho_{2}$};
		\draw[thick, <-,shorten >=0.1cm, shorten <=0.1cm] (1,0) -- (2,0); 
		\draw[thick, ->,shorten >=0.1cm, shorten <=0.1cm] (2,0) -- (3,0); \node at (2.5,0.25) {$\scriptsize\rho_{3}$};
	\end{tikzpicture}\]
to produce a bounded modified loop.\end{remark}

\parpic[r]{\begin{tikzpicture}[>=latex, thick] 
\node (a) at (2,0) {$\bullet$};
\node (c) at (0,0) {$\bullet$};
\node (e) at (0,2) {$\bullet$};
\draw[->] (a) to (c);
\node (b) at (5,0) {$\circ$};
\node (d) at (4,1) {$\circ$};
\node (f) at (5,2) {$\circ$};
\node (g) at (6,1) {$\circ$};
\draw[->] (f) to (d);
\end{tikzpicture}} 
The standard filling picks out the $\bullet$-idempotent, in practice, and adds a differential for each type $a_k$ chain. The dual filling picks out the $\circ$-idempotent and adds a differential for each $b^*_k$ chain. For instance, when $M$ is the exterior of the left-hand trefoil and $\CFD(M,\mu,\lambda)$ is represented by $\lp\sim (a_1b_1\bar{d}_2) \sim (a^*_1e^*b^*_1\bar{d}^*_1)$, the resulting complexes $\CFhat(S^3)\cong\CFhat(M(\mu))\cong\lp^A_\bullet\boxtimes\lp$ and $\CFhat(M(\lambda))\cong\lp^A_\circ\boxtimes\lp$ are shown on the right.

We regard the chain complexes $C_{\lp}(\frac{1}{0})$ and $C_{\lp}(0)$ as the result of abstract Dehn fillings along a pair of abstract slopes in $\lp$ which we identify with $\infty=\frac{1}{0}$ (corresponding to the standard filling) and $0$ (corresponding to the dual filling). In fact, a given loop $\lp$ gives rise to a natural $\hat\Q$-family of chain complexes: Choosing an even-length continued fraction $\frac{p}{q}=[a_1,a_2\ldots,a_n]$ let $\lp^{\frac{p}{q}}=\du^{a_n} \circ \cdots \circ \tw^{a_3} \circ \du^{a_2} \circ \tw^{a_1}(\lp)$ and define \[C_{\lp}(\textstyle\frac{p}{q}) = \lp_\bullet^A\boxtimes\lp^{\frac{p}{q}}\] We will regard the complex $C_{\lp}(\textstyle\frac{p}{q})$ as an abstract Dehn filling of the loop $\lp$ along the abstract slope $\frac{p}{q} \in \hat\Q$.

The reason for this definition is illustrated as follows:

\begin{prop}
If a given loop $\lp$ represents $\CFD(M, \alpha, \beta)$ for some bordered manifold $(M, \alpha, \beta)$, the chain complex $C_{\lp}(\frac{p}{q})$ is (homotopy equivalent to) the chain complex $\CFhat(M(p\alpha + q\beta ) )$. That is, abstract filling along abstract slopes correpsonds to Dehn filling along slopes whenever $\lp$ describes the type D structure of a bordered three-manifold. 
\end{prop}

\begin{proof}Immediate from the definitions: Note that $\lp_{\frac{p}{q}}$ represents $\CFD(M, p\alpha+q\beta, r\alpha + s\beta)$, where $\frac{p}{q}=[a_1,a_2\ldots,a_n]$ is an even length continued fraction and $\frac{r}{s}=[a_1,a_2\ldots,a_{n-1}]$.\end{proof}

Fixing a relative $(\Ztwo)$-grading for a loop $\lp$, consider the idempotent Euler characteristics $\chi_\bullet(\lp)$ and $\chi_\circ(\lp)$. If $\chi_\bullet(\lp)$ and $\chi_\circ(\lp)$ are not both 0 then $\lp$ has a preferred slope $-\frac{\chi_\circ(\lp)}{\chi_\bullet(\lp)}$; we call this slope the \emph{abstract rational longitude}. Note that by Proposition \ref{prop:rational_longitude}, if $\lp$ represents $\CFD(M, \alpha, \beta)$ for a rational homology solid torus $(M, \alpha, \beta)$ then the abstract rational longitude for $\lp$ is the rational longitude for $(M, \alpha, \beta)$.

Recall that a slope $\frac{p}{q}$ is an L-space slope for a bordered manifold $(M, \alpha, \beta)$ if Dehn filling along the curve $p\alpha + q\beta$ yields an L-space. We end this section by defining a similar notion of L-space slopes for loops.

\begin{definition}\label{def:abstractL}
Given a loop $\lp$, we say an abstract slope $\frac{p}{q}$ in $\hat\Q$ is an L-space slope for $\lp$ if the relatively $(\Ztwo)$-graded chain complex $C_{\lp}(\frac{p}{q})$ is an L-space chain complex in the sense that $\dim H_*(C_{\lp}(\frac{p}{q})) = |\chi(C_{\lp}(\frac{p}{q}))|\ne0$.
\end{definition}

\begin{remark} With these notions in place, we will now drop the modifier {\em abstract} when treating loops despite the fact that a given loop may or may not describe the type D structure of a three-manifold. In particular, we will not make a distinction between slopes and abstract slopes in the sequel.
\end{remark}

By considering loops in the abstract, a particular class of loops is singled out.

\begin{definition}\label{def:solid torus-like}
A  loop $\lp$ is solid torus-like if it may be obtained from the loop $(ee\cdots e)$ via applications of $\tw^{\pm 1}$ and $\du^{\pm 1}$.  
\end{definition}

Note that $\chi_\bullet(ee\cdots e)$ counts the number of $e$ segments appearing and $\chi_\circ(ee\cdots e)=0$ identifies the rational longitude of a solid torus-like loop. In particular, $\lp_\bullet$ (representing $\CFD(D^2\times S^1,l,m)$) is solid torus-like. Justifying the chosen terminology, we have the following behaviour:

\begin{prop}\label{prp:torus like split}
If $\lp$ is solid torus like with $\chi_\circ(\lp)=0$ then
$\lp^A\boxtimes \lp'\cong \bigoplus_{\chi_\bullet(\lp)}\lp^A_\bullet\boxtimes \lp'$
for any loop $\lp'$.
\end{prop}

\begin{proof}
Recall that $\lp_\bullet^A$ has a single generator $x$ and operations
\[m_{3+i}(x, \rho_3, \underbrace{ \rho_{23}, \ldots, \rho_{23}}_{i \text{ times}}, \rho_2) = x \]
so that for generators $x\otimes u,x\otimes v \in \lp^A_\bullet\boxtimes \lp'$, with $x\otimes v$ a summand of $\partial(x\otimes u)$, there must be  \[\delta^{i+2}(u)=\rho_3\otimes\underbrace{ \rho_{23}\otimes \cdots\otimes \rho_{23}}_{i \text{ times}}\otimes \rho_2\otimes v\] in the type D structure for $\lp'$. Now consider the type A structure described by $\lp^A$. This has $n=\chi_\bullet(\lp)$ generators $x_0,\ldots x_{n-1}$ and operations 
\[m_{3+i}(x_j, \rho_3, \underbrace{ \rho_{23}, \ldots, \rho_{23}}_{i \text{ times}}, \rho_2) = x_{i+j+1} \]
where the subscripts are understood to be reduced modulo $n$. (Note that the cyclic ordering on the generators is determined by $m_{3}(x_j, \rho_3, \rho_2) = x_{j+1}$.) So given generators $u$ and $v$ in the type D structure for $\lp'$ as above,  $x_{i+j+1}\otimes v$ is in the image of $\partial(x_j\otimes u)$ for each $j\in\{0,\ldots,n-1\}$.  This achieves the desired splitting.  
\end{proof}

\begin{cor}\label{cor:STlike}Suppose $\lp$ is solid torus-like with $\chi_\circ(\lp)=0$. Then $\lp^A\boxtimes\lp'$ is an L-space complex if and only if $\infty$ is an L-space slope for the loop $\lp'$.\end{cor}

As a result, with respect to gluing, solid torus-like loops will need to be treated like solid tori. Consequently, manifolds with solid torus-like invariants (should they exist) will need to be singled out. 

\begin{definition}\label{def:solid torus-like manifold}
A loop-type manifold $M$ is solid torus like if it is a rational homology solid torus and every loop in the representation of $\CFD(M,\alpha,\beta)$ is solid torus-like. 
\end{definition}

\parpic[r]{\begin{tikzpicture}[>=latex, thick] 
\node (a) at (2,0) {$\bullet$};
\node (b) at (1,0) {$\circ$};
\node (c) at (0,0) {$\bullet$};
\node (d) at (0,1) {$\circ$};
\node (e) at (0,2) {$\bullet$};
\node (f) at (1,2) {$\circ$};
\node (g) at (2,1) {$\circ$};
\node (h) at (2,2) {$\bullet$};
\draw[->] (a) to node[below]{$\rho_3$} (b);
\draw[->] (b) to node[below]{$\rho_2$} (c);
\draw[->] (c) to node[left]{$\rho_{123}$} (d);
\draw[->] (e) to node[left]{$\rho_1$} (d);
\draw[->] (f) to node[above]{$\rho_2$} (e);
\draw[->] (h) to node[above]{$\rho_{3}$} (f);
\draw[->] (h) to node[right]{$\rho_{1}$} (g);
\draw[->] (a) to node[right]{$\rho_{123}$} (g);
\end{tikzpicture}} 
We conclude this section with an explicit example, demonstrating that not every abstract loops arises as the type D structure of a bordered three-manifold. Consider the loop $\lp$ described by the cyclic word $(a_1b_1\bar a_1\bar b_1)$ and illustrated on the right. Suppose that $\lp$ describes the invariant $\CFD(M,\alpha,\beta)$ for some orientable three-manifold with torus boundary $M$. Then the abstract Dehn filling $\lp^A_\bullet\boxtimes\lp$ yields the chain complex $\CFhat(M(\alpha))$. However, observe that $H_*(\lp^A_\bullet\boxtimes\lp)=0$ (in particular, this slope is not an L-space slope according to Definition \ref{def:abstractL}). This shows that no such $M$ exists: In general, for a closed orientable three-manifold $Y$, $\HFhat(Y)$ does not vanish \cite{OSz2004-b}. However, we remark that this particular loop arises as a component of the graph describing $\CFD(M,\mu,\lambda)$, where $M=S^3\smallsetminus \nu(K)$ and $K$ is any thin knot in $S^3$ that does not admit L-space surgeries; see \cite{Petkova-thin}.

\section{Characterizing slopes}\label{sec:char}

We now turn to an application of the loop calculus developed above. For a given loop $\lp$ we give an explicit description of the L-space and non-L-space slopes. As a consequence, we will prove the following:

\begin{thm}\label{thm:intervals}
Given a loop $\lp$, the set of non-L-space slopes is an interval in $\hat\Q$, that is, the restriction of a connected subset in $\hat\R$. As a result, if $M$ is a loop-type manifold then there is a decomposition $\hat\R=U\cup V$ into disjoint, connected subsets $U,V\in \hat\R$ such that $\mathcal{L}_M = U\cap \hat\Q$ and $\mathcal{L}^c_M=V\cap \hat\Q$. \end{thm}

\begin{remark}The subset $U$ determining $\mathcal{L}_M$ in this statement may be empty.\end{remark}

\subsection{L-space slopes} A loop $\lp$ has two preferred slopes, 0 and $\infty$; we begin by giving explicit conditions on the loop $\lp$ (in terms of its representatives in standard or dual notation) under which these are L-space slopes.

\begin{prop}
\label{prop:Dehn_filling_Lspaces}Given a loop $\lp$, $\infty$ is an L-space slope if and only if $\lp$ can be written in standard notation with at least one $d_k$ letter and no $c_k$ letters (where $k$ can be any integer). $0$ is an L-space slope if and only if $\lp$ can be written in dual notation with at least one $d^*_k$ letter no $c^*_k$ letters.
\end{prop}

\begin{proof}
The slope $\infty$ is an L-space slope if $C(\frac{1}{0}) = \lp_\bullet^A \boxtimes \lp$ is an L-space complex, that is, if $H_*(C(\frac{1}{0}))$ is nontrivial and each generator of $H_*(C(\frac{1}{0}))$ has the same ($\Ztwo$)-grading.

Recall that $\lp_\bullet^A$ has a single generator $x$ with idempotent $\iota_0$ and operations
\[m_{3+i}(x, \rho_3, \underbrace{ \rho_{23}, \ldots, \rho_{23}}_{i \text{ times}}, \rho_2) = x \]
for each $i \ge 0$. The box tensor product of this module with $\lp$ is easy to describe if $\lp$ is written in standard notation. There is one generator $x \otimes y$ for each $\bullet$-vertex $y$ in $\lp$ (by abuse $y$ is both the vertex in the loop $\lp$ and the corresponding generator in the type D structure corresponding to $\lp$). Given a cyclic word represting $\lp$ in standard notation, each letter represents a chain between adjacent $\bullet$-vertices; for each type $a_k$ chain from $y_1$ to $y_2$ there is a differential from $x\otimes y_1$ to $x\otimes y_2$. Since sequential occurences of type $a$ chains are impossible in any loop, the differentials on $C(\frac{1}{0})$ are isomorphisms mapping single generators to single generators. Therefore, the contribution to $H_*(\lp_\bullet^A\boxtimes\lp)$ (with relative ($\Ztwo$)-grading) is simply given by the $\bullet$-vertices of $\lp$ which do not lie at the end of a type $a$ segment.

The ($\Ztwo$)-grading on $C(\frac{1}{0})$ can be recovered as follows: the generators corresponding to two adjacent $\bullet$-vertices in $\lp$ have the same grading if the vertices are connected by an unstable chain and opposite gradings if the vertices are connected by a stable chain. It follows that endpoints of chains of type $d_k$ all have the same grading and endpoints of chains of type $c_k$ all have the opposite grading. This is because $d_k$ segments must be separated from each other by an even number of stable chains, and from $c_k$ segments by an odd number of stable chains.

Suppose the $\lp$ can be written in standard notation with at least one $d_k$ and no $c_k$. Every generator of $C(\frac{1}{0})$ either comes from an endpoint of a $d_k$ or from the common endpoint of two stable chains. The latter generators vanish in homology, since one of the two stable chains must be type $a$, and the former generators all have the same ($\Ztwo$)-grading. Thus in this case, $C(\frac{1}{0})$ is an L-space complex.

Suppose now that $\lp$ contains both $d_k$ and $c_k$ segments. For any unstable chain, at least one of the two endpoints must correspond to a generator of $C(\frac{1}{0})$ that survives in homology, since an unstable chain can be adjacent to a type $a$ chain on at most one side. Since endpoints of type $d_k$ and type $c_k$ unstable chains produce generators of opposite grading, it follows that $C(\frac{1}{0})$ has generators of both gradings that survive in homology, and thus $C(\frac{1}{0})$ is not an L-space complex.

If $\lp$ has no unstable chains when written in standard notation, then it has only stable chains, which alternate between type $a$ and $b$. In this case every $\bullet$-vertex is the endpoint of a type $a$ chain. Thus $H_*(C(\frac{1}{0}))$ is trivial, and $C(\frac{1}{0})$ is not an L-space complex. Finally, if $\lp$ cannot be written in standard notation, then it consists only of $e^*$ segments. $H_*(\lp^A_\bullet\boxtimes\lp)$ has two generators with opposite gradings; it follows that $\infty$ is not an L-space slope.

The proof for $0$-filling is almost identical. $\lp_\circ^A$ has a single generator $x$ with idempotent $\iota_1$ and operations
\[m_{3+i}(x, \rho_2, \underbrace{ \rho_{12}, \ldots, \rho_{12}}_{i \text{ times}}, \rho_1) = x \]
for each $i \ge 0$. The box tensor product of this module with $\lp$ has one generator for each generator $y$ of $\lp$ with idempotent $\iota_1$ (that is, each $\circ$-vertex) and a differential for each type $a^*$ segment. Expressing $\lp$ in dual notation, the contribution to $H_*(\lp^A_\circ\boxtimes\lp)$ (with relative ($\Ztwo$)-grading) is given by the $\iota_0$-generators of $\lp$ which do not lie at the end of a type $a^*$ segment. If $\CFD(Y, \spinc)$ is a loop that cannot be written in dual notation-- that is, it is a collection of $e$ segments -- then the contribution to $H_*(\lp^A_\circ\boxtimes\lp^D)$ is two generators of opposite grading.

Since $a^*$ and $b^*$ segments change the ($\Ztwo$)-grading while $c^*$ and $d^*$ segments do not, the rest of the proof is completely analogous to the proof for the $\infty$-filling.\end{proof}

Combining the two conditions in Proposition \ref{prop:Dehn_filling_Lspaces} gives a stronger condition on the loop.

\begin{prop}
\label{prop:Dehn_filling_Lspaces2}
Given a loop $\lp$, both $\infty$ and $0$ L-space slopes if and only if the following equivalent conditions hold:
\begin{itemize}
\item[$(i)$] $\lp$ can be written in standard notation with at least one $d_k$ letter and no $c_k$ letters (with $k \in \Z)$, and $\lp$ contains a subword from exactly one of the following sets
\begin{align*}
A_+ &= \{b_i a_j, a_i e^n b_j, a_i e^n d_j, d_i e^n b_j, d_i e^n d_j, d_\ell, a_\ell, b_\ell \quad | \quad i,j \ge 1, n\ge0, \ell\ge2 \}\\
A_- &= \{b_i a_j, a_i e^n b_j, a_i e^n d_j, d_i e^n b_j, d_i e^n d_j, d_\ell, a_\ell, b_\ell \quad | \quad i,j \le 1, n\ge0, \ell\le2 \}\end{align*}
We will say that $\lp$ satisfies condition $(i)_\pm$ if it satisfies condition $(i)$ with a subword in $A_\pm$.

\item[$(ii)$] $\lp$ can be written in dual notation with at least one $d^*_k$ letter and no $c^*_k$ letters (with $k \in \Z)$, and $\lp$ contains a subword from exactly one of the following sets
\begin{align*}
A^*_+ &= \{b^*_i a^*_j, a^*_i (e^*)^n b^*_j, a^*_i (e^*)^n d^*_j, d^*_i (e^*)^n b^*_j, d^*_i (e^*)^n d^*_j, d^*_\ell, a^*_\ell, b^*_\ell \quad | \quad i,j \ge 1, n\ge0, \ell\ge2 \}\\
A^*_- &= \{b^*_i a^*_j, a^*_i (e^*)^n b^*_j, a^*_i (e^*)^n d^*_j, d^*_i (e^*)^n b^*_j, d^*_i (e^*)^n d^*_j, d^*_\ell, a^*_\ell, b^*_\ell \quad | \quad i,j \le 1, n\ge0, \ell\le2 \}\end{align*}
We will say that $\lp$ satisfies condition $(i)_\pm$ if it satisfies condition $(ii)$ with a subword in $A^*_\pm$.

\end{itemize}
Moreover, condition $(i)_+$ is equivalent to condition $(ii)_+$ and condition $(i)_-$ is equivalent to condition $(ii)_-$.
\end{prop}
\begin{proof}
By Proposition \ref{prop:Dehn_filling_Lspaces}, 0 and $\infty$ are both L-space slopes for $\lp$ if and only if $\lp$ contains either $d_k$ segments or $c_k$ segments in standard notation, but not both, and $\lp$ contains either $d^*_k$ segments or $c^*_k$ segments in dual notation, but not both. It will be convenient to choose standard notation and dual notation representatives for $\lp$ with respect to the same orientation on the loop. This means that by reversing the orientation we can either ensure that the standard unstable chains are of type $d_k$ or we can ensure that the dual unstable chains are of type $d^*_k$.

First assume that $\lp$ contains standard unstable chains of type $d_k$ but not $c_k$. Under this assumption, the set $A_+$ is precisely the subwords in standard notation that dualize give a $d^*_k$ chain. This follows from the discussion on dualizing in Section \ref{sec:notations_for_loops}. A $d^*_{n+1}$ segment with $n\ge0$ arises from a subword $a_i e^n b_j$, $a_i e^n d_j$, $d_i e^n b_j$ or $d_i e^n d_j$ with $i,j \ge 1$. A segment $d^*_{-n-1}$ arises from a subword $b_i \bar{e}^n a_j$, $b_i \bar{e}^n c_j$, $c_i \bar{e}^n a_j$ or $c_i \bar{e}^n c_j$ with $i,j \ge 1$; since we assume that $\lp$ contains no $c_k$ segments, including $\bar{e} = c_0$, the only relevant case is $b_i a_j$. Finally, a $d^*_0 = e^*$ segment arises from a standard letter of type $a_\ell$, $b_\ell$, $c_\ell$, or $d_\ell$ with subscript $\ell > 1$; we can ignore the case of $c_\ell$ by assumption. By similar reasoning, we can check that under the assumption that $\lp$ contains no $c_k$ letters, the set $A_-$ is precisely the subwords in standard notation that dualize to give $c^*_k$ chains.

If we instead assume that $\lp$ contains dual unstable chains of type $d^*_k$ but not $c^*_k$, the argument is simlar. Note that the sets $A^*_\pm$ are the same as the sets $A_\pm$ with stars added to each letter, and the process for switching from dual to standard notation is the same as the process for switching from standard to dual (up to adding/removing stars from letters). We have immediately that $A^*_+$ is the set of subwords in that dualize to give a $d_k$ letter, and $A^*_-$ is the set of subwords that dualize to give a $c_k$ letter.

We have shown that conditions (i) and (ii) are both equivalent to $\lp$ containing an unstable chain of exactly one of the two types in both standard notation and dual notation, and this is equivalent to both 0 and $\infty$ being L-space slopes. Finally, if $\lp$ satisfies condition $(i)_+$ then $\lp$ contains both a $d_k$ segment and a $d^*_k$ segment (with respect to the same orientation of the loop). It follows that $\lp$ satisfies condition $(ii)_+$. If $\lp$ satisfies condition $(i)_-$ then it contains both a $d_k$ segment and a $c^*_k$ segment. After reversing the orientation of the $\lp$, it contains a $c_k$ segment and a $d^*_k$ segments; thus $\lp$ satisfies condition $(ii)_-$.
\end{proof}

\begin{cor}\label{cor:L-space-closed-interval}
If $0$ and $\infty$ are L-space slopes for a loop $\lp$ then the set of L-sapce slopes for $\lp$ contains either all positive slopes (that is, the interval $[0,\infty]$) or all negative slopes (that is, the interval $[-\infty,0]$).
\end{cor}

\begin{proof}
Let $L_+$ be the set of loops satisfying condition $(i)_+$ and $(ii)_+$ in Proposition \ref{prop:Dehn_filling_Lspaces2}, and let $L_-$ be the set of loops satisfying conditions $(i)_-$ and $(ii)_-$. It is not difficult to see that condition $(i)_+$ is preserved by the operation $\tw$, since for any ${\bf w} \in A_+$, $\tw({\bf w})$ is in $A_+$ or contains an element of $A_+$ as a subword. Similarly, condition $(ii)_+$ is preserved by the operation $\du$. Therefore the set $L_+$ is preserved by $\tw$ and $\du$. In the same way, the set $L_-$ is preserved by $\twi$ and $\dui$.

If $\frac{p}{q} > 0$, let $[a_1, \ldots, a_n]$ be a continued fraction for $\frac{p}{q}$ of even length with all positive terms. To see if $\frac{p}{q}$ is an L-space slope, we reparametrize by
$ \du^{a_n} \circ \cdots \circ \tw^{a_3} \circ \du^{a_2} \circ \tw^{a_1}$, taking the slope $\frac{p}{q}$ to the slope $\infty$. It follows that if a loop is in $L_+$ then any $\frac{p}{q}>0$ is an L-space slope, since the reparametrized loop is also in $L_+$.

If $\frac{p}{q} < 0$, let $[-a_1, \ldots, -a_n]$ be a continued fraction for $\frac{p}{q}$ of even length with all negative terms. To see if $\frac{p}{q}$ is an L-space slope, we reparametrize by
$\du^{-a_n} \circ \cdots \circ \tw^{-a_3} \circ \du^{-a_2} \circ \tw^{-a_1}$, taking the slope $\frac{p}{q}$ to the slope $\infty$. It follows that if a loop is in $L_-$, then any $\frac{p}{q}<0$ is an L-space slope.\end{proof}

\subsection{Non-L-space slopes}
The goal of this section is to establish easily-checked conditions certifying that the standard and dual slope of a given loop are non-L-space slopes. Our focus will be on the following result. 

\begin{prop}\label{prp:non-L-space endpoints}
Suppose that $\lp$ is a loop for which both standard and dual fillings give rise to non-L-spaces. Then the standard and dual slopes bound an interval of non-L-space slopes in $\mathcal{L}_M^c$.  That is, one of  $[0,\infty]$ or $[-\infty,0]$ consists entirely of non-L-space slopes. \end{prop}

The proof of this result is similar to the proof of Corollary \ref{cor:L-space-closed-interval} but is more technical and will therefore be built up in a series of lemmas. 

To begin, note that Proposition \ref{prop:Dehn_filling_Lspaces} can be restated in terms of non-L-space slopes as follows:

\begin{prop}\label{prop:nonLspace_slopes}
$\infty$ is a non-L-space slope for a loop $\lp$ if and only if either (i) $\lp$ contains both $c_k$ and $d_k$ unstable chains in standard notation, (ii) $\lp$ contains no unstable chains in standard notation, or (iii) $\lp$ cannot be written in standard notation. $0$ is a non-L-space slope if and only if either (i) $\lp$ contains both $c^*_k$ and $d^*_k$ unstable chains in dual notation, (ii) $\lp$ contains no unstable chains in dual notation, or (iii) $\lp$ cannot be written in dual notation. 
\end{prop}

It will be helpful to give conditions in standard notation under which 0 is a non-L-space slope. Let $k>1$, $i,j>0$, and $n\ge0$  be integers and consider the sets
\begin{align*}
A_1 &= \{a_k, b_k, c_k,c_ic_0^nc_j, c_ic_0^na_j, b_ic_0^nc_j, b_ic_0^na_j\}\\[5pt]
A_2&=\{a_{-k}, b_{-k}, d_{-k} ,d_{-i} d_0^n d_{-j}, d_{-i} d_0^n b_{-j}, a_{-i} d_0^n d_{-j}, a_{-i} d_0^n b_{-j} \}\\[5pt]
A_3 &= \{a_k, b_k, d_k, d_id_0^nd_j, d_id_0^nb_j, a_id_0^nd_j,a_id_0^nb_j \}\\[5pt]
A_4 &= \{a_{-k}, b_{-k}, c_{-k}, c_{-i}c_0^n c_{-j}, c_{-i} c_0^n a_{-j}, b_{-i} c_0^n c_{-j}, b_{-i}c_0^n a_{-j}\}
\end{align*}
These four sets may be interpreted as follows:
\begin{lem}\label{lem:four-sets}
A loop $\lp$ written in standard notation contains a word in $A_1\cup A_3$ if and only if it contains $d^*_n$ in dual notation for some integer $n$. It contains a word in $A_2\cup A_4$ if and only if it contains $c^*_n$ in dual notation for some integer $n$.

\end{lem}
\begin{proof} We prove the first statement and leave the second to the reader. 

First notice that if $k>1$ then any letter $a_k,b_k,c_k,d_k$ contains at least one instance of $\raisebox{-4.5pt}{\begin{tikzpicture}[>=latex] 
		\node at (0,0) {$\circ$}; 
		\node at (1,0) {$\circ$}; 
				\draw[thick, ->,shorten >=0.1cm, shorten <=0.1cm] (0,0) -- (1,0); \node at (0.5,0.25) {$\scriptsize\rho_{23}$};
	\end{tikzpicture}}$ which, in dual notation, gives an $e^*=d_0^*$. 
If $i,j>0$ then each word of the form 	$c_ic_0^nc_j$, $c_ic_0^na_j$, $b_ic_0^nc_j$, or $b_ic_0^na_j$ gives an instance of 
\[
\begin{tikzpicture}[>=latex] 
		\node at (0,0) {$\circ$}; 
		\node at (1,0) {$\bullet$}; 
		\node at (2,0) {$\bullet$};
		\node at (3,0) {$\circ$}; 
					\draw[thick, <-,shorten >=0.1cm, shorten <=0.1cm] (0,0) -- (1,0); \node at (0.5,0.25) {$\scriptsize\rho_{1}$};

		\draw[dashed, thick, ->,shorten >=0.1cm, shorten <=0.1cm] (1,0) -- (2,0); \node at (1.5,0.25) {$\scriptsize\rho_{12}$};
		\draw[thick, ->,shorten >=0.1cm, shorten <=0.1cm] (2,0) -- (3,0); \node at (2.5,0.25) {$\scriptsize\rho_{3}$};
	\end{tikzpicture}
	\]
	which, in dual notation, gives $\bar c^*_{n+1}=d^*_{-n-1}$. 
Similarly, each word of the form $d_i d_0^n d_j$, $d_i d_0^n b_j$, $a_i d_0^n d_j$, or $a_i d_0^n b_j$ gives an instance of \[
\begin{tikzpicture}[>=latex] 
		\node at (0,0) {$\circ$}; 
		\node at (1,0) {$\bullet$}; 
		\node at (2,0) {$\bullet$};
		\node at (3,0) {$\circ$}; 
					\draw[thick, ->,shorten >=0.1cm, shorten <=0.1cm] (0,0) -- (1,0); \node at (0.5,0.25) {$\scriptsize\rho_{2}$};

		\draw[dashed, thick, ->,shorten >=0.1cm, shorten <=0.1cm] (1,0) -- (2,0); \node at (1.5,0.25) {$\scriptsize\rho_{12}$};
		\draw[thick, ->,shorten >=0.1cm, shorten <=0.1cm] (2,0) -- (3,0); \node at (2.5,0.25) {$\scriptsize\rho_{123}$};
	\end{tikzpicture}
	\] which, in dual notation, gives $d^*_{n+1}$. 
For the converse, observe that the segments $d^*_0$, $d^*_{n+1}$ and $d^*_{-n-1}$ with $n \ge 0$ in dual notation can only arise from the words mentioned above in standard notation.	
\end{proof}


\begin{cor}\label{cor:non-L-space-dual-filling}
$0$ is not an L-space slope for $\lp$ if and only if either (i) $\lp$ contains a sub-word from each of $A_i$ and $A_j$, where $i$ and $j$ have opposite parity, (ii) $\lp$ does not contain a subword from any $A_i$, or (iii) $\lp$ cannot be written in standard notation.
\end{cor}
\begin{proof}
This follows immediately from Proposition \ref{prop:nonLspace_slopes} and Lemma \ref{lem:four-sets}. 
\end{proof}
Based on this observation, the proof of Proposition \ref{prp:non-L-space endpoints} will reduce to several cases, which are treated in the following Lemmas. 

\begin{lem}\label{lem:1-2-closure} If  $\lp$ is a loop containing a sub-word from each of the sets $A_1$ and $A_2$ then  $\twi(\lp)$ and $\dui(\lp)$ are both loops containing a sub-word from each of the sets $A_1$ and $A_2$. \end{lem}
\begin{proof} 
First consider $\twi({\bf w})$ for any sub-word ${\bf w} \in A_1$. We have $\twi(a_k)=a_k$, $\twi(b_k)=b_k$, and $\twi(c_k)=c_{k+1}$, $\twi(c_i\bar{e}^nc_j)=c_{i+1}c_1^nc_{j+1}$, $\twi(c_i\bar{e}^na_j)=c_{i+1}c_1^na_{j}$, $\twi(b_i\bar{e}^nc_j)=b_ic_1^nc_{j+1}$, and $\twi(b_i\bar{e}^na_j)=b_i c_1^na_j$. Note that in each case,  $\twi({\bf w})$ is in $A_1$ or contains a sub-word in $A_1$. Thus the operation $\twi$ preserves the property that $\lp$ contains a word from $A_1$. On the other hand, revisiting the proof of Lemma \ref{lem:four-sets} we see that each of $a_k$, $b_k$ or $c_k$ produces at least one $e^*$  when dualized, while each of $c_i\bar{e}^nc_j$, $c_i\bar{e}^na_j$, $b_i\bar{e}^nc_j$, $b_i\bar{e}^na_j$ produces a $\bar c^*_{n+1}$. Since $\twi(e^*)=\bar c_1^*$ and $\twi(\bar c^*_{n+1})=\bar c^*_{n+2}$, it is enough to observe that 
\[\{c_i\bar{e}^nc_j, c_i\bar{e}^na_j, b_i\bar{e}^nc_j, b_i\bar{e}^na_j\}\subset A_1\]
is the set of sub-words that give rise to a $\bar c_{n+1}^*$ under dualizing (where $n\ge 0$ and $i,j>0$). It follows that  the operation $\dui$ preserves the property that $\lp$ contains a word from $A_1$. 

Next consider $\twi({\bf w})$ for any ${\bf w} \in A_2$. Proceeding as above, observe that $\twi({\bf w})$ is in the set
\[\{ \bar{a}_k, \bar{b}_k, \bar{c}_{k+1} ,\bar{c}_{i+1}\bar c^n_1\bar{c}_{j+1}, \bar{c}_{i+1}\bar c_1^n\bar{b}_j, \bar{a}_i\bar c_1^n\bar{c}_{j+1}, \bar{a}_i\bar c_1^n\bar{b}_j  \}. \] 
Each of these words is in $A_2$ or contains a sub-word in $A_2$, so the operation $\twi$ preserves the property that $\lp$ contains a word from $A_2$. Each of $\bar a_k$, $\bar b_k$ or $\bar c_k$ produces at least one $\bar e^*$  when dualized, while each of $\bar c_i{e}^n\bar c_j$, $\bar c_i{e}^n\bar b_j$, $\bar a_i{e}^n\bar c_j$, $\bar a_i{e}^n\bar b_j$ produces a $c^*_{n+1}$. Since $\twi(\bar e^*)= c_1^*$ and $\twi( c^*_{n+1})= c^*_{n+2}$, it is enough to observe that 
\[\{\bar c_i{e}^n\bar c_j, \bar c_i{e}^n\bar  b_j, \bar a_i{e}^n\bar c_j, \bar a_i{e}^n\bar  b_j\}\subset A_2\]
is the set of sub-words that give rise to a $c_{n+1}^*$ under dualizing (where $n\ge 0$ and $i,j>0$). It follows that  the operation $\dui$ preserves the property that $\lp$ contains a word from $A_2$. 
\end{proof}

\begin{lem}\label{lem:3-4-closure} If  $\lp$ is a loop containing a sub-word from each of the sets $A_3$ and $A_4$ then  $\tw(\lp)$ and $\du(\lp)$  are both loops containing a sub-word from each of the sets $A_3$ and $A_4$. \end{lem}
\begin{proof}Analogous to the proof of Lemma \ref{lem:1-2-closure} and left to the reader.\end{proof}

We next consider the case that $\lp$ contains a sub-word from $A_1$ and a sub-word from $A_4$ but does not contain a sub-word from $A_2$ or $A_3$. While this third case is ultimately similar to the two cases we have already treated, it is more technical owing in part to three sub-cases. 

\begin{lem}\label{lem:hard-case-1} Suppose $\lp$ is a loop containing a sub-word from each of the sets $A_1$ and $A_4$ but that $\lp$ does not contain a sub-word from either of the sets $A_2$ or $A_3$. If $\lp$ contains a $\bar c_1$ then the loop $\twi(\lp)$ contains a sub-word from each of $A_1$ and $A_2$ while the loop $\dui(\lp)$ either \begin{itemize}
\item[(i)] contains a sub-word from each of $A_1$ and $A_2$; or
\item[(ii)] contains a $\bar c_1$ and a sub-word from each of the sets $A_1$ and $A_4$ but does not contain a sub-word from $A_2$ or from $A_3$.
\end{itemize}
\end{lem}
\begin{proof} First notice that $\twi(\bar{c}_1)=\bar{c}_2\in A_2$. Since the set $A_1$ is closed under $\twi$ (compare Lemma \ref{lem:1-2-closure}), it follows that $\twi(\lp)$ contains a sub-word from both $A_1$ and $A_2$.

Turning now to the behaviour under $\dui$: Since $\lp$ does not contain a word  from $A_2$, any occurrence of $\bar{c}_1$ can be preceded only by $a_1e^i$ or $d_1e^i$ (for $i\ge0$), and can be followed only by a $e^jb_1$ or $e^jd_1$ (for $j\ge0$). (Note that intances of $a_k$ or $b_k$ for $k>1$ are ruled out since $\lp$ does not contain a word from $A_3$.) Regardless of which arises, this ensures that a $\bar{c}_1$ is part of a segment in $\lp$ of the form
\[\begin{tikzpicture}[>=latex] 
		\node at (0,0) {$\circ$}; 
		\node at (1,0) {$\bullet$}; 
		\node at (2,0) {$\bullet$};
		\node at (3,0) {$\circ$}; 
		\node at (4,0) {$\bullet$}; 
		\node at (5,0) {$\bullet$}; 
		\node at (6,0) {$\circ$}; 
		\draw[thick, ->,shorten >=0.1cm, shorten <=0.1cm] (0,0) -- (1,0); \node at (0.5,0.25) {$\scriptsize\rho_{2}$};
		\draw[dashed, thick, ->,shorten >=0.1cm, shorten <=0.1cm] (1,0) -- (2,0); \node at (1.5,0.25) {$\scriptsize\rho_{12}$};
		\draw[thick, ->,shorten >=0.1cm, shorten <=0.1cm] (2,0) -- (3,0); \node at (2.5,0.25) {$\scriptsize\rho_{1}$};
		\draw[thick, <-,shorten >=0.1cm, shorten <=0.1cm] (3,0) -- (4,0); \node at (3.5,0.25) {$\scriptsize\rho_{3}$};
		\draw[dashed, thick, ->,shorten >=0.1cm, shorten <=0.1cm] (4,0) -- (5,0); \node at (4.5,0.25) {$\scriptsize\rho_{12}$};
		\draw[thick, ->,shorten >=0.1cm, shorten <=0.1cm] (5,0) -- (6,0); \node at (5.5,0.25) {$\scriptsize\rho_{123}$};
	\end{tikzpicture}\]
and hence a subword $b^*_{i+1}a^*_{j+1}$ when $\lp$ is expressed in dual notation. It follows that the property that $\lp$ contains a $\bar{c}_1$ is closed under $\dui$, under the assumption that $\lp$ does not contain a sub-word from $A_2$ or $A_3$.  

Now observe that the set $A_1$ is closed under $\dui$, so it must be the case that $\dui(\lp)$ contains sub-words from $A_1$ and $A_2$ or, if this is not the case, from $A_1$ and $A_4$. In the latter case, suppose that $\dui(\lp)$ contains a word from $A_3$. Then $\du(\dui(\lp))=\lp$ contains a word from $A_3$ also, which is a contradiction.  
\end{proof}


\begin{lem}\label{lem:hard-case-2} Suppose $\lp$ is a loop containing a sub-word from each of the sets $A_1$ and $A_4$ but that $\lp$ does not contain a sub-word from either of the sets $A_2$ or $A_3$. If $\lp$ contains a $d_1$ then the loop $\tw(\lp)$ contains a sub-word from each of $A_3$ and $A_4$ while the loop $\du(\lp)$ either \begin{itemize}
\item[(i)] contains a sub-word from each of $A_3$ and $A_4$; or
\item[(ii)] contains a $d_1$ and a sub-word from each of the sets $A_1$ and $A_4$ but does not contain a sub-word from $A_2$ or from $A_3$.
\end{itemize}
\end{lem}
\begin{proof}Similar to the proof of Lemma \ref{lem:hard-case-1}. The key observation is that the instance of $d_1$ is part of a segment in $\lp$ of the form
\[\begin{tikzpicture}[>=latex] 
		\node at (0,0) {$\circ$}; 
		\node at (1,0) {$\bullet$}; 
		\node at (2,0) {$\bullet$};
		\node at (3,0) {$\circ$}; 
		\node at (4,0) {$\bullet$}; 
		\node at (5,0) {$\bullet$}; 
		\node at (6,0) {$\circ$}; 
		\draw[thick, ->,shorten >=0.1cm, shorten <=0.1cm] (0,0) -- (1,0); \node at (0.5,0.25) {$\scriptsize\rho_{3}$};
		\draw[dashed, thick, ->,shorten >=0.1cm, shorten <=0.1cm] (1,0) -- (2,0); \node at (1.5,0.25) {$\scriptsize\rho_{12}$};
		\draw[thick, ->,shorten >=0.1cm, shorten <=0.1cm] (2,0) -- (3,0); \node at (2.5,0.25) {$\scriptsize\rho_{123}$};
		\draw[thick, ->,shorten >=0.1cm, shorten <=0.1cm] (3,0) -- (4,0); \node at (3.5,0.25) {$\scriptsize\rho_{2}$};
		\draw[dashed, thick, ->,shorten >=0.1cm, shorten <=0.1cm] (4,0) -- (5,0); \node at (4.5,0.25) {$\scriptsize\rho_{12}$};
		\draw[thick, <-,shorten >=0.1cm, shorten <=0.1cm] (5,0) -- (6,0); \node at (5.5,0.25) {$\scriptsize\rho_{1}$};
	\end{tikzpicture}\]
which, in dual form, gives 	$a^*_{j+1}b^*_{i+1}$. Thus the property that $\lp$ contains a $d_1$ implies that $\du(\lp)$ contains a $d_1$.  
 \end{proof}

\begin{lem}\label{lem:hard-case-3} Suppose $\lp$ is a loop containing a sub-word from each of the sets $A_1$ and $A_4$ but that $\lp$ does not contain a sub-word from either of the sets $A_2$ or $A_3$, and suppose that $\infty$ is a non-L-space filling for $\lp$. If $\lp$ contains neither a $d_{-1}$ nor a $d_1$ then two cases arise: Either $\lp$ contains the sub-word $\w_1= a_1e^n\bar b_1$ or $\lp$ contains the sub-word $\w_2=\bar a_1e^n b_1$. Regardless of which occurs, 
\begin{itemize}
\item[(i)] $\tw(\lp)$ contains sub-words from both $A_3$ and $A_4$; 
\item[(ii)] $\twi(\lp)$ contains sub-words from both $A_1$ and $A_2$; 
\item[(iii)] $\du(\lp)$ contains sub-words either from $A_3$ and $A_4$ or from $A_1$ and $A_4$ and contains a $\w_i$; and
\item[(iv)] $\dui(\lp)$ contains sub-words either from $A_1$ and $A_2$ or from $A_1$ and $A_4$ and contains a $\w_i$.
\end{itemize}\end{lem}
\begin{proof}
Suppose $\lp$ contains sub-words from $A_1$ and $A_4$ but not from $A_2$ or $A_3$, and $\lp$ contains neither $d_1$ nor $d_{-1}$. If $\lp$ has a non-L-space standard filling, then either $\lp$ contains no standard unstable chains or $\lp$ contains both $c_k$ and $d_k$ unstable chains. In the former case $\lp$ consists of alternating $a_k$ and $b_k$ segments; in the latter case, $\lp$ must contain a sequence of $d_0$ segments preceded by an $a_k$ segment and followed by a $b_k$ segment. In either case, it is routine to check that the subscripts on the type $a$ and $b$ stable chains must have oppostie signs, and thus that $\lp$ contains $\w_1= a_1e^n\bar b_1$ or $\w_2=\bar a_1e^n b_1$.

Note then that $\tw(\w_1) = a_1d_1^n\bar b_1$ which gives a word in $A_3$, $\twi(\w_1) = a_1\bar c_1^n\bar b_1$ which gives a word in $A_2$, $\tw(\w_2) = \bar a_1d_1^nb_1$ which gives a word in $A_3$, and  $\twi(\w_2) = \bar a_1\bar c_1^nb_1$ which gives a word in $A_2$. Since $\tw$ preserves the property that there is a sub-word from $A_4$, and $\twi$ preserves the property that there is a sub-word from $A_1$, we have shown that conclusions (i) and (ii) hold.


Consider the sub-word $\w_1 = a_1e^n\bar b_1$. If  a loop $\lp$ contains $\w_1$, then the graph for $\lp$ contains the segment
 \[\begin{tikzpicture}[>=latex] 
		\node at (1,0) {$\bullet$}; 
		\node at (2,0) {$\circ$};
		\node at (3,0) {$\bullet$}; 
		\node at (4,0) {$\bullet$}; 
		\node at (5,0) {$\circ$}; 
		\node at (6,0) {$\bullet$}; 
		\draw[thick, ->,shorten >=0.1cm, shorten <=0.1cm] (1,0) -- (2,0); \node at (1.5,0.25) {$\scriptsize\rho_{3}$};
		\draw[thick, ->,shorten >=0.1cm, shorten <=0.1cm] (2,0) -- (3,0); \node at (2.5,0.25) {$\scriptsize\rho_{2}$};
		\draw[dashed, thick, ->,shorten >=0.1cm, shorten <=0.1cm] (3,0) -- (4,0); \node at (3.5,0.25) {$\scriptsize\rho_{12}$};
		\draw[thick, ->,shorten >=0.1cm, shorten <=0.1cm] (4,0) -- (5,0); \node at (4.5,0.25) {$\scriptsize\rho_{1}$};
		\draw[thick, <-,shorten >=0.1cm, shorten <=0.1cm] (5,0) -- (6,0); \node at (5.5,0.25) {$\scriptsize\rho_{123}$};
	\end{tikzpicture}\]
It follows that in dual notation $\lp$ contains a $b^*_{n+1}$. Moreover, this $b^*_{n+1}$ is preceded by a segment ending with a $\rho_3$ arrow (that is, a segment of type $\bar a^*$ or $\bar c^*$) and followed by a segment beginning with a backwards $\rho_{123}$ arrow (that is, a segment of type $\bar a^*$ or $\bar d^*$). Thus $\lp$ contains $\w_1$ in standard notation if and only if it contains one of the following in dual notation: $\bar a^*_i b^*_{n+1} \bar a^*_j$, $\bar a^*_i b^*_{n+1} \bar d^*_j$, $\bar c^*_i b^*_{n+1} \bar a^*_j$, or $\bar c^*_i b^*_{n+1} \bar d^*_j$, where $i, j > 0$.

If $\lp$ contains $\w_1$, then in dual notation it contains one of $\bar a^*_i b^*_{n+1} \bar a^*_j$, $\bar a^*_i b^*_{n+1} \bar d^*_j$, $\bar c^*_i b^*_{n+1} \bar a^*_j$, or $\bar c^*_i b^*_{n+1} \bar d^*_j$. It is easy to see then that $\du(\lp)$ contains one of these subwords (and thus the subword $\w_1$ in standard notation) or it contains $e^* b^*_{n+1}$. In the latter case, $\lp$ contains a $\rho_{23}$ arrow followed by a $\rho_2$ arrow. This means that $\lp$ contains an $a_k$ or $d_k$ segment with $k>1$, each of which are in $A_3$. Similarly, $\dui(\lp)$ either contains $\w_1$ in standard notation or it contains $b^*_{n+1} \bar e^*$, which implies that $\lp$ contains a sub-word in $A_2$.

The details of a similar argument for the subword $\w_2 = \bar a_1 e^n \bar b_1$ are left to the reader. We similarly find that if $\lp$ contains $\w_2$, then $\du(\lp)$ contains either $\w_2$ or a sub-word in $A_3$ and $\dui(\lp)$ contains either $\w_2$ or a sub-word in $A_2$.

Since by assumption $\lp$ contains sub-words from $A_1$ and $A_4$, $\du(\lp)$ contains a sub-word from $A_4$ and from either $A_1$ or $A_3$. Since $\lp$ also contains a sub-word $\w_i$, the discussion above implies that $\du(\lp)$ contains the sub-word $\w_i$ or a subword in $A_3$; that is, $(iii)$ is satisfied. Similarly, $(iv)$ is satisfied since $\dui(\lp)$ contains a sub-word in $A_2$, a sub-word in either $A_1$ or $A_3$, and a subword $\w_i$ if there is no sub-word from $A_1$.\end{proof}

\begin{lem}\label{lem:hard-case-4}
Suppose $\lp$ contains no subwords in any $A_i$ and that $\lp$ contains both a $d_k$ segment and a $c_k$ segment. Then $\twi(\lp)$ contains a subword from each of $A_1$ and $A_2$ and $\tw(\lp)$ contains a subword from each of $A_3$ and $A_4$.
\end{lem}
\begin{proof}
In standard notation $\lp$ consists of the segments $a_{\pm 1}$, $b_{\pm 1}$, $c_{\pm 1}$, $d_{\pm 1}$, $c_0$, and $d_0$, and the nonzero subscripts alternate between $+1$ and $-1$. By assumption, $\lp$ must contain at least one $d_1$, $d_0$, or $d_{-1}$ segement. If $\lp$ contains $d_1$ then $\tw(\lp)$ contains $d_2 \in A_3$. Moreover, $d_1$ must be preceded by $d_0$, $d_{-1}$, or $a_{-1}$ and followed by $d_0$, $d_{-1}$, or $b_{-1}$. It follows that $\twi(\lp)$ contains a $d_0$ preceded by a $d_{-1}, d_{-2}$ or $a_{-1}$ and followed by a  $d_{-1}, d_{-2}$ or $b_{-1}$; thus $\twi(\lp)$ contains a subword in $A_2$. Similarly if $\lp$ contains $d_{-1}$ then $\twi(\lp)$ contains $d_{-2} \in A_2$, and since $d_{-1}$ must be preceded by $d_0$, $d_1$, or $a_1$ and followed by $d_0$, $d_1$, or $b_1$, $\tw(\lp)$ contains a subword in$A_3$. Finally, suppose $\lp$ contains a $d_0$ segment but no $d_1$ or $d_{-1}$. If the $d_0$ segment is adjacent to another $d_0$ segment, then $\tw(\lp)$ contains $d_1 d_1 \in A_3$ and $\twi(\lp)$ contains $d_{-1} d_{-1} \in A_2$. Otherwise, $\lp$ must contain either $a_{-1} d_0 b_1$ or $a_1 d_0 b_{-1}$. It is easy to check that $\tw(\lp)$ contains either $d_1 b_1$ or $a_1 d_1$, both elements of $A_3$, and that $\twi(\lp)$ contains either $a_{-1} d_{-1}$ or $d_{-1} b_{-1}$ in $A_2$.

A similar argument shows that if $\lp$ contains $c_1$, $c_0$, or $c_{-1}$ then $\tw(\lp)$ contains a subword in $A_4$ and $\twi(\lp)$ contains a subword in $A_1$. This can also be deduced by reversing the orienation on $\lp$, which takes $d_k$ segments to $c_{-k}$ segments and takes the sets $A_2$ and $A_3$ to $A_1$ and $A_4$.
\end{proof}

\begin{proof}[Proof of Proposition \ref{prp:non-L-space endpoints}]
Let $\lp$ be a loop for which both standard and dual filling gives an L-space. By Corollary \ref{cor:non-L-space-dual-filling} we have the following cases to consider:
\begin{itemize}
\item[(1)] $\lp$ contains a sub-word from $A_1$ and from $A_2$;
\item[(2)] $\lp$ contains a sub-word from $A_3$ and from $A_4$;
\item[(3)] $\lp$ contains a sub-word from $A_1$ and from $A_4$, but not $A_2$ or $A_3$;
\item[(4)] $\lp$ in does not contain a sub-word from any $A_i$;
\end{itemize}
Note that, \emph{a priori}, there is another case similar to (3) in which $\lp$ contains subwords in $A_2$ and $A_3$; however, this is equivalent to (3) after reversing orientation on $\lp$.

In case $(1)$, Lemma \ref{lem:1-2-closure} implies that applying any combination of $\twi$ and $\dui$ to $\lp$ gives rise to a loop $\lp'$ which contains a subword in each of $A_1$ and $A_2$, and thus has non L-space dual filling. $\frac{p}{q}$ filling on $\lp$ is given by dual filling $\tw^{a_n} \circ \cdots \circ \du^{a_2} \circ \tw^{a_1} (\lp)$, where $[a_1, \ldots, a_n]$ is an odd length continued fraction for $\frac{p}{q}$. For any $-\infty < \frac{p}{q} < 0$, we can choose a continued fraction with each $a_i \le 0$, and so $\frac{p}{q}$ is not an L-space slope for $\lp$. Similarly, in case (2) we choose a continued fraction with postive terms and use Lemma \ref{lem:3-4-closure} to see that any $0 < \frac{p}{q} < \infty$ is a non-L-space.

Case $(3)$ has three subcases, depending on whether $\lp$ contains a $\bar{c}_1$, a $d_1$, or neither. If $\lp$ contains a $\bar{c}_1$ then by Lemmas \ref{lem:hard-case-1} and \ref{lem:1-2-closure} applying any combination of $\twi$ and $\dui$ to $\lp$ with at least one application of $\twi$ produces a loop $\lp'$ with a subword in $A_1$ and $A_2$. Given any $-\infty < \frac{p}{q} < 0$ we can choose an odd length continued fraction $[a_1, \ldots, a_n]$ for which each $a_i$ is strictly negative except that $a_1$ may be 0, and if $a_1 = 0$ then $n \ge 3$. It follows that $\frac{p}{q}$ is a non-L-space slope for $\lp$. Similarly, if $\lp$ contains a $d_1$ then Lemmas \ref{lem:hard-case-2} and \ref{lem:3-4-closure} imply that any $0< \frac{p}{q}< \infty$ is a non-L-space slope. If $\lp$ contains neither a $d_1$ nor a $c_1$, then any $\frac{p}{q}$ is a non-L-space slope by Lemma \ref{lem:hard-case-3}

In Case (4), $\lp$ has no dual unstable chains. We have immediately that $\du(\lp) = \dui(\lp) = \lp$. Since $\infty$ is a non-L-space slope, $\lp$ must have either no standard unstable chains or unstable chains of both type $c_k$ and type $d_k$. If $\lp$ has no standard unstable chains then $\tw(\lp) = \twi(\lp) = \lp$; it follows that any combination of Dehn twists preserves $\lp$, and so any slope $\frac{p}{q}$ is a non-L-space slope (in fact in this case any filling of $\lp$ has trivial homology). If $\lp$ has standard unstable chains of both types, then Lemma \ref{lem:hard-case-4} implies that any slope $\frac{p}{q}$ is a non-L-space slope.
\end{proof}

\subsection{Proof of Theorem \ref{thm:intervals}}

We are now ready to prove that the set of L-space slopes for a loop $\lp$ is a (possibly empty) interval in $\hat\Q$. We will use the fact that applying $\tw$ and $\du$ to a loop $\lp$ changes the set of L-space slopes in a controlled way: the slope $\frac{p}{q}$ for $\lp$ is equivalent to the slope $\frac{p-nq}{q}$ for $\tw^n(\lp)$ and the slope $\frac{p}{q-np}$ for $\du^n(\lp)$. These transformations preserve the cyclic ordering on $\hat\Q$, and thus preserve the connectedness of the set of L-space slopes. In particular, while Corollary \ref{cor:L-space-closed-interval} and Proposition \ref{prp:non-L-space endpoints} are stated in terms of the 0 and $\infty$ slopes, the same statement holds for any two slopes of distance 1. In addition to these results, we will need the following Lemma:

\begin{lem}\label{lem:integral_slopes}
For any loop $\lp$, there do not exist integers $n_1 < n_2 < n_3 < n_4$ such that $n_1$ and $n_3$ are L-space slopes and $n_2$ and $n_4$ are not L-space slopes, or vice versa 
\end{lem}
\begin{proof}
An integer $n$ is an L-space slope for $\lp$ iff 0 is an L-space slope for $\tw^n(\lp)$. By Lemma \ref{lem:four-sets} this happens when $\tw^n(\lp)$ contains a word from the set $A_1 \cup A_3$ or from the set $A_2 \cup A_4$, but not both. Recall from the proof of Lemmas \ref{lem:1-2-closure} and \ref{lem:3-4-closure} that $A_1$ and $A_2$ are closed under $\twi$ and $A_3$ and $A_4$ are closed under $\tw$.

Suppose that $n_1 < n_2 < n_3$ with $n_2$ an L-space slope for $\lp$ and $n_1$ and $n_3$ are non-L-space slopes; we will show that all $n<n_1$ and all $n>n_3$ are non-L-space slopes. Up to reversing the loop, we can assume $\tw^{n_2}(\lp)$ contains a word in $A_1\cup A_3$ and not in $A_2\cup A_4$. We consider the case that $\tw^{n_2}(\lp)$ contains a word in $A_1$; the case that it contains a word in $A_3$ is similar and left to the reader.

By the closure property mentioned above, $\tw^{n_1}(\lp)$ contains a word in $A_1$. $\tw^{n_1}(\lp)$ also does not contain a word in $A_4$, since $\tw^{n_2}(\lp)$ does not, and so it must contain a word in $A_2$. It follows that, for any $n < n_1$, $\tw^n(\lp)$ contains words in $A_1$ and $A_2$ and thus $n$ is a non-L-space slope for $\lp$. We would like to show that $\tw^{n_3}(\lp)$ contains a word in $A_3$, since this would imply it also contains a word in $A_4$ and that any $n > n_3$ is a non-L-space slope for $\lp$. Suppose, by way of contradiction, that $\tw^{n_3}(\lp)$ does not contain a word in $A_3$. Then we see immediately that $\tw^{n_2}$ does not contain a word in $A_3$. It also does not contain a word in $A_2$, so it must consist only of the segments
\[ \{a_1, a_{-1}, b_1, b_{-1}, d_1, d_0, d_{-1}, c_k \}, \]
where $k$ can be any integer. Moreover since $n_2+1 \le n_3$, $\tw^{n_2+1}(\lp)$ does not contain a word in $A_3$. In particular this implies that $\tw^{n_2}(\lp)$ does not contain $d_1$. Any $d_{-1}$ cannot be preceded by $d_{-1}$ or $a_{-1}$ or followed by a $d_{-1}$ or $b_{-1}$, since this would give a word in $A_2$. The alternative, that $d_{-1}$ is preceded by a $d_0$ or $a_1$ and followed by a $d_0$ or $b_1$, is also ruled out since this would produce a word in $A_3$ in $\tw^{n_2+1}(\lp)$; thus $\tw^{n_2}(\lp)$ does not contain $d_{-1}$. If  $\tw^{n_2}(\lp)$ contains $d_0$, then it contains either $a_1 d_0^m b_1$, $a_{-1} d_0^m b_{1}$, $a_1 d_0^m b_{-1}$, or $a_{-1} d_0^m b_1$. None of these are possible; the first two are words in $A_3$ and $A_2$, respectively, and the last two give rise to words in $A_3$ in $\tw^{n_2+1}(\lp)$. We have now shown that $\tw^{n_2}(\lp)$ does not contain any $d_k$ segments, but since it also does not contain $a_{-1} b_{-1} \in A_2$ this contradicts the fact that $\tw^{n_1}(\lp)$ has a word in $A_2$. Therefore, $\tw^{n_3}(\lp)$ contains a word in $A_3$ and a word in $A_4$, and $n$ is a non-L-space slope for all $n > n_3$.\end{proof}

\begin{proof}[Proof of Theorem \ref{thm:intervals}]
Any two distinct slopes $\frac{r}{s}$ and $\frac{p}{q}$ determine two intervals in $\hat\Q$; we need to show that if $\frac{r}{s}$ and $\frac{p}{q}$ are L-space slopes then one of these intervals lies entirely in the set of L-space slopes for $\lp$. Note that by reparametrizing with $\tw$ and $\du$ as discussed above, we may assume that $\frac{r}{s} = 0$. We will assume $\frac{p}{q} > 0$ below; similar arguments apply when $\frac{p}{q} < 0$.

Suppose that $0$ and $\frac{p}{q}$ are L-space slopes and there exist non L-space slopes $u \in (0, \frac{p}{q})$ and $v\in[0, \frac{p}{q}]^c$. To reach a contradiction, we choose a continued fraction $[a_1, \ldots, a_n]$ for $\frac{p}{q}$ with $a_1 \ge 0$ and $a_i > 0$ for $i>1$ and proceed by induction on the length $n$. In the base case of $n = 1$, $\frac{p}{q} = a_1$ is an integer. We claim that there is an integer $u' \in (0, a_1)$ that is a non-L-space slope. To see this, if $u$ is not an integer consider the slopes $\lfloor u \rfloor$ and $\lceil u \rceil$. These slopes are distance 1 and both intervals between them contain non-L-space slopes ($u\in (\lfloor u \rfloor, \lceil u \rceil)$ and $v\in [\lfloor u \rfloor, \lceil u \rceil]^c$ ). Thus by Corollary \ref{cor:L-space-closed-interval} at least one of them is a non-L-space slope. Similarly, if $v \neq \infty$ then either $\lfloor v \rfloor$ or $\lceil v \rceil$ is a non-L-space slope. If $v = \infty$ is a non-L-space slope, note that $\lp$ contains either no unstable chains or both types of unstable chains in dual notation. In the former case, it follows that all integer slopes behave the same, but this contradicts the fact that $0$ is an L-space slope and $u'$ is not. In the latter case, it is easy to see that an integer $m$ is a non-L-space slope if $|m| \gg 0$. Thus we have integers $0 < u' < a_1$ and $v' <0$ or $v' > a_1$ such that $0$ and $a_1$ are L-space slopes and $u'$ and $v'$ are not; this contradicts Lemma \ref{lem:integral_slopes}.

Suppose the continued fraction $[a_1, \ldots, a_n]$ for $\frac{p}{q}$ has length $n > 1$ and $a_1 = 0$. We consider the loop $\lp' = \du^{a_2} (\lp)$. Note that the slope $\frac{p}{q}$ for $\lp_1$ corresponds to the slope $\frac{p'}{q'} = \frac{p}{q-a_2 p}$, which has continued fraction $[a_3, \ldots, a_n]$. The slope 0 for $\lp$ corresponds to 0 for $\lp'$, and the slopes $u$ and $v$ for $\lp$ correspond to slopes $u'$ and $v'$ for $\lp'$. Thus $\lp'$ has L-space slopes $0$ and $\frac{p'}{q'}$ and non-L-space slopes $u' \in (0, \frac{p'}{q'})$ and $v' \in (-\infty,0)\cup(\frac{p'}{q'}, \infty]$. Since $\frac{p'}{q'}$ has a continued fraction of length $n-2$, this produces a contradiction by induction.

Suppose $\frac{p}{q}$ has continued fraction $[a_1, \ldots, a_n]$ of length $n > 1$ with $a_1 > 0$. Consider the distance 1 slopes $a_1 = \lfloor \frac{p}{q} \rfloor$ and $a_1 + 1 = \lceil \frac{p}{q} \rceil$; both intervals between these slopes contain L-space slopes (either 0 or $\frac{p}{q}$), so by Proposition \ref{prp:non-L-space endpoints} one of them must be an L-space slope. First suppose that $a_1$ is an L-space slope. By the base case of induction, we cannot have both $u \in (0, a_1)$ and $v \in [0, \frac{p}{q}]^c \subset [0, a_1]^c$, and so we must have $u\in(a_1,\frac{p}{q})$. Consider the loop $\lp' = \tw^{a_1}(\lp)$. The slopes 0 and $\frac{p'}{q'} = \frac{p}{q} - a_1$ are L-space slopes for $\lp'$ while the slopes $u' = u - a_2 \in (0, \frac{p'}{q'})$ and $v' = v - a_2 \in [0, \frac{p'}{q'}]^c$ are non-L-space slopes. A continued fraction for $\frac{p'}{q'}$ is $[0, a_2, \ldots, a_n]$; as shown above this produces a contradiction.

Finally, suppose instead that $a_1$ is a non-L-space slope for $\lp$ and $a_1 + 1$ is an L-space slope. The base case of induction rules out the possibility that $v \in [0, a_1+1]^c$, so we must have $v \in (\frac{p}{q}, a_1 + 1)$. Consider the loop $\lp' = \du^{a_2-1} \circ \tw^{a_1 + 1}(\lp)$. The slope $a_1$ for $\lp$ corresponds to the slope $0$ for $\lp'$, and the slope $\frac{p}{q}$ corresponds to a slope $\frac{p'}{q'}$ with continued fraction $[-1, 1, a_3, \ldots, a_n] \sim [0, -a_3, \ldots, -a_n]$. There are non-L-space slopes for $\lp'$ in both intervals between $0$ and $\frac{p'}{q'}$. By induction (using the analog of the above cases when $\frac{p}{q} < 0$), this is a contradiction.

This proves that the set of L-space slopes for a loop $\lp$ is an interval. To prove the statement for bordered manifolds $(M, \alpha, \beta)$ of loop-type, we simply observe that if $M$ is a loop-type manifold, $M(p\alpha + q\beta)$ is an L-space if and only if $\frac{p}{q}$ is an L-space slope for each loop $\lp$ in $\CFD(M, \alpha, \beta)$. The set of L-space slopes for $M$ is the intersection of the intervals of L-space slopes for each loop, and hence an interval.
\end{proof}

\subsection{Strict L-space slopes}

The notion of L-space slope is fairly natural however, in the context of the theorems in this paper, it is not quite the right condition. We will be interested primarily in slopes that are not only L-space slopes but are also surrounded by a neighborhood of L-space slopes.

\begin{definition}
A slope in the boundary of a three-manifold $M$ with torus boundary is a strict L-space slope if it is an L-space slope in the interior of $\mathcal{L}_M$. Denote the set of strict L-Space slopes by $\mathcal{L}^\circ_M$. 
\end{definition}

We will give a geometric interpretation of strict L-space slopes in Section \ref{sec:proofs}. For now, we will use loop notation and the results in this section about L-space  slopes to determine when an L-space slope is strict.

\begin{prop}
\label{prop:strict_Lspace_slopes}
Given a loop $\lp$,
\begin{enumerate}
\item $\infty$ is strict L-space slope for $\lp$ if and only if $\lp$ can be written in standard notation using only the subwords
$$d_k, \quad b_1 a_{-1}, \quad \text{ and } \quad b_{-1} a_1, $$
where $k$ can be any integer, with at least one $d_k$ and such that $b_1a_{-1}$ is never adjacent to $b_{-1} a_1$.

\item $0$ is strict L-space slope for $\lp$ if and only if $\lp$ can be written in dual notation using only the subwords
$$d^*_k, \quad b^*_1 a^*_{-1}, \quad \text{ and } \quad b^*_{-1} a^*_1, $$
where $k$ can be any integer, with at least one $d^*_k$ and such that $b^*_1a^*_{-1}$ is never adjacent to $b^*_{-1} a^*_1$.

\end{enumerate}
\end{prop}

\begin{proof}
Suppose $\infty$ is a strict L-space slope. In particular, it is an L-space slope, so $\lp$ can be written in standard notation with no $c_k$ and at least one $d_k$. Since it does not contain $c_k$, $\lp$ can be broken into pieces of the form $d_k$ or $b_i a_j$ for any integer $k$ and nonzero integers $i$ and $j$,  with at least one $d_k$. We also have that $n$ and $-n$ are L-space slopes for sufficiently large $n$. Equivalently, $0$ is an L-space slope after we apply $\tw^n$ or $\tw^{-n}$ for sufficiently large $n$. The twist $\tw^n$ has the effect of replacing all unstable chains with chains of type $d_k$ for $k\gg0$, while the twist $\tw^{-n}$ replaces all unstable chains with $d_k$ for $k \ll 0$.

Consider a loop consisting of the pieces $d_k$ with $k \gg 0$ and $b_i a_j$ with any nonzero integers $i$ and $j$, with at least one $d_k$. The loop certainly contains $d^*_0 = e^*$, since it contains $d_k$ with $k>1$. Thus 0 is L-space detected if and only if the loop contains no $c_k^*$ segments. A loop of this form contains $c^*_k$ with $k>0$ if and only if it contains $a_i b_j$ with $i,j < 0$. It contains $c^*_k$ with $k < 0$ if and only if it contains $b_i a_j$ with $i, j < 0$. It contains $c^*_0 = \bar{e}^*$ segments if and only if it contains $a_\ell$ or $b_\ell$ with $\ell < -1$.

Similarly, consider a loop consisting of the pieces $d_k$ with $k << 0$ and $b_i a_j$ with any nonzero integers $i$ and $j$, with at least one $d_k$. The loop contains $c^*_0 = \bar{e}^*$, so 0 is L-space detected if and only if the loop contains no $d^*_k$ segments. A loop of this form contains $d^*_k$ if and only if it contains $b_i a_j$ or $a_i b_j$ with $i,j > 0$ or $a_\ell$ or $b_\ell$ with $\ell > 1$.

Therefore, a loop for which $\infty$ is an L-space slope also has $n$ and $-n$ as L-space slopes for all sufficiently large $n$ if and only if it contains no chains of type $a_k$, $b_k$, with $k>1$, $a_1$ segments are never adjacent to $b_1$ segments, and $a_{-1}$ segments are never adjacent to $b_{-1}$ segments.

The proof of part $(2)$ is completely analogous. Given that 0 is an L-space slope, we must check that $1/n$ and $-1/n$ are L-space slopes for sufficiently large $n$, which is equivalent to checking that $\infty$ is an L-space slope after applying $\du^n$ or $\du^{-n}$. These twists have the effect of replacing unstable chains in dual notation with $d^*_k$ segments with either $k\gg0$ or $k\ll0$. The rest of the proof is identical to the proof above after adding/removing stars on each segment.\end{proof}

\subsection{Simple loops}\label{sub:simple}
We will often restrict to a special class of loops which we call simple.

\begin{definition}\label{def:simple}
A loop $\lp$ is \emph{simple} if there is a loop $\lp'$ consisting only of unstable chains such that $\lp$ can be obtained from $\lp'$ using the operations $\tw^{\pm 1}$ and $\du^{\pm 1}$. A collection of loops $\{ \lp_i \}_{i=1}^n$ is said to be simple if, possibly after applying a sequence of the operations $\tw^{\pm 1}$ and $\du^{\pm 1}$, every loop consists only of unstable chains.
\end{definition}

\begin{remark}
In the above definition, we do not specify whether the unstable chains are in standard or dual notation. In fact, it does not matter: If $\lp$ contains only unstable chains in standard notation then $\tw^n (\lp)$ contains only unstable chains in dual notation for sufficiently large $n$, and if $\lp$ contains only dual unstable chains then $\du^n(\lp)$ contains only standard unstable chains. Definition \ref{def:simple} can be stated in terms of standard notation only (compare \cite[Definition 4]{HRRW}), but it is sometimes convenient to check the condition in dual notation.
\end{remark}

The notion of simple loops gives rise to a refinement of loop-type manifolds: $M$ is said to be of \emph{simple loop-type} if it is loop-type and, for some choice of $\alpha$ and $\beta$ the loops representing $\CFD(M, \alpha, \beta)$ consist only of unstable chains. Equivalently, $M$ is of simple loop-type if, for any choice of $\alpha$ and $\beta$, $\CFD(M, \alpha, \beta)$ is represented by a simple collection of loops.

Note that solid torus-like loops (Definition \ref{def:solid torus-like}) are examples of simple loops. 
Although we do not give an explicit description of all simple loops, we prove the following useful property.
\begin{prop}
\label{prop:simple_loops_std}
If $\lp$ is simple then, up to reorienting the loop, $\lp$ has no $a_k$ or $b_k$ segments with $k < 0$.
\end{prop}

The proof makes use of the following two observations.

\begin{lem}
\label{lem:simple_loops_pos_twist}
If $\lp$ contains no $a_k$, $b_k$, or $c_k$ segments with $k<0$, then the same is true for $\tw(\lp)$ and $\du(\lp)$.
\end{lem}
\begin{proof}
The first statement is obvious. For the second, we observe that $\lp$ contains no $a^*$, $b^*$, or $c^*$ segments in dual notation. Indeed, an $a^*$ or a $c^*$ contains a backwards $\rho_3$ arrow, which in standard notation implies the presence of an $a_k$ or $c_k$ with $k<0$, and a $b^*$ segment contains a forward $\rho_1$ segment, which implies the presence of a $b_k$ or $c_k$ with $k <0$. It is clear that this property is preserved by $\du$. Since $\du(\lp)$ has no $a^*$, $b^*$, or $c^*$ segments it is straightforward to check that $\du(\lp)$ has no forward $\rho_1$ arrows or backward $\rho_3$ arrows, and thus it does not contain any $a_k$, $b_k$, or $c_k$ segments with $k<0$.
\end{proof}

\begin{lem}
\label{lem:simple_loops_neg_twist}
If $\lp$ contains no $a_k$, $b_k$, or $d_k$ segments with $k<0$, then the same is true for $\twi(\lp)$ and $\dui(\lp)$.
\end{lem}
\begin{proof}
The proof is completely analogous to the previous Lemma.
\end{proof}

\begin{proof}[Proof of Proposition \ref{prop:simple_loops_std}]
Since $\lp$ is simple, there is some $\lp'$ consisting of only unstable chains such that $\lp$ is obtained from $\lp'$ by a sequence of Dehn twists. That sequence of Dehn twists determines an element of the mapping class group, which can be represented by the matrix
$\left(\begin{smallmatrix} p & q\\ r & s\end{smallmatrix}\right).$
Let $[k_1, \ldots, k_{2n}]$ be a continued fraction for $p/q$ of even length; the element of the mapping class group above can be decomposed as the following sequence of Dehn twists:
$$\tw^m \circ \du^{k_{2n}} \circ \tw^{k_{2n-1}} \circ \cdots \circ \du^{k_2}\circ \tw^{k_1},$$
where $m$ is an integer. Let
$$\lp'' =  \du^{k_{2n}} \circ \tw^{k_{2n-1}} \circ \cdots \circ \du^{k_2}\circ \tw^{k_1}(\lp).$$

First suppose that $p/q$ is positive. We may assume that each $k_i$ is positive for $1\le i \le 2n$. $\lp'$ can clearly be written with no $a_k$, $b_k$, or $c_k$ segments with $k < 0$, and by Lemma \ref{lem:simple_loops_pos_twist} this property is closed under all positive twists, so $\lp''$ can also be written with no $a_k$, $b_k$, or $c_k$ segments with $k < 0$. It follows that $\lp = \tw^m(\lp'')$ can be written with no $a_k$ or $b_k$ segments with $k < 0$.

If instead $p/q$ is negative, we may chose a continued fraction with each $k_i$ negative. $\lp'$ can be written with no $a_k$, $b_k$, or $d_k$ segments with $k < 0$, and by Lemma \ref{lem:simple_loops_neg_twist} the same is true for $\lp''$. It follows that $\lp = \tw^m(\lp'')$ can be written with no $a_k$ or $b_k$ segments with $k < 0$.
\end{proof}

The argument above can be repeated in dual notation, taking a continued fraction for $r/s$ instead of $p/q$, to prove the following:
\begin{prop}
\label{prop:simple_loops_dual}
If $\lp$ is simple then, up to reorienting the loop, $\lp$ has no $a^*_k$ or $b_k^*$ segments with $k<0$.
\end{prop}

The main advantage of restricting to simple loops is that it greatly simplifies the conditions that a given slope is a strict L-space slope. Recall that $\infty$ is a strict L-space slope for $\lp$ if and only if $\lp$ can be decomposed into words of the form $d_k$, $b_1 a_{-1}$ or $b_{-1} a_1$. If $\lp$ is a simple loop, then the last two words can not appear by Proposition \ref{prop:simple_loops_std} and $\infty$ is a strict L-space slope if and only if $\lp$ consists only of unstable chains in standard notation. Equivalently (using Observation \ref{obs:dual_stable_chains}), $\infty$ is a strict L-space slope if and only if $\lp$ does not contain both positive and negative subscripts in dual notation. Similarly, $0$ is a strict L-space slope if and only if $\lp$ consists only of unstable chains in dual notation, which is equivalent to $\lp$ having only nonnegative or only nonpositive subscripts in standard notation.

We conclude this section by refining Theorem \ref{thm:intervals} in the case of simple loops. Theorem \ref{thm:intervals} states that set of L-space slopes for a loop $\lp$ is a (possibly empty) interval in $\hat\Q$; the endpoints of this interval could be irrational, \emph{a priori}, but we find that only rational endpoints are possible.

\begin{prop}\label{prp:rational endpoints}
Let $\lp$ be a simple loop. The set of L-space slopes for $\lp$ is either: 
(i) identified with $\Q$ (in practice, every slope other than the rational longitude); 
or (ii) the restriction to $\hat\Q$ of a closed interval $U$ in $\hat\R$ with rational endpoints.
\end{prop}
\begin{proof}
Note that applying the twist operations $\tw$ and $\du$ to a loop preserves the cyclic ordering on abstract slopes, and so properties $(i)$ and $(ii)$ are preserved under these operations. Given an arbitrary simple loop $\lp$, we will describe an algorithm for applying twists to produce a loop $\lp'$ with one of the following properties
\begin{itemize}
\item[(1)] $\lp'$ contains only $d_0$ segments.
\item[(2)] $\lp'$ contains only $d_1$, $d_0$, and $d_{-1}$ segments, with $d_1$ and $d_{-1}$ segments alternating (ignoring $d_0$ segments).
\item[(3)] $\lp'$ contains only $d_k$ segments with $k \ge -1$, including at least one $d_{-1}$ segment, such that each $d_{-1}$ is followed by $(d_0)^j d_k$ for some $j\ge0$ and $k>0$, and $\lp'$ does not satisfy $(2)$.
\end{itemize}
By definition, some sequence of twists produces a loop $\lp_1$ which can be written in standard notation using only type $d_k$ segments. We may assume that the minimum subscript for the $d_k$ segments in $\lp_1$ is 0. Given $\lp_i$ (starting with $i = 1$), the algorithm proceeds as follows: If the subscripts in $\lp_i$ are all 0, then $\lp_i$ satisfies $(1)$ and the algorithm stops. Otherwise, consider the loop $\twi(\lp_i)$; this loop consists of $d_k$ segments with $k \ge -1$ including at least one $d_{-1}$. If $\twi(\lp_1)$ satisfies $(2)$ or $(3)$, the algorithm stops. Otherwise, consider the loop $\ex (\lp_i)$, which by Lemma \ref{lem:ex_operation} contains only $d^*_k$ segments with $k \le 0$ in dual notation. It follows that $\ex (\lp_i)$ contains only $c_k$ segments with $k \ge 0$ in standard notation; reversing the loop, we have that $\ex (\lp_i)$ can be written with only $d_k$ segments with $k \le 0$ in standard notation. Let $m_i$ denote the minimum subscript for a $d_k$ segment in $\ex( \lp_i )$, and define $\lp_{i+1}$ to be $\tw^{-m_i} \circ \ex( \lp_i )$. Note that $\lp_{i+1}$ consists of $d_k$ segments with $k \ge 0$, including at least one $d_0$. We now repeat the algorithm using $\lp_{i+1}$.

To see that this algorithm terminates, let $\kappa_i$ denote the number of $d_0$ segments in the loop $\lp_i$. Observe that $d_0$ segments in $\lp_{i+1}$ come from minimal subscripts in $\ex( \lp_i )$, which come from maximal sequences of $d_0$'s in $\lp_i$; in particular, there is at least one $d_0$ in $\lp_i$ for each $d_0$ in $\lp_{i+1}$, and so $\kappa_{i+1} \le \kappa_i$. Moreover, the inequality is strict unless $\lp_i$ has no consecutive $d_0$ segments. If $\twi( \lp_i )$ satisfies $(2)$ or $(3)$ then the algorithm terminates; otherwise, $\lp_i$ must contain $d_0 (d_1)^j d_0$ for some $j \ge 0$. Let $\lambda_i$ denote the minimal such $j$. As observed above, if $\lambda_i = 0$ then $\kappa_{i+1} < \kappa_i$. If $\lambda_i > 0$, then the subword $d_0 (d_1)^{\lambda_i} d_0$ in $\lp_i$ gives rise to the subword $d_{-2} (d_{-1})^{\lambda_i - 1} d_{-2}$ in $\ex(\lp_i)$ and the subword $d_0 (d_1)^{\lambda_i -1} d_0$ in $\lp_{i+1}$. Thus at each step in the algorithm, either $\kappa_{i+1} < \kappa_i$ or $\kappa_{i+1} = \kappa_i$ and $\lambda_{i+1} < \lambda_{i}$. Since $\kappa_i$ and $\lambda_i$ are non-negative integers, the algorithm must terminate after finitely many steps.

Let $\lp'$ be the result of the algorithm above. In cases $(1)$ and $(2)$, we can easily check that condition $(i)$ is satisfied. First note that $\infty$ is an L-space slope for $\lp'$ since it contains only unstable chains in standard notation. Moreover, for any $n\in\Z$, $\frac{1}{n}$ is an L-space slope for $\lp'$ if and only if $\infty$ is an L-space slope for $\du^{n}(\lp')$. But $\du^{n}(\lp') = \lp'$ since $\lp'$ either cannot be written in dual notation (case $(1)$) or contains only stable chains in dual notation (case $(2)$). Since the set of L-space slopes is an interval and it contains slopes arbitrarily close to 0 on both sides, it must be all of $\hat\Q \setminus \{0\}$. In case $(3)$, note that writing $\lp'$ in dual notation produces stable chains, but each $b^*$ segment is immediately followed by an $a^*$ segment. Moreover, since $\lp'$ does not satisfy $(2)$, it has at least one unstable chain in dual notation. It follows from \ref{prop:Dehn_filling_Lspaces} and \ref{prop:strict_Lspace_slopes} that 0 is an L-space slope for $\lp'$ but not a strict L-space slope. Thus 0 must be a boundary of the interval of L-space slopes. Since $\lp'$ was obtained from $\lp$ by applying a finite number of twists, the slope 0 for $\lp'$ can be expressed as a rational slope for $\lp$.

In case $(3)$, we have found one rational boundary of the interval of L-space slopes. In fact, we can check that it is the left boundary. In this case $\chi_\circ(\lp')$ and $\chi_\bullet(\lp')$ have the same sign, and so the slope of the rational longitude $-\frac{\chi_\circ}{\chi_\bullet}$ is negative. By Corollary \ref{cor:L-space-closed-interval} it follows that the set of L-space slopes for $\lp'$ contains $[0,\infty]$, and so 0 must be the left boundary of the interval of L-space slopes for $\lp'$. A similar algorithm, with opposite signs for subscripts in property $(3)$, shows that the right endpoint is also rational.
\end{proof}

The proof of Theorem \ref{thm:detection} is now complete: It follows from (and is made more precise by) Theorem \ref{thm:intervals} in combination with Proposition \ref{prp:rational endpoints}

\section{Gluing results}\label{sec:glue}

This section is devoted to proving Theorem \ref{thm:gluing}. We first prove the analogous result on the level of abstract loops, and then deduce the gluing theorem for simple loop-type manifolds. We end the section with an application to generalized splicing of L-space knot complements and give the proof of Theorem \ref{thm:gen-splice}.

\subsection{A gluing result for abstract loops}
We will say that two loops $\lp_1$ and $\lp_2$ are \emph{L-space aligned} if for every slope $\frac{r}{s} \in \hat\Q$, either $\frac{r}{s}$ is a strict L-space slope for $\lp_1$ or $\frac{s}{r}$ is a strict L-space slope for $\lp_2$.  This section is devoted to proving the following proposition:

\begin{prop}
\label{prop:gluing_loops}
If $\lp_1$ and $\lp_2$ are simple loops which are not solid-torus-like, then $\lp^A_1 \boxtimes \lp_2$ is an L-space chain complex if and only if $\lp_1$ and $\lp_2$ are L-space aligned.
\end{prop}

An essential observation is that $\lp_1$ and $\lp_2$ are L-space aligned if and only if $\tw(\lp_1)$ and $\du(\lp_2)$ are L-space aligned, since $\tw$ takes the slope $\frac{r}{s}$ to $\frac{r+s}{s}$ while $\du$ takes the slope $\frac{s}{r}$ to $\frac{s}{r+s}$. More generally, we can apply a sequence of twists $\{ \tw^{k_1}, \du^{k_2}, \ldots, \tw^{k_{2n-1}}, \du^{k_{2n}} \}$ to $\lp_1$ and a corresponding sequence of twists $\{ \du^{k_1}, \tw^{k_2}, \ldots, \du^{k_{2n-1}}, \tw^{k_{2n}} \}$ to $\lp_2$ without changing whether or not the pair of loops is L-space aligned. By Proposition \ref{prop:slope-trick}, the quasi-isomorphism type of $\lp_1^A \boxtimes \lp_2$ is also unchanged. Thus in proving Proposition \ref{prop:gluing_loops}, we may first reparametrize the pair of loops to get a more convenient form.

\begin{proof}[Proof of Proposition \ref{prop:gluing_loops}, only if direction]
Suppose $\lp_1$ and $\lp_2$ are not L-space aligned; that is, there exists a slope $\frac{p}{q}$ that is not a strict L-space slope for $\lp_1$ and $\frac{q}{p}$ is not a strict L-space slope for $\lp_2$. In fact, we may assume that $\frac{p}{q} = \infty$; if not, we reparametrize as described above, replacing $\lp_1$ with 
\[ \lp'_1 = \du^{k_{2n}} \circ \tw^{k_{2n-1}} \circ \cdots \circ \du^{k_2} \circ \tw^{k_1} (\lp_1)\]
such that the slope $\frac{p}{q}$ for $\lp_1$ becomes the slope $\infty$ for $\lp'_1$, and replacing $\lp_2$
\[ \lp'_1 = \tw^{k_{2n}} \circ \du^{k_{2n-1}} \circ \cdots \circ \tw^{k_2} \circ \du^{k_1} (\lp_2).\]
Furthermore, we may assume that $\lp_1$ contains no segments of type $d_k$, $\bar{d}_k$, $e$, or $\bar e$, and that $\lp_2$ contains no segments of type $d^*_k$, $\bar{d}^*_k$, $e^*$, or $\bar{e}^*$; if necessary, we replace $\lp_1$ with $\tw^{-n}(\lp_1)$ and $\lp_2$ with $\du^{-n}(\lp_2)$ for sufficiently large $n$.

Since $\infty$ is not a strict L-space slope for $\lp_1$, the loop $\lp_1$ must contain stable chains in standard notation (note that since $\lp_1$ is not solid torus-like, it can be written in standard notation). In particular, after possibly reversing the loop, $\lp_1$ contains a $b_k$ segment. The corresponding segment in $\lp_1^A$ is shown below:
\[
\begin{tikzpicture}[>=latex, scale = 2] 
		\node at (0,0) {$\bullet$}; 
		\node at (1,0) {$\circ$}; 
		\node at (2,0) {$\circ$};
		\node at (3,0) {$\bullet$}; 
		
		\node[below] at (0,0) {$x_1$};
		\node[below] at (1,0) {$y_1$};
		\node[below] at (2,0) {$y_k$};
		\node[below] at (3,0) {$x_2$};
		
		\draw[thick, ->,shorten >=0.1cm, shorten <=0.1cm] (0,0) -- (1,0); \node at (0.5,0.125) {$\scriptsize\rho_3, \rho_2, \rho_1$};
		\draw[dashed, thick, ->,shorten >=0.1cm, shorten <=0.1cm] (1,0) -- (2,0); \node at (1.5,0.125) {$\scriptsize\rho_2, \rho_1$};
		\draw[thick, <-,shorten >=0.1cm, shorten <=0.1cm] (2,0) -- (3,0); \node at (2.5,0.125) {$\scriptsize\rho_{3}$};
		\end{tikzpicture}
	\]
Since $0$ is not a strict L-space slope for $\lp_2$, this loop must contain stable chains in dual notation. In particular it contains an $a^*_\ell$ segment; we label the corresponding generators as follows:
\[
\begin{tikzpicture}[>=latex, scale = 2] 
		\node at (0,0) {$\circ$}; 
		\node at (1,0) {$\bullet$}; 
		\node at (2,0) {$\bullet$};
		\node at (3,0) {$\circ$};
		
		\node[below] at (0,0) {$w_1$};
		\node[below] at (1,0) {$z_1$};
		\node[below] at (2,0) {$z_\ell$};
		\node[below] at (3,0) {$w_2$};
		
		\draw[thick, <-,shorten >=0.1cm, shorten <=0.1cm] (0,0) -- (1,0); \node at (0.5,0.125) {$\scriptsize\rho_{3}$};
		\draw[dashed, thick, ->,shorten >=0.1cm, shorten <=0.1cm] (1,0) -- (2,0); \node at (1.5,0.125) {$\scriptsize\rho_{12}$};
		\draw[thick, ->,shorten >=0.1cm, shorten <=0.1cm] (2,0) -- (3,0); \node at (2.5,0.125) {$\scriptsize\rho_{123}$};
	\end{tikzpicture}	\]

Consider the generator $y_k$ in $\lp_1^A$, which has no outgoing $\Ainfty$ operations. To determine the possible incoming operations, note that the segment $b_k$ must be followed by either a type $c$ segment or a type $a$ segment. This is because we assumed that $\lp_1$ contains no $\bar{e}$ or $\bar{d}_j$ segments, and $\lp_1$ can not contain both $b_k$ and $\bar{a}_j$ segments since it is simple. In either case, $x_2$ has an outgoing $\rho_3$ labeled arrow in $\lp_1^A$ and no incoming arrows. It follows that $y_k$ has only the following incoming $\Ainfty$ operations:
\begin{eqnarray*}
m_2(x_2, \rho_3) &=& y_k, \\
m_{2+i}(y_{k-i}, \rho_2, \rho_{12}, \ldots, \rho_{12}, \rho_2) &=& y_k \quad \text{ for } 1\le i < k, \\
m_{3+k}(x_1, \rho_3, \rho_2, \rho_{12}, \ldots, \rho_{12}, \rho_2) &=& y_k,
\end{eqnarray*}
and possibly more operations whose inputs end with $\rho_2, \rho_{12}, \ldots, \rho_{12}, \rho_1$.

Consider the generator $w_2$ in $\lp_2$. Note that $a^*_\ell$ must be followed by either a $b^*$ segment or a $\bar{c}^*$ segment. It follows that the only incoming sequences of arrows consist of a $\rho_{123}$ or $\rho_1$ arrow preceeded by some number of $\rho_{12}$ arrows. Comparing this with the $\Ainfty$ operations terminating at $y_k$ described above, it is clear that the generator $y_k\otimes w_2$ in $\lp_1^A\boxtimes\lp_2$ has no incoming differentials. It also has no outgoing differentials, since $y_k$ has no outgoing $\Ainfty$ operations. Thus $y_k\otimes w_2$ survives in homology.

Similarly, consider the generator $x_1$ in $\lp_1^A$ and $z_1$ in $\lp_2$. The generator $z_1$ has no incoming sequences of arrows, and the only outgoing seqeuences consist of a single $\rho_3$ arrow or begin with some number of $\rho_{12}$ arrows followed by a $\rho_{123}$. Here we use the fact that the segment $a^*_\ell$ in the simple loop $\lp_2$ can only be preceded by a $b^*$ segment or a $c^*$ segment, so the outgoing $\rho_3$ arrow can not be followed by another outgoing arrow. In $\lp_1^A$, the segment $b_k$ must be preceded by an $a$ segment or a $\bar{c}$ segment. It follows that for any nontrivial operation $m_{n+1}(x_1, \rho_{I_1}, \ldots, \rho_{I_n})$, we have 
\begin{itemize}
\item $\rho_{I_1} \neq \rho_{123}$;
\item if $\rho_{I_1} = \rho_{12}$, then $\rho_{I_i} = \rho_{12}$ for $1 \le i \le n-1$ and $\rho_{I_n} = \rho_1$; 
\item if $\rho_{I_1} = \rho_3$, then $n>1$.
\end{itemize}
We see that no $\Ainfty$ operations starting at $x_1$ match with the $\delta^n$ maps starting at $z_1$. Thus the generator $x_1\otimes z_1$ in $\lp_1^A\boxtimes\lp_2$ has no incoming or outgoing differentials and survives in homology.

Finally, we observe that $gr(z_1) = gr(w_2)$ since $z_1$ and $w_2$ are connected by only $\rho_{12}$ and $\rho_{123}$ arrows, but $gr(y_k) = gr(y_1) = -gr(x_1)$, since arrows labelled $(\rho_2, \rho_1)$ preserve grading but arrows labelled $(\rho_3, \rho_2, \rho_1)$ flip grading. It follows that $y_k\otimes w_2$ and $x_1\otimes z_1$ have opposite $\Z/2\Z$ grading. Since both survive in homology, $\lp_1^A\boxtimes\lp_2$ is not an L-space complex.
\end{proof}

To prove the converse we will use the fact that $\lp_1$ and $\lp_2$ are L-space aligned to put strong restrictions on the segments that may appear in the loops $\lp_1$ and $\lp_2$. Once again, we can apply twists to $\lp_1$ and $\lp_2$ to obtain $\lp'_1$ and $\lp'_2$ with convenient parametrizations, such that $\lp'_1$ and $\lp'_2$ are still L-space aligned and $\lp_1^A\boxtimes\lp_2$ is homotopy equivalent to $(\lp'_1)^{A}\boxtimes\lp'_2$. The set of strict L-space slopes for $\lp_2$ is some nonempty open interval in $\hat\Q$. This interval is either all of $\hat\Q$ except the rational longitude or it has distinct rational endpoints; see Proposition \ref{prp:rational endpoints}. In the latter case, we can reparametrize so that for $\lp'_2$ these boundaries have slope $0$ and $\frac{p}{q}$ for some $1 < \frac{p}{q} \le \infty$. To see this, take $p>0$ and choose $n$ so that $0\le q+np<p$ and $1<\frac{p}{q+np}\le\infty$; we can replace $\frac{p}{q}$ with $\frac{p}{q+np}$ by applying dual twists, in particular, leaving the slope 0 fixed. Now the set of strict L-space  slopes for $\lp'_2$ is exactly $(0, \frac{p}{q})$. The fact that $\lp'_1$ and $\lp'_2$ are L-space aligned then implies that the set of strict L-space slopes for $\lp'_1$ contains $[-\infty, \frac{q}{p}]$. If the set of strict L-space slopes for $\lp_2$ is all of $\hat\Q$ except the rational longitude, then we can chose a parametrization such that the rational longitude of $\lp'_2$ is 0 and such that the set of strict L-space slopes for $\lp'_1$ contains $[-\infty, 0]$.

\begin{lem}
\label{lem:loop_restrictions_L1}
If $\frac{q}{p} \in [0, 1)$ and $\lp$ is a simple loop for which the interval of strict L-space slopes contains $[-\infty, \frac{q}{p}]$ then $\lp$ consists only of  segments $c^*_k$ with $0\le k \le  \lceil \frac{p}{q} \rceil$.
\end{lem}
\begin{proof}
Since $0$ is a strict L-space slope, $\lp$ can be written with only dual unstable chains. Up to reading the loop in reverse order, we can assume the unstable chains are $c^*_k$ segments. Moreover, the fact that $\infty$ is a strict L-space slope implies that $\lp$ can not contain $c_k^*$ segments with both positive and negative subscripts. Since the rational longitude is given by $-\chi_\circ(\lp)/\chi_\bullet(\lp)$ and falls in the interval $(\frac{q}{p}, \infty)$, we must have that $\chi_\circ(\lp)$ and $\chi_\bullet(\lp)$ have opposite signs. This only happens if $\lp$ is composed of $c^*_k$ segments with $k\ge0$. Let $n=\lceil \frac{p}{q} \rceil$. Observe that $\lp$ must contain at least one $c^*_k$ with $0\le k < n$ (recall that $c_0=\bar{e}^*$), since otherwise the rational longitude is less than $\frac{1}{n}$. Finally, the fact that $\frac{1}{n}$ is a strict L-space slope implies that $\infty$ is a strict L-space slope for the loop $\du^{n}( \lp )$. Since $\lp$ contains $\bar{e}^*$ or $c^*_k$ with $k < n$, $\du^{n}( \lp )$ contains at least one $c^*_k$ with $k<0$ and therefore does not contain any $c^*_k$ with $k>0$. Therefore $\lp$ does not contain $c^*_k$ with $k > n$.
\end{proof}

\begin{lem}
\label{lem:loop_restrictions_L2}
If $\frac{p}{q} \in (1, \infty]$ and $\lp$ is a simple loop that is not solid-torus like for which the interval of strict L-space slopes contains $(0, \frac{p}{q})$ then  $\lp$ consists only of $a_k$, $b_k$, $c_k$ and $d_k$ segments (for $k>0$)  and $e$ and $\bar{c}_1$ segments. Moreover 
\begin{itemize}
\item $\lp$ contains no two $\bar{c}_1=d_{-1}$ segments separated only by $e=d_0$ segments;
\item $\lp$ contains no $c_k$ segments with $k < \lceil \frac{p}{q} \rceil-1$; and 
\item if $0$ is not a strict L-space slope for $\lp$ then there is at least one $\bar{c}_1$ segment.
\end{itemize}
\end{lem}

\begin{proof}
Since 1 is a strict L-space slope, $\tw(\lp)$ can be written with no $\bar{a}, \bar{b}, \bar{c},$ or $\bar{d}$ segments. It follows that $\lp$ can be written with no $\bar{a}, \bar{b}, \bar{d}$, or $\bar{e}$ segments and no $\bar{c}_k$ segments with $k > 1$.

Let $n=\lceil \frac{p}{q} \rceil$. Since $n-1 < \frac{p}{q}$ is a strict L-space slope, $\tw^{n-1}(\lp)$ does not contain both barred and unbarred segments (ignoring $e$'s). Suppose $\lp$ contains $c_k$ with $k < n-1$. Then  
$\tw^{n-1}(\lp)$ contains at least one $\bar{d}$ segment and can not contain any unbarred segments. Any $a$, $b$, or $c_k$ segments with $k > n-1$ in $\lp$ produces an unbarred segment in $\tw^{n-1}(\lp)$, so we must have that $\lp$ consists only of $c_k$ segments with $k \le n-1$. However, in this case it is easy to see that $\chi_\bullet(\lp)$ and $\chi_\circ(\lp)$ have opposite signs and $|\chi_\circ(\lp)| \le (n-1) |\chi_\bullet(\lp)|$, which contradicts the fact that the rational longitude $-\chi_\circ/\chi_\bullet$ does not fall in the interval $(0, p/q)$. Thus $\lp$ does not contain $c_k$ with $k < n-1$.

$\lp$ must contain an $a_k, b_k, c_k$, or $d_k$ segment with $k>1$ or two segments of type $a_1, b_1, c_1,$ or $d_1$ separated only by $e$'s. Otherwise, $\lp$ would consist of only $\bar{c}_1$, $e$, and $d_1$ segments with at least one $\bar{c}_1$ for each $d_1$; in this case, $\chi_\bullet(\lp)$ and $\chi_\circ(\lp)$ have opposite signs and $|\chi_\circ(\lp)| \le |\chi_\bullet(\lp)|$, so the rational longitude falls in $(0, 1]$. It follows that in dual notation $\lp$ contains $\bar{c}^*_k, e^*$, or $d^*_k$, and thus $\du^m(\lp)$ contains a $d^*$ segment for sufficiently large $m$. Since $1/m$ is a strict L-space slope for sufficiently large $m$ we must have that $\du^m(\lp)$ does not contain a $\bar{d}^*$ segment. Therefore $\lp$ does not contain any $c^*_k$ segments, and thus in standard notation it does not have two $\bar{c}_1$ segments separated only by $e$'s.

Finally, if $0$ is not a strict L-space slope for $\lp$ then $\lp$ must contain both barred and unbarred segments in standard notation; it follows that $\lp$ must contain at least one $\bar{c}_1$.
\end{proof}

The two previous Lemmas only depend on $\lceil \frac{p}{q} \rceil$. If $\frac{p}{q}$ is not an integer, it is possible to give further restrictions on subwords that can appear in the loop. We will prove one such restriction using two properties for a pair of loops. For an integer $r \ge 0$, we will say that two loops $\lp_1$ and $\lp_2$ satisfy \proplambda\ (or, \proplambda\ for $r \ge 0$) if:
\begin{itemize}
\item[$(\boldsymbol\lambda\boldsymbol 1)$] 
$\lp_1$ consists only of $c^*_k$ with $0\le k \le n$, for some $n$, with at least one $c^*_n$;
\item[$(\boldsymbol\lambda\boldsymbol 2)$] $\lp_2$ consists only of $a_k$, $b_k$ ($k>0$), $d_k$ ($k\ge-1$), and $c_l$ ($\l \ge m$) segments for some $m>0$, with at least one $c_m$ and at least one $\bar{c}_1=d_{-1}$;
\item[$(\boldsymbol\lambda\boldsymbol 3)$] There is an integer $N>0$ and subscripts $k_i \in \{N, N-1\}$ for $1\le i\le r$ such that $\lp_1$ contains the subword $c^*_N c^*_{k_1} \ldots c^*_{k_r} c^*_N$ and $\lp_2$ contains the subword $c_{N-1} c_{k_1} \ldots c_{k_r} c_{N-1}$.
\end{itemize}
The integer $r\ge0$ appearing in condition $(\boldsymbol\lambda\boldsymbol 3)$ is the complexity of the pair $(\lp_1,\lp_2)$ satisfying \proplambda; when we appeal to pairs satisfying \proplambda\ the aim will be to decrease this complexity. Similarly, two loops $\lp_1$ and $\lp_2$ satisfy property \proplambdaDUAL\ (or, \proplambdaDUAL\ for $r \ge 0$) if:
\begin{itemize}
\item[$(\boldsymbol\lambda^{\!*}\boldsymbol 1)$]  
$\lp_1$ consists only of $\bar{d}_k$ with $0\le k \le n$, for some $n$, with at least one $\bar{d}_n$;
\item[$(\boldsymbol\lambda^{\!*}\boldsymbol 2)$]  
  $\lp_2$ consists only of $\bar{a}^*_k$, $\bar{b}^*_k$ (for $k>0$), $\bar{c}^*_k$ (for $k>-1$), and $\bar{d}^*_l$ (for $l\ge m$) segments for some $m$, with at least one $\bar{d}^*_m$ and at least one $d^*_1=\bar{c}^*_{-1}$;
\item[$(\boldsymbol\lambda^{\!*}\boldsymbol 3)$]  
 There is an integer $N>0$ and subscripts $k_i\in\{N, N-1\}$ for $1\le k\le r$ such that $\lp_1$ contains the subword $\bar{d}_N \bar{d}_{k_1} \ldots \bar{d}_{k_r} \bar{d}_N$ and $\lp_2$ contains the subword $\bar{d}^*_{N-1} \bar{d}^*_{k_1} \ldots \bar{d}^*_{k_r} \bar{d}^*_{N-1}$.
 \end{itemize}

\begin{lem}
\label{lem:proplambda}
If two simple loops $\lp_1$ and $\lp_2$ satisfy either \proplambda\ or \proplambdaDUAL\ then $\lp_1$ and $\lp_2$ are not L-space aligned.
\end{lem}

\begin{figure}[ht!]
\begin{tikzpicture}
\draw [fill=lightgray,lightgray] (1.85,-0.25) rectangle (3.25,1.25);
\draw [fill=lightgray,lightgray] (6.85,-0.25) rectangle (8.25,1.25);
\node at (0,1) {$\lp_1$}; \node at (2,1) {$\phantom{a_1}\bar e^* c_3^* c_2^*c_2^*c_1^*c_2^*$};
\node at (0,0) {$\lp_2$}; \node at (2,0) {$a_1\bar c_1b_1 c_1c_2c_1c_1$};
\node at (9,1) {$\lp_1'$}; \node at (7,1) {$\phantom{a_1} \bar d_2^* c_1^* \bar e^*\bar e^*\bar d_1^*\bar e^*$};
\node at (9,0) {$\lp_2'$}; \node at (7,0) {$a_1d_1b_1 \bar d_1\bar e\bar d_1\bar d_1$};
\draw [thick, -> ,>=latex] (3.75,0) to (5.25,0);\draw [thick, -> ,>=latex] (3.75,1) to (5.25,1);
\node at (4.5, 0.35) {$\tw^2$};\node at (4.5, 1.35) {$\du^2$};
\end{tikzpicture}
\caption{Loops $\lp_1$ and $\lp_2$ satisfying \proplambda\ for $n=3$ and $m=1$. The relevant subwords, with  $N=r=2$, have been highlighted. Note that this illustrates the key step in the proof of Lemma \ref{lem:proplambda} when $n>m+1$: One checks that $0$ is not a strict L-space slope for $\lp_1'$ (this loop contains a $\bar d^*$ segment) while the set of strict L-space slopes for $\lp_2'$ does not intersect $[0,\infty]$ thus $\lp_1'$ and $\lp_2'$ (and hence $\lp_1$ and $\lp_2$) are not L-space aligned.} 
\label{fig:property lambda example}
\end{figure}

\begin{proof}
Suppose first that \proplambda\ is satisfied, where $n$ is the maximum subscript for $c^*_k$ segments in $\lp_1$ and $m$ is the minimum subscript of $c_k$ segments in $\lp_2$. Since $\lp_1$ contains $c^*_N$ and $\lp_2$ contains $c_{N-1}$ for some $N$, we have that $n \ge m+1$. Consider the reparametrized loops $\lp'_1 = \du^{m+1}(\lp_1)$ and $\lp'_2 = \tw^{m+1}(\lp_2)$; $\lp'_1$ and $\lp'_2$ are L-space aligned if and only if $\lp_1$ and $\lp_2$ are. We  now have:
\begin{itemize}
\item 
$\lp'_1$ consists only of $c^*_k$ with $-m\le k \le n-m-1$ with at least one $c_{n-m-1}^*$ (where, as usual, $c_0^*=\bar{e}^*$);
\item 
$\lp'_2$ consists only of $a_k$, $b_k$ (for $k>0$), $c_k$ (for $k\ge-1$), and $d_l$ segments with $l \ge m$, with at least one $c_{-1}=\bar{d}_1$ and at least one $d_m$.
\end{itemize}
Note that for $\lp'_2$, $0$ is not a strict L-space slope because the loop contains at least one barred segment, $\bar{d}_1$, and at least one unbarred segment, $d_m$. Note also that $\infty$ is not a strict L-space slope, since $\lp'_2$ contains unstable chains with both orientations. The slope $-1$ is a strict L-space slope, since $\tw^{-1}(\lp_2')$ has no barred segments (ignoring $e$'s). It follows that $\lp'_2$ has no strict L-space slopes in $[0, \infty]$.

First consider the case that $n > m+1$. In this case, $\lp'_1$ contains at least one $c_k^*$ segment with $k>0$. If $\lp'_1$ also contains a $c^*_k$ segment with $k<0$, then $0$ is not a strict L-space slope for $\lp'_1$. Since $\infty$ is not a strict L-space slope for $\lp'_2$,  $\lp'_1$ and $\lp'_2$ are not L-space aligned. If $\lp'_1$ does not contain a $c_k^*$ segment with $k<0$, then it consists only of $c^*_k$ segments with $k\ge0$. It follows that the rational longitude $-\chi_\circ(\lp'_1) / \chi_\bullet(\lp'_1)$ is positive. Since the rational longitude is not a strict L-space slope, and all positive slopes for $\lp'_2$ are not strict L-space slopes, $\lp'_1$ and $\lp'_2$ are not L-space aligned.

In the case that $n = m+1$, we have $N = n$ in the statement of \proplambda\ (with complexity $r$) for $\lp_1$ and $\lp_2$. The shifted loops $\lp'_1$ and $\lp'_2$ satisfy an additional condition:
\begin{itemize}
\item 
There are subscripts $k'_i \in \{0, 1\}$ for $1\le i \le r$ such that $\lp'_1$ contains the subword $\bar{e}^* \bar{d}^*_{k'_1} \ldots \bar{d}^*_{k'_r} \bar{e}^*$ and $\lp'_2$ contains the subword $\bar{d}_1 \bar{d}_{k'_1} \ldots \bar{d}_{k'_r} \bar{d}_1$, following (as usual) the convention that $\bar{d}_0 = \bar{e}$ and $\bar{d}^*_0 = \bar{e}^*$.
\end{itemize}
The next step is to write $\lp'_1$ in standard notation and $\lp'_2$ in dual notation. $\lp'_1$ consits only of $\bar{e}^*$ and $\bar{d}^*$ segments, so in standard notation it consits only of $\bar{e}$ and $\bar{d}$ segmetents. There is some maximum subscript on the $\bar{d}$ segments; call it $n'$. Note that this says that $\lp'_1$ and $\lp'_2$ satisfy condition $(\boldsymbol\lambda^{\! *}\boldsymbol 1)$ for \proplambdaDUAL. Dualizing $\lp'_2$ is slightly harder. Types of segments in dual notation are determined by adjacent pairs of standard segments (ignoring $e$'s and $\bar{e}$'s); see Section \ref{sec:notations_for_loops}. The possible pairs of segments in $\lp'_2$ are $ab, ad, dd, db, ba, bc, b\bar{d}, ca, cc, c\bar{d}, \bar{d}a, \bar{d}c$, and $\bar{d}\bar{d}$. These correspond to dual segments $d^*, d^*, d^*, d^*, \bar{c}^*, \bar{c}^*, \bar{b}^*, \bar{c}^*, \bar{c}^*, \bar{b}^*, \bar{a}^*, \bar{a}^*$, and $\bar{d}^*$, respectively. Since $\lp'_2$ has no $e$ segments, the subscripts on the $d^*$ segments are at most 1. Since the only bar segments in $\lp'_2$ in standard notation have subscript 1, there are no $\bar{e}^*$ segments in $\lp'_2$. Thus $\lp'_2$ consists only of $\bar{a}^*_k, \bar{b}^*_k, \bar{c}^*_k, \bar{d}^*_k$ (for $k>0$), $e^*$, and $d^*_1$ segments. There is some minimum subscript on the $\bar{d}$ segments; call it $m'$. Moreover, since $\lp'_2$ contains at least one $d_k$ segment, it also contains at least one $d^*_1$ segment. Note that this says that $\lp'_1$ and $\lp'_2$ satisfy condition $(\boldsymbol\lambda^{\! *}\boldsymbol 2)$ of \proplambdaDUAL.

If $n' > m'+1$, then we proceed in a similar fashion to the $n>m+1$ case for \proplambda. We replace $\lp'_1$ with $\lp''_1 = \tw^{-m'-1}(\lp'_1)$ and $\lp'_2$ with $\lp''_2 = \du^{-m'-1}(\lp'_2)$. We can observe that $\lp''_2$ contains at least one $c^*_1$ segment and at least one $\bar{c}_{m'}$ segment; thus $0$ and $\infty$ are not strict L-space slopes for $\lp''_2$, and in fact no slope in $[-\infty, 0]$ is a strict L-space slope. We can also observe that $\lp''_1$ consists of $c$, $\bar{e}$, and $\bar{d}$ segements with at least one $\bar{d}$. Thus either $0$ is not a strict L-space slope or the rational longitude is negative. In either case, $\lp''_1$ and $\lp''_2$ are not L-space aligned.

Now assume that $n' = m'+1$. Consider the sequence $k'_1, \ldots, k'_r$. This sequence consists of some number of (possibly empty) strings of 0's each separated by a single 1; let $l_1, \ldots, l_s$ be the sequence of the lengths of these strings of 0's. Note that $s$ is at most $r+1$. Since $\lp'_2$ contains the word $\bar{d}_1 \bar{d}_{k'_1} \ldots \bar{d}_{k'_r} \bar{d}_1$, it follows that it contains the dual word $\bar{d}^*_{l_1} \bar{d}^*_{l_2} \ldots \bar{d}^*_{l_{s-1}} \bar{d}^*_{l_s}$. Similarly, $\lp'_1$ contains the dual word $\bar{e}^* \bar{d}^*_{k'_1} \ldots \bar{d}^*_{k'_r} \bar{e}^*$, which must be followed and proceeded by more $\bar{d}^*$ segments, possibly with additional $\bar{e}^*$ segments in between. It follows that $\lp'_1$ contains the the word $\bar{d}_{l'_1} \bar{d}_{l_2} \ldots \bar{d}_{l_{s-1}} \bar{d}_{l'_s}$, where $l'_1 > l_1$ and $l'_s > l_s$. Since we have assumed that $n' = m'+1$, we must have that $l'_1 = l_1 + 1$ and $l'_s = l_s + 1$, and moreover that $l_1 = l_s$ and $l_i\in\{l_1,l_1 + 1\}$ for every other $i$. Let $N' = l_1 + 1$ so that $\lp'_1$ contains the subword $\bar{d}_{N'} \bar{d}_{l_2} \ldots \bar{d}_{l_{s-1}} \bar{d}_{N'}$ and $\lp'_2$ contains the subword $\bar{d}^*_{N'-1} \bar{d}^*_{l_1} \ldots \bar{d}^*_{l_{s-1}} \bar{d}^*_{N'-1}$. In other words, the property $(\boldsymbol\lambda^{\! *}\boldsymbol 3)$ of \proplambdaDUAL\ is satisfied with complexity $r' = s-2$. Note that $r' < r$, since $s \le r+1$. We also have that $r' \ge 0$, since if $s=1$ then each $k'_i$ is 0 for each $1 \le i \le r$; it would follow that $\lp'_2$ contains $\bar{d}^*_{r+1}$ and $\lp'_1$ contains a $\bar{d}$ segment with subscript at least $r+3$, and so $n' > m' + 1$.

We have shown that if $\lp_1$ and $\lp_2$ satisfy \proplambda\ (for some integer $r\ge 0$), then either they are not L-space aligned or they can be modified to $\lp'_1$ and $\lp'_2$ which satisfy \proplambdaDUAL\ (for some integer $r'\ge 0$) where $0 \le r' < r$. A similar proof shows that  if $\lp_1$ and $\lp_2$ satisfy \proplambdaDUAL\, then either they are not L-space aligned or they can be modified to $\lp'_1$ and $\lp'_2$ which satisfy \proplambda\ with $0 \le r' < r$. In both cases, the complexity is reduced, so by induction on $r$, we have that $\lp_1$ and $\lp_2$ are not L-space aligned if they satisfy either property.
\end{proof}

\parpic[r]{\begin{tikzpicture}[>=latex, thick] 
\draw [fill=lightgray,lightgray] (-0.5,-0.5) rectangle (0.2,0.2);
\draw [fill=lightgray,lightgray] (3.8,-0.5) rectangle (5.2,0.2);
\node (a) at (2,0) {$\bullet$};
\node (b) at (1,0) {$\circ$};
\node (c) at (0,0) {$\bullet$};
\node (d) at (0,1) {$\circ$};
\node (e) at (0,2) {$\bullet$};
\node (f) at (1,2) {$\circ$};
\node (g) at (2,1) {$\circ$};
\draw[->] (a) to node[below]{$\rho_3$} (b);
\draw[->] (b) to node[below]{$\rho_2$} (c);
\draw[->] (c) to node[left]{$\rho_{123}$} (d);
\draw[->] (e) to node[left]{$\rho_1$} (d);
\draw[->] (f) to node[above]{$\rho_2$} (e);
\draw[->] (g) to node[above right]{$\rho_{23}$} (f);
\draw[->] (a) to node[right]{$\rho_{123}$} (g);
\node (a) at (7,0) {$\bullet$};
\node (c) at (5,0) {$\bullet$};
\node (e) at (5,2) {$\bullet$};
\draw[->] (a) to (c);
\draw[->,very thick] (2.75,1) to (4.4,1); \node at (3.5,1.35) {$\lp^A_\bullet\boxtimes\cdot$};
\node at (-0.2,-0.2) {$y$};\node at (4.45,-0.2) {$x\otimes y$};
\end{tikzpicture}} 
Now that we have placed restrictions on loops $\lp_1$ and $\lp_2$ which are L-space aligned, we can complete the proof of Proposition \ref{prop:gluing_loops} by analyzing the box tensor product segment by segment. In the proof, we will determine when certain generators in $\lp_1^A\boxtimes \lp_2$ cancel in homology. To aid in this, we introduce the following terminology: We refer to differentials starting or ending at $x\otimes y$ in $\lp_1^A\boxtimes \lp_2$ as being ``on the right" or ``on the left" depending on whether the type D operations in $\lp_2$ that give rise to the differential in the tensor product are to the right or left of $y$, relative to the cyclic ordering on $\lp_2$. This is motivated by picturing the tensor product on a grid with rows indexed by generators of $\lp_1^A$ and columns indexed by generators of $\lp_2$, as in Figure \ref{fig:c2_tensor_c3}. In the example shown (having fixed the loop $(d_2\bar b_1\bar a_1)$, that is, reading the loop counter-clockwise) the generator $x\otimes y$ cancels on the right when paired with the standard solid torus (recall that this example may be identified with the trivial surgery on the right hand trefoil). With this terminology in place, we observe:

\begin{lem}
For any loops $\lp_1$ and $\lp_2$ and any generator $x\otimes y \in \lp_1^A\boxtimes \lp_2$, there is at most one differential into or out of $x\otimes y$ on the right, and at most one on the left.
\end{lem}
\begin{proof}
This follows from examining the arrows of type ${\bf I}_\bullet$, ${\bf II}_\bullet$, ${\bf I}_\circ$, and ${\bf II}_\circ$ introduced in Section \ref{sub:loop-type} and the corresponding type A arrows. For instance, suppose $y$ has a type ${\bf I}_\bullet$ arrow on the right.  A differential on the right of $x\otimes y$ must arise from and $\Ainfty$ operation starting at $x$ with inputs starting with $\rho_1$, $\rho_{12}$, or $\rho_{123}$. Such an $\Ainfty$ operation only exists if the directed graph representing $\lp_1^A$ has an arrow from $x$ labeled by $\rho_1$ (that is, if $m_2(x, \rho_1)$ is nontrivial); there is at most one such arrow from $x$ in $\lp_1^A$, since the corresponding vertex in $\lp_1$ is adjacent to only one arrow of type ${\bf II}_\bullet$. Any sequence of arrows in $\lp_1^A$ starting at $x$ with the first labelled by $\rho_1$ gives rise to a sequence of outgoing $\Ainfty$ operations at $x$, each with the first input $\rho_1$, $\rho_{12}$, or $\rho_{123}$. However, none of these operations have inputs which are the first $n$ inputs of another operation; thus at most one operation can pair with a sequence of arrows on the right of $y$ in $\lp_2$. It follows that $x\otimes y$ has at most one differential on the right.

Similar arguments show that there is at most one differential on the right if the arrow on the right of $y$ is type ${\bf II}_\bullet$, ${\bf I}_\circ$, or ${\bf II}_\circ$, and the same is true on the left.
\end{proof}

Thus, all differentials in $\lp_1^A\boxtimes \lp_2$ appear in linear chains and $x\otimes y$ will cancel in homology if it has an odd length chain of differentials on either side; we say that $x\otimes y$ cancels on the right (resp. left) if there is an odd length chain of differentials on the right (resp. left). If $x\otimes y$ does not cancel on the right or the left, we say $x\otimes y$ does not cancel in homology, since it, or potentially a linear combination of it with other generators of the same $(\Ztwo)$-grading, survives in homology.

\begin{proof}[Proof of Proposition \ref{prop:gluing_loops}, if direction]
Suppose $\lp_1$ and $\lp_2$ are L-space aligned. Up to changing parametrization, we can assume that the interval of strict L-space slopes for $\lp_2$ contains $(0, \frac{p}{q})$ for some $1 < \frac{p}{q} \le \infty$ and does not contain 0 and the interval of strict L-space slopes for $\lp_1$ contains $[-\infty, \frac{q}{p}]$. By Lemmas \ref{lem:loop_restrictions_L1} and \ref{lem:loop_restrictions_L2}, $\lp_1$ can be written with only $e^*$ and $c^*_k$ segments and $\lp_2$ consists of $a_k$, $b_k$, $c_k$, $d_k$, $e$, and $\bar{c}_1$ segments. We will fix the $(\Ztwo)$-grading on each loop so that every generator of $\lp_1^A$ has grading 0 and all $\iota_1$ generators in $\lp_2$ have grading 0 except those coming from $\bar{c}_1$ segments. 

Consider a generator $x \otimes y$ in $\lp_1^A\boxtimes \lp_2$. $x$ belongs to a segment $s_1$ in $\lp_1^A$ and $y$ belongs to a segment $s_2$ in $\lp_2$. We will consider cases depending on the type of the segments $s_1$ and $s_2$ and in each case show that either $x\otimes y$ has grading 0 or it cancels in homology. Therefore $\lp_1^A\boxtimes \lp_2$ is an L-space complex.

First note that $\gr(x)$ is always 0 by assumption. If $s_2$ is and $e$, $d_k$, or $b_k$ segment then $y$ also has grading 0 and the grading of $x\otimes y$ is 0. If $s_1$ is $\bar{e}^*$ then $x$ is in idempotent $\iota_1$ and so $y$ must also have idempotent $\iota_1$. All $\iota_1$ generators of of $\lp_2$ have grading 0 except those in $\bar{c}_1$ segments, so $x\otimes y$ has grading 0 if $s_1$ is $\bar{e}^*$ and $s_2$ is not $\bar{c}_1$.

Suppose that $s_1$ is $c^*_k$ and $s_2$ is $a_\ell$ with generators labeled as follows:
\begin{center}
\begin{tikzpicture}[>=latex]
\node at (-1,0) {$s_1 =$};

\node (x0) at (0,0) {}; 
\node (x1) at (2,0) {$\bullet$};
\node (x2) at (4,0) {$\bullet$};
\node (x3) at (6,0) {$\bullet$};
\node (x4) at (8,0) {$\circ$};

\node[below] at (x1) {$x_1$};
\node[below] at (x2) {$x_2$};
\node[below] at (x3) {$x_k$};
\node[below] at (x4) {$x_{k+1}$};

\node[above right = -2pt] at (x1) {\tiny $+$};
\node[above right = -2pt] at (x2) {\tiny $+$};
\node[above right = -2pt] at (x3) {\tiny $+$};
\node[above right = -2pt] at (x4) {\tiny $+$};

\draw[->] (x1) to node[above]{$\rho_1$} (x0);
\draw[->] (x1) to node[above]{$\rho_3, \rho_2$} (x2);
\draw[->, dashed] (x2) to node[above]{$\rho_3, \rho_2$} (x3);
\draw[->] (x3) to node[above]{$\rho_{3}$} (x4);
\end{tikzpicture}
\begin{tikzpicture}[>=latex]
\node at (-1,0) {$s_2 =$};

\node (x0) at (0,0) {$\bullet$};
\node (x1) at (2,0) {$\circ$};
\node (x2) at (4,0) {$\circ$};
\node (x3) at (6,0) {$\circ$};
\node (x4) at (8,0) {$\bullet$};

\node[below] at (x0) {$y_0$};
\node[below] at (x1) {$y_1$};
\node[below] at (x2) {$y_2$};
\node[below] at (x3) {$y_\ell$};
\node[below] at (x4) {$y_{\ell+1}$};

\node[above right = -2pt] at (x0) {\tiny $-$};
\node[above right = -2pt] at (x1) {\tiny $+$};
\node[above right = -2pt] at (x2) {\tiny $+$};
\node[above right = -2pt] at (x3) {\tiny $+$};
\node[above right = -2pt] at (x4) {\tiny $+$};

\draw[->] (x0) to node[above]{$\rho_3$} (x1);
\draw[->] (x1) to node[above]{$\rho_{23}$} (x2);
\draw[->, dashed] (x2) to node[above]{$\rho_{23}$} (x3);
\draw[->] (x3) to node[above]{$\rho_{2}$} (x4);
\end{tikzpicture}
\end{center}
Every generator in each segment has grading 0 except for $y_0$, so $x\otimes y$ has grading 0 unless $y = y_0$ and $x = x_i$ for $i \in \{1, \ldots, k\}$. For each $i$, $\lp_1^A$ has the operation
$$m_{2+k-i}(x_i, \rho_3, \rho_{23}, \ldots, \rho_{23}) = x_{k+1}.$$
If $k-i < \ell$ it follows that in the tensor product there is a differential from $x_i \otimes y_0$ to $x_{k+1} \otimes y_{k-i+1}$. It is not difficult to check that $x_{k+1} \otimes y_{k-i+1}$ does not cancel from the right; the arrow in $\lp_2$ to the right of $y_{k-i+1}$ is outgoing and $s_1$ is followed by a $c^*$ or $\bar{e}^*$ segment so $x_{k+1}$ has only incoming $\Ainfty$ operations. Therefore in this case $x_i \otimes y_0$ cancels in homology. If instead $k-i \ge \ell$ then $\lp_1^A$ has the operation
$$m_{2+\ell}(x_i, \rho_3, \rho_{23}, \ldots, \rho_{23}, \rho_2) = x_{i + \ell}$$
and there is a differential in the box tensor product from $x_i \otimes y_0$ to $x_{i+\ell} \otimes y_{\ell+1}$. Again it is not difficult to see that $x_{i+\ell} \otimes y_{\ell+1}$ does not cancel from the right, since the arrow in $\lp_2$ to the right of $y_{\ell+1}$ must be an outgoing $\rho_1, \rho_{12}$ or $\rho_{123}$ arrow and $x_{i+\ell}$ has no outgoing $\Ainfty$ operations with first input $\rho_1, \rho_{12}$ or $\rho_{123}$. Thus $x_i \otimes y_0$ cancels in homology.

Suppose that $s_1$ is $c^*_k$, with generators labeled as above, and $s_2$ is $\bar{c}_1$ with generators labeled as follows (note that the generator $y_0$ is not actually part of the segment $s_2$):
\begin{center}
\begin{tikzpicture}[>=latex]
\node at (-1,0) {$s_2 =$};
\node (x0) at (0,0) {$\bullet$};
\node (x1) at (2,0) {$\circ$};
\node (x2) at (4,0) {$\bullet$};
\node[above right = -2pt] at (x0) {\tiny $+$};
\node[above right = -2pt] at (x1) {\tiny $-$};
\node[above right = -2pt] at (x2) {\tiny $+$};
\node[below] at (x0) {$y_0$};
\node[below] at (x1) {$y_1$};
\node[below] at (x2) {$y_2$};
\draw[->] (x0) to node[above]{$\rho_1$} (x1);
\draw[->] (x2) to node[above]{$\rho_3$} (x1);
\end{tikzpicture}
\end{center}
The generator $y_2$ has grading 0 and the generator $y_1$ has grading 1, so $x\otimes y$ only has grading 1 if $x = x_{k+1}$ and $y = y_1$. In this setting, $x_{k+1} \otimes y_1$ has an incoming differential on the right which starts from $x_k \otimes y_2$. To ensure that $x_{k+1} \otimes y_1$ cancels in homology we need to check that $x_k \otimes y_2$ does not cancel on the right. To the right of $y_2$ in $\lp_2$ is an outgoing sequence of arrows that starts with some number of $\rho_{12}$ arrows followed by a $\rho_{123}$ arrow (here we use that a $\bar{c}_1$ segment in $\lp_2$ is not followed by another $\bar{c}_1$ segment with only $e$'s in between). If $k > 1$ then $x_k$ has no outgoing $\Ainfty$ operations except $m_2(x_k, \rho_3) = x_{k+1}$. If $k = 1$, then $x_k$ has additional operations, but since $s_1$ is preceded by some number of $\bar{e}^*$ segments and then a $c^*$ segment, the inputs for these operations can only be some number of $\rho_{12}$'s followed by a $\rho_1$. In either case, it is clear that $x_k \otimes y_2$ does not cancel on the right.

Suppose that $s_1$ is $\bar{e}^*$ and $s_2$ is $\bar{c}_1$. In this case, $x$ is the only generator of $\bar{e}^*$, $y$ is the only generator of $\bar{c}_1$ with idempotent $\iota_1$, and $x\otimes y$ has grading 1. In $\lp_2$, $s_2$ is preceded by some number $i$ of $e$ segments, which are preceded by a $b_k$ or $d_k$ segment. Thus to the left of $y$ in $\lp_2$ there is an incoming sequence of arrows that ends with $\rho_2$, $i$ $\rho_{12}$'s, and $\rho_1$. In $\lp_1^A$, $s_1$ is followed by some number $j$ of $\bar{e}^*$ segments followed by a $c^*_k$. If $j > i$ then $x$ has an incoming operation with inputs $(\rho_2, \rho_{12}, \ldots, \rho_{12}, \rho_1)$ with $i$ $\rho_{12}$'s, and if $j \le i$ then $x$ has an incoming operation with inputs $(\rho_{12}, \ldots, \rho_{12}, \rho_1)$ with $j$ $\rho_{12}$'s. In either case, these operations give rise to a differential in the box tensor product ending at $x\otimes y$. In both cases it is also easy to check that the initial generator of this differential has no other differentials, and so $x\otimes y$ cancels in homology.

The only case remaining is the case that $s_1$ is $c^*_k$ and $s_2$ is $c_\ell$. This case is depicted in Figure \ref{fig:c2_tensor_c3} for $k = 3$ and $\ell = 2$. We label the generators of $s_1$ sequentially as $x_1, \ldots, x_{k+1}$, and we label the generators of $s_2$ as $y_0, \ldots, y_\ell$. For the remainder of this proof, assume that $s_1$ is $c^*_k$, $s_2$ is $c_\ell$, $x\otimes y$ has grading 1, and $x\otimes y$ does not cancel in homology; we will produce a contradiction, proving the proposition.

Since $x\otimes y$ has grading 1, $y$ is $y_0$ and $x$ is $x_i$ with $i \in \{1, \ldots, k\}$.  For each $i > k-\ell$, $x_i \otimes y_0$ has an outgoing differential which ends at $x_{k+1} \otimes y_{k-i+1}$. If $i > k-\ell + 1$, then $x_{k+1} \otimes y_{k-i+1}$ does not cancel on the right, since the arrow to the right of $y_{k-i+1}$ is an outgoing $\rho_{23}$ arrow but $x_{k+1}$ has only incoming $\Ainfty$ operations. Thus if $i > k-\ell+1$, $x_i \otimes y_0$ cancels on the right. By Lemmas \ref{lem:loop_restrictions_L1} and \ref{lem:loop_restrictions_L2}, $k \le n$ and $\ell \ge n-1$; it follows that $k - \ell$ is at most 1 and $x_i \otimes y_0$ cancels on the right for any $i > 2$. Thus $x$ must be either $x_1$ or $x_2$.

If $k\le \ell$ then $x_2 \otimes y_0$ cancels on the right.  If $k=n$ and $\ell=n-1$, $x_2 \otimes y_0$ potentially cancels from the right. It has an outgoing differential on the right ending at $x_{k+1}\otimes y_\ell$, so it cancels from the right if and only if $x_{k+1}\otimes y_\ell$ does not cancel on the right. $x_{k+1}\otimes y_\ell$ has a differential on the right only if $s_1 = c^*_k$ is followed by a segment $c^*_{k'}$, in which case the differential starts with the generator $x'_1\otimes y'_0$, where $x'_1$ is the first generator of the $c^*_{k'}$ segment following $s_1$ and $y'_0$ is the first generator of the segment following $s_2$. If $x'_1\otimes y'_0$ cancels on the right then so does $x_2\otimes y_0$, and $x'_1\otimes y'_0$ automatically cancels on the right if $s_2$ is followed by a type $a$ segment. Thus if $x = x_2$ we must have that $s_1$ is followed by $c^*_{k'}$, $s_2$ is followed by $c_\ell'$, and $x'_1\otimes y'_0$ does not cancel from the right.

If $k<\ell$ then $x_1 \otimes y_0$ cancels on the right, and if $k>\ell$ then $x_1\otimes y_0$ does not cancel on the right. If $k=\ell=n$ or $k=\ell=n-1$ then $x_1\otimes y_0$ potentially cancels on the right. By the same reasoning as above, it does not cancel on the right if and only if $s_1$ is followed by $c^*_{k'}$, $s_2$ is followed by $c_\ell'$, and $x'_1\otimes y'_0$ does not cancel from the right. $x_1\otimes y_0$ may also cancel from the left. It has an outgoing differential on the left ending at $x_0\otimes y'_{\ell'}$, where $x_0$ is the last generator in the segment immediately preceding $s_1$ and $y'_{\ell'}$ is the last generator in the segment preceding $s_2$. If $s_2$ is preceded by $b_{\ell'}$ then it is easy to see that $x_0\otimes y'_{\ell'}$ has no differentials on the left, so $x_1\otimes y_0$ cancels from the left. If instead $s_2$ is preceded by $c_{\ell'}$, $x_0\otimes y'_{\ell'}$ cancels from the left unless $s_1$ is preceded by $c^*_{k'}$ and $k' < \ell'$ or $k' = \ell'$ and $x'_1\otimes y'_0$ cancels from the left.

Suppose that $x\otimes y = x_2\otimes y_0$ does not cancel in homology (in particular it does not cancel from the right). Then we have shown that $k = n$, $\ell = n-1$, $s_1 = c^*_k$ is followed by $c^*_{k'}$, $s_2 = c_\ell$ is followed by $c_{\ell'}$, and the generator $x'_1 \otimes y'_0$ does not cancel from the right. Furthermore, this last fact implies that either $k' = n$ and $\ell'= n-1$ or that $k' = \ell'$, $c^*_{k'}$ is followed by $c^*_{k''}$, $c_{\ell'}$ is followed by $c_{\ell''}$, and $x_1'' \otimes y_0''$ does not cancel from the right, where $x_1''$ and $y_0''$ are the appropriate generators of $c^*_{k''}$ and $c_{\ell'}$. Repeating this argument, we see that $s_1 = c^*_n$ is followed by a sequence of $c^*_n$ and $c^*_{n-1}$ segments ending with a $c^*_n$ and $s_2 = c_{n-1}$ is followed by a sequence of $c_n$ and $c_{n-1}$ segments ending with $c_{n-1}$ but with indices otherwise the same as the indices in the sequence of $c^*$ segments following $s_1$.

Suppose that $x\otimes y = x_1\otimes y_0$ does not cancel in homology. The fact that it does not cancel from the right implies that $k = \ell = n$ or $k=\ell=n-1$, $s_1 = c^*_k$ is followed by $c^*_{k'}$, $s_2 = c_\ell$ is followed by $c_{\ell'}$, and the generator $x'_1 \otimes y'_0$ does not cancel from the right. As in the preceding paragraph, this implies that $s_1 = c^*_n$ is followed by a sequence of $c^*_n$ and $c^*_{n-1}$ segments ending with a $c^*_n$ and $s_2 = c_{n-1}$ is followed by a corresponding sequence of $c_n$ and $c_{n-1}$ segments ending with $c_{n-1}$. The fact that $x\otimes y = x_1\otimes y_0$ does not cancel from the left implies that $s_1 = c^*_k$ is preceded by $c^*_{k'}$, $s_2 = c_\ell$ is preceded by $c_{\ell'}$, and either $k' = n$ and $\ell'  = n-1$ or $k' = \ell'$ and the generator $x'_1 \otimes y'_0$ does not cancel from the left. Repeating the argument, we see that $s_1$ is preceded by a sequence of $c^*_n$ and $c^*_{n-1}$ segments starting with $c^*n$ and $s_2$ is preceded by a sequence of $c_n$ and $c_{n-1}$ segments with matching sequence of indeces except that the initial segment is $c_{n-1}$.

\begin{figure}
\begin{tikzpicture}[>=latex]
\node (y0) at (2,5) {$\bullet$};
\node (y1) at (3.5,5) {$\circ$};
\node (y2) at (5,5) {$\circ$};
\node (y3) at (6.5,5) {};
\node[above right = -2pt] at (y0) {\tiny $-$};
\node[above right = -2pt] at (y1) {\tiny $+$};
\node[above right = -2pt] at (y2) {\tiny $+$};

\node (x0) at (0,4) {};
\node (x1) at (0,3) {$\bullet$};
\node (x2) at (0,2) {$\bullet$};
\node (x3) at (0,1) {$\bullet$};
\node (x4) at (0,0) {$\circ$};
\node[above right = -2pt] at (x1) {\tiny $+$};
\node[above right = -2pt] at (x2) {\tiny $+$};
\node[above right = -2pt] at (x3) {\tiny $+$};
\node[above right = -2pt] at (x4) {\tiny $+$};

\node[below] at (y0) {$y_0$};
\node[below] at (y1) {$y_1$};
\node[below] at (y2) {$y_2$};

\node[right = 5pt] at (x1) {$x_1$};
\node[right = 5pt] at (x2) {$x_2$};
\node[right = 5pt] at (x3) {$x_3$};
\node[right = 5pt] at (x4) {$x_4$};

\draw[->] (y0) to node[above]{$\rho_3$} (y1);
\draw[->] (y1) to node[above]{$\rho_{23}$} (y2);
\draw[->] (y3) to node[above]{$\rho_1$} (y2);

\draw[->] (x1) to node[left]{$\rho_1$} (x0);
\draw[->] (x1) to node[left]{$\rho_3, \rho_2$} (x2);
\draw[->] (x2) to node[left]{$\rho_3, \rho_2$} (x3);
\draw[->] (x3) to node[left]{$\rho_3$} (x4);

\draw (1,-.5) -- (1,4) -- (7, 4);

\node (x1y0) at (2, 3) {$-$};
\node (x2y0) at (2, 2) {$-$};
\node (x3y0) at (2, 1) {$-$};
\node (x4y1) at (3.5, 0) {$+$};
\node (x4y2) at (5, 0) {$+$};

\draw[->] (x2y0) to (x4y2);
\draw[->] (x3y0) to (x4y1);
\end{tikzpicture}

\caption{A portion of the chain complex $\lp_1^A \boxtimes \lp_2$ coming from a segment $c^*_3$ in $\lp_1^A$ (left edge) and a segment $c_2$ in $\lp_2$ (top edge). Signs $+$ and $-$ indicate generators in the box tensor product with grading 0 and 1, respectively.}
\label{fig:c2_tensor_c3}
\end{figure}
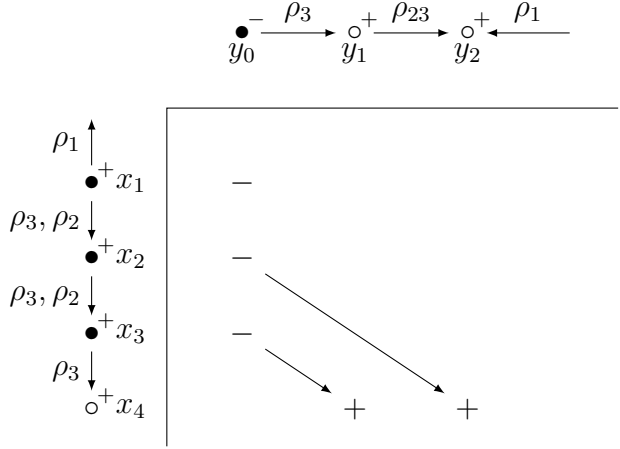

Regardless of whether $x = x_1$ or $x = x_2$, we find that there is a sequence of indices $k_1, k_2, \ldots, k_r$ with $k_i \in \{n, n-1\}$ such that $\lp_1^A$ contains  $c^*_n c^*_{k_1} c^*_{k_2} \ldots c^*_{k_r} c^*_n$ and $\lp_2$ contains $c_{n-1} c_{k_1} c_{k_2} \ldots c_{k_r} c_{n-1}$. It follows that $\lp_1$ and $\lp_2$ satisfy \proplambda. By Lemma \ref{lem:proplambda}, this implies that $\lp_1$ and $\lp_2$ are not L-space aligned, a contradiction.
\end{proof}

As noted previously, the gluing statement in Propostion \ref{prop:gluing_loops} needs to be modified if either loop is solid torus-like. Note that for a solid torus-like loop, all slopes are L-space slopes except the rational longitude. If $\lp_1$ is solid torus-like with rational longitude $\frac{r}{s}$ and $\lp_2$ is simple, then $\lp_1$ and $\lp_2$ are L-space aligned if and only if $\frac{s}{r}$ is a strict L-space slope for $\lp_2$. This is a sufficient, but not a necessary, condition for $\lp^A_1 \boxtimes \lp_2$ to be an L-space complex.

\begin{prop}
\label{prop:gluing_loops_solid_torus_like}
If $\lp_1$ and $\lp_2$ are simple loops and $\lp_1$ is solid torus-like with rational longitude $\frac{r}{s}$, then $\lp^A_1 \boxtimes \lp_2$ is an L-space chain complex if and only if $\frac{s}{r}$ is an L-space slope for $\lp_2$.
\end{prop}

\begin{proof}
We may choose a framing so that $\lp_1$ is a collection of $e$ segments, that is, so that the rational longitude is represented by $0$. Correspondingly, the slope $\frac{s}{r}$ for $\lp_2$ is represented by $\infty$. The result now follows from Corollary \ref{cor:STlike}.
\end{proof}

\subsection{A gluing result for loop-type manifolds} Returning to loop-type manifolds, we are now in a position to collect the material proved in this section and, in particular, apply Proposition \ref{prop:gluing_loops} to establish a gluing theorem. 

\begin{thm}
\label{thm:gluing_manifolds}
Let $(M_1, \alpha_1, \beta_1)$ and $(M_2, \alpha_2, \beta_2)$ be simple loop-type bordered manifolds with torus boundary which are not solid torus-like, and let $Y$ be the closed manifold $(M_1, \alpha_1, \beta_1) \cup (M_2, \alpha_2, \beta_2)$ (with the gluing map $\alpha_1\mapsto\beta_2, \beta_1\mapsto\alpha_2$, as in Section \ref{sub:pairing}). Then $Y$ is an L-space if and only if every essential simple closed curve on $\partial M_1 = \partial M_2 \subset Y$ determines a strict L-space slope for either $M_1$ or $M_2$.
\end{thm}

\begin{remark} There is an alternate statement of the conclusion on Theorem \ref{thm:gluing_manifolds} using the notation laid out in  this paper: The closed manifold $M_1\cup M_2$ is an L-space if and only if for each rational $\frac{p}{q}$ either $\frac{p}{q}\in\mathcal{L}^\circ(M_1, \alpha_1, \beta_1)$ or $\frac{q}{p}\in\mathcal{L}^\circ(M_2, \alpha_2, \beta_2)$. Recall that, following the conventions in Section \ref{sub:pairing}, $\frac{p}{q}\in\mathcal{L}^\circ(M, \alpha, \beta)$ if and only if $\pm(p\alpha+q\beta)\in\mathcal{L}_M^\circ$. \end{remark}

Note that Theorem \ref{thm:gluing_manifolds} implies Theorem \ref{thm:gluing}.

\begin{remark}If either bordered manifold in the statement of Theorem \ref{thm:gluing_manifolds} is solid torus-like, then $Y$ is an L-space if and only if every essential simple closed curve on $\partial M_1 = \partial M_2 \subset Y$ determines an L-space slope for either $M_1$ or $M_2$. This is an immediate consequence of Proposition \ref{prop:gluing_loops_solid_torus_like}, and amounts to replacing $\mathcal{L}^\circ(M_i, \alpha, \beta)$ by $\mathcal{L}(M_i, \alpha, \beta)$. Notice that Dehn surgery -- by definition of an L-space slope -- is a special case of this version of the gluing result, since the solid torus is solid torus-like, in the sense of Definition \ref{def:solid torus-like manifold}. \end{remark}

\begin{proof}[Proof of Theorem \ref{thm:gluing_manifolds}]
Let $\CFD(M_1, \alpha_1, \beta_1)$ be represented by simple loops $\lp_1^1, \dots, \lp_1^n$ and let $\CFD(M_2, \alpha_2, \beta_2)$ be represented by simple loops $\lp_2^1, \dots, \lp_2^m$. A given slope is a (strict) L-space slope for $(M_i, \alpha_i, \beta_i)$ if and only if it is a (strict) L-space slope (abstractly) for each loop $\lp_i^k$. $Y$ is an L-space if and only if $\lp_1^k \boxtimes \lp_2^j$ is and L-space complex (again, abstractly) for each $1 \le k \le n$, $1 \le j \le m$.

Suppose there is a slope $\frac{p}{q}$ such that $\frac{p}{q} \notin \mathcal{L}^\circ(M_1, \alpha_1, \beta_1)$ and $\frac{q}{p} \notin \mathcal{L}^\circ(M_2, \alpha_2, \beta_2)$. Then $\frac{p}{q}$ is not a strict L-space slope for $\lp_1^k$ for some $k$ and $\frac{q}{p}$ is not a strict L-space slope for $\lp_2^j$ for some $j$. We may assume that $\lp_1^k$ is not solid torus-like; if it is solid torus-like, then $\frac{p}{q}$ must be the rational longitude for $(M_1, \alpha_1, \beta_1)$ and thus not a strict L-space slope for any of the loops $\lp_1^1, \ldots, \lp_1^n$. Similarly, we may assume that $\lp_2^j$ is solid torus-like. Since $\lp_1^k$ and $\lp_2^j$ are not L-space aligned, $\lp_1^k \boxtimes \lp_2^j$ is not an L-space complex, and $Y$ is not an L-space.

If every slope is a strict L-space slope for either $M_1$ or $M_2$, then every slope is a strict L-space slope either for each $\lp_1^k$ or for each $\lp_2^j$. It follows that every pair $(k,j)$, $\lp_1^k$ and $\lp_2^j$ are L-space aligned, and thus $Y$ is an L-space.
\end{proof}

\subsection{L-space knot complements}
In light of Theorem \ref{thm:gluing_manifolds}, it is natural to ask which 3-manifolds $M$ with torus boundary have simple loop-type bordered invariants. We first observe that complements of L-space knots in $S^3$, that is, those knots admitting an L-space surgery, have this property.

\begin{prop}\label{prp:L-space-knots-loop}
If $K$ is an L-space knot in $S^3$ and $M = S^3 \smallsetminus\nu(K)$, then for any choice of parametrizing curves $\alpha$ and $\beta$, $\CFD(M, \alpha, \beta)$ can be represented by a single simple loop.
\end{prop}
\begin{proof}
We only need to show that for some choice of parametrizing curves, $\CFD(M, \alpha, \beta)$ is a loop which consists of only unstable chains in either standard notation or dual notation. Consider then $\CFD(M,\mu, n\mu+\lambda)$ where $\mu$ is the meridian of the knot $K$ and $\lambda$ is the Seifert longitude of $K$. Suppose that $|n|$ is chosen sufficiently large so that the result of Dehn surgery $S^3_n(K)=M(n\mu+\lambda)$ is an L-space. 

The knot Floer homology $\mathit{CFK}^{-}(K)$ has a staircase shape of alternating horizontal and vertical arrows (see \cite{Petkova-thin}, for example). By the construction in \cite[Section 11.5]{LOT-bordered}, $\CFD(M,\mu, n\mu+\lambda)$ consists of alternating horizontal and vertical chains, with the ends connected by a single unstable chain. Thus $\CFD(M, \mu, n\mu+\lambda)$ is a loop. In standard loop notation, the horizontal chains are type $a$ segments, the vertical chains are type $b$ segments, and the unstable chain is a type $d$ segment if $n > 0$ and a type $c$ segment if $n<0$. More precisely, 
\[\CFD(M,\mu, n\mu+\lambda)=\begin{cases}
b_{k_1} a_{k_2} b_{k_3} a_{k_4} \cdots b_{k_{2r-1}} a_{k_{2r}} d_j & n >0 \\
a_{k_1} b_{k_2} a_{k_3} b_{k_4} \cdots a_{k_{2r-1}} b_{k_{2r}} c_j& n<0
\end{cases} \]
In either case, the loop has no bar segments in standard notation, so when we switch to dual notation it has no stable chains. Thus, $\CFD(M,\mu, n\mu+\lambda)$ is a simple loop.
\end{proof}

Note that the behaviour established in this application of Theorem \ref{thm:gluing_manifolds} is expected in general and, in particular, should not require the hypothesis that the knot complement be a loop-type manifold; see \cite[Conjecture 1]{Hanselman2014} and \cite[Conjecture 4.3]{CLW2013}.

\begin{conj}\label{conj:splice} For knots $K_1$ and $K_2$ in $S^3$, with  $M_i=S^3\smallsetminus\nu(K)$ for $i=1,2$, the generalized splicing $M_1\cup_hM_2$ is an L-space if and only if either $\gamma\in\mathcal{L}^\circ_{M_1}$ or $h(\gamma)\in\mathcal{L}^\circ_{M_2}$ for every slope $\gamma$.\end{conj}

On the other hand, it is clear that the case where two L-space knot complements are identified in the generalized splice is the interesting case, and therefore the loop restriction in this setting is quite natural. In this case, we remark that the question is now settled.

\begin{cor}\label{cor:splice} Conjecture \ref{conj:splice} holds when the $K_i$ are L-space knots.  \end{cor}
\begin{proof} Immediate on combining Theorem \ref{thm:gluing_manifolds} and Proposition \ref{prp:L-space-knots-loop}.\end{proof}

Note that Corollary \ref{cor:splice} implies Theorem \ref{thm:gen-splice}.

Some authors define L-space knots to be the class of knots for which some {\em positive} surgery yields an L-space.  We will not follow this convention because, while this is a natural definition for knots in $S^3$ in that certain statements become simpler, our interest is in the more general setting of manifolds with torus boundary admitting L-space fillings wherein the distinction seems to be less meaningful. In particular, the next section is concerned with this more general setting.

\section{Graph manifolds}\label{sec:manifolds}

Note that knots in the three-sphere admitting L-space surgeries contain torus knots -- those knots admitting a Seifert structure on their complement -- as a strict subset. Our goal now is to establish another class of loop-type manifolds: Seifert fibered rational homology solid tori. This should be viewed as the natural geometric enlargement of the class of torus knot exteriors. To do so, we study these as a subset of graph manifolds, and establish sufficient conditions for a rational homology solid torus admitting a the structure of a graph manifold to be of loop-type. 

\subsection{Preliminaries on graph manifolds}\label{sub:prelim} We will represent a graph manifold rational homology sphere $Y$ by a plumbing tree, following the notation developed in \cite{Neumann}. A plumbing tree is an acyclic graph with integer weights associated to each vertex. Such a graph specifies a graph manifold as follows: To each vertex $v_i$ of weight $e_i$ and valence $d_i$ we assign the Euler number $e_i$ circle bundle over the sphere minus $d_i$ disks; and for each edge connecting vertices $v_i$ and $v_j$ glue the corresponding bundles along a torus boundary component by a gluing map that takes the fiber of one bundle to a curve in the base surface of the other bundle, and vice versa.

To allow for graph manifolds with boundary, we associate an additional integer $b_i \ge 0$ to each vertex. To construct the corresponding manifold, we associate to each vertex a bundle over $S^2$ minus $(d_i + b_i)$; $d_i$ boundary components of this bundle will glue to bundles corresponding to other vertices, but the remaining $b_i$ boundary components remain unglued. The resulting graph manifold has $\sum_i b_i$ toroidal boundary components. In diagrams of plumbing trees, we will indicate the presence of boundary tori by drawing $b_i$ half-edges (dotted lines which do not connect to another vertex) at each vertex $v_i$.

Given a plumbing tree $\Gamma$ with a single boundary half-edge, let $M_\Gamma$ denote the corresponding graph manifold. The torus $\partial M_\Gamma$ has a natural choice of parametrizing curves: One corresponds to a fiber in the $S^1$ bundle containing the boundary, and one corresponds to a curve in the base surface of that bundle (note that these are precisely the curves used above to specify the gluing maps in the construction based on a given graph). Call these two slopes $\alpha$ and $\beta$, respectively. These slopes do not have a preferred orientation, but reversing the orientation on the bundle associated to every vertex in the construction of $M_\Gamma$ gives a diffeomorphism between $(M_\Gamma, \alpha, \beta)$ and $(M_\Gamma, -\alpha, -\beta)$ as bordered manifolds. 

Thus $(M_\Gamma, \alpha, \beta)$ is a canonical bordered manifold associated to $\Gamma$. Since we will be interested in the bordered invariants of such graph manifolds, for ease of notation we will refer to $\CFD(M_\Gamma, \alpha, \beta)$ simply as $\CFD(\Gamma)$.

We will make use of three important operations on single boundary plumbing trees. These are summarized in Figure \ref{fig:tree_operations}, and described as follows:
\begin{itemize}
\item[\emph{twist}] The operation $\twist^\pm$ adds $\pm 1$ to the weight of the vertex containing the boundary edge.
\item[\emph{extend}] The operation $\extend$ inserts a new 0-framed vertex between the boundary vertex and the boundary edge.
\item[\emph{merge}] The operation $\merge$ takes two single boundary plumbing trees and identifies their boundary vertices, removing one of the two boundary edges to produce a new single boundary plumbing tree; the weight of the new boundary vertex is the sum of the weights of the original boundary vertices. 
\end{itemize}
With these operations we can, in particular, construct any single-boundary plumbing tree.

\begin{figure}
\begin{center}
\begin{tikzpicture}
\tikzstyle{every node}=[draw,circle,fill=white,minimum size=4pt,
                            inner sep=0pt]
\node at (0,0) [circle, fill=black] {};
\draw (-1,-1) -- (0,0) -- (1,-1);
\draw (0,0) -- (-.5, -1);
\put(-2, -20){$\cdots$}

\draw [dashed, thick] (0,0) -- (0,1);
\put(5, 2){$n$}

\draw[->] (1,0) to (2,0);
\put(40, 4){$\twist^\pm$}
\end{tikzpicture}
\begin{tikzpicture}
\tikzstyle{every node}=[draw,circle,fill=white,minimum size=4pt,
                            inner sep=0pt]
\node at (0,0) [circle, fill=black] {};
\draw (-1,-1) -- (0,0) -- (1,-1);
\draw (0,0) -- (-.5, -1);
\put(-2, -20){$\cdots$}

\draw [dashed, thick] (0,0) -- (0,1);
\put(5, 2){$n\pm 1$}
\end{tikzpicture}
\hspace{1 in}
\begin{tikzpicture}
\tikzstyle{every node}=[draw,circle,fill=white,minimum size=4pt,
                            inner sep=0pt]
\node at (0,0) [circle, fill=black] {};
\draw (-1,-1) -- (0,0) -- (1,-1);
\draw (0,0) -- (-.5, -1);
\put(-2, -20){$\cdots$}

\draw [dashed, thick] (0,0) -- (0,1);
\put(5, 2){$n$}

\draw[->] (1,0) to (2,0);
\put(40, 4){$\extend$}
\end{tikzpicture}
\begin{tikzpicture}
\tikzstyle{every node}=[draw,circle,fill=black,minimum size=4pt,
                            inner sep=0pt]
\node at (0,0) [circle] {};
\node at (0,1) [circle] {};

\draw (-1,-1) -- (0,0) -- (1,-1);
\draw (0,1) -- (0,0) -- (-.5, -1);
\put(-2, -20){$\cdots$}

\draw [dashed, thick] (0,1) -- (0,2);
\put(5, 2){$n$}
\put(5, 22){$0$}

\end{tikzpicture}


\begin{tikzpicture}
\tikzstyle{every node}=[draw,circle,fill=white,minimum size=4pt,
                            inner sep=0pt]
\node at (0,0) [circle, fill=black] {};
\draw (-1,-1) -- (0,0) -- (1,-1);
\draw (0,0) -- (-.5, -1);
\put(-2, -20){$\cdots$}

\draw [dashed, thick] (0,0) -- (0,1);
\put(5, 2){$n_1$}

\end{tikzpicture}
\quad,\quad
\begin{tikzpicture}
\tikzstyle{every node}=[draw,circle,fill=white,minimum size=4pt,
                            inner sep=0pt]
\node at (0,0) [circle, fill=black] {};
\draw (-1,-1) -- (0,0) -- (1,-1);
\draw (0,0) -- (-.5, -1);
\put(-2, -20){$\cdots$}

\draw [dashed, thick] (0,0) -- (0,1);
\put(5, 2){$n_2$}

\draw[->] (1,0) to (2,0);
\put(40, 4){$\merge$}
\end{tikzpicture}
\begin{tikzpicture}
\tikzstyle{every node}=[draw,circle,fill=white,minimum size=4pt,
                            inner sep=0pt]
\node at (0,0) [circle, fill=black] {};
\draw (-2,-1.5) -- (0,0) -- (-.5,-1.5);
\draw (0,0) -- (-1.5, -1.5);
\put(-28, -35){$\cdots$}

\draw (2,-1.5) -- (0,0) -- (.5,-1.5);
\draw (0,0) -- (1, -1.5);
\put(25, -35){$\cdots$}

\draw [dashed, thick] (0,0) -- (0,1);
\put(5, 2){$n_1 + n_2$}
\end{tikzpicture}

\end{center}

\caption{Three operations on graphs for constructing and graph manifold with torus boundary.}
\label{fig:tree_operations}
\end{figure}
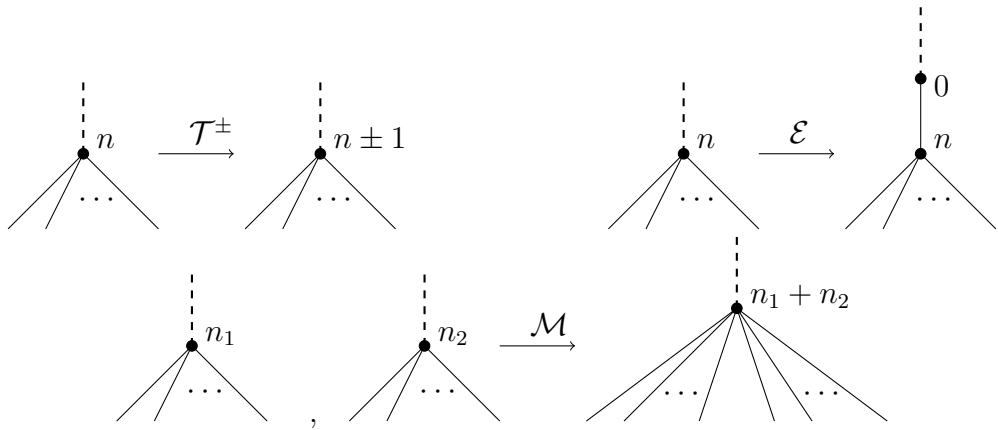

\subsection{Operations on bordered manifolds} The graph operations described above may be thought of as operations on the corresponding (bordered) graph manifolds. In particular, we will abuse notation and write $M_{\twist(\Gamma)} = \twist(M_\Gamma)$. $\twist(M_\Gamma)$ is obtained from $M_\Gamma$ by attaching the mapping cylinder of a positive (standard) Dehn twist to the boundary of $M_\Gamma$. More precisely, by considering the gluing conventions prescribed by the plumbing tree, the effect on the bordered manifold is $\twist(M_\Gamma,\alpha,\beta) = (M_\Gamma,\alpha,\beta+\alpha)$. 

Similarly, $\extend(M_\Gamma) = M_{\extend(\Gamma)}$ is obtained from $M_\Gamma$ by attaching the bimodule corresponding to the two boundary plumbing tree
\begin{center}
\begin{tikzpicture}
\node (a) at (0,0) {$\bullet$};
\node[above = 1pt] at (a) {$0$};
\draw[dashed] (-1,0)--(0,0)--(1,0);
\end{tikzpicture}
\end{center}
The graph manifold assigned to this plumbing tree is the trivial $S^1$-bundle over the annulus, or $T^2 \times [0, 1]$.  Recall that the bordered structure on each boundary component of this manifold is given by  the convention that $\alpha$ is an $S^1$ fiber and $\beta$ lies in the base surface. One checks that the resulting bordered manifold can be realised as the mapping cylinder for the diffeomorphism represented by the matrix 
$\left( \begin{smallmatrix} 0 & 1\\ -1 & 0\end{smallmatrix} \right) \in \mathit{SL}_2(\Z). $ 
The effect on the level of bordered manifolds is  $\extend(M_\Gamma,\alpha,\beta)=(M_\Gamma,-\beta,\alpha)$.

Finally, $\merge(M_{\Gamma_1},M_{\Gamma_2})$ is the result of attaching the bundle $S^1\times \pants$ where $\pants$ is a pair of pants. The bordered structure on $S^1\times\pants$ (which we suppress from the notation) is determined as follows: the two ``input" boundary components which glue to $M_{\Gamma_1}$ and $M_{\Gamma_2}$ are parametrized by pairs $(\alpha, \beta)$ with $\beta$ a fiber and $\alpha$ a curve in the base surface $\pants$, while the third ``output" boundary component is parametrized by $(\alpha, \beta)$ with $\alpha$ a fiber and $\beta$ a curve in $\pants$. As usual, the bordered structure on $S^1\times\pants$ determines the gluing as well as the bordered structure on the resulting three-manifold.

To see that $\merge(M_{\Gamma_1},M_{\Gamma_2})$ is indeed the manifold corresponding to $\merge(\Gamma_1, \Gamma_2)$, note that $S^1\times\pants$ with the specified bordered structure is the manifold associated with the three boundary plumbing tree $\Gamma_\merge$ below. \begin{center}
\begin{tikzpicture}
\tikzstyle{every node}=[draw,circle,fill=white,minimum size=4pt,
                            inner sep=0pt]
\node at (0,0) [circle, fill=black] {};
\node at (-.8,-.8) [circle, fill=black] {};
\node at (.8,-.8) [circle, fill=black] {};
\draw (-.8,-.8) -- (0,0) -- (.8,-.8);
\draw [dashed, thick] (0,0) -- (0,1);
\draw [dashed, thick] (1.6,-1.6) -- (.8,-.8);
\draw [dashed, thick] (-1.6,-1.6) -- (-.8,-.8);

\put(-31,-18){$0$}
\put(26,-18){$0$}
\put(5, 2){$0$}

\put(-80, -15){$\Gamma_\merge = $}

\end{tikzpicture}
\end{center}
On the other hand, attaching $\Gamma_1$ and $\Gamma_2$ to the two lower boundary edges of $\Gamma_\merge$ produces a single boundary plumbing tree that is equivalent to $\merge(\Gamma_1, \Gamma_2)$ (by equivalent, we mean that the corresponding manifolds are diffeomorphic; to see this equivalence, use rule R3 of \cite{Neumann} to contract the 0-framed valence two vertices).

From this discussion, it should be clear that the operations $\twist$, $\extend$, and $\merge$ may be extended to natural operations on arbitrary bordered manifolds. We are interested in the effect of these operations on bordered invariants. Following the conventions laid out, we have that $\twist^\pm(M,\alpha,\beta)= (M,\alpha,\beta\pm\alpha)$ so that, at the level of bordered invariants, 
\begin{equation}\label{eqn:twist-it}
\CFD(\twist^\pm(M,\alpha,\beta)) \cong \Ts^{\pm1} \boxtimes \CFD(M,\alpha,\beta)
\end{equation}

Similarly,  the action of $\left(\begin{smallmatrix} 0&1\\ -1&0\end{smallmatrix}\right)\in\mathit{SL}_2(\Z)$ giving $\extend(M,\alpha,\beta)= (M,-\beta,\alpha)$ is realized by
 \begin{equation}\label{eqn:extend-it}
 \CFD(\extend(M,\alpha,\beta)) \cong \Ts\boxtimes\Td\boxtimes\Ts \boxtimes \CFD(M,\alpha,\beta)
 \end{equation}
 on type D structures. 

The most involved is the merge operation: Given a pair of bordered manifolds $(M_i,\alpha_i,\beta_i)$, for $i=1,2$, this produces a bordered manifold  \[\merge_{1,2} = \merge\left((M_1,\alpha_1\beta_1),(M_2,\alpha_2,\beta_2)\right)\] using  $S^1\times\pants$. As a bordered 3-manifold with three boundary components, $S^1\times\pants$ gives rise to a trimodule in the bordered theory (the object of study in work of the first author \cite{Hanselman2013}). More precisely, we consider the type DAA trimodule  $\CFDAA(S^1\times\pants)$.

We will extract $\CFDAA(S^1\times\pants )$ from the trimodule $\widehat{\mathit{CFDDD}}( \mathcal{Y}_\pants )$ computed in \cite{Hanselman2013}. Note that $S^1\times\pants$ and $\mathcal{Y}_\pants$ agree as manifolds, but have different bordered structure. The trimodule $\widehat{\mathit{CFDDD}}(\mathcal{Y}_\pants)$ has five generators, $v$, $w$, $x$, $y$, and $z$, and the differential is given by the following (c.f. \cite[Figure 10]{Hanselman2013}):
\begin{align*}
\partial(v) &= \rho_3 \otimes x + \rho_1 \sigma_3 \tau_{123} \otimes y + \rho_{123}\sigma_{123}\tau_{123} \otimes y + \tau_3 \otimes z \\
\partial(w) &= \rho_3 \sigma_{12} \otimes x + \rho_1 \sigma_3 \tau_1 \otimes y + \rho_{123}\sigma_{123}\tau_1 \otimes y\\
\partial(x) &= \rho_2 \sigma_{12} \otimes v + \rho_2 \otimes w + \sigma_1\tau_3 \otimes y\\
\partial(y) &= \sigma_2\tau_2 \otimes x + \rho_2 \sigma_2 \otimes z\\
\partial(z) &= \rho_3 \sigma_1 \otimes y + \tau_2 \otimes w
\end{align*}
Since the trimodule has three commuting actions by three copies of the torus algebra (one corresponding to each boundary component), we use $\rho$, $\sigma$, and $\tau$ to distinguish between algebra elements in each copy of $\Alg$. The $\sigma$ boundary of $\mathcal{Y}_\pants$ is parametrized by a pair $(\alpha, \beta)$ such that $\beta$ is a fiber of $\pants \times S^1$ and $\alpha$ lies in the base $\pants$, while the $\rho$ and $\tau$ boundaries have the opposite parametrization. Thus the bordered manifold $S^1\times\pants$ can be obtained from $\mathcal{Y}_\pants$ by switching the role of $\alpha$ and $\beta$ on (say) the $\rho$ boundary. 

\begin{figure}\labellist
\small
\pinlabel {$\widehat{\mathit{CFDDD}}(\mathcal{Y}_\mathcal{P})$} at 148 142 
\pinlabel {$\tau$} at 154 215

\pinlabel {\rotatebox{45}{$\widehat{\mathit{CFAA}}(\mathbb{I})$}} at 44 44
\pinlabel {\rotatebox{45}{$\Ts$}} at 78 78
\pinlabel {\rotatebox{45}{$\Td$}} at 95 95
\pinlabel {\rotatebox{45}{$\Ts$}} at 113 113
\pinlabel {$\rho$} at 3 10

\pinlabel {\rotatebox{-45}{$\widehat{\mathit{CFAA}}(\mathbb{I})$}} at 201 99.5 
\pinlabel {$\sigma$} at 242 67
\endlabellist
\includegraphics[scale=1]{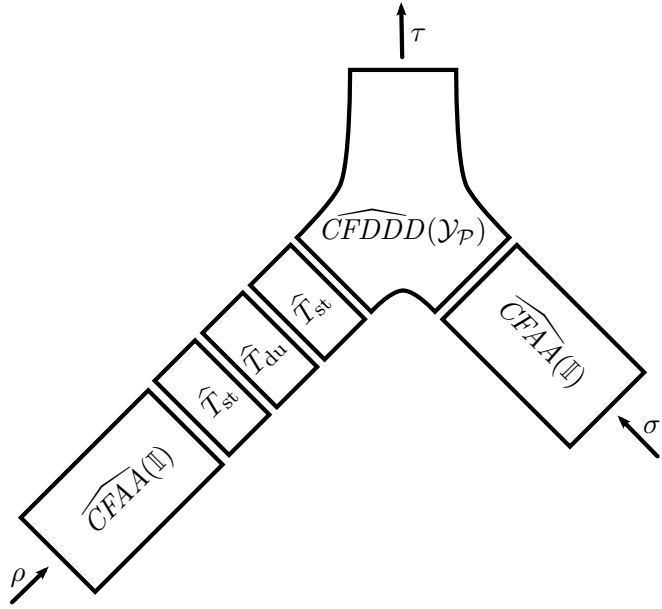}
\caption{A schematic description of the trimodule $\CFDAA(S^1\times\pants)$ extracted from the calculation and conventions in \cite{Hanselman2013}.}
\label{fig:merge_schematic}
\end{figure}

It is now clear how to compute $\widehat{\mathit{CFDAA}}( S^1\times\pants )$ from $\widehat{\mathit{CFDDD}}( \mathcal{Y}_\pants )$: we change the parametrization of the $\rho$-boundary by applying the bimodule $\Ts \boxtimes \Td \boxtimes \Ts$ and then change the $\rho$- and $\sigma$-boundaries to type A by tensoring with $\CFAA(\mathbb{I})$ (see Figure \ref{fig:merge_schematic} for a schematic description). These are both straightforward computations; the resulting trimodule has five generators and the operations are listed in Table \ref{table:trimodule}.

With this trimodule in hand, we have \[\CFD(\merge_{1,2})\cong \CFDAA(S^1\times \pants)\boxtimes\left((M_1,\alpha_1,\beta_1),(M_2,\alpha_2,\beta_2)\right)\] where $\CFDAA(S^1\times\pants)\boxtimes\left(\cdot,\cdot\right)$, by convention, tensors against the $\rho$-boundary in the first factor and against the $\sigma$-boundary in the second factor; compare Figure \ref{fig:merge_schematic}. 

\begin{table}
\centering
  \begin{tabular}{|l|l|}
    \hline
    $m_2(x, \sigma_3)= z $ & $m_3(v, \rho_1, \sigma_1) = \tau_1\otimes y$\\
$m_2(y, \sigma_2)= x $ & $m_3(v, \rho_1, \sigma_{12})= \tau_1\otimes x$\\
$m_2(y, \sigma_{23})= z$ & $m_3(v, \rho_1, \sigma_{123})= \tau_1\otimes z$\\
       $m_2(w, \rho_3)= z $ & $m_3(v, \rho_1, \sigma_{123})= \tau_{123}\otimes y$\\
$m_2(y, \rho_2)= w$ & $m_3(v, \rho_{12}, \sigma_1)= \tau_1\otimes w$ \\
$m_2(y, \rho_{23})= z$& $m_3(v, \rho_{12}, \sigma_{12})= \tau_{12}\otimes  v$\\
$m_2(w, \sigma_{23})= \tau_{23}\otimes w $ & $m_3(v, \rho_{12}, \sigma_{123})= \tau_{123}\otimes w$\\
$m_2(x, \sigma_3)= \tau_{23}\otimes y $ & $m_3(v, \rho_{123}, \sigma_1)= \tau_1\otimes z$\\
$m_2(x, \rho_2)= \tau_2\otimes v$ & $m_3(v, \rho_{123}, \sigma_{12})= \tau_{123}\otimes x$\\
$m_2(x, \rho_{23})= \tau_{23}\otimes x $ & $ m_3(v, \rho_{123}, \sigma_{123})= \tau_{123}\otimes z$\\
  \cline{2-2}
 $m_2(v, \sigma_3) = \tau_3\otimes w$& $m_5(v, \rho_3, \rho_2, \rho_1, \sigma_1)= \tau_1\otimes z$ \\
$m_2(y, \sigma_{23})= \tau_{23}\otimes y$ & $m_5(v, \rho_3, \rho_2, \rho_1, \sigma_1)= \tau_{123}\otimes y$ \\
$m_2(v, \rho_3)= \tau_3\otimes x$ & $m_5(v, \rho_3, \rho_2, \rho_{12}, \sigma_1)= \tau_{123}\otimes w$ \\
$m_2(w, \sigma_2)= \tau_2\otimes v $ & $m_5(v, \rho_3, \rho_2, \rho_{123}, \sigma_1)= \tau_{123}\otimes z$ \\
  \cline{2-2}
$m_2(z, \sigma_2)= \tau_{23}\otimes x$ & $m_7(v, \rho_3, \rho_2, \rho_3, \rho_2, \rho_1, \sigma_1)= \tau_{123}\otimes z$\\
$m_2(z, \sigma_{23})= \tau_{23}\otimes z$ & \\
    \hline
  \end{tabular}
  \caption{Operations for $\widehat{\mathit{CFDAA}}( \Gamma_\merge )$}  
\label{table:trimodule}
\end{table}

\subsection{The effect of twist, extend, and merge on loops} Restricting to the case of loop-type bordered invariants, the effect of the operations $\twist^{\pm 1}, \extend$, and $\merge$ can be given simpler descriptions. Recall that $\ex=\tw \circ \du^{-1} \circ \tw$, and that $\ex$ can be easily calculated using Lemma \ref{lem:ex_operation}.

\begin{prop}\label{prp:tw-and-ex-on-loops}
If $\CFD(M,\alpha,\beta)$ is of loop-type and represented by the collection $\{\lp_i\}_{i=1}^n$ then $\CFD(\twist^\pm(M,\alpha,\beta))$ is represented by $\{ \tw^{\pm 1}( \lp_i ) \}_{i=1}^n$ and $\CFD( \extend( M,\alpha,\beta))$ is represented by $\{\ex(\lp_i) \}_{i=1}^n$.
\end{prop}

\begin{proof}
The proof is immediate from Proposition \ref{prop:effect_of_Tstd} with Equation \ref{eqn:twist-it} and Proposition \ref{prop:effect_of_Tdul} with Equation \ref{eqn:extend-it}.
\end{proof}

It follows easily that the operations $\twist^\pm$ and $\extend$ preserve the simple loop-type property.

\begin{lem}\label{lem:twist_extend_loop_type}
Given a (simple) loop-type bordered three-manifold, the operations $\twist^\pm$ and $\extend$ produce (simple) loop-type manifolds. 
\end{lem}

\begin{proof}
It is an immediate consequence of Proposition \ref{prp:tw-and-ex-on-loops} that if $\CFD(M,\alpha,\beta)$ is represented by a collection of loops then the same is true for $\CFD(\twist(M,\alpha,\beta))$ and $\CFD(\extend(M,\alpha,\beta))$, since the operations $\tw^{\pm1}$ and $\ex$ take loops to loops. Moreover, if the loops defining $\CFD(M,\alpha,\beta)$ are simple then the loops resulting from $\twist^\pm$ and $\extend$ are simple, since changing framing by Dehn twists does not, by definition, change whether or not a loop is simple. It only remains to check that $\CFD(\twist(M,\alpha,\beta))$ and $\CFD(\extend(M,\alpha,\beta))$ have exactly one loop for each $\operatorname{spin}^c$-structure on the corresponding manifolds. This is again clear since it is true for $\CFD(M,\alpha,\beta)$ and the operations $\twist^{\pm1}$ and $\extend$ amount to changing the parametrization on the boundary of $M$; changing the parametrization does not change the number of $\operatorname{spin}^c$-structures, and the loop operations $\tw^{\pm1}$ and $\ex$ do not change the number of loops.
\end{proof}

For the merge operation, we will restrict further to the case that the loop(s) representing $\CFD(M_1,\alpha_1,\beta_1)$ can all be written in standard notation with no stable chains. In this case, $\CFD(\merge_{1,2})$ is also a collection of loops.

\begin{remark}
By contrast, if $\CFD(M_1,\alpha_1,\beta_1)$ and $\CFD(M_2,\alpha_2,\beta_2)$ both contain stable chains in standard notation, $\CFD(\merge_{1,2})$ is not obviously of loop-type. However, in many cases it can be realized as a collection of loops after a homotopy equivalence.
\end{remark}

It is enough to describe $\merge$ on individual loops; we use $\me$ to denote the corresponding operation on abstract loops. If $\CFD(M_1,\alpha_1,\beta_1)$ is represented by a collection of loops $\{\lp_i\}_{i=1}^n$ and $\CFD(M_2,\alpha_2,\beta_2)$ is represented by a collection of loops $\{\lp_j\}_{j=1}^m$ then $\CFD(\merge_{1,2})$ is given by
$\bigcup_{1\le i\le n, 1\le j\le m} \me(\lp_i, \lp_j).$ Determining $\me(\lp_i,\lp_j)$ is a direct calculation using the trimodule and and a key application of loop calculus.

\begin{prop}\label{prop:merge_unstable_chains}
Let $\lp_1$ be a loop which can be written in standard notation with only type $d_k$ unstable chains and let $\lp_2$ be any loop. Then $\me(\lp_1, \lp_2)$ is a collection of loops. If $\lp_2$ cannot be written in standard notation, then $\me(\lp_1, \lp_2)$ is one copy of $\lp_2$ for each segment in $\lp_1$. Otherwise, $\me(\lp_1, \lp_2)$ is determined as follows: The $\iota_0$-vertices of the $\Alg$-decorated graph correspond to pairs $(u, v)$, where $u$ is a $\iota_0$-vertex of $\lp_1$ and $v$ is a $\iota_0$-vertex of $\lp_2$, and for each $d_k$ segment from $u_1$ to $u_2$ in $\lp_1$ we have:
\begin{enumerate}
\item for each $a_\subL$ segment from $v_1$ to $v_2$ in $\lp_2$, there is an $a_\subL$ segment from $(u_2, v_1)$ to $(u_2, v_2)$;

\item for each $b_\subL$ segment from $v_1$ to $v_2$ in $\lp_2$, there is n $b_\subL$ segment from $(u_1, v_1)$ to $(u_1, v_2)$;

\item for each $c_\subL$ segment from $v_1$ to $v_2$ in $\lp_2$, there is a $c_{\subL-k}$ segment from $(u_2, v_1)$ to $(u_1, v_2)$;

\item for each $d_\subL$ segment from $v_1$ to $v_2$ in $\lp_2$, there is a $d_{k + \subL}$ segment from $(u_1, v_1)$ to $(u_2, v_2)$;

\end{enumerate}
\end{prop}
\begin{proof}
The type D module represented by $\me(\lp_1, \lp_2)$ is obtained by tensoring  the type D modules corresponding to $\lp_1$ and $\lp_2$ with the $\rho$ and $\sigma$ boundaries, respectively, of the trimodule $\CFDAA( \mathcal{Y}_\pants )$. We will denote this tensor by $\CFDAA( \mathcal{Y}_\pants )\boxtimes(\lp_1,\lp_2)$. This trimodule has five generators; the idempotents associated with each generator on the $\rho$, $\tau$, and $\sigma$ boundaries are as follows:

 \begin{center} \begin{tabular}{|l|l|l|l|l|l|}
    \hline
    Generator: & $v$ & $w$ & $x$ & $y$ & $z$ \\
    \hline
    Idempotents: &$(\iota_0^\rho, \iota_0^\sigma, \iota_0^\tau)$ & $(\iota_0^\rho, \iota_1^\sigma, \iota_1^\tau)$ &$(\iota_1^\rho, \iota_0^\sigma, \iota_1^\tau)$ & $(\iota_1^\rho, \iota_1^\sigma, \iota_1^\tau)$& $(\iota_1^\rho, \iota_1^\sigma, \iota_1^\tau)$\\ 
    \hline
\end{tabular}\end{center}

Note that the $\iota_0$ generators in $\widehat{\mathit{CFDAA}}( S^1\times\pants )\boxtimes (\lp_1,\lp_2)$ arise precisely from the the generator $v$ tensored with $\iota_0$ generators in $\lp_1$ and $\lp_2$.

First suppose that $\lp_2$ is written in standard notation. Note also that since $\lp_1$ is assumed to have no $a_k$ segments, we may ignore the $m_5$ and $m_7$ operations in Table \ref{table:trimodule}. As a result, there are no operations in $\me(\lp_1, \lp_2)$ that arise from more than one segment in either loop. This means to compute $\me(\lp_1, \lp_2)$ we can feed $\lp_1$ and $\lp_2$ into the trimodule one segment at a time. For each combination of segment in $\lp_1$ and segment in $\lp_2$, the resulting portion of $\widehat{\mathit{CFDAA}}( S^1\times\pants )\boxtimes (\lp_1,\lp_2)$ is homotopy equivalent to the segment determined by $(1)-(4)$ in the statement of the proposition. The proof is essentially contained in Figures \ref{fig:merge1}, \ref{fig:merge2}, and \ref{fig:merge3}; we will describe one case in detail and leave the details of the other cases to the reader, with the figures as a guide.

Consider a segment $d_k$ in $\lp_1$ and a segment $a_\subL$ in $\lp_2$ with $k,\subL > 0$ (see the top left box in Figure \ref{fig:merge1} for the case of $k = \subL = 2$). Let the generators in these two segments be labelled as follows:
$$d_k = \quad \overset{\rho_{123}}\longrightarrow \underset{x_1}\circ \overset{\rho_{23}}\longrightarrow \cdots \overset{\rho_{23}}\longrightarrow \underset{x_k}\circ \overset{\rho_2}\longrightarrow \underset{u_2}\bullet \to$$
$$a_\subL = \quad \leftarrow \underset{v_1}\bullet \overset{\sigma_{3}}\longrightarrow \underset{y_1}\circ \overset{\sigma_{23}}\longrightarrow \cdots \overset{\sigma_{23}}\longrightarrow \underset{y_\subL}\circ \overset{\sigma_2}\longrightarrow \underset{v_2}\bullet \to$$
Note that the arrow adjacent to $u_2$ on the right is determined by the segment following $d_k$ in $\lp_1$, but following Section \ref{sub:loop-type} it must be an outgoing $\rho_1$, $\rho_{12}$ or $\rho_{123}$ arrow. Similarly, the arrows to the left of $v_1$ and the right of $v_2$ are outgoing $\sigma_1$, $\sigma_{12}$, or $\sigma_{123}$ arrows. The portion of $\widehat{\mathit{CFDAA}}( S^1\times\pants )\boxtimes (\lp_1,\lp_2)$ coming from $d_k$ and $a_\subL$ has the following generators:
\begin{align*}
x \otimes x_i \otimes v_j \quad &\text{for } 1\le i\le k, j\in\{1,2\} \\
y \otimes x_i \otimes y_j \quad &\text{for } 1\le i\le k, 1\le j\le \subL \\
z \otimes x_i \otimes y_j \quad &\text{for } 1\le i\le k, 1\le j\le \subL \\
v \otimes u_2 \otimes v_j \quad &\text{for } j\in\{1,2\} \\
w \otimes u_2 \otimes y_j \quad &\text{for } 1\le j\le \subL
\end{align*}
For each $i$ in $\{1, \ldots, k\}$, the trimodule operations $m_2(x, \sigma_3)$, $m_2(y, \sigma_{23})$, and $m_2(y, \sigma_2)$ give rise to the following unlabeled edges:
\begin{align*}
x \otimes x_i \otimes v_1 \to z\otimes x_i\otimes y_{1\phantom{+1}} \quad &\text{for } 1\le i\le k \\
y \otimes x_i \otimes y_j \to z\otimes x_i\otimes y_{j+1} \quad &\text{for } 1\le i\le k, 1\le j\le \subL-1 \\
y \otimes x_i \otimes y_\subL \to x\otimes x_i\otimes v_{2\phantom{+1}} \quad &\text{for } 1\le i\le k
\end{align*}
We can cancel these unlabeled edges using the edge reduction algorithm described in Section \ref{sub:typeD}. We cancel them in order of increasing $i$, and for fixed $i$ in the order above. It is not difficult to check that the only additional incoming arrows at $z\otimes x_i\otimes y_j$ and $z\otimes x_i\otimes v_2$ are given by
\begin{align*}
z \otimes x_i \otimes y_j \overset{\tau_{23}}\longrightarrow z\otimes x_i\otimes y_{j+1} \quad &\text{for } 1\le i\le k, 1\le j\le \subL-1 \\
z \otimes x_i \otimes y_\subL \overset{\tau_{23}}\longrightarrow x\otimes x_i\otimes v_{2\phantom{+1}} \quad &\text{for } 1\le i\le k \\
y \otimes x_{i-1} \otimes y_j \longrightarrow z\otimes x_i\otimes y_{j\phantom{+1}} \quad &\text{for } 1\le i\le k, 1\le j\le \subL-1 \\
x \otimes x_{i-1} \otimes y_\subL \longrightarrow x\otimes x_i\otimes v_{2\phantom{+1}} \quad &\text{for } 1\le i\le k
\end{align*}
It follows that each time we use the edge reduction algorithm on one of the unlabeled arrows mentioned above, there are no other incoming arrows at the terminal vertex that have not already been canceled, and so canceling the arrow produces no new arrows. After canceling all of the unlabeled arrows, the only remaining generators are $v \otimes u_2 \otimes v_1$, $v \otimes u_2 \otimes v_2$, and $v \otimes u_2 \otimes y_j$ for $1\le j\le \subL$. Arrows between these generators arise from trimodule operations involving only the generators $v$ and $w$ and with no $\rho$ inputs; there are only three:
$$m_2(v, \sigma_3) = \tau_3\otimes w, \quad m_2(w, \sigma_{23})=\tau_{23}\otimes w, \quad \text{and} \quad m_2(w, \sigma_2) = \tau_2 \otimes v.$$
It follows that the only arrows in the portion of $\widehat{\mathit{CFDAA}}( S^1\times\pants)\boxtimes (\lp_1,\lp_2)$ coming from the segments $d_k$ and $a_\subL$ are
\begin{align*}
v \otimes u_2 \otimes v_1 &\overset{\tau_3}\longrightarrow w\otimes u_2\otimes y_1, \\
w \otimes u_2 \otimes y_j &\overset{\tau_{23}}\longrightarrow w\otimes u_2\otimes y_{j+1} \quad \text{for } 1\le j\le \subL-1, \\
w\otimes u_2 \otimes y_\subL &\overset{\tau_2}\longrightarrow v\otimes u_2\otimes v_2.
\end{align*}
That is, the portion of $\widehat{\mathit{CFDAA}}( S^1\times\pants)\boxtimes (\lp_1,\lp_2)$ coming from the segments $d_k$ and $a_\subL$ is an $a_\subL$ segment from $v \otimes u_2 \otimes v_1$ to $v \otimes u_2 \otimes v_2$.

Finally, note that there can be no arrows connecting $w\otimes u_2\otimes y_j$ to any other portions of $\CFDAA( S^1\times\pants)\boxtimes (\lp_1,\lp_2)$ arising from different segments, since the only arrow connecting $u_2$ to a generator in a different segment of $\lp_1$ is an outgoing $\rho_1$, $\rho_{12}$ or $\rho_{123}$ arrow, and the only trimodule operations involving these inputs also have $\sigma_1$, $\sigma_{12}$, or $\sigma_{123}$ as an input. The outgoing arrows from $u_2$ and from $v_1$ and $v_2$ do give rise to additional arrows out of $v \otimes u_2 \otimes v_1$ and $v \otimes u_2 \otimes v_2$; these show up in the portion of $\CFDAA( \mathcal{Y}_\pants )\boxtimes (\lp_1,\lp_2)$ coming from the segment following $d_k$ in $\lp_1$ and the segment following or preceding $a_\subL$ in $\lp_2$.

For other pairs of segments in $\lp_1$ and $\lp_2$, the proof is similar. Figure \ref{fig:merge1} depicts the relevant portion of $\CFDAA( S^1\times\pants )\boxtimes (\lp_1,\lp_2)$ for $d_2$, $d_0$, or $d_{-2}$ paired with $a_2$ or $b_2$. Any $d_k$ paired with $a_\subL$ or $b_\subL$ behaves like one of these cases, depending on the sign of $k$. Note that if $\subL<0$ we simply take the mirror image of these diagrams, since $a_{-\subL} = \bar{a}_\subL$. This proves $(1)$ and $(2)$.

$(3)$ can be deduced from $(4)$ by observing that a $c_\subL$ segment from $v_1$ to $v_2$ is the same as a $d_{-\subL}$ segment from $v_2$ to $v_1$. To prove $(4)$, consider pairing $d_k$ in $\lp_1$ with $d_\subL$ in $\lp_2$. The behavior depends on the sign of $k$ and $\subL$; Figures \ref{fig:merge2} and \ref{fig:merge3} depict the cases with $k$ in $\{2, 0, -2\}$ and $\subL$ in $\{3, 0, -3\}$. If $k = 0$, it is clear that the result is a segment of type $d_\subL$, and similarly if $\subL = 0$ the result is a segment of type $d_k$. The case of $k$ and $\subL$ positive behaves like the top left box in Figure \ref{fig:merge2}, and the case of $k$ and $\subL$ negative behaves like the bottom box in Figure \ref{fig:merge3}. In each case all generators cancel except for those along the top and right edges, resulting in a segment of type $d_{k+\subL}$.

The case that $k$ and $\subL$ have opposite signs is slightly more complicated. Assume first that $k$ is negative and $\subL$ is positive. If $k \ge -\subL$, the resulting complex looks like the bottom left box in Figure \ref{fig:merge2}. Note that starting from the top left corner, there is a path to the bottom right corner consisting of a $\tau_1$ arrow, an odd length ``zig-zag" sequence of unlabelled arrows, $k+\subL$ $\tau_{23}$ arrows, and a $\tau_2$ arrow. Everything in the diagram not involved in this sequence can be canceled without adding new arrows. Cancelling the remaining unlabelled arrows turns the $\tau_1$ arrow and the first $\tau_{23}$ arrow into a $\tau_{123}$ arrow, if $k+\subL > 0$, or it turns the $\tau_1$ arrow and $\tau_2$ arrow into a $\tau_{12}$ arrow if $k+\subL = 0$. The result is a segment of type $d_{k+\subL}$. The case that $k< -\subL$ is slightly different. It is not pictured separately, but the main difference is that the zig-zag sequence of unlabelled arrows starting at the end of the $\tau_1$ arrow has even length and ends on the right side of the diagram instead of the bottom. It is then followed by $-k-\subL-1$ backwards $\tau_{23}$ arrows and a backwards $\tau_3$ arrow. Everything not involved in this sequence cancels, and removing the unlabelled arrows produces a segment of type $d_{k+\subL}$ from the top left corner to the bottom right corner.

To complete the proof of (4), consider the case that $k$ is positive and $\subL$ is negative. If $\subL < -k$, the resulting complex looks like the top box in Figure \ref{fig:merge3}. The complex reduces to a $\tau_1$ arrow followed by a chain of unlabelled arrows, $-k-\subL-1$ backwards $\tau_{23}$ arrows and a backwards $\tau_{3}$ arrow. This further reduces to a chain of type $d_{k+\subL}$. If instead $\subL \ge -k$, the chain of unlabelled arrows following the $\tau_1$ arrow ends on the right side of the diagram instead of the bottom. Again, the complex reduces to a chain of type $d_{k+\subL}$.

Finally, we must consider the case the case that $\lp_2$ cannot be written in standard notation. In this case, $\lp_2$ is a collection of $e^*$ segments.  Note that in this case we can ignore all operations in Table \ref{table:trimodule} that involve $v$, $x$, or $\sigma$ inputs other than $\sigma_{23}$ (this only leaves seven operations). It is easy to see that the complex $\widehat{\mathit{CFDAA}}(S^1\times\pants) \boxtimes (\lp_1 , \lp_2)$ collapses to a copy of of $\lp_2$ for each $d_k$ segment in $\lp_1$ (see Figure \ref{fig:merge4}). 
\end{proof}

\begin{figure}
\begin{center}
\begin{tikzpicture}[scale = .8, > = latex]
\scriptsize
\node [color = gray] (g00) at (0,9) {$\circ$};
\node [color = gray] (g01a) at (3,9.5) {$\circ$};
\node [color = gray] (g01b) at (3,8.5) {$\circ$};
\node [color = gray] (g02a) at (6,9.5) {$\circ$};
\node [color = gray] (g02b) at (6,8.5) {$\circ$};
\node [color = gray] (g03) at (9,9) {$\circ$};

\node [color = gray] (g10) at (0,6) {$\circ$};
\node [color = gray] (g11a) at (3,6.5) {$\circ$};
\node [color = gray] (g11b) at (3,5.5) {$\circ$};
\node [color = gray] (g12a) at (6,6.5) {$\circ$};
\node [color = gray] (g12b) at (6,5.5) {$\circ$};
\node [color = gray] (g13) at (9,6) {$\circ$};

\node (g20) at (0,3) {$\bullet$};
\node (g21) at (3,3) {$\circ$};
\node (g22) at (6,3) {$\circ$};
\node (g23) at (9,3) {$\bullet$};

\node (top-1) at (-1, 12) {};
\node (top0) at (0,12) {$\bullet$};
\node (top1) at (3,12) {$\circ$};
\node (top2) at (6,12) {$\circ$};
\node (top3) at (9,12) {$\bullet$};
\node (top4) at (10,12) {};

\node (left-1) at (-2,11) {};
\node (left0) at (-2,9)  {$\circ$};
\node (left1) at (-2,6) {$\circ$};
\node (left2) at (-2,3) {$\bullet$};
\node (left3) at (-2,2) {};

\node (top left corner) at (-1, 2) {};
\node (bottom right corner) at (10, 2) {};

\draw (-1, 2) -- (10, 2) -- (10, 11) -- (-1, 11) -- (-1, 2);

\draw [->] (left-1) to node[left]{$\rho_{123}$} (left0); 
\draw [->] (left0) to node[left]{$\rho_{23}$} (left1); 
\draw [->] (left1) to node[left]{$\rho_{2}$} (left2); 
\draw [->] (left2) to (left3); 

\draw [->] (top0) to (top-1);
\draw [->] (top0) to node[above]{$\sigma_3$} (top1); 
\draw [->] (top1) to node[above]{$\sigma_{23}$} (top2); 
\draw [->] (top2) to node[above]{$\sigma_2$} (top3);
\draw [->] (top3) to (top4);

\draw [->] (g20) to (top left corner);

\draw [->, ultra thick] (g00) to (g01a);
\draw [->, ultra thick] (g01b) to (g02a);
\draw [->, ultra thick] (g02b) to (g03);
\draw [->, ultra thick] (g10) to (g11a);
\draw [->, ultra thick] (g11b) to (g12a);
\draw [->, ultra thick] (g12b) to (g13);
\draw [->, color = gray] (g01b) to (g11a);
\draw [->, color = gray] (g02b) to (g12a);
\draw [->, color = gray] (g11b) to (g21);
\draw [->, color = gray] (g12b) to (g22);

\draw [->, color = gray] (g00) to node[below]{$\tau_{23}$} (g01b);
\draw [->, color = gray] (g01a) to node[above]{$\tau_{23}$} (g02a);
\draw [->, color = gray] (g01b) to node[below]{$\tau_{23}$} (g02b);
\draw [->, color = gray] (g02a) to node[above]{$\tau_{23}$} (g03);
\draw [->, color = gray] (g10) to node[below]{$\tau_{23}$} (g11b);
\draw [->, color = gray] (g11a) to node[above]{$\tau_{23}$} (g12a);
\draw [->, color = gray] (g11b) to node[below]{$\tau_{23}$} (g12b);
\draw [->, color = gray] (g12a) to node[above]{$\tau_{23}$} (g13);
\draw [->] (g20) to node[above]{$\tau_{3}$} (g21);
\draw [->] (g21) to node[above]{$\tau_{23}$} (g22);
\draw [->] (g22) to node[above]{$\tau_{2}$} (g23);

\draw [->, color = gray] (g00) to node[left]{$\tau_{23}$} (g10);
\draw [->, color = gray] (g10) to node[left]{$\tau_{2}$} (g20);
\draw [->, color = gray] (g03) to node[left]{$\tau_{23}$} (g13);
\draw [->, color = gray] (g13) to node[left]{$\tau_{2}$} (g23);
\draw [->, color = black] (g23) to (bottom right corner);

\end{tikzpicture}
\begin{tikzpicture}[scale = .8, > = latex]
\scriptsize
\node (g01a) at (3,9.5) {$\circ$};
\node [color = gray] (g01b) at (3,8.5) {$\circ$};
\node (g02a) at (6,9.5) {$\circ$};
\node [color = gray] (g02b) at (6,8.5) {$\circ$};

\node [color = gray] (g11a) at (3,6.5) {$\circ$};
\node [color = gray] (g11b) at (3,5.5) {$\circ$};
\node [color = gray] (g12a) at (6,6.5) {$\circ$};
\node [color = gray] (g12b) at (6,5.5) {$\circ$};

\node [color = gray] (g21) at (3,3) {$\circ$};
\node [color = gray] (g22) at (6,3) {$\circ$};

\node (top0) at (1,12) {};
\node (top1) at (3,12) {$\circ$};
\node (top2) at (6,12) {$\circ$};
\node (top3) at (8,12) {};

\node (top left corner) at (1, 11) {};
\node (top right corner) at (8, 11) {};
\node at (1,2) {};

\draw (1, 2) -- (8, 2) -- (8, 11) -- (1, 11) -- (1, 2);

\draw [->] (top0) to node[above]{$\sigma_{123}$} (top1); 
\draw [->] (top1) to node[above]{$\sigma_{23}$} (top2); 
\draw [->] (top3) to node[above]{$\sigma_1$} (top2);

\draw [->] (top left corner) to node[above right = -2pt]{$\tau_{123}$} (g01a);

\draw [->, color = gray] (g01b) to (g02a);
\draw [->, color = gray] (g11b) to (g12a);
\draw [->, ultra thick] (g01b) to (g11a);
\draw [->, ultra thick] (g02b) to (g12a);
\draw [->, ultra thick] (g11b) to (g21);
\draw [->, ultra thick] (g12b) to (g22);

\draw [->] (g01a) to node[above]{$\tau_{23}$} (g02a);
\draw [->, color = gray] (g01b) to node[below]{$\tau_{23}$} (g02b);
\draw [->, color = gray] (g11a) to node[above]{$\tau_{23}$} (g12a);
\draw [->, color = gray] (g11b) to node[below]{$\tau_{23}$} (g12b);
\draw [->, color = gray] (g21) to node[above]{$\tau_{23}$} (g22);

\draw [->, color = black] (top right corner) to node[above left = -2pt]{$\tau_1$} (g02a);

\end{tikzpicture}

\begin{tikzpicture}[scale = .8, >=latex]
\scriptsize
\node (g00) at (0,1) {$\bullet$};
\node (g01) at (3,1) {$\circ$};
\node (g02) at (6,1) {$\circ$};
\node (g03) at (9,1) {$\bullet$};

\node (left-1) at (-2,3) {};
\node (left0) at (-2,1)  {$\bullet$};
\node (left1) at (-2,0) {};

\node (top left corner) at (-1, 3) {};
\node (bottom right corner) at (10, 0) {};
\node (bottom left corner) at (-1, 0) {};

\draw (-1, 0) -- (10, 0) -- (10, 3) -- (-1, 3) -- (-1, 0);

\draw [->] (left-1) to node[left]{$\rho_{12}$} (left0); 
\draw [->] (left0) to (left1); 

\draw [->] (g00) to (bottom left corner);
\draw [->] (g00) to node[above]{$\tau_{3}$} (g01);
\draw [->] (g01) to node[above]{$\tau_{23}$} (g02);
\draw [->] (g02) to node[above]{$\tau_{2}$} (g03);
\draw [->] (g03) to (bottom right corner);

\end{tikzpicture} 
\begin{tikzpicture}[scale = .8, >=latex]
\scriptsize
\node (g00) at (0,0) {$\circ$};
\node (g01) at (3,0) {$\circ$};

\node (top left corner) at (-2, 2) {};
\node (bottom right corner) at (5, -1) {};
\node (top right corner) at (5, 2) {};
\draw (-2, -1) -- (5, -1) -- (5, 2) -- (-2, 2) -- (-2, -1);

\draw [->] (top left corner) to node[above right = -2pt]{$\tau_{123}$} (g00);
\draw [->] (g00) to node[above]{$\tau_{23}$} (g01);
\draw [->] (top right corner) to node[above left = -2pt]{$\tau_{1}$} (g01);

\end{tikzpicture}

\begin{tikzpicture}[scale = .8, > = latex]
\scriptsize
\node [color = gray] (g00) at (0,9) {$\circ$};
\node [color = gray] (g01a) at (3,9.5) {$\circ$};
\node [color = gray] (g01b) at (3,8.5) {$\circ$};
\node [color = gray] (g02a) at (6,9.5) {$\circ$};
\node [color = gray] (g02b) at (6,8.5) {$\circ$};
\node [color = gray] (g03) at (9,9) {$\circ$};

\node [color = gray] (g10) at (0,6) {$\circ$};
\node [color = gray] (g11a) at (3,6.5) {$\circ$};
\node [color = gray] (g11b) at (3,5.5) {$\circ$};
\node [color = gray] (g12a) at (6,6.5) {$\circ$};
\node [color = gray] (g12b) at (6,5.5) {$\circ$};
\node [color = gray] (g13) at (9,6) {$\circ$};

\node (g20) at (0,3) {$\bullet$};
\node (g21) at (3,3) {$\circ$};
\node (g22) at (6,3) {$\circ$};
\node (g23) at (9,3) {$\bullet$};

\node (left-1) at (-2,11) {};
\node (left0) at (-2,9)  {$\circ$};
\node (left1) at (-2,6) {$\circ$};
\node (left2) at (-2,3) {$\bullet$};
\node (left3) at (-2,2) {};

\node (bottom left corner) at (-1, 2) {};
\node (bottom right corner) at (10, 2) {};
\node (top left corner) at (-1, 11) {};

\draw (-1, 2) -- (10, 2) -- (10, 11) -- (-1, 11) -- (-1, 2);

\draw [->] (left-1) to node[left]{$\rho_{1}$} (left0); 
\draw [->] (left1) to node[left]{$\rho_{23}$} (left0); 
\draw [->] (left2) to node[left]{$\rho_{3}$} (left1); 
\draw [->] (left2) to (left3); 

\draw [->] (g20) to (bottom left corner);

\draw [->, ultra thick] (g00) to (g01b);
\draw [->, ultra thick] (g01a) to (g02b);
\draw [->, ultra thick] (g02a) to (g03);
\draw [->, ultra thick] (g10) to (g11b);
\draw [->, ultra thick] (g11a) to (g12b);
\draw [->, ultra thick] (g12a) to (g13);
\draw [->, color = gray] (g11a) to (g01b);
\draw [->, color = gray] (g12a) to (g02b);
\draw [->, color = gray] (g21) to (g11b);
\draw [->, color = gray] (g22) to (g12b);

\draw [->, color = gray] (g00) to node[above]{$\tau_{23}$} (g01a);
\draw [->, color = gray] (g01a) to node[above]{$\tau_{23}$} (g02a);
\draw [->, color = gray] (g01b) to node[below]{$\tau_{23}$} (g02b);
\draw [->, color = gray] (g02b) to node[below]{$\tau_{23}$} (g03);
\draw [->, color = gray] (g10) to node[above]{$\tau_{23}$} (g11a);
\draw [->, color = gray] (g11a) to node[above]{$\tau_{23}$} (g12a);
\draw [->, color = gray] (g11b) to node[below]{$\tau_{23}$} (g12b);
\draw [->, color = gray] (g12b) to node[below]{$\tau_{23}$} (g13);
\draw [->] (g20) to node[above]{$\tau_{3}$} (g21);
\draw [->] (g21) to node[above]{$\tau_{23}$} (g22);
\draw [->] (g22) to node[above]{$\tau_{2}$} (g23);

\draw [->, color = gray] (g10) to node[left]{$\tau_{23}$} (g00);
\draw [->, color = gray] (g20) to node[left]{$\tau_{3}$} (g10);
\draw [->, color = gray] (g13) to node[left]{$\tau_{23}$} (g03);
\draw [->, color = gray] (g23) to node[left]{$\tau_{3}$} (g13);
\draw [->, color = black] (g23) to (bottom right corner);

\end{tikzpicture}
\begin{tikzpicture}[scale = .8, > = latex]
\scriptsize
\node (g01a) at (3,9.5) {$\circ$};
\node [color = gray] (g01b) at (3,8.5) {$\circ$};
\node (g02a) at (6,9.5) {$\circ$};
\node [color = gray] (g02b) at (6,8.5) {$\circ$};

\node [color = gray] (g11a) at (3,6.5) {$\circ$};
\node [color = gray] (g11b) at (3,5.5) {$\circ$};
\node [color = gray] (g12a) at (6,6.5) {$\circ$};
\node [color = gray] (g12b) at (6,5.5) {$\circ$};

\node [color = gray] (g21) at (3,3) {$\circ$};
\node [color = gray] (g22) at (6,3) {$\circ$};

\node (top left corner) at (1, 11) {};
\node (top right corner) at (8, 11) {};
\node at (1,2) {};

\draw (1, 2) -- (8, 2) -- (8, 11) -- (1, 11) -- (1, 2);

\draw [->] (top left corner) to node[above right = -2pt]{$\tau_{123}$} (g01a);
\draw [->, color = gray] (top left corner) to node[below left = -2pt]{$\tau_{1}$} (g01b);

\draw [->, color = gray] (g01a) to (g02b);
\draw [->, color = gray] (g11a) to (g12b);
\draw [->, ultra thick] (g11a) to (g01b);
\draw [->, ultra thick] (g12a) to (g02b);
\draw [->, ultra thick] (g21) to (g11b);
\draw [->, ultra thick] (g22) to(g12b);

\draw [->] (g01a) to node[above]{$\tau_{23}$} (g02a);
\draw [->, color = gray] (g01b) to node[below]{$\tau_{23}$} (g02b);
\draw [->, color = gray] (g11a) to node[above]{$\tau_{23}$} (g12a);
\draw [->, color = gray] (g11b) to node[below]{$\tau_{23}$} (g12b);
\draw [->, color = gray] (g21) to node[above]{$\tau_{23}$} (g22);

\draw [->, color = black] (top right corner) to node[above left = -2pt]{$\tau_1$} (g02a);

\end{tikzpicture}

\caption{Each box contains the portion of $\widehat{\mathit{CFDAA}}(S^1\times\pants) \boxtimes (\lp_1 , \lp_2)$ coming from a $d_2$, $d_0$, or $d_{-2}$ segment in $\lp_1$ (top, middle, bottom, respectively) and an $a_2$ or $b_2$ segment in $\lp_2$ (left and right, respectively). Thick unlabeled arrows can be removed with the edge reduction algorithm; gray indicates generators and arrows that are eliminated when the differentials are canceled.}
\label{fig:merge1}
\end{center}
\end{figure}
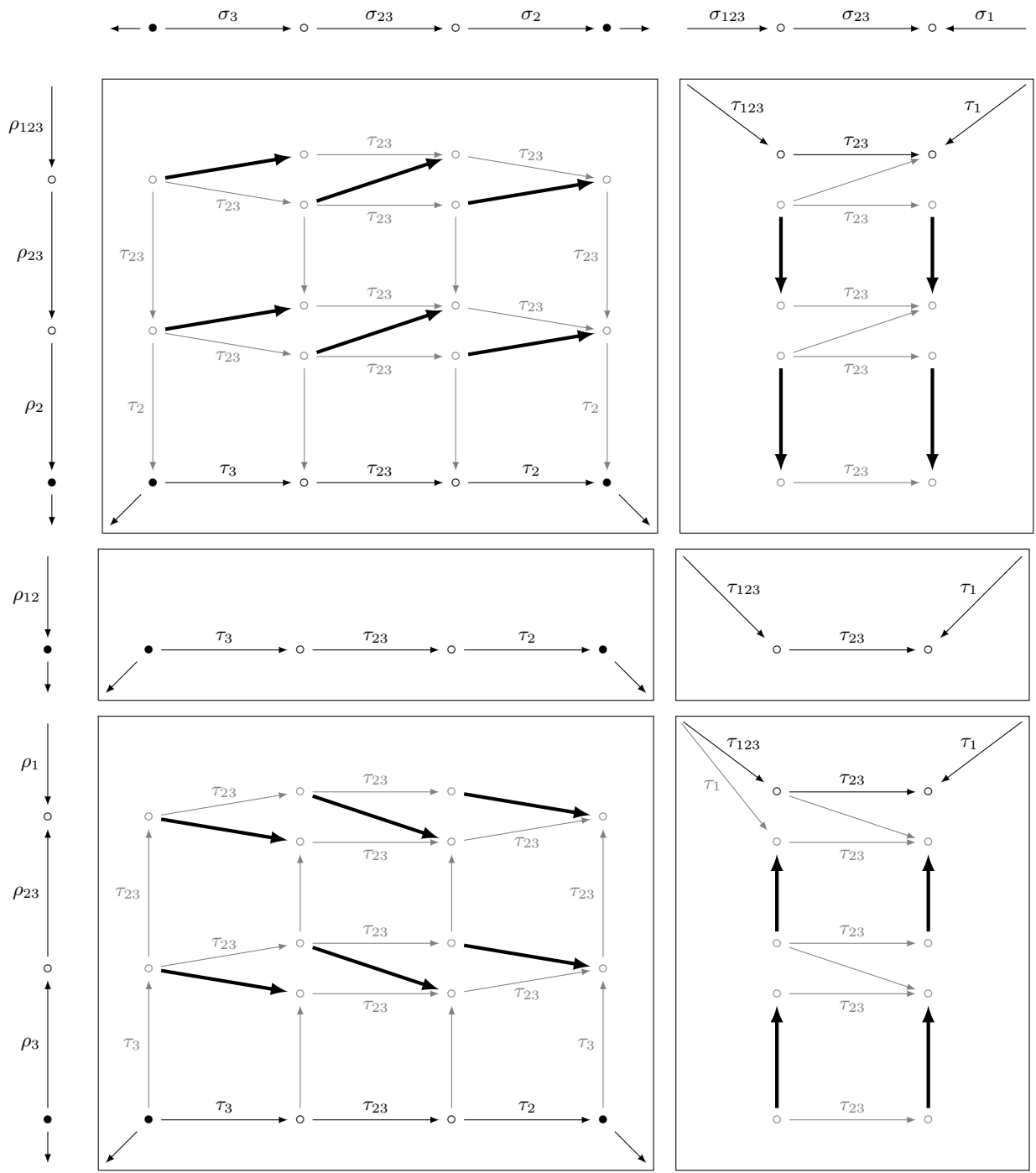

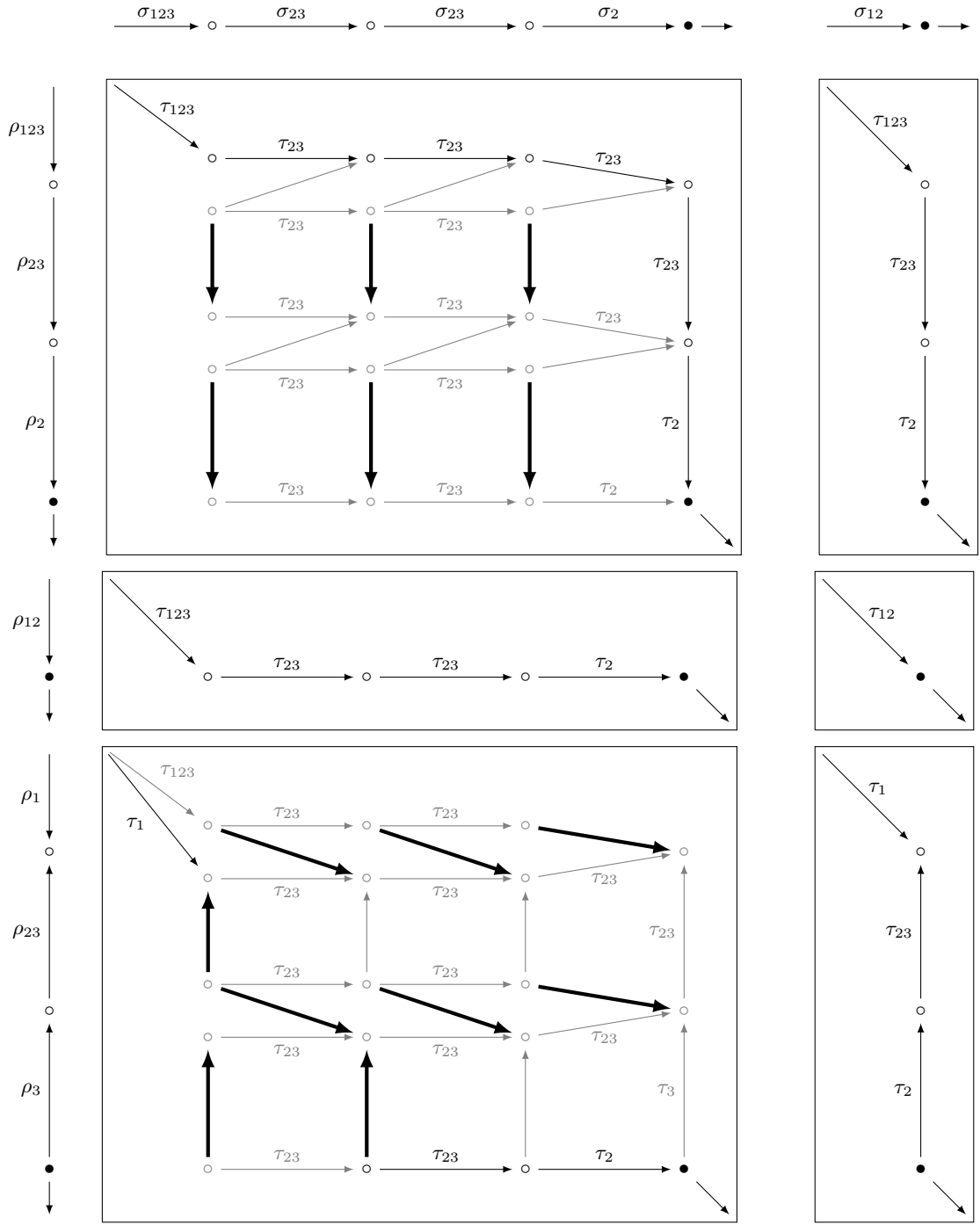
\begin{figure}
\begin{center}
\begin{tikzpicture}[scale = .8, >=latex]
\scriptsize
\node (g00a) at (0,9.5) {$\circ$};
\node [color = gray] (g00b) at (0,8.5) {$\circ$};
\node (g01a) at (3,9.5) {$\circ$};
\node [color = gray] (g01b) at (3,8.5) {$\circ$};
\node (g02a) at (6,9.5) {$\circ$};
\node [color = gray] (g02b) at (6,8.5) {$\circ$};
\node (g03) at (9,9) {$\circ$};

\node [color = gray] (g10a) at (0,6.5) {$\circ$};
\node [color = gray] (g10b) at (0,5.5) {$\circ$};
\node [color = gray] (g11a) at (3,6.5) {$\circ$};
\node [color = gray] (g11b) at (3,5.5) {$\circ$};
\node [color = gray] (g12a) at (6,6.5) {$\circ$};
\node [color = gray] (g12b) at (6,5.5) {$\circ$};
\node  (g13) at (9,6) {$\circ$};

\node [color = gray] (g20) at (0,3) {$\circ$};
\node [color = gray] (g21) at (3,3) {$\circ$};
\node [color = gray] (g22) at (6,3) {$\circ$};
\node (g23) at (9,3) {$\bullet$};

\node (top-1) at (-2, 12) {};
\node (top0) at (0,12) {$\circ$};
\node (top1) at (3,12) {$\circ$};
\node (top2) at (6,12) {$\circ$};
\node (top3) at (9,12) {$\bullet$};
\node (top4) at (10,12) {};

\node (left-1) at (-3,11) {};
\node (left0) at (-3,9)  {$\circ$};
\node (left1) at (-3,6) {$\circ$};
\node (left2) at (-3,3) {$\bullet$};
\node (left3) at (-3,2) {};

\node (top left corner) at (-2, 11) {};
\node (bottom right corner) at (10, 2) {};

\draw (-2, 2) -- (10, 2) -- (10, 11) -- (-2, 11) -- (-2, 2);

\draw [->] (left-1) to node[left]{$\rho_{123}$} (left0); 
\draw [->] (left0) to node[left]{$\rho_{23}$} (left1); 
\draw [->] (left1) to node[left]{$\rho_{2}$} (left2); 
\draw [->] (left2) to (left3); 

\draw [->] (top-1) to node[above]{$\sigma_{123}$} (top0); 
\draw [->] (top0) to node[above]{$\sigma_{23}$} (top1); 
\draw [->] (top1) to node[above]{$\sigma_{23}$} (top2); 
\draw [->] (top2) to node[above]{$\sigma_2$} (top3);
\draw [->] (top3) to (top4);

\draw [->] (top left corner) to node[above right = -2pt]{$\tau_{123}$} (g00a);

\draw [->, color = gray] (g00b) to (g01a);
\draw [->, color = gray] (g01b) to (g02a);
\draw [->, color = gray] (g02b) to (g03);
\draw [->, color = gray] (g10b) to (g11a);
\draw [->, color = gray] (g11b) to (g12a);
\draw [->, color = gray] (g12b) to (g13);
\draw [->, ultra thick] (g00b) to (g10a);
\draw [->, ultra thick] (g01b) to (g11a);
\draw [->, ultra thick] (g02b) to (g12a);
\draw [->, ultra thick] (g10b) to (g20);
\draw [->, ultra thick] (g11b) to (g21);
\draw [->, ultra thick] (g12b) to (g22);

\draw [->, color = black] (g00a) to node[above]{$\tau_{23}$} (g01a);
\draw [->, color = gray] (g00b) to node[below]{$\tau_{23}$} (g01b);
\draw [->, color = black] (g01a) to node[above]{$\tau_{23}$} (g02a);
\draw [->, color = gray] (g01b) to node[below]{$\tau_{23}$} (g02b);
\draw [->, color = black] (g02a) to node[above]{$\tau_{23}$} (g03);
\draw [->, color = gray] (g10a) to node[above]{$\tau_{23}$} (g11a);
\draw [->, color = gray] (g10b) to node[below]{$\tau_{23}$} (g11b);
\draw [->, color = gray] (g11a) to node[above]{$\tau_{23}$} (g12a);
\draw [->, color = gray] (g11b) to node[below]{$\tau_{23}$} (g12b);
\draw [->, color = gray] (g12a) to node[above]{$\tau_{23}$} (g13);
\draw [->, color = gray] (g20) to node[above]{$\tau_{23}$} (g21);
\draw [->, color = gray] (g21) to node[above]{$\tau_{23}$} (g22);
\draw [->, color = gray] (g22) to node[above]{$\tau_{2}$} (g23);

\draw [->, color = black] (g03) to node[left]{$\tau_{23}$} (g13);
\draw [->, color = black] (g13) to node[left]{$\tau_{2}$} (g23);
\draw [->, color = black] (g23) to (bottom right corner);

\end{tikzpicture}  \qquad 
\begin{tikzpicture}[scale = .8, >=latex]
\scriptsize
\node (g00) at (0,9) {$\circ$};
\node (g10) at (0,6) {$\circ$};
\node (g20) at (0,3) {$\bullet$};

\node (top-1) at (-2, 12) {};
\node (top0) at (0,12) {$\bullet$};
\node (top1) at (1,12) {};

\node (top left corner) at (-2, 11) {};
\node (bottom right corner) at (1, 2) {};

\draw (-2, 2) -- (1, 2) -- (1, 11) -- (-2, 11) -- (-2, 2);

\draw [->] (top-1) to node[above]{$\sigma_{12}$} (top0); 
\draw [->] (top0) to (top1);

\draw [->] (top left corner) to node[above right = -2pt]{$\tau_{123}$} (g00);

\draw [->, color = black] (g00) to node[left]{$\tau_{23}$} (g10);
\draw [->, color = black] (g10) to node[left]{$\tau_{2}$} (g20);
\draw [->, color = black] (g20) to (bottom right corner);

\end{tikzpicture}


\begin{tikzpicture}[scale = .8, >=latex]
\scriptsize
\node (g00) at (0,1) {$\circ$};
\node (g01) at (3,1) {$\circ$};
\node (g02) at (6,1) {$\circ$};
\node (g03) at (9,1) {$\bullet$};

\node (left-1) at (-3,3) {};
\node (left0) at (-3,1)  {$\bullet$};
\node (left1) at (-3,0) {};

\node (top left corner) at (-2, 3) {};
\node (bottom right corner) at (10, 0) {};

\draw (-2, 0) -- (10, 0) -- (10, 3) -- (-2, 3) -- (-2, 0);

\draw [->] (left-1) to node[left]{$\rho_{12}$} (left0); 
\draw [->] (left0) to (left1); 

\draw [->] (top left corner) to node[above right = -2pt]{$\tau_{123}$} (g00);
\draw [->] (g00) to node[above]{$\tau_{23}$} (g01);
\draw [->] (g01) to node[above]{$\tau_{23}$} (g02);
\draw [->] (g02) to node[above]{$\tau_{2}$} (g03);
\draw [->] (g03) to (bottom right corner);

\end{tikzpicture}  \qquad
\begin{tikzpicture}[scale = .8, >=latex]
\scriptsize
\node (g00) at (0,0) {$\bullet$};

\node (top left corner) at (-2, 2) {};
\node (bottom right corner) at (1, -1) {};

\draw (-2, -1) -- (1, -1) -- (1, 2) -- (-2, 2) -- (-2, -1);

\draw [->] (top left corner) to node[above right = -2pt]{$\tau_{12}$} (g00);
\draw [->] (g00) to (bottom right corner);

\end{tikzpicture}

\begin{tikzpicture}[scale = .8, >=latex]
\scriptsize
\node [color = gray] (g00a) at (0,9.5) {$\circ$};
\node [color = gray] (g00b) at (0,8.5) {$\circ$};
\node [color = gray] (g01a) at (3,9.5) {$\circ$};
\node [color = gray] (g01b) at (3,8.5) {$\circ$};
\node [color = gray] (g02a) at (6,9.5) {$\circ$};
\node [color = gray] (g02b) at (6,8.5) {$\circ$};
\node [color = gray] (g03) at (9,9) {$\circ$};

\node [color = gray] (g10a) at (0,6.5) {$\circ$};
\node [color = gray] (g10b) at (0,5.5) {$\circ$};
\node [color = gray] (g11a) at (3,6.5) {$\circ$};
\node [color = gray] (g11b) at (3,5.5) {$\circ$};
\node [color = gray] (g12a) at (6,6.5) {$\circ$};
\node [color = gray] (g12b) at (6,5.5) {$\circ$};
\node [color = gray] (g13) at (9,6) {$\circ$};

\node [color = gray] (g20) at (0,3) {$\circ$};
\node [color = black] (g21) at (3,3) {$\circ$};
\node [color = black] (g22) at (6,3) {$\circ$};
\node (g23) at (9,3) {$\bullet$};

\node (left-1) at (-3,11) {};
\node (left0) at (-3,9)  {$\circ$};
\node (left1) at (-3,6) {$\circ$};
\node (left2) at (-3,3) {$\bullet$};
\node (left3) at (-3,2) {};

\node (top left corner) at (-2, 11) {};
\node (bottom right corner) at (10, 2) {};

\draw (-2, 2) -- (10, 2) -- (10, 11) -- (-2, 11) -- (-2, 2);

\draw [->] (left-1) to node[left]{$\rho_{1}$} (left0); 
\draw [->] (left1) to node[left]{$\rho_{23}$} (left0); 
\draw [->] (left2) to node[left]{$\rho_{3}$} (left1); 
\draw [->] (left2) to (left3); 

\draw [->, color = gray] (top left corner) to node[above right]{$\tau_{123}$} (g00a);
\draw [->] (top left corner) to node[below left]{$\tau_{1}$} (g00b);

\draw [->, ultra thick] (g00a) to (g01b);
\draw [->, ultra thick] (g01a) to (g02b);
\draw [->, ultra thick] (g02a) to (g03);
\draw [->, ultra thick] (g10a) to (g11b);
\draw [->, ultra thick] (g11a) to (g12b);
\draw [->, ultra thick] (g12a) to (g13);
\draw [<-, ultra thick] (g00b) to (g10a);
\draw [<-, color = gray] (g01b) to (g11a);
\draw [<-, color = gray] (g02b) to (g12a);
\draw [<-, ultra thick] (g10b) to (g20);
\draw [<-, ultra thick] (g11b) to (g21);
\draw [<-, color = gray] (g12b) to (g22);

\draw [->, color = gray] (g00a) to node[above]{$\tau_{23}$} (g01a);
\draw [->, color = gray] (g00b) to node[below]{$\tau_{23}$} (g01b);
\draw [->, color = gray] (g01a) to node[above]{$\tau_{23}$} (g02a);
\draw [->, color = gray] (g01b) to node[below]{$\tau_{23}$} (g02b);
\draw [->, color = gray] (g02b) to node[below]{$\tau_{23}$} (g03);
\draw [->, color = gray] (g10a) to node[above]{$\tau_{23}$} (g11a);
\draw [->, color = gray] (g10b) to node[below]{$\tau_{23}$} (g11b);
\draw [->, color = gray] (g11a) to node[above]{$\tau_{23}$} (g12a);
\draw [->, color = gray] (g11b) to node[below]{$\tau_{23}$} (g12b);
\draw [->, color = gray] (g12b) to node[below]{$\tau_{23}$} (g13);
\draw [->, color = gray] (g20) to node[above]{$\tau_{23}$} (g21);
\draw [->, color = black] (g21) to node[above]{$\tau_{23}$} (g22);
\draw [->, color = black] (g22) to node[above]{$\tau_{2}$} (g23);

\draw [->, color = gray] (g13) to node[left]{$\tau_{23}$} (g03);
\draw [->, color = gray] (g23) to node[left]{$\tau_{3}$} (g13);
\draw [->, color = black] (g23) to (bottom right corner);

\end{tikzpicture}  \qquad 
\begin{tikzpicture}[scale = .8, >=latex]
\scriptsize
\node (g00) at (0,9) {$\circ$};
\node (g10) at (0,6) {$\circ$};
\node (g20) at (0,3) {$\bullet$};

\node (top left corner) at (-2, 11) {};
\node (bottom right corner) at (1, 2) {};

\draw (-2, 2) -- (1, 2) -- (1, 11) -- (-2, 11) -- (-2, 2);

\draw [->] (top left corner) to node[above right = -2pt]{$\tau_{1}$} (g00);

\draw [->, color = black] (g10) to node[left]{$\tau_{23}$} (g00);
\draw [->, color = black] (g20) to node[left]{$\tau_{2}$} (g10);
\draw [->, color = black] (g20) to (bottom right corner);

\end{tikzpicture}

\caption{The portions of $\widehat{\mathit{CFDAA}}(S^1\times\pants) \boxtimes (\lp_1 ,\lp_2)$ coming from a $d_2$, $d_0$, or $d_{-2}$ segment in $\lp_1$ and a $d_3$ or $d_0$ segment in $\lp_2$. Thick unlabeled arrows can be removed with the edge reduction algorithm; gray indicates generators and arrows that are eliminated when the differentials are canceled.}
\label{fig:merge2}
\end{center}
\end{figure}

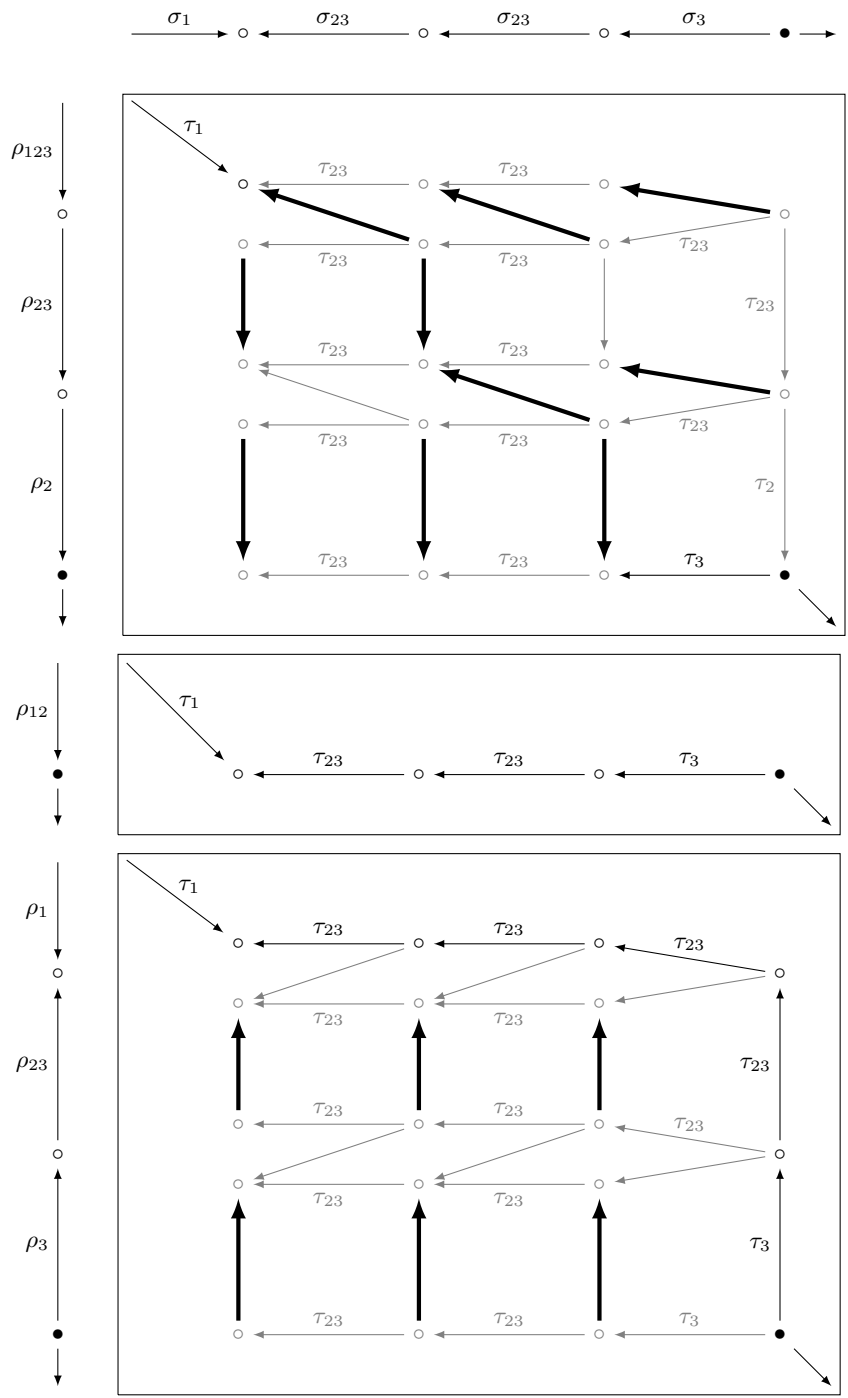
\begin{figure}
\begin{center}
\begin{tikzpicture}[scale = .8, >=latex]
\scriptsize
\node (g00a) at (0,9.5) {$\circ$};
\node [color = gray] (g00b) at (0,8.5) {$\circ$};
\node [color = gray] (g01a) at (3,9.5) {$\circ$};
\node [color = gray] (g01b) at (3,8.5) {$\circ$};
\node [color = gray] (g02a) at (6,9.5) {$\circ$};
\node [color = gray] (g02b) at (6,8.5) {$\circ$};
\node [color = gray] (g03) at (9,9) {$\circ$};

\node [color = gray] (g10a) at (0,6.5) {$\circ$};
\node [color = gray] (g10b) at (0,5.5) {$\circ$};
\node [color = gray] (g11a) at (3,6.5) {$\circ$};
\node [color = gray] (g11b) at (3,5.5) {$\circ$};
\node [color = gray] (g12a) at (6,6.5) {$\circ$};
\node [color = gray] (g12b) at (6,5.5) {$\circ$};
\node [color = gray] (g13) at (9,6) {$\circ$};

\node [color = gray] (g20) at (0,3) {$\circ$};
\node [color = gray] (g21) at (3,3) {$\circ$};
\node [color = gray] (g22) at (6,3) {$\circ$};
\node (g23) at (9,3) {$\bullet$};

\node (top-1) at (-2, 12) {};
\node (top0) at (0,12) {$\circ$};
\node (top1) at (3,12) {$\circ$};
\node (top2) at (6,12) {$\circ$};
\node (top3) at (9,12) {$\bullet$};
\node (top4) at (10,12) {};

\node (left-1) at (-3,11) {};
\node (left0) at (-3,9)  {$\circ$};
\node (left1) at (-3,6) {$\circ$};
\node (left2) at (-3,3) {$\bullet$};
\node (left3) at (-3,2) {};

\node (top left corner) at (-2, 11) {};
\node (bottom right corner) at (10, 2) {};

\draw (-2, 2) -- (10, 2) -- (10, 11) -- (-2, 11) -- (-2, 2);

\draw [->] (left-1) to node[left]{$\rho_{123}$} (left0); 
\draw [->] (left0) to node[left]{$\rho_{23}$} (left1); 
\draw [->] (left1) to node[left]{$\rho_{2}$} (left2); 
\draw [->] (left2) to (left3); 

\draw [->] (top-1) to node[above]{$\sigma_{1}$} (top0); 
\draw [<-] (top0) to node[above]{$\sigma_{23}$} (top1); 
\draw [<-] (top1) to node[above]{$\sigma_{23}$} (top2); 
\draw [<-] (top2) to node[above]{$\sigma_3$} (top3);
\draw [->] (top3) to (top4);

\draw [->] (top left corner) to node[above right = -2pt]{$\tau_{1}$} (g00a);

\draw [->, ultra thick] (g01b) to (g00a);
\draw [->, ultra thick] (g02b) to (g01a);
\draw [->, ultra thick] (g03) to (g02a);

\draw [->, color = gray] (g11b) to (g10a);
\draw [->, ultra thick] (g12b) to (g11a);
\draw [->, ultra thick] (g13) to (g12a);

\draw [->, ultra thick] (g00b) to (g10a);
\draw [->, ultra thick] (g01b) to (g11a);
\draw [->, color = gray] (g02b) to (g12a);
\draw [->, ultra thick] (g10b) to (g20);
\draw [->, ultra thick] (g11b) to (g21);
\draw [->, ultra thick] (g12b) to (g22);

\draw [<-, color = gray] (g00a) to node[above]{$\tau_{23}$} (g01a);
\draw [<-, color = gray] (g00b) to node[below]{$\tau_{23}$} (g01b);
\draw [<-, color = gray] (g01a) to node[above]{$\tau_{23}$} (g02a);
\draw [<-, color = gray] (g01b) to node[below]{$\tau_{23}$} (g02b);
\draw [<-, color = gray] (g02b) to node[below]{$\tau_{23}$} (g03);
\draw [<-, color = gray] (g10a) to node[above]{$\tau_{23}$} (g11a);
\draw [<-, color = gray] (g10b) to node[below]{$\tau_{23}$} (g11b);
\draw [<-, color = gray] (g11a) to node[above]{$\tau_{23}$} (g12a);
\draw [<-, color = gray] (g11b) to node[below]{$\tau_{23}$} (g12b);
\draw [<-, color = gray] (g12b) to node[below]{$\tau_{23}$} (g13);
\draw [<-, color = gray] (g20) to node[above]{$\tau_{23}$} (g21);
\draw [<-, color = gray] (g21) to node[above]{$\tau_{23}$} (g22);
\draw [<-, color = black] (g22) to node[above]{$\tau_{3}$} (g23);

\draw [->, color = gray] (g03) to node[left]{$\tau_{23}$} (g13);
\draw [->, color = gray] (g13) to node[left]{$\tau_{2}$} (g23);
\draw [->, color = black] (g23) to (bottom right corner);

\end{tikzpicture}


\begin{tikzpicture}[scale = .8, >=latex]
\scriptsize
\node (g00) at (0,1) {$\circ$};
\node (g01) at (3,1) {$\circ$};
\node (g02) at (6,1) {$\circ$};
\node (g03) at (9,1) {$\bullet$};

\node (left-1) at (-3,3) {};
\node (left0) at (-3,1)  {$\bullet$};
\node (left1) at (-3,0) {};

\node (top left corner) at (-2, 3) {};
\node (bottom right corner) at (10, 0) {};

\draw (-2, 0) -- (10, 0) -- (10, 3) -- (-2, 3) -- (-2, 0);

\draw [->] (left-1) to node[left]{$\rho_{12}$} (left0); 
\draw [->] (left0) to (left1); 

\draw [->] (top left corner) to node[above right = -2pt]{$\tau_{1}$} (g00);
\draw [<-] (g00) to node[above]{$\tau_{23}$} (g01);
\draw [<-] (g01) to node[above]{$\tau_{23}$} (g02);
\draw [<-] (g02) to node[above]{$\tau_{3}$} (g03);
\draw [->] (g03) to (bottom right corner);

\end{tikzpicture} 

\begin{tikzpicture}[scale = .8, >=latex]
\scriptsize
\node (g00a) at (0,9.5) {$\circ$};
\node [color = gray] (g00b) at (0,8.5) {$\circ$};
\node (g01a) at (3,9.5) {$\circ$};
\node [color = gray] (g01b) at (3,8.5) {$\circ$};
\node (g02a) at (6,9.5) {$\circ$};
\node [color = gray] (g02b) at (6,8.5) {$\circ$};
\node (g03) at (9,9) {$\circ$};

\node [color = gray] (g10a) at (0,6.5) {$\circ$};
\node [color = gray] (g10b) at (0,5.5) {$\circ$};
\node [color = gray] (g11a) at (3,6.5) {$\circ$};
\node [color = gray] (g11b) at (3,5.5) {$\circ$};
\node [color = gray] (g12a) at (6,6.5) {$\circ$};
\node [color = gray] (g12b) at (6,5.5) {$\circ$};
\node  (g13) at (9,6) {$\circ$};

\node [color = gray] (g20) at (0,3) {$\circ$};
\node [color = gray] (g21) at (3,3) {$\circ$};
\node [color = gray] (g22) at (6,3) {$\circ$};
\node (g23) at (9,3) {$\bullet$};

\node (left-1) at (-3,11) {};
\node (left0) at (-3,9)  {$\circ$};
\node (left1) at (-3,6) {$\circ$};
\node (left2) at (-3,3) {$\bullet$};
\node (left3) at (-3,2) {};

\node (top left corner) at (-2, 11) {};
\node (bottom right corner) at (10, 2) {};

\draw (-2, 2) -- (10, 2) -- (10, 11) -- (-2, 11) -- (-2, 2);

\draw [->] (left-1) to node[left]{$\rho_{1}$} (left0); 
\draw [<-] (left0) to node[left]{$\rho_{23}$} (left1); 
\draw [<-] (left1) to node[left]{$\rho_{3}$} (left2); 
\draw [->] (left2) to (left3); 

\draw [->] (top left corner) to node[above right = -2pt]{$\tau_{1}$} (g00a);

\draw [->, color = gray] (g01a) to (g00b);
\draw [->, color = gray] (g02a) to (g01b);
\draw [->, color = gray] (g03) to (g02b);

\draw [->, color = gray] (g11a) to (g10b);
\draw [->, color = gray] (g12a) to (g11b);
\draw [->, color = gray] (g13) to (g12b);

\draw [->, ultra thick] (g10a) to (g00b);
\draw [->, ultra thick] (g11a) to (g01b);
\draw [->, ultra thick] (g12a) to (g02b);
\draw [->, ultra thick] (g20) to (g10b);
\draw [->, ultra thick] (g21) to (g11b);
\draw [->, ultra thick] (g22) to (g12b);

\draw [<-, color = black] (g00a) to node[above]{$\tau_{23}$} (g01a);
\draw [<-, color = gray] (g00b) to node[below]{$\tau_{23}$} (g01b);
\draw [<-, color = black] (g01a) to node[above]{$\tau_{23}$} (g02a);
\draw [<-, color = gray] (g01b) to node[below]{$\tau_{23}$} (g02b);
\draw [<-, color = black] (g02a) to node[above]{$\tau_{23}$} (g03);
\draw [<-, color = gray] (g10a) to node[above]{$\tau_{23}$} (g11a);
\draw [<-, color = gray] (g10b) to node[below]{$\tau_{23}$} (g11b);
\draw [<-, color = gray] (g11a) to node[above]{$\tau_{23}$} (g12a);
\draw [<-, color = gray] (g11b) to node[below]{$\tau_{23}$} (g12b);
\draw [<-, color = gray] (g12a) to node[above]{$\tau_{23}$} (g13);
\draw [<-, color = gray] (g20) to node[above]{$\tau_{23}$} (g21);
\draw [<-, color = gray] (g21) to node[above]{$\tau_{23}$} (g22);
\draw [<-, color = gray] (g22) to node[above]{$\tau_{3}$} (g23);

\draw [->, color = black] (g13) to node[left]{$\tau_{23}$} (g03);
\draw [->, color = black] (g23) to node[left]{$\tau_{3}$} (g13);
\draw [->, color = black] (g23) to (bottom right corner);

\end{tikzpicture}

\caption{Each box contains the portion of $\widehat{\mathit{CFDAA}}(S^1\times\pants) \boxtimes (\lp_1, \lp_2)$ coming from a $d_2$, $d_0$, or $d_{-2}$ segment in $\lp_1$ and a $d_{-3}$ segment in $\lp_2$. Thick unlabeled arrows can be removed with the edge reduction algorithm; gray indicates generators and arrows that are eliminated when the differentials are canceled.}
\label{fig:merge3}
\end{center}
\end{figure}

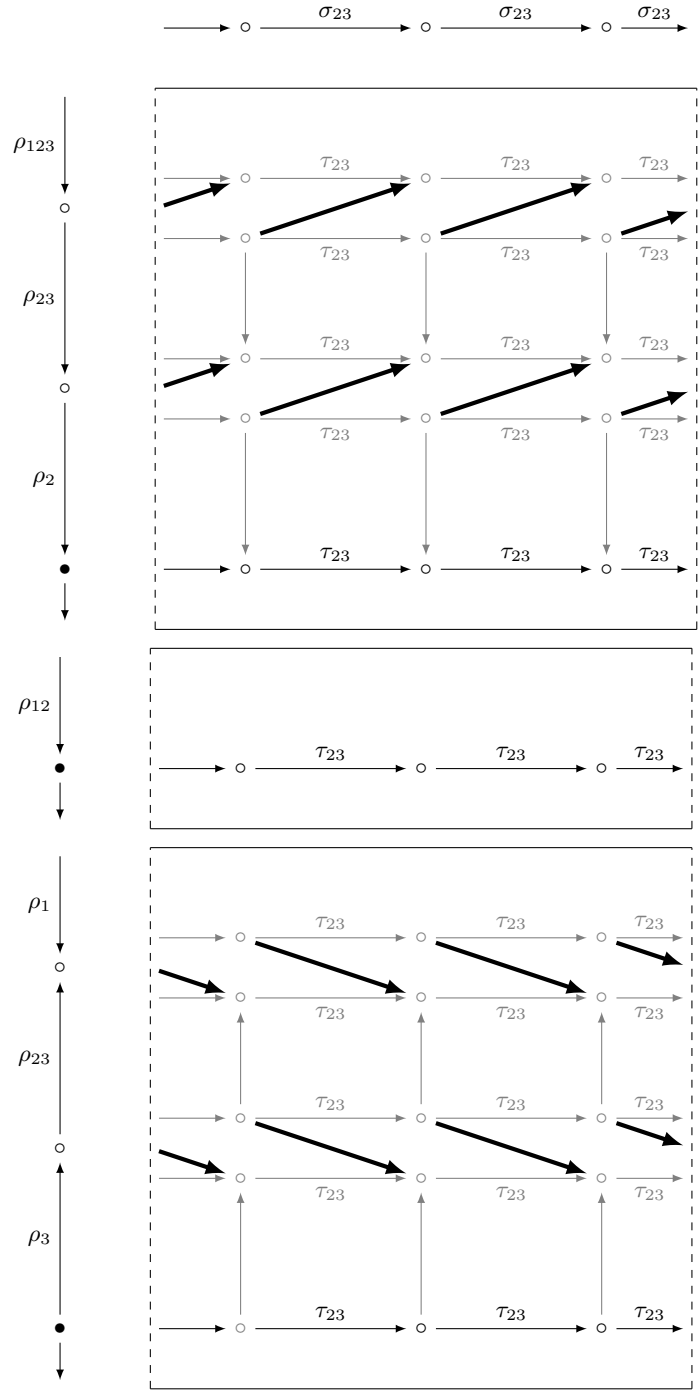
\begin{figure}
\begin{center}
\begin{tikzpicture}[scale = .8, >=latex]
\scriptsize
\node [color = gray] (g00a) at (0,9.5) {$\circ$};
\node [color = gray] (g00b) at (0,8.5) {$\circ$};
\node [color = gray] (g01a) at (3,9.5) {$\circ$};
\node [color = gray] (g01b) at (3,8.5) {$\circ$};
\node [color = gray] (g02a) at (6,9.5) {$\circ$};
\node [color = gray] (g02b) at (6,8.5) {$\circ$};

\node [color = gray] (g10a) at (0,6.5) {$\circ$};
\node [color = gray] (g10b) at (0,5.5) {$\circ$};
\node [color = gray] (g11a) at (3,6.5) {$\circ$};
\node [color = gray] (g11b) at (3,5.5) {$\circ$};
\node [color = gray] (g12a) at (6,6.5) {$\circ$};
\node [color = gray] (g12b) at (6,5.5) {$\circ$};

\node (g03) at (7.5, 9) {};
\node (g03a) at (7.5, 9.5) {};
\node (g03b) at (7.5, 8.5) {};
\node (g0-1) at (-1.5, 9) {};
\node (g0-1a) at (-1.5, 9.5) {};
\node (g0-1b) at (-1.5, 8.5) {};
\node (g13) at (7.5, 6) {};
\node (g13a) at (7.5, 6.5) {};
\node (g13b) at (7.5, 5.5) {};
\node (g1-1) at (-1.5, 6) {};
\node (g1-1a) at (-1.5, 6.5) {};
\node (g1-1b) at (-1.5, 5.5) {};
\node (g2-1) at (-1.5, 3) {};
\node (g23) at (7.5, 3) {};

\node (g20) at (0,3) {$\circ$};
\node (g21) at (3,3) {$\circ$};
\node (g22) at (6,3) {$\circ$};

\node (top-1) at (-1.5, 12) {};
\node (top0) at (0,12) {$\circ$};
\node (top1) at (3,12) {$\circ$};
\node (top2) at (6,12) {$\circ$};
\node (top3) at (7.5,12) {};

\node (left-1) at (-3,11) {};
\node (left0) at (-3,9)  {$\circ$};
\node (left1) at (-3,6) {$\circ$};
\node (left2) at (-3,3) {$\bullet$};
\node (left3) at (-3,2) {};

\node (top left corner) at (-1.5, 3) {};
\node (bottom right corner) at (7.5, 3) {};

\draw (-1.5, 2) -- (7.5, 2);
\draw (-1.5,11) -- (7.5,11);
\draw[dashed] (-1.5, 2) -- (-1.5,11);
\draw[dashed] (7.5,2) -- (7.5,11);

\draw [->] (left-1) to node[left]{$\rho_{123}$} (left0); 
\draw [->] (left0) to node[left]{$\rho_{23}$} (left1); 
\draw [->] (left1) to node[left]{$\rho_{2}$} (left2); 
\draw [->] (left2) to (left3); 

\draw [->] (top-1) to (top0); 
\draw [->] (top0) to node[above]{$\sigma_{23}$} (top1); 
\draw [->] (top1) to node[above]{$\sigma_{23}$} (top2); 
\draw [->] (top2) to node[above]{$\sigma_{23}$} (top3);

\draw [->, ultra thick] (g00b) to (g01a);
\draw [->, ultra thick] (g01b) to (g02a);
\draw [->, ultra thick] (g10b) to (g11a);
\draw [->, ultra thick] (g11b) to (g12a);
\draw [->, ultra thick] (g02b) to (g03);
\draw [->, ultra thick] (g0-1) to (g00a);
\draw [->, ultra thick] (g12b) to (g13);
\draw [->, ultra thick] (g1-1) to (g10a);
\draw [->, color = gray] (g00b) to (g10a);
\draw [->, color = gray] (g01b) to (g11a);
\draw [->, color = gray] (g02b) to (g12a);
\draw [->, color = gray] (g10b) to (g20);
\draw [->, color = gray] (g11b) to (g21);
\draw [->, color = gray] (g12b) to (g22);

\draw [->, color = gray] (g00a) to node[above]{$\tau_{23}$} (g01a);
\draw [->, color = gray] (g00b) to node[below]{$\tau_{23}$} (g01b);
\draw [->, color = gray] (g01a) to node[above]{$\tau_{23}$} (g02a);
\draw [->, color = gray] (g01b) to node[below]{$\tau_{23}$} (g02b);
\draw [->, color = gray] (g02a) to node[above]{$\tau_{23}$} (g03a);
\draw [->, color = gray] (g02b) to node[below]{$\tau_{23}$} (g03b);
\draw [->, color = gray] (g0-1a) to (g00a);
\draw [->, color = gray] (g0-1b) to (g00b);

\draw [->, color = gray] (g10a) to node[above]{$\tau_{23}$} (g11a);
\draw [->, color = gray] (g10b) to node[below]{$\tau_{23}$} (g11b);
\draw [->, color = gray] (g11a) to node[above]{$\tau_{23}$} (g12a);
\draw [->, color = gray] (g11b) to node[below]{$\tau_{23}$} (g12b);
\draw [->, color = gray] (g12a) to node[above]{$\tau_{23}$} (g13a);
\draw [->, color = gray] (g12b) to node[below]{$\tau_{23}$} (g13b);
\draw [->, color = gray] (g1-1a) to (g10a);
\draw [->, color = gray] (g1-1b) to (g10b);

\draw [->, color = black] (g20) to node[above]{$\tau_{23}$} (g21);
\draw [->, color = black] (g21) to node[above]{$\tau_{23}$} (g22);
\draw [->, color = black] (g22) to node[above]{$\tau_{23}$} (g23);
\draw [->, color = black] (g2-1) to (g20);

\end{tikzpicture}

\begin{tikzpicture}[scale = .8, >=latex]
\scriptsize
\node (g00) at (0,1) {$\circ$};
\node (g01) at (3,1) {$\circ$};
\node (g02) at (6,1) {$\circ$};
\node (g03) at (7.5,1) {};
\node (g0-1) at (-1.5,1) {};

\node (left-1) at (-3,3) {};
\node (left0) at (-3,1)  {$\bullet$};
\node (left1) at (-3,0) {};

\draw (-1.5, 0) -- (7.5, 0);
\draw (-1.5,3) -- (7.5,3);
\draw[dashed] (-1.5,0) -- (-1.5,3);
\draw[dashed] (7.5,0) -- (7.5,3);

\draw [->] (left-1) to node[left]{$\rho_{12}$} (left0); 
\draw [->] (left0) to (left1); 

\draw [->] (g0-1) to (g00);
\draw [->] (g00) to node[above]{$\tau_{23}$} (g01);
\draw [->] (g01) to node[above]{$\tau_{23}$} (g02);
\draw [->] (g02) to node[above]{$\tau_{23}$} (g03);

\end{tikzpicture}

\begin{tikzpicture}[scale = .8, >=latex]
\scriptsize
\node [color = gray] (g00a) at (0,9.5) {$\circ$};
\node [color = gray] (g00b) at (0,8.5) {$\circ$};
\node [color = gray] (g01a) at (3,9.5) {$\circ$};
\node [color = gray] (g01b) at (3,8.5) {$\circ$};
\node [color = gray] (g02a) at (6,9.5) {$\circ$};
\node [color = gray] (g02b) at (6,8.5) {$\circ$};
\node [color = gray] (g03) at (7.5,9) {};
\node [color = gray] (g03a) at (7.5,9.5) {};
\node [color = gray] (g03b) at (7.5,8.5) {};
\node [color = gray] (g0-1) at (-1.5,9) {};
\node [color = gray] (g0-1a) at (-1.5,9.5) {};
\node [color = gray] (g0-1b) at (-1.5,8.5) {};

\node [color = gray] (g10a) at (0,6.5) {$\circ$};
\node [color = gray] (g10b) at (0,5.5) {$\circ$};
\node [color = gray] (g11a) at (3,6.5) {$\circ$};
\node [color = gray] (g11b) at (3,5.5) {$\circ$};
\node [color = gray] (g12a) at (6,6.5) {$\circ$};
\node [color = gray] (g12b) at (6,5.5) {$\circ$};
\node [color = gray] (g13) at (7.5,6) {};
\node [color = gray] (g13a) at (7.5,6.5) {};
\node [color = gray] (g13b) at (7.5,5.5) {};
\node [color = gray] (g1-1) at (-1.5,6) {};
\node [color = gray] (g1-1a) at (-1.5,6.5) {};
\node [color = gray] (g1-1b) at (-1.5,5.5) {};

\node [color = gray] (g20) at (0,3) {$\circ$};
\node [color = black] (g21) at (3,3) {$\circ$};
\node [color = black] (g22) at (6,3) {$\circ$};
\node (g23) at (7.5,3) {};
\node (g2-1) at (-1.5,3) {};

\node (left-1) at (-3,11) {};
\node (left0) at (-3,9)  {$\circ$};
\node (left1) at (-3,6) {$\circ$};
\node (left2) at (-3,3) {$\bullet$};
\node (left3) at (-3,2) {};

\draw (-1.5, 2) -- (7.5, 2);
\draw (-1.5,11) -- (7.5,11);
\draw[dashed] (-1.5, 2) -- (-1.5,11);
\draw[dashed] (7.5,2) -- (7.5,11);

\draw [->] (left-1) to node[left]{$\rho_{1}$} (left0); 
\draw [->] (left1) to node[left]{$\rho_{23}$} (left0); 
\draw [->] (left2) to node[left]{$\rho_{3}$} (left1); 
\draw [->] (left2) to (left3); 

\draw [->, ultra thick] (g00a) to (g01b);
\draw [->, ultra thick] (g01a) to (g02b);
\draw [->, ultra thick] (g02a) to (g03);
\draw [->, ultra thick] (g0-1) to (g00b);
\draw [->, ultra thick] (g10a) to (g11b);
\draw [->, ultra thick] (g11a) to (g12b);
\draw [->, ultra thick] (g12a) to (g13);
\draw [->, ultra thick] (g1-1) to (g10b);
\draw [<-, color = gray] (g00b) to (g10a);
\draw [<-, color = gray] (g01b) to (g11a);
\draw [<-, color = gray] (g02b) to (g12a);
\draw [<-, color = gray] (g10b) to (g20);
\draw [<-, color = gray] (g11b) to (g21);
\draw [<-, color = gray] (g12b) to (g22);

\draw [->, color = gray] (g00a) to node[above]{$\tau_{23}$} (g01a);
\draw [->, color = gray] (g00b) to node[below]{$\tau_{23}$} (g01b);
\draw [->, color = gray] (g01a) to node[above]{$\tau_{23}$} (g02a);
\draw [->, color = gray] (g01b) to node[below]{$\tau_{23}$} (g02b);
\draw [->, color = gray] (g02a) to node[above]{$\tau_{23}$} (g03a);
\draw [->, color = gray] (g02b) to node[below]{$\tau_{23}$} (g03b);
\draw [->, color = gray] (g0-1a) to (g00a);
\draw [->, color = gray] (g0-1b) to (g00b);

\draw [->, color = gray] (g10a) to node[above]{$\tau_{23}$} (g11a);
\draw [->, color = gray] (g10b) to node[below]{$\tau_{23}$} (g11b);
\draw [->, color = gray] (g11a) to node[above]{$\tau_{23}$} (g12a);
\draw [->, color = gray] (g11b) to node[below]{$\tau_{23}$} (g12b);
\draw [->, color = gray] (g12a) to node[above]{$\tau_{23}$} (g13a);
\draw [->, color = gray] (g12b) to node[below]{$\tau_{23}$} (g13b);
\draw [->, color = gray] (g1-1a) to (g10a);
\draw [->, color = gray] (g1-1b) to (g10b);

\draw [->] (g2-1) to (g20);
\draw [->, color = black] (g20) to node[above]{$\tau_{23}$} (g21);
\draw [->, color = black] (g21) to node[above]{$\tau_{23}$} (g22);
\draw [->, color = black] (g22) to node[above]{$\tau_{23}$} (g23);

\end{tikzpicture}

\caption{The portions of $\widehat{\mathit{CFDAA}}(S^1\times\pants) \boxtimes (\lp_1,\lp_2)$ coming from a $d_2$, $d_0$, or $d_{-2}$ segment in $\lp_1$ when $\lp_2$ is a collection of $e^*$ segments. The left and right edges of each box should be identified. Thick unlabeled arrows can be removed with the edge reduction algorithm; gray indicates generators and arrows that are eliminated when the differentials are canceled. The result is a copy of $\lp_2$ for each $\iota_0$-generator in $\lp_1$.}
\label{fig:merge4}
\end{center}
\end{figure}

In practice, it is helpful to compute $\me(\lp_1, \lp_2)$ by creating an $i$ by $j$ grid, where $i$ is the length of $\lp_1$ and $j$ is the length of $\lp_2$. The $(i,j)$ entry of this grid is a square containing a single segment connecting two of its corners, as dictated by Proposition \ref{prop:merge_unstable_chains} (see, for example, Figure \ref{fig:merge_example}). The collection of loops $\me(\lp_1, \lp_2)$ can now be read off the grid by identifying the top and bottom edges and the left and right edges. Note that the number of disjoint loops in $\me(\lp_1, \lp_2)$ is given by $\gcd(i, j')$, where $i$ is the number of $d_k$ segments in $\lp_1$ (by assumption, this is the length of $\lp_1$) and $j'$ is the number of $d_k$ segments in $\lp_2$ minus the number of $c_k$ segments in $\lp_2$ (if $j' = 0$, we use the convention that $\gcd(i, 0) = i$; our assumptions rule out the case that $i = 0$). Note that $i$ and $j'$ can be given in terms of the $(\Ztwo)$-grading on $\lp_1$ and $\lp_2$: up to sign, we have $i = \chi_\bullet(\lp_1)$ and $j' = \chi_\bullet(\lp_2)$. To see this, observe that, for an appropriate choice of relative grading, each $d_k$ segment contributes $1$ to $\chi_\bullet$, each $c_k$ segment contributes $-1$, and each $a_k$ or $b_k$ segment contributes 0.

\begin{lem}\label{lem:merge_loop_type}
Consider a pair of loop-type, bordered, rational homology solid tori $(M_1,\alpha_1,\beta_1)$ and $(M_2,\alpha_2,\beta_2)$. If the loops representing $\CFD(M_1,\alpha_1,\beta_1)$ contain only standard unstable chains then \[\merge_{1,2} = \merge\left((M_1,\alpha_1,\beta_1),(M_2,\alpha_2,\beta_2)\right)\] is of loop-type. If, in addition, the loops representing $\CFD(M_2,\alpha_2,\beta_2)$ contain only standard unstable chains, then $\merge_{1,2}$ is of simple loop-type.
\end{lem}
\begin{proof}
If $\CFD(M_1,\alpha_1,\beta_1)$ and $\CFD(M_2,\alpha_2,\beta_2)$ are collections of loops and the loops in $\CFD(M_1,\alpha_1,\beta_1)$ contain only standard unstable chains, then by Proposition \ref{prop:merge_unstable_chains} $\CFD(\merge_{1,2})$ is a collection of loops. Moreover, if the loops in $\CFD(M_2,\alpha_2,\beta_2)$ also contain only standard unstable chains, Proposition \ref{prop:merge_unstable_chains} implies that the resulting loops in $\CFD(\merge_{1,2})$ contain only standard unstable chains, and in particular are simple. It only remains to check that $\CFD(\merge_{1,2})$ has exactly one loop for each spin$^c$-structure on $\merge_{1,2}$.

Recall that the operation $\merge$ corresponds to gluing two bordered manifolds to two of the three boundary components of $\pants\times S^1$, where $\pants$ is $S^2$ with three disks removed. Thus $\partial(S^1\times\pants)$ has three connected components; denote the $i^{\text{th}}$ connected component by $\partial(S^1\times\pants)_i$. For each  $i \in \{1, 2, 3\}$ let $f_i$ denote a curve in $\partial(S^1\times\pants)_i$ which is a fiber $\{\text{pt}\}\times S^1$ and let $b_i$ denote the relevant component of $\partial\pants\times \{\text{pt}\}$. According to the conventions introduced in Section \ref{sub:prelim}, applying the operation $\merge$ to two bordered manifolds $(M_1, \alpha_1, \beta_1)$ and $(M_2, \alpha_2, \beta_2)$ corresponds to gluing $M_1$ and $M_2$ to $\pants\times S^1$ by identifying $\alpha_1$ with $f_1$, $\beta_1$ with $b_1$, $\alpha_2$ with $f_2$, and $\beta_2$ with $b_2$. The result is the manifold $M_3=\merge_{1,2}$.

We consider homology groups with coefficients in $\Z$. The spin$^c$-structures on a 3-manifold $M$ with boundary are indexed by $H^2(M) \cong H_1(M, \partial M)$. Using the appropriate Mayer-Vietoris sequences, we have that $H_1(M_3, \partial M_3)$ is homomorphic to the quotient 
\[H_1(S^1\times\pants, \partial (S^1\times\pants)_3) \oplus H_1(M_1) \oplus H_1(M_2) / \{\alpha_1 \sim f_1, \beta_1\sim b_1, \alpha_2\sim f_2, \beta_2\sim b_2\}\]
Note that $H_1(S^1\times\pants, \partial (S^1\times\pants)_3)$ is generated by $f_i$ and $b_i$ for $i \in \{1,2,3\}$ with the relations $f_1 = f_2 = f_3$ and $b_1 + b_2 = b_3 = 0$. For $i = 1,2$, since $M_{i}$ is a rational solid torus there is a unique (possibly disconnected) curve in $\partial M_i$ that bounds a surface in $M_i$, to which we associate $p_i/ q_i$ (where $p_i$ and $q_i$ may not be relatively prime) so that $p_i \alpha_i + q_i \beta_i$ generates the kernel of the inclusion of $H_1(\partial M_i)$ into $H_1(M_i)$. The long exact sequence for relative homology gives
\[ H_1(M_i) \cong H_1(M_i, \partial M_i) \oplus H_1(\partial M_i) / \langle p_i \alpha_i + q_i \beta_i \rangle \]
In $H_1(M_3, \partial M_3)$, the relation $p_i \alpha_i + q_i \beta_i = 0$ translates to $p_i f_i + q_i b_i = q_i b_i = 0$. It follows that
\[ H_1(M_3, \partial M_3) \cong H_1(M_1, \partial M_1) \oplus H_1(M_2, \partial M_2) \oplus \langle b_1 \rangle / (q_1 b_1 = q_2 b_1 = 0) \]

Thus for each spin$^c$-structure on $M_1$ and each spin$^c$-structure on $M_2$, there are $\gcd(q_1, q_2)$ spin$^c$-structures on $M_3$. Note that $q_2$ may be 0; in this case we use the convention that $\gcd(q, 0) = q$. The assumption that $\CFD(M_1,\alpha_1,\beta_1)$ contains only standard unstable chains implies that $q_1 \neq 0$.

Recall that the rational longitude of a rational homology solid torus can be read off of the bordered invariants. More precisely, by Proposition \ref{prop:rational_longitude}, the curves in $\partial M_i$ which are nullhomologous in $M_i$ are integer multiples of 
\[ \chi_\circ(\CFD(M_i, \alpha_i, \beta_i; \spinc_i)) \alpha_i + \chi_\bullet(\CFD(M_i, \alpha_i, \beta_i; \spinc_i)) \beta_i,\]
where $\spinc_i$ is any spin$^c$-structure on $M_i$. Thus $q_i = \chi_\bullet(\CFD(M_i, \alpha_i, \beta_i; \spinc_i))$. By assumption $\CFD(M_i, \alpha_i, \beta_i; \spinc_i)$ is a loop for each $\spinc_i$. It was observed above that given two loops $\lp_1$ and $\lp_2$, with $\lp_1$ consisting only of standard unstable chains, $\me(\lp_1, \lp_2)$ is a collection of $\gcd(\chi_\bullet(\lp_1), \chi_\bullet(\lp_2))$ loops. Thus for each loop in $\CFD(M_1,\alpha_1,\beta_1)$ and for each loop in $\CFD(M_2,\alpha_2,\beta_2)$, there are $\gcd(q_1, q_2)$ loops in $\CFD(\merge_{1,2})$. It follows that $\CFD(\merge_{1,2})$ has exactly one loop for each spin$^c$-structure on $M_3$ (where $\merge_{1,2}\cong (M_3,\alpha_3,\beta_3)$ as bordered manifolds).
\end{proof}

\begin{figure}
\begin{tikzpicture}[scale = 1.5,>=latex]

\scriptsize
\draw[->] (1.1, 4.5) to node[above]{$d_{\subL_1}$} (1.9, 4.5);
\draw[->] (2.1, 4.5) to node[above]{$d_{\subL_2}$} (2.9, 4.5);
\draw[->] (3.1, 4.5) to node[above]{$b_{\subL_3}$} (3.9, 4.5);
\draw[->] (4.1, 4.5) to node[above]{$c_{\subL_4}$} (4.9, 4.5);
\draw[->] (5.1, 4.5) to node[above]{$a_{\subL_5}$} (5.9, 4.5);
\draw[->] (6.1, 4.5) to node[above]{$d_{\subL_6}$} (6.9, 4.5);

\draw[->] (.5, 3.9) to node[left]{$d_{k_1}$} (.5, 3.1);
\draw[->] (.5, 2.9) to node[left]{$d_{k_2}$} (.5, 2.1);
\draw[->] (.5, 1.9) to node[left]{$d_{k_3}$} (.5, 1.1);
\draw[->] (.5, .9) to node[left]{$d_{k_4}$} (.5, 0.1);

\draw[dotted] (1,0) -- (7,0);
\draw[dotted] (1,1) -- (7,1);
\draw[dotted] (1,2) -- (7,2);
\draw[dotted] (1,3) -- (7,3);
\draw[dotted] (1,4) -- (7,4);

\draw[dotted] (1,0) -- (1,4);
\draw[dotted] (2,0) -- (2,4);
\draw[dotted] (3,0) -- (3,4);
\draw[dotted] (4,0) -- (4,4);
\draw[dotted] (5,0) -- (5,4);
\draw[dotted] (6,0) -- (6,4);
\draw[dotted] (7,0) -- (7,4);

\draw[->] (1.1, 3.9) to node[above right = -2pt]{$d_{k_1+\subL_1}$} (1.9,3.1);
\draw[->] (2.1, 3.9) to node[above right = -2pt]{$d_{k_1+\subL_2}$} (2.9,3.1);
\draw[->, bend right = 45] (3.1, 3.9) to node[below = -2pt]{$b_{\subL_3}$} (3.9,3.9);
\draw[->] (4.1, 3.1) to node[below right = -2pt]{$c_{\subL_4-k_1}$} (4.9,3.9);
\draw[->, bend left = 45] (5.1, 3.1) to node[above = -2pt]{$a_{\subL_5}$} (5.9,3.1);
\draw[->] (6.1, 3.9) to node[above right = -2pt]{$d_{k_1+\subL_6}$} (6.9,3.1);

\draw[->] (1.1, 2.9) to node[above right = -2pt]{$d_{k_1+\subL_1}$} (1.9,2.1);
\draw[->] (2.1, 2.9) to node[above right = -2pt]{$d_{k_1+\subL_2}$} (2.9,2.1);
\draw[->, bend right = 45] (3.1, 2.9) to node[below = -2pt]{$b_{\subL_3}$} (3.9,2.9);
\draw[->] (4.1, 2.1) to node[below right = -2pt]{$c_{\subL_4-k_1}$} (4.9,2.9);
\draw[->, bend left = 45] (5.1, 2.1) to node[above = -2pt]{$a_{\subL_5}$} (5.9,2.1);
\draw[->] (6.1, 2.9) to node[above right = -2pt]{$d_{k_1+\subL_6}$} (6.9,2.1);

\draw[->] (1.1, 1.9) to node[above right = -2pt]{$d_{k_1+\subL_1}$} (1.9,1.1);
\draw[->] (2.1, 1.9) to node[above right = -2pt]{$d_{k_1+\subL_2}$} (2.9,1.1);
\draw[->, bend right = 45] (3.1, 1.9) to node[below = -2pt]{$b_{\subL_3}$} (3.9,1.9);
\draw[->] (4.1, 1.1) to node[below right = -2pt]{$c_{\subL_4-k_1}$} (4.9,1.9);
\draw[->, bend left = 45] (5.1, 1.1) to node[above = -2pt]{$a_{\subL_5}$} (5.9,1.1);
\draw[->] (6.1, 1.9) to node[above right = -2pt]{$d_{k_1+\subL_6}$} (6.9,1.1);

\draw[->] (1.1, 0.9) to node[above right = -2pt]{$d_{k_1+\subL_1}$} (1.9,0.1);
\draw[->] (2.1, 0.9) to node[above right = -2pt]{$d_{k_1+\subL_2}$} (2.9,0.1);
\draw[->, bend right = 45] (3.1, 0.9) to node[below = -2pt]{$b_{\subL_3}$} (3.9,0.9);
\draw[->] (4.1, 0.1) to node[below right = -2pt]{$c_{\subL_4-k_1}$} (4.9,0.9);
\draw[->, bend left = 45] (5.1, 0.1) to node[above = -2pt]{$a_{\subL_5}$} (5.9,0.1);
\draw[->] (6.1, 0.9) to node[above right = -2pt]{$d_{k_1+\subL_6}$} (6.9,0.1);

\end{tikzpicture}

\caption{Computing $\me(\lp_1, \lp_2)$ for two loops $\lp_1 = (d_{k_1} d_{k_2} d_{k_3} d_{k_4})$ and ${\lp_2 = (d_{\subL_1} d_{\subL_2} b_{\subL_3} c_{\subL_4} a_{\subL_5} d_{\subL_6})}$. }
\label{fig:merge_example}

\end{figure}

\subsection{Bordered invariants of graph manifolds}\label{sec:graph_mfd_algorithm} With this description of $\twist$, $\extend$, and $\merge$ on bordered invariants in hand, we may now return focus to graph manifolds. The first author has described (and implemented) an algorithm for computing the (bordered) Heegaard Floer invariants of graph manifolds \cite{Hanselman2013}; we will now outline a version of this algorithm for graph manifolds with a single boundary component and adapt it to the loops setup. Recall that, given a graph $\Gamma$ with associated bordered manifold $(M_\Gamma, \alpha,\beta)$ (as described in Section \ref{sub:prelim}), we write $\CFD(\Gamma)$ for $\CFD(M_\Gamma, \alpha,\beta)$.

In order to compute $\CFD(\Gamma)$, we inductively build up the plumbing tree $\Gamma$ using the three plumbing tree operations $\twist^{\pm 1}$, $\extend$, and $\merge$ depicted in Figure \ref{fig:tree_operations}, starting from the plumbing tree 
\begin{center}
\begin{tikzpicture}
\node (a) at (0,0) {$\bullet$};
\node[left = 5pt] at (a) {$\Gamma_0 = $};
\node[above = 1pt] at (a) {$0$};
\draw[dashed] (0,0)--(1,0);
\end{tikzpicture}
\end{center}
Note that $M_{\Gamma_0}$ is a solid torus, with bordered structure $(M_{\Gamma_0}, \alpha, \beta)$ where $\alpha$ is a fiber of the $S^1$-bundle over $D^2$ (i.e. a longitude of the solid torus) and $\beta$ is a curve in the base surface (i.e. a meridian of the solid torus, identified by $\partial D^2\times\{\text{pt}\}$). Thus $\CFD(\Gamma_0)$ is represented by the loop $\lp_\bullet \sim (d_0)$. As we apply the operations $\twist$, $\extend$, and $\merge$, we keep track of the bordered invariants of the relevant manifolds. Let $\Gamma_1$ and $\Gamma_2$ be single boundary plumbing trees. As shown in the previous section:
\begin{align*}
\CFD( \twist ( \Gamma_1) ) &\cong \Ts \boxtimes\CFD( \Gamma_1 )\\
\CFD( \extend( \Gamma_1) ) &\cong \Ts\boxtimes\Td\boxtimes\Ts\boxtimes\CFD( \Gamma_1 )\\
\CFD( \merge(\Gamma_1, \Gamma_2))&\cong\CFDAA(\Gamma_\merge)\boxtimes\left(\CFD(\Gamma_1),\CFD(\Gamma_2)\right)
\end{align*} 

Specializing Lemmas \ref{lem:twist_extend_loop_type} and \ref{lem:merge_loop_type} to graph manifold leads to the following observation: 
\begin{lem}
Let $\Gamma$ be a single boundary plumbing tree constructed as described above from copies of $\Gamma_0$ using the operations $\twist^\pm$, $\extend$, and $\merge$. $\CFD(\Gamma)$ has simple loop-type as long as each time the operation $\merge$ is applied, the two input plumbing trees have simple loop-type bordered invariants with only unstable chains in standard notation.
\end{lem}
\begin{proof}
This follows from Lemma \ref{lem:twist_extend_loop_type}, which says that $\twist^\pm$ and $\extend$ take simple loop-type manifolds to simple loop-type manifolds, and Lemma \ref{lem:merge_loop_type}, which says that $\merge$ takes two simple loop-type manifolds to a simple loop-type manifold provided the simple loops corresponding to both inputs consist only of unstable chains.
\end{proof}

With this observation, we can describe a large family of simple loop-type manifolds.

For each vertex $v$ of a plumbing tree $\Gamma$, let $w(v)$ denote the Euler weight associated to $v$ and let $n_+(v)$ and $n_-(v)$ denote the number of neighboring vertices $v'$ for which $w(v') \ge 0$ and $w(v')\le 0$, respectively. We will say that $v$ is a \emph{bad vertex} if $-n_-(v) < w(v) < n_+(v)$, and otherwise $v$ is a \emph{good vertex} (this should be viewed as a generalization of the notion of bad vertices defined for negative definite plumbing trees in \cite{OSz-plumbing}).

\begin{prop}\label{prop:good_vertices}
Let $\Gamma$ be a plumbing tree with a single boundary edge at the vertex $v_0$, and suppose that every vertex other than $v_0$ is good. Then $\CFD(\Gamma)$ is a collection of loops consisting only of standard unstable chains; up to reversal we can assume these unstable chains are of type $d_k$. Moreover, if $w(v_0)$ is (strictly) greater than $n_+(v_0)$ then these unstable chains all have subscripts (strictly) greater than 0. If $w(v_0)$ is (strictly) less than $n_-(v_0)$ then the unstable chains all have subscript (strictly) less than 0.
\end{prop}
\begin{proof}
We proceed by induction on the number of vertices of $\Gamma$. The base case, where $\Gamma$ has only one vertex $v_0$, is trivial; $\CFD(\Gamma)$ in this case is given by the loop $(d_{w(v_0)})$.

For the inductive step, first suppose the boundary vertex $v_0$ of $\Gamma$ has valence two (including the boundary edge). That is, $\Gamma$ has the form
\begin{center}
\begin{tikzpicture}
\node (a) at (0, 0) {$\bullet$};
\node at (-2, 0) {$\Gamma = $};
\node at (1.5, 0) {$\Gamma'$};
\node[above] at (a) {$w_0$};

\draw[dashed] (-1,0) -- (0, 0);
\draw (0,0) -- (1,0);
\draw (1.5,0) circle (5 mm);

\end{tikzpicture}
\end{center}
where $w_0 = w(v_0)$. Let $v_1$ be the boundary vertex of $\Gamma'$, and let $w_1$ be the corresponding weight. Let $n_\pm(v_1)$ denote the counts of neighboring vertices defined above for $v_1$ as a vertex in $\Gamma$, and let $n'_\pm(v_1)$ denote these counts for $v_1$ as a vertex in $\Gamma'$ (i.e., ignoring the vertex $v_0$). Since $v_1$ is a good vertex, we have one of the following two cases:
\begin{enumerate}
\item $w_1 \le -n_-(v_1) \le -n'_-(v_1) \le 0$
\item $w_1 \ge n_+(v_1)\ge n'_+(v_1) \ge 0$
\end{enumerate}
Note that $\Gamma = \twist^{w_0}( \extend( \Gamma' ) )$, and so $\CFD(\Gamma)$ is given by $\tw^{w_0}(\ex(\CFD(\Gamma')))$. Furthermore, we assume by induction that the proposition holds for $\Gamma'$.

In case (1), we have that $\CFD(\Gamma')$ consists of $d_k$ segments with nonpositive subscripts. Thus in dual notation $\ex(\CFD(\Gamma'))$ is a loop consisting of $d^*_k$ segments with nonnegative subscripts, and in standard notation $\ex(\CFD(\Gamma')$ consists of $d_k$ segments with nonnegative subscripts. Moreover if $w_0 \le 0$, then $w_1$ is strictly less than $-n'_-(v_1) = -n_-(v_1) + 1$, so the $d_k$ segments in $\CFD(\Gamma)$ have strictly negative subscripts, the $d^*_k$ segments in $\ex(\CFD(\Gamma'))$ have strictly positive subscripts, and in standard notation $\ex(\CFD(\Gamma'))$ consists only of $d_0$ and $d_1$ segments. It follows that $\CFD(\Gamma)$ is a collection of loops consisting of $d_k$ segments with subscripts satisfying $k < 0$ if $w_0 < -n_-(v_0) = -1$, $k\le0$ if $w_0 = n_-(v_0) = -1$, $k \ge 0$ if $w_0 = n_+(v_0) \ge 0$ and $k > 0$ if $w_0 > n_+(v_0) \ge 0$. Thus the proposition holds for $\Gamma$.

In case (2), $\CFD(\Gamma')$ consists of $d_k$ segments with nonnegative subscripts. Thus $\ex(\CFD(\Gamma'))$ consists of $d^*_k$ with nonpositive subscripts in dual notation and of $d_k$ with nonpositive subscripts in standard notation. Moreover, if $w_0 \ge 0$ then $n_+(v_1) > n'_+(v_1)$, so the $d_k$ segments in $\CFD(\Gamma')$ have strictly positive subscripts, the $d^*_k$ segments in $\ex(\CFD(\Gamma'))$ have strictly negative subscripts, and the $d_k$ segments in $\ex(\CFD(\Gamma'))$ have subscripts in $\{0, -1\}$.  It follows that $\CFD(\Gamma)$ is a collection of loops consisting of $d_k$ segments with subscripts satisfying $k < 0$ if $w_0 < -n_-(v_0) \le 0$, $k\le0$ if $w_0 = n_-(v_0) \le 0$, $k \ge 0$ if $w_0 = n_+(v_0) = 1$ and $k > 0$ if $w_0 > n_+(v_0) = 1$. Thus the proposition holds for $\Gamma$.

Now suppose that $v_0$ has valence higher than two. This means that $\Gamma$ can be obtained as $\merge(\Gamma_1, \Gamma_2)$, where $\Gamma_1$ and $\Gamma_2$ have fewer vertices than $\Gamma$. By induction, the proposition holds for $\Gamma_1$ and $\Gamma_2$, and so $\CFD(\Gamma_1)$ and $\CFD(\Gamma_2)$ may be represented by a collection of loops consisting only of standard type $d_k$ chains. By Proposition \ref{prop:merge_unstable_chains}, merging two (collections of) loops with only $d_k$ chains produces a new collection of loops with only $d_k$ chains. Moreover, each chain in the $\CFD(\Gamma)$ is of the form $d_{k+\subL}$ for some chains $d_k$ in $\CFD(\Gamma_1)$ and $d_\subL$ in $\CFD(\Gamma_2)$, so the maximum (resp. minimum) subscript in $\CFD(\Gamma)$ is the sum of the maximum (resp. minimum) subscripts in $\CFD(\Gamma_1)$ and $\CFD(\Gamma_2)$.

For $i\in\{0,1\}$, let $v_i$ be the boundary vertex of $\Gamma_i$ and let $w_i = w(v_i)$ be the corresponding weight. We can choose any values for $w_i$ provided $w_1 + w_2 = w_0$. If $w_0 \ge n_+(v_0) = n_+(v_1) + n_+(v_2)$, then we can choose $w_1 = n_+(v_1)$ and $w_2 \ge n_+(v_2)$. By the inductive assumption, $\CFD(\Gamma_1)$ and $\CFD(\Gamma_2)$ both have only nonnegative subscripts, so the same is true of $\CFD(\Gamma_2)$. Furthermore if $w_0 > 0$ then $w_2 > n_+(v_2)$, so $\CFD(\Gamma_2)$ has only strictly positive subscripts and the same follows for $\CFD(\Gamma)$. Similarly, if $w_0$ is (strictly) less than $n_-(v_0) = n_-(v_1) + n_-(v_2)$, we can choose $w_1 = n_-(v_1)$ and $w_2$ (strictly) less than $n_-(v_2)$ and conclude that $\CFD(\Gamma)$ only has subscripts (strictly) less than 0.
\end{proof}

Note that if $\Gamma$ is a plumbing tree with at most one bad vertex, we can compute $\HFhat(M_\Gamma)$ by taking the dual filling of $\CFD(\Gamma')$, where $\Gamma'$ is obtained from $\Gamma$ by adding a boundary edge to the bad vertex, or to any vertex if there are no bad vertices. $\CFD(\Gamma')$, which has simple loop-type by Proposition \ref{prop:good_vertices}, can be computed using the operations $\tw^{\pm 1}, \ex$, and $\me$. In particular, we have the following generalization of \cite[Lemma 2.6]{OSz-plumbing}:

\begin{cor}
If $\Gamma$ is a closed plumbing tree with no bad vertices then the manifold $M_\Gamma$ is an $L$-space. \end{cor}
\begin{proof}
Add a boundary edge to any vertex of $\Gamma$ to produce a single boundary plumbing tree $\Gamma'$. By Proposition \ref{prop:good_vertices}, $\CFD(\Gamma')$ is a collection of loops consisting only of type $d_k$ segments with either all subscripts greater than or equal to zero or all subscripts less than or equal to zero. By Observation \ref{obs:dual_stable_chains}, it follows that in dual notation the loops representing $\CFD(\Gamma')$ have no stable chains. $M_\Gamma$ is obtained from $M_{\Gamma'}$ by dual filling, and dual filling is an $L$-space if there are no stable chains in dual notation. \end{proof}

Recall that in the context of Theorem \ref{thm:gluing} it is important to distinguish solid torus-like manifolds from other manifolds of simple loop-type. Toward that end, we check the following:

\begin{prop}\label{prop:solid_torus_graph_manifolds}
Let $\Gamma$ be a plumbing tree with a single boundary edge at the vertex $v_0$, and suppose that every vertex other than $v_0$ is good. Then $M_\Gamma$ is solid torus-like if and only if it is a solid torus.
\end{prop}
\begin{proof}
We proceed by induction on the size of $\Gamma$ as in the proof of Proposition \ref{prop:good_vertices}. Applying Dehn twists does not change whether or not a manifold is a solid torus, nor does it change whether or not the corresponding bordered invariants are solid torus-like. It follows that the proposition holds for $\twist^\pm(\Gamma)$ and $\extend(\Gamma)$ if it holds for $\Gamma$. We only need to check that it holds for $\merge(\Gamma_1, \Gamma_2)$, assuming by induction that it holds for $\Gamma_1$ and $\Gamma_2$. By Proposition \ref{prop:good_vertices}, we may assume that the loops representing $\CFD(\Gamma_1)$ and $\CFD(\Gamma_2)$ consist only of $d_k$ chains.

Suppose $\CFD(\merge(\Gamma_1, \Gamma_2))$ is solid torus-like. Recall that, by the appropriate generalization of Lemma \ref{lem:solid-torus-restrictions}, the loops representing $\CFD(\merge(\Gamma_1, \Gamma_2))$ consist only of $d_k$ segments such that the maximum and minimum subscripts appearing differ by at most one. Note also that the difference between maximum and minimum subscripts is additive under $\merge$. It follows that either $\Gamma_1$ or $\Gamma_2$ must have bordered invariant consisting only of $d_n$ segments for a fixed $n\in \Z$; say $\Gamma_2$ has this property. $M_{\Gamma_2}$ is solid torus-like; by induction, $M_{\Gamma_2}$ is a solid torus, and $\CFD(\Gamma_2)$ is given by the single loop $(d_n)$. Proposition \ref{prop:merge_unstable_chains} then implies that $\CFD(\merge(\Gamma_1, \Gamma_2))$ is obtained from $\CFD(\Gamma_1)$ by adding $n$ to the subscript of each segment. In other words, $\CFD(\merge(\Gamma_1, \Gamma_2)) \cong \CFD(\twist^n(\Gamma_1))$.

This equivalence does not only hold on the level of bordered invariants. Indeed, $\merge(\Gamma_1, \Gamma_2)$ is equivalent to $\twist^n(\Gamma_1)$ up to the graph moves in \cite{Neumann}; in particular, the corresponding graph manifolds are diffeomorphic. To see this, note that $\Gamma_2$ must be equivalent to the plumbing tree
\begin{center}
\begin{tikzpicture}
\node (a) at (0,0) {$\bullet$};
\node[above = 1pt] at (a) {$n$};
\draw[dashed] (0,0)--(1,0);
\end{tikzpicture}
\end{center}
since $\CFD$ of this tree is $(d_n)$, and it is clear from Figure \ref{fig:tree_operations} that merging with this tree has the same effect as applying $\twist^n$. Since $\CFD(\twist^n(\Gamma_1))$ is solid torus like, $\CFD(\Gamma_1)$ is solid torus-like and, by the inductive hypothesis, $M_{\Gamma_1}$ is a solid torus. It follows that $M_{\merge(\Gamma_1, \Gamma_2)} = M_{\twist^n(\Gamma_1)} = \twist^n(M_{\Gamma_1})$ is a solid torus.
\end{proof}

\subsection{An explicit example: The Poincar\'e homology sphere} As an example of the algorithm and loop operations described above, we will compute $\HFhat$ of the Poincare homology sphere using the plumbing tree
\begin{center}
\begin{tikzpicture}

\node (a) at (0,1) {$\bullet$};
\node (b) at (-1,0) {$\bullet$};
\node (c) at (0,2) {$\bullet$};
\node (d) at (1,0) {$\bullet$};
\node at (-1.5, 1) {$\Gamma = $};

\draw (0,1) -- (-1,0);
\draw (0,2) -- (0,1) -- (1,0);

\node[right] at (a) {$-1$};
\node[left] at (b) {$-3$};
\node[right] at (c) {$-2$};
\node[right] at (d) {$-5$};

\end{tikzpicture}
\end{center}

We start with the loop $(d_0)$ representing $\CFD(\Gamma_0)$. For this example, by abuse of notation, we will equate the bordered invariants with their loop representatives, thus $\CFD(\Gamma_0)= (d_0)$. We use the twist and extend operations to compute invariants for the following plumbing trees:
\begin{center}
\begin{tikzpicture}
\node (a) at (0,0) {$\bullet$};
\node[left = 5pt] at (a) {$\Gamma_1 = $};
\node[above = 1pt] at (a) {$-2$};
\node (b) at (1,0) {$\bullet$};
\node[above = 1pt] at (b) {$0$};
\draw (0,0)--(1,0);
\draw[dashed] (1,0)--(2,0);
\end{tikzpicture}
\qquad
\begin{tikzpicture}
\node (a) at (0,0) {$\bullet$};
\node[left = 5pt] at (a) {$\Gamma_2 = $};
\node[above = 1pt] at (a) {$-3$};
\node (b) at (1,0) {$\bullet$};
\node[above = 1pt] at (b) {$0$};
\draw (0,0)--(1,0);
\draw[dashed] (1,0)--(2,0);\end{tikzpicture}
\qquad 
\begin{tikzpicture}
\node (a) at (0,0) {$\bullet$};
\node[left = 5pt] at (a) {$\Gamma_3 = $};
\node[above = 1pt] at (a) {$-5$};
\node (b) at (1,0) {$\bullet$};
\node[above = 1pt] at (b) {$0$};
\draw (0,0)--(1,0);
\draw[dashed] (1,0)--(2,0);\end{tikzpicture}
\end{center}
Note that $\Gamma_1 = \extend(\twist^{-2}(\Gamma_0))$, so
$$\CFD(\Gamma_1) = \ex(\tw^{-2}( (d_0) )) = \ex( (d_{-2}) ) = (d^*_2) \sim (d_1 d_0).$$
Similarly, we find that
$$\CFD(\Gamma_2) = \ex(\tw^{-3}( (d_0) )) = \ex( (d_{-3}) ) = (d^*_3) \sim (d_1 d_0 d_0),$$
$$\CFD(\Gamma_3) = \ex(\tw^{-5}( (d_0) )) = \ex( (d_{-5}) ) = (d^*_5) \sim (d_1 d_0 d_0 d_0 d_0).$$
Now let
\begin{center}
\begin{tikzpicture}
\node (a) at (0,0) {$\bullet$};
\node (b) at (-1,-1) {$\bullet$};
\node (c) at (1,-1) {$\bullet$};

\node at (-1.5,0) {$\Gamma_4 = $};
\node[right = 3pt] at (a) {$0$};
\node[below = 1pt] at (b) {$-2$};
\node[below = 1pt] at (c) {$-3$};
\draw (-1,-1)--(0,0)--(1,-1);
\draw[dashed] (0,0)--(0,1);
\end{tikzpicture}
\qquad
\begin{tikzpicture}
\node (a) at (0,0) {$\bullet$};
\node (b) at (-1,-1) {$\bullet$};
\node (c) at (0,-1) {$\bullet$};
\node (d) at (1,-1) {$\bullet$};

\node at (-1.5,0) {$\Gamma_5 = $};
\node[right = 3pt] at (a) {$0$};
\node[below = 1pt] at (b) {$-2$};
\node[below = 1pt] at (c) {$-3$};
\node[below = 1pt] at (d) {$-5$};
\draw (-1,-1)--(0,0)--(1,-1);
\draw (0,0) -- (0,-1);
\draw[dashed] (0,0)--(0,1);
\end{tikzpicture}
\qquad 
\begin{tikzpicture}
\node (a) at (0,0) {$\bullet$};
\node (b) at (-1,-1) {$\bullet$};
\node (c) at (0,-1) {$\bullet$};
\node (d) at (1,-1) {$\bullet$};

\node at (-1.5,0) {$\Gamma_6 = $};
\node[right = 3pt] at (a) {$-1$};
\node[below = 1pt] at (b) {$-2$};
\node[below = 1pt] at (c) {$-3$};
\node[below = 1pt] at (d) {$-5$};
\draw (-1,-1)--(0,0)--(1,-1);
\draw (0,0) -- (0,-1);
\draw[dashed] (0,0)--(0,1);
\end{tikzpicture}
\end{center}

We have that $\Gamma_4 = \merge(\Gamma_1, \Gamma_2)$, $\Gamma_5= \merge(\Gamma_4, \Gamma_3)$, and $\Gamma_6 = \twist^{-1}(\Gamma_5)$. Reading diagonally from the first grid below, we see that $\CFD(\Gamma_4) = \me\left( (d_1 d_0), (d_1 d_0 d_0) \right) = (d_2 d_0 d_1 d_1 d_1 d_0)$.
$$\begin{array}{c | ccc}
& d_1 & d_0 & d_0 \\
\hline
d_1 & d_2 & d_1 & d_1 \\
d_0 & d_1 & d_0 & d_0 
\end{array}
\hspace{1 in}
\begin{array}{c | ccccc}
& d_1 & d_0 & d_0 & d_0 & d_0\\
\hline
d_2 & d_3 & d_2 & d_2 & d_2 & d_2\\
d_0 & d_1 & d_0 & d_0 & d_0 & d_0\\
d_1 & d_2 & d_1 & d_1 & d_1 & d_1\\
d_1 & d_2 & d_1 & d_1 & d_1 & d_1\\
d_1 & d_2 & d_1 & d_1 & d_1 & d_1\\
d_0 & d_1 & d_0 & d_0 & d_0 & d_0
\end{array}$$
The second grid tells us that
$$\CFD(\Gamma_5) = (d_3 d_0 d_1 d_1 d_1 d_1 d_2 d_0 d_1 d_1 d_2 d_0 d_2 d_0 d_1 d_2 d_1 d_0 d_2 d_0 d_2 d_1 d_1 d_0 d_2 d_1 d_1 d_1 d_1 d_0).$$
Applying the operation $\tw^{-1}$, we find that
$$\CFD(\Gamma_6) = (d_2 d_{-1} d_0 d_0 d_0 d_0 d_1 d_{-1} d_0 d_0 d_1 d_{-1} d_1 d_{-1} d_0 d_1 d_0 d_{-1} d_1 d_{-1} d_1 d_0 d_0 d_{-1} d_1 d_0 d_0 d_0 d_0 d_{-1}).$$
Finally, to compute $\CFhat$ of the closed manifold $M_\Gamma$ we must fill in the boundary of $(M_{\Gamma_6}, \alpha, \beta)$ with a $D^2 \times S^1$ such that the meridian $\partial D^2\times \{\text{pt}\}$ glues to $\beta$ and the longitude $\{\text{pt}\}\times S^1$ glues to $\alpha$. In other words, $M_\Gamma$ is the $0$-filling of $(M_{\Gamma_6}, \alpha, \beta)$, so $\CFhat(M_\Gamma)$ is obtained from $\CFD(\Gamma_6)$ by tensoring with $\lp_\circ^A$. To do this, we first write $\CFD(\Gamma_6)$ in dual notation (using the procedure described in Section \ref{sec:notations_for_loops}):
$$\CFD(\Gamma_6) = (d^*_0 b^*_1 a^*_5 b^*_1 a^*_3 b^*_1 a^*_1 b^*_1 a^*_2 b^*_2 a^*_1 b^*_1 a^*_1 b^*_3 a^*_1 b^*_5 a^*_1).$$
Tensoring with $\lp_\circ$ produces one generator for each segment in dual notation, but also one differential for each type $a^*$ segment. Since $\CFD(\Gamma_6)$ has 17 dual segments and 8 $a^*$ segments, all but one generator in $\lp_\circ^A \boxtimes \CFD(\Gamma_6) \cong \CFhat(M_\Gamma)$ cancels in homology. Thus $\HFhat(M_\Gamma)$ has dimension 1, as is now well known (this was first calculated in \cite[Section 3.2]{OSz2004-b}).

\begin{remark}
While computer computation is not our primary motivation, it is worth pointing out that using loop calculus as described in this section instead of taking box tensor products of modules, bimodules, and trimodules greatly improves the efficiency of the algorithm in \cite{Hanselman2013} for rational homology sphere graph manifolds. This is illustrated by the fact that the example above can easily be done by hand, while computing the relevant tensor products would be tedious without a computer. The largest computation in \cite{Hanselman2013} (for which  the dimension of $\HFhat$ is $213312$) took roughly 12 hours; when the computer implementation is adapted to use loop calculus, the same computation runs in 30 seconds. The caveat is that the purely loop calculus algorithm does not work for all rational homology sphere graph manifolds, since there may be graph manifold rational homology solid tori that are not of loop-type, but in practice it works for most examples. 
\end{remark}

\section{L-spaces and non-left-orderability}\label{sec:proofs}

We conclude by proving the remaining results quoted in the introduction: Theorem \ref{thm:Seifert-simple}, Theorem \ref{thm:detection-via-TIBOKB}, and finally, Theorem \ref{thm:main}.

To begin, we observe that Seifert fibered rational homology tori have simple loop-type bordered invariants. Any Seifert fibered space over $S^2$ can be given a star shaped plumbing tree in which all vertices of valence one or two have weight at most $-2$; in particular, there are no bad vertices except for the central vertex. Adding a boundary edge to the central vertex corresponds to removing a neighborhood of a regular fiber, creating a Seifert fibered space over $D^2$. By Proposition \ref{prop:good_vertices}, such a plumbing tree has simple loop-type $\CFD$. Moreover, by Proposition \ref{prop:solid_torus_graph_manifolds}, such a plumbing tree is non-solid torus-like unless the corresponding manifold is a solid torus. 

The only other option to consider is a Seifert fibered space over the M\"obius band, since a Seifert fibered space over any other base orbifold has $b_1 > 0$. Such a manifold can be obtained from a Seifert fibered space over $D^2$ by removing a neighborhood of a regular fiber and gluing in the Euler number 0 bundle over the M\"obius band, fiber to fiber and base to base. On the plumbing tree, this corresponds to adding
\begin{center}
\begin{tikzpicture}
\node (b) at (1,0) {$\bullet$};
\node (c) at (2,.6) {$\bullet$};
\node (d) at (2,-.6) {$\bullet$};

\node[above = 2pt] at (b) {$0$};
\node[right] at (c) {$2$};
\node[right] at (d) {$-2$};

\draw (2, .6) -- (1, 0) -- (2,-.6);
\draw[dashed] (0,0) -- (1,0);
\end{tikzpicture}
\end{center}
to the central vertex, or equivalently to merging with the plumbing tree
\begin{center}
\begin{tikzpicture}
\node (a) at (0,0) {$\bullet$};
\node (b) at (1,0) {$\bullet$};
\node (c) at (2,.6) {$\bullet$};
\node (d) at (2,-.6) {$\bullet$};

\node at (-2,0) {$\Gamma_N =$};

\node[above = 2pt] at (a) {$0$};
\node[above = 2pt] at (b) {$0$};
\node[right] at (c) {$2$};
\node[right] at (d) {$-2$};

\draw (2, .6) -- (1, 0) -- (2,-.6);
\draw (0,0) -- (1,0);
\draw[dashed] (-1,0) -- (0,0);
\end{tikzpicture}
\end{center}

\begin{prop}\label{lem:CFD_N2}
$\CFD$ for the plumbing tree $\Gamma_N$ above consists of two loops, one $(a_1 b_1) \sim (d^*_1 d^*_{-1})$ and the other $(e^* e^*)$.
\end{prop}
\begin{proof}
Using the loop operations described in the preceding section, this is a simple computation. Note that $\Gamma_N = \extend( \merge( \extend( \twist^2( \Gamma_0) ), \extend( \twist^{-2}( \Gamma_0 ) ) ) )$. Thus $\CFD(\Gamma_N)$ is given by
\begin{align*}
\ex( \me( \ex( \tw^2( (d_0) ) ), \ex( \tw^{-2}( (d_0) ) ) ) ) 
&= \ex( \me( \ex( (d_2) ), \ex( (d_{-2}) ) ) ) \\
& = \ex( \me( (d_{-1} d_0), (d_1 d_0) ) ) \\
& = \ex(  (d_0 d_0) \amalg (d_1 d_{-1}) ) \\
& = (d^*_0 d^*_0) \amalg (d^*_1 d^*_{-1}) \qedhere\end{align*}
\end{proof}
\begin{remark}Note that this result was first established by a direct calculation by Boyer, Gordon and Watson \cite{BGW2013}; this calculation is greatly simplified by appealing to loop calculus. \end{remark}

We now complete the proof that Seifert fibered rational homology tori have simple loop-type.

\begin{proof}[Proof of Theorem \ref{thm:Seifert-simple}]
As observed above, the case of Seifert fibered manifolds over $D^2$ is a special case of Proposition \ref{prop:good_vertices}. For a Seifert fibered manifold over the M\"obius band, a plumbing tree $\Gamma$ is given by $\merge(\Gamma', \Gamma_N)$, where $\Gamma'$ is a star shaped plumbing tree for a Seifert fibered manifold over $D^2$. By Proposition \ref{prop:good_vertices}, the loops in $\CFD(\Gamma')$ contain only unstable chains in standard notation. By Proposition \ref{prop:merge_unstable_chains} and Proposition \ref{lem:CFD_N2}, we find that $\CFD(\Gamma)$ is a collection of disjoint copies of $\CFD(\Gamma_N)$, and in particular is a collection of simple loops. Moreover, by Lemma \ref{lem:merge_loop_type} there is one loop for each spin$^c$-structure, and so $\CFD(\Gamma)$ is of simple loop-type. Finally, with the foregoing in place (in particular, Proposition \ref{prop:solid_torus_graph_manifolds}) one checks that the only solid torus-like manifold in this class is the solid torus itself. 
\end{proof}

We are now in a position to assemble the pieces and give the proof of Theorem \ref{thm:main}. A key observation is the following consequence of our gluing theorem, which provides some alternate characterizations of the set of strict L-space slopes of a given simple loop-type manifold. 

\begin{thm}\label{thm:characterize}
Let $M$ be a simple loop-type manifold that is not solid torus-like. The following are equivalent:
\begin{enumerate}
\item[(i)] $\gamma$ is a strict L-space slope for $M$, that is, $\gamma\in\mathcal{L}_M^\circ$;
\item[(ii)] $M\cup_h M'$ is an L-space, where $h(\gamma)=\lambda$ and $M'$ is a non solid torus-like, simple, loop-type manifold with rational longitude $\lambda$ for which $\mathcal{L}_{M'}$ includes every slope other than $\lambda$;
\item[(iii)]  $M\cup_h N$ is an L-space, where $h(\gamma)=\lambda$ and $N$ is the twisted $I$-bundle over the Klein bottle with rational longitude $\lambda$. 
\end{enumerate}
\end{thm}

\begin{proof}
Given $\gamma\in\mathcal{L}_M^\circ$ and a simple, loop-type manifold $M'$ for which $\mathcal{L}_{M'}$ includes every slope other than $\lambda$, Theorem \ref{thm:gluing_manifolds} ensures that $M\cup_h M'$ is an L-space. Indeed, for any $\gamma'\ne\gamma$ we have that $\lambda\ne h(\gamma')\in\mathcal{L}_{M'}^\circ$. This proves that (i) implies (ii). 

To see that (ii) implies (iii), is suffices to observe that $N(\gamma)$ is an L-space for all $\gamma$ other than the rational longitude; this is indeed the case as observed, for example, in \cite[Proposition 5]{BGW2013} (alternatively, this fact is an exercise in loop calculus).   

Finally, suppose that  $M\cup_h N$ is an L-space and consider the slope $\gamma$ in $\partial M$ determined by $h^{-1}(\lambda)$. Since $N(\lambda)$ is not an L-space, by Theorem \ref{thm:gluing_manifolds} it must be that $\gamma\in\mathcal{L}_M^\circ$ as required, so that (iii) implies (i). \end{proof}

Boyer and Clay consider a collection of Seifert fibered rational homology solid tori $\{N_t\}$ for integers $t>1$. In this collection $N_2=N$, the twisted $I$-bundle over the Klein bottle. More generally, the $N_t$ are examples of Heegaard Floer homology solid tori considered in \cite{Watson}. These manifolds are easily described by the plumbing tree:
\[
\begin{tikzpicture}
\node (a) at (0,0) {$\bullet$};
\node (b) at (1,0) {$\bullet$};
\node (c) at (2,.6) {$\bullet$};
\node (d) at (2,-.6) {$\bullet$};

\node[above = 2pt] at (a) {$0$};
\node[above = 2pt] at (b) {$0$};
\node[right] at (c) {$t$};
\node[right] at (d) {$-t$};

\draw (2, .6) -- (1, 0) -- (2,-.6);
\draw (0,0) -- (1,0);
\draw[dashed] (-1,0) -- (0,0);
\end{tikzpicture}
\]
As observed in \cite{Watson}, translated into loop notation, the invariant $\CFD(N_t,\varphi,\lambda)$ is simple described by \[(d_0^*)^t\amalg\left[\amalg_{i=1}^{t-1} (d^*_i d^*_{i-t})\right]\]
This calculation is similar to that of Proposition \ref{lem:CFD_N2}; since these are members of a much larger class of manifolds that are interesting in their own right, we leave a detailed study of their properties to \cite{Watson}. Note, in particular, that $N_t(\gamma)$ is an L-space for all $\gamma$ other than the rational longitude; see Section \ref{sec:char} but compare also \cite{BGW2013}. Therefore, the $N_t$ satisfy the conditions of (ii) in Theorem \ref{thm:characterize} and we have the following:

\begin{cor}\label{cor:characterize}
Let $M$ be a simple loop-type manifold that is not solid torus-like. The following are equivalent:
\begin{enumerate}
\item[(i)] $\gamma$ is a strict L-space slope for $M$, that is, $\gamma\in\mathcal{L}_M^\circ$;
\item[(ii)] $M\cup_h N_t$ is an L-space, where $h(\gamma)=\lambda$ is the rational longitude, for any integer $t>1$;
\item[(iii)]  $M\cup_h N_2$ is an L-space, where $h(\gamma)=\lambda$ is the rational longitude.
\end{enumerate}
\end{cor}

Note that Theorem \ref{thm:detection-via-TIBOKB} follows from the equivalence between (i) and (iii) in Corollary \ref{cor:characterize}. This answers \cite[Question 1.8]{BC} and considerably simplifies \cite[Theorem 1.6]{BC} when  restricting to Seifert fibered rational homology solid tori. Indeed, we have shown:

\begin{thm}\label{thm:detect-new}
Suppose $M\ncong D^2\times S^1$ is a Seifert fibered rational homology solid torus. The following are equivalent: 
\item[(i)] $\gamma\in(\mathcal{L}_M^\circ)^c$;
\item[(ii)] $\gamma$ is detected by a left-order (in the sense of Boyer and Clay \cite{BC});
\item[(iii)]   $\gamma$ is detected by a taut foliation (in the sense of Boyer and Clay \cite{BC}).
\end{thm}
\begin{proof}This follows immediately from \cite[Theorem 1.6]{BC} combined with Corollary \ref{cor:characterize}.\end{proof}

\begin{proof}[The proof of Theorem \ref{thm:main}]
The equivalence between (ii) and (iii) are due to Boyer and Clay \cite{BC}. To see that (i) is equivalent to either of these we first note that, if $M$ is one of the two seifert fibered pieces in $Y$ then, according to Theorem \ref{thm:intervals}, the set of all slopes $\hat\Q$ is divided into (the restriction to $\hat \Q$ of) two disconnected intervals $\mathcal{L}_M^\circ$ and $(\mathcal{L}_M^\circ)^c$. The latter is precisely the set of {\em NLS detected slopes} in the sense of Boyer and Clay \cite[Definition 7.16]{BC}, according to Theorem \ref{thm:detect-new}. (In particular, this observation should be compared with \cite[Theorem 8.1]{BC}.) Thus the desired equivalence follows from Theorem \ref{thm:gluing}, on comparison with \cite[Theorem 1.7]{BC} restricted to rational homology solid tori.    
\end{proof}

\bibliographystyle{plain}
\bibliography{references/loops_bibliography}

\end{document}